%% file: main.tex
\documentclass[11pt, reqno]{amsart}

\usepackage{amssymb}
\usepackage{amsmath}
\usepackage{bm}
\usepackage{array}
\usepackage{diagbox}
\usepackage{graphicx}
\usepackage{pinlabel}
\usepackage{graphics}
 \usepackage[mathscr]{euscript} 
\usepackage{amscd}
\usepackage{enumerate}
\usepackage[all]{xy}
\usepackage{mathtools} 
\usepackage{color}
\usepackage[table]{xcolor}
\usepackage{xspace} 
\usepackage[margin=3cm, marginparwidth=2.5cm]{geometry}
\usepackage{hyperref}
\usepackage[normalem]{ulem}
\usepackage{cite}
\usepackage{tikz-cd}

\usepackage{tikz}
\usepackage{ifthen}
\usetikzlibrary{shapes.misc, arrows, arrows.meta, decorations.markings}
\tikzset{cross/.style={cross out, draw=black, minimum size=2*(#1-\pgflinewidth), inner sep=0pt, outer sep=0pt},
cross/.default={1pt}}

\usepackage{import}
\usepackage{xifthen}
\usepackage{pdfpages}
\usepackage{transparent}

\definecolor{InkscapePurple}{RGB}{128,0,128}

\newtheorem{thm}{Theorem}[section]
\newtheorem{theorem}[thm]{Theorem}

\newtheorem{corollary}[thm]{Corollary}

\newtheorem{lemma}[thm]{Lemma}
\newtheorem{prop}[thm]{Proposition}
\newtheorem{proposition}[thm]{Proposition}

\newtheorem{question}[thm]{Question}
\theoremstyle{definition}

\newtheorem{definition}[thm]{Definition}

\theoremstyle{remark}

\newtheorem{remark}[thm]{Remark}

\newtheorem{example}[thm]{Example}

\renewcommand{\theta}{\vartheta}

\newcommand{\set}[1]{\left\{#1\right\}}

\DeclareMathOperator{\lk}{lk}

\newcommand{\Z}{\mathbb{Z}}
\newcommand{\Q}{\mathbb{Q}}
\newcommand{\R}{\mathbb{R}}

\DeclareMathOperator{\CF}{CF}

\DeclareMathOperator{\HF}{HF}

\DeclareMathOperator{\Kh}{Kh}

\renewcommand{\v}{\mathbf{v}}

\newcommand{\x}{\mathbf{x}}
\newcommand{\y}{\mathbf{y}}

\renewcommand{\S}{\Sigma}

\newcommand{\gr}{\mathbf{gr}}

\DeclareMathOperator{\Spin}{Spin}
\newcommand{\spinc}{spin$^c$\xspace} 
\newcommand{\Spinc}{\Spin^c} 
\newcommand{\s}{\mathfrak{s}}

\DeclareMathOperator{\Wu}{Wu}

\newcommand{\Ainfty}{\mathcal{A}_\infty}
\newcommand{\Alg}{\mathcal{A}}
\newcommand{\HFhat}{\widehat{\mathit{HF}}}
\newcommand{\CFD}{\widehat{\mathit{CFD}}}
\newcommand{\CFA}{\widehat{\mathit{CFA}}}

\newcommand{\Deltadb}{\Delta_{sym}}

\DeclareMathOperator{\rot}{rot}

\newcommand{\twist}{\mathcal{T}}
\newcommand{\extend}{\mathcal{E}}
\newcommand{\merge}{\mathcal{M}}
\newcommand{\twistloop}{\textsc{t}} 
\newcommand{\extendloop}{\textsc{e}} 
\newcommand{\mergeloop}{\textsc{m}}

\newcommand{\bpt}{\bullet}
\newcommand{\drawimmersedcurve}{
\draw[red] (-2*\U+7*\s, -2*\U+0*\s) .. controls (-2*\U+8*\s, -2*\U+0*\s) and (-8*\s, -2*\U+0*\s) .. (-7*\s, -2*\U+0*\s);
\foreach \x in {1,-1} {
   \begin{scope}[xscale=\x, yscale=\x]
      \draw[red] (-7*\s, -2*\U+0*\s) .. controls (-7*\s+3*\s, -2*\U+0*\s) and (-1*\s, -2*\U+6*\s-3*\s) ..
      (-1*\s, -2*\U+6*\s) .. controls (-1*\s, -2*\U+7*\s) and (-1*\s, -7*\s) ..
      (-1*\s, -6*\s) .. controls (-1*\s, -6*\s+4*\s) and (7*\s-4*\s, 2*\s) ..
      (7*\s, 2*\s) .. controls (8*\s, 2*\s) and (2*\U-8*\s, 2*\s) ..
      (2*\U-7*\s, 2*\s) .. controls (2*\U-7*\s+2*\s, 2*\s) and (2*\U-3*\s, 6*\s-2*\s) ..
      (2*\U-3*\s, 6*\s) .. controls (2*\U-3*\s, 7*\s) and (2*\U-3*\s, 2*\U-7*\s) ..
      (2*\U-3*\s, 2*\U-6*\s) .. controls (2*\U-3*\s, 2*\U-6*\s+\s) and (2*\U-7*\s+2*\s, 2*\U-4*\s) ..
      (2*\U-7*\s, 2*\U-4*\s) .. controls (2*\U-8*\s, 2*\U-4*\s) and (8*\s, 2*\U-4*\s) ..
      (7*\s, 2*\U-4*\s) .. controls (6*\s, 2*\U-4*\s) and (5*\s, 2*\U-5*\s) ..
      (5*\s, 2*\U-6*\s) .. controls (5*\s, 2*\U-7*\s) and (5*\s, 7*\s) ..
      (5*\s, 6*\s) .. controls (5*\s, 4*\s) and (\s, \s) .. 
      (0,0);
   \end{scope}
   }
}

\newcommand{\drawimmersedcurvevar}{
\draw[red] (-\U+2*\S, -\U) -- (-2*\S, -\U);
\foreach \x in {1,-1} {
   \begin{scope}[xscale=\x, yscale=\x]
      \draw[red] (-2*\S, -\U) .. controls (0, -\U) and (\S, -\U) ..
      (\S, -\U+\S) .. controls (\S, -\U+2*\S) and (\S, -2*\S) ..
      (\S, -\S) .. controls (\S, 0) and (\S, 0) ..
      (2*\S, 0) .. controls (3*\S,0) and (\U-2*\S,0) ..
      (\U-\S, 0) .. controls (\U,0) and (\U,0) ..
      (\U, \S) .. controls (\U, 2*\S) and (\U, \U-2*\S) ..
      (\U, \U-2*\S) .. controls (\U, \U-\S) and (\U,\U-\S) ..
      (\U-\S, \U-\S) .. controls (\U-2*\S, \U-\S) and (2*\S, \U-\S) ..
      (\S, \U-\S) .. controls (0, \U-\S) and (0, \U-\S) ..
      (0, \U-2*\S) .. controls (0, \U-3*\S) and (0, \S) ..
      (0,0);
   \end{scope}
   }
}

\newcommand{\othercolor}{InkscapePurple}
\newcommand{\drawimmersedcurvesym}{
\foreach \x in {1,-1} {
   \begin{scope}[xscale=\x, yscale=\x]
      \draw[\othercolor, rounded corners=0.5*\u] (-0.5*\U, -\U) -- (-0.5*\U, -\U+\u) -- (\u,-\U+\u) -- (\u,-\U+2*\u);
      \draw[red] (\u,-\U+2*\u) -- (\u,-2*\u);
      \draw[\othercolor, rounded corners=0.5*\u] (\u, -2*\u) -- (\u, -1*\u) -- (0.5*\U-0.5*\u, -1*\u) -- (0.5*\U-0.5*\u, 2*\u) -- (\U-2*\u, 2*\u) -- (\U-2*\u, 3*\u);
      \draw[red] (\U-2*\u, 3*\u) -- (\U-2*\u, \U-3*\u);
      \draw[\othercolor, rounded corners=0.5*\u] (\U-2*\u, \U-3*\u) -- (\U-2*\u, \U-2*\u) -- (0, \U-2*\u) -- (0, \U-3*\u);
      \draw[red] (0, \U-3*\u) -- (0,\u);
      \draw[\othercolor] (0, \u) -- (0,0);
   \end{scope}
   }
}

\begin{document}

\title[Correction terms and symmetries of immersed curves]{Correction terms of double branched covers\\ and symmetries of immersed curves}

\author{Jonathan Hanselman}%
\address{Department of Mathematics, Princeton University \\ Princeton, NJ 08544}%
\email{\href{mailto:jh66@math.princeton.edu}{jh66@math.princeton.edu}}%

\author{Marco Marengon}%
\address{Alfr\'ed R\'enyi Institute for Mathematics \\ Budapest, Hungary 1053}%
\email{\href{mailto:marengon@renyi.hu}{marengon@renyi.hu}}%

\author{Biji Wong}%
\address{Department of Mathematics, Duke University \\ Durham, NC 27708}%
\email{\href{mailto:biji.wong@duke.edu}{biji.wong@duke.edu}}%


\begin{abstract}

We use the immersed curves description of bordered Floer homology to study $d$-invariants of double branched covers $\Sigma_2(L)$ of arborescent links $L \subset S^3$. We define a new invariant $\Delta_{sym}$ of bordered $\mathbb{Z}_2$-homology solid tori from an involution of the associated immersed curves and relate it to both the $d$-invariants and the Neumann-Siebenmann $\bar\mu$-invariants of certain fillings. We deduce that if $L$ is a 2-component arborescent link and $\Sigma_2(L)$ is an L-space, then the spin $d$-invariants of $\Sigma_2(L)$ are determined by the signatures of $L$. By a separate argument, we show that the same relationship holds when  $L$ is a 2-component link that admits a certain symmetry. 

\end{abstract}

\maketitle

\tableofcontents


\input{sections/intro.tex}

\input{sections/bordered.tex}
\input{sections/gradings.tex}

\input{sections/delta-d.tex}
\input{sections/mu-bar.tex}
\input{sections/main_proof.tex}

\input{sections/involutions.tex}


\bibliographystyle{alpha}
\bibliography{HFsignaturespaper}


\end{document}

%% file: sections/intro.tex

\section{Introduction}

In the last two decades new invariants of knots and links have been defined that share some similarities with the classical signature $\sigma$, such as $\tau$ \cite{OS:tau} and $s$ \cite{Rasmussen:s}. While these invariants agree (up to multiplication by a universal constant) with the signature for quasi-alternating knots \cite{ManoOzsvath}, they are in general different -- for example, $\tau$ and $s$ can be used to prove the local Thom conjecture, while $\sigma$ cannot.

In this paper we focus on $d$-invariants $\delta$ \cite{OSd} of double branched covers $\Sigma_2(L)$ of links $L \subset S^3$.
Like the invariants mentioned above, $\delta$ is most properly defined when $L$ is equipped with a \emph{quasi-orientation} $\omega$, i.e.\ an orientation of each component of $L$ up to an overall reversal. We have a correspondence, due to Turaev \cite{Turaev}, between the set $QO(L)$ of quasi-orientations $\omega$ on $L$ and the set $\textrm{Spin}(\Sigma_2(L))$ of spin structures $\mathfrak{s}_{\omega}$ on $\Sigma_2(L)$. With it one defines
\[
\delta(L, \omega) := d(\Sigma_2(L), \mathfrak{s}_{\omega}).
\]
When compared to the invariants $\tau$ and $s$, the invariant $\delta$ shows some peculiar behavior: $\delta$ cannot be used to prove the local Thom conjecture (since it differs from $\tau$ -- and also from $\sigma$ -- on a family of torus knots \cite[Section 4.2]{ManoOwens}) unlike $\tau$ and $s$, but does agree with $\sigma$ (up to multiplication by $-1/4$) for quasi-alternating links \cite{ManoOwens, DO, LiscaOwensQuasi} like $\tau$ and $s$. 

When $\Sigma_2(L)$ is an $L$-space, it is natural to expect to have greater control over the $d$-invariant $\delta$, because then $d(\Sigma_2(L), \mathfrak{s})$ is simply the Maslov grading of the unique non-trivial element in the corresponding Heegaard Floer homology of $\Sigma_2(L)$. In fact, Lin-Ruberman-Saveliev \cite[Theorem A]{LRS} showed that if $\Sigma_2(L)$ is an L-space, then
\begin{equation}\label{eq:LRS}
 \sum_{\mathfrak{s} \in \textrm{Spin}(\Sigma_2(L))} d(\Sigma_2(L), \mathfrak{s}) = - \frac{1}{4} \sum_{\omega \in QO(L)} \sigma(L, \omega).
\end{equation}

Our first main theorem is a refinement of Equation \eqref{eq:LRS} for 2-component arborescent links.
\begin{theorem}\label{thm:main}
If $L$ is a 2-component arborescent link with $\Sigma_2(L)$ an L-space, then
\begin{equation*}
\set{d(\Sigma_2(L), \mathfrak{s_1}), d(\Sigma_2(L), \mathfrak{s_2})} = \set{-\frac{1}{4}\sigma(L, \omega_1), -\frac{1}{4}\sigma(L, \omega_2)},
\end{equation*}
where $\mathfrak{s_1}$, $\mathfrak{s_2}$ are the two spin structures of $\Sigma_2(L)$ and $\omega_1$, $\omega_2$ are the two quasi-orientations of $L$.
\end{theorem}

Recall that a link is called \emph{arborescent} if it can be constructed from a weighted tree (or disjoint union of trees, called a forest) using the procedure explained in \cite[\textsection 2]{Siebenmann} and illustrated in Figure \ref{fig:tree+link}.
A key property of arborescent links is that their double branched covers are graph manifolds. The Heegaard Floer homology of graph manifolds is better understood than that of general 3-manifolds, and this enables us to relate the $d$-invariants of the double branched covers of arborescent links to the signatures of the links. We also make extensive use of a re-interpretation of the bordered Heegaard Floer homology of manifolds with torus boundary as immersed curves by the first author, Rasmussen, and Watson \cite{HRW, HRW-companion}. See Section \ref{sec:proof_immersedcurves} for an outline of the proof of Theorem \ref{thm:main}.

Along the way, we define a new invariant $\Delta_{sym}$ of bordered $\Z_2$-homology solid tori $M$ from the symmetries of the immersed multicurve for $M$. When a certain filling $Y$ of $M$ is an L-space, $\Delta_{sym}$ agrees with the difference of the $d$-invariants of $Y$ in the two spin structures on $Y$. For a precise statement, see Lemma \ref{lem:Delta-sym-is-Delta-d}. 
We remark that the definition of $\Delta_{sym}$ is geometric and uses the immersed curves description of $\CFD(M)$; it is less transparent how to define $\Delta_{sym}$ in the algebraic formulation of $\CFD(M)$.

While Theorem \ref{thm:main} is an improvement on Equation \eqref{eq:LRS} for 2-component arborescent links, it does not show the matching of quasi-orientation $\omega$ on $L$ with spin structure $\mathfrak{s}$ on $\Sigma_2(L)$ that gives $d(\Sigma_2(L), \mathfrak{s}) = -\frac{1}{4} \sigma(L, \omega)$. Thus, the following question is still open, even for arborescent links.

\begin{question}\label{Q:dsig}
If $\Sigma_2(L)$ is an L-space, is $d(\Sigma_2(L), \mathfrak{s}_\omega) = -\frac{1}{4} \sigma(L, \omega)$ for every quasi-orientation $\omega$ on $L$?
\end{question}

Our second theorem gives a positive answer to Question \ref{Q:dsig} when $L$ admits an orientation-preserving involution.

\begin{theorem}
\label{thm:d-invariantinvol}
Let $L \subset S^3$ be a 2-component link with $\det L \neq 0$, and let $\iota:S^3 \rightarrow S^3$ be an orientation-preserving involution such that
\begin{enumerate}
\item $\iota$ fixes $L$ set-wise and
\item the fixed point set of $\iota$ is a circle that intersects $L$ in two points (necessarily on the same component of $L$).
\end{enumerate}
Suppose $\Sigma_2(L)$ is an L-space, and let $\mathfrak{s}_1$ and $\mathfrak{s}_2$ denote the two spin structures on $\Sigma_2(L)$.
Then for $i \in \{1, 2\}$,
\[
d(\Sigma_2(L), \mathfrak{s}_i) = - \frac{1}{4} \sigma(L, \omega_{\mathfrak{s}_i}).
\]
\end{theorem}

We remark that Theorem \ref{thm:d-invariantinvol} is not specific to arborescent links, but works for general links. The core of the proof of Theorem \ref{thm:d-invariantinvol} is to show that Turaev's correspondence between the set of quasi-orientations on $L$ and the set of spin structures on $\Sigma_2(L)$ is natural, and that the involution swaps $\mathfrak{s}_1$ and $\mathfrak{s}_2$, which implies that the two values $d(\Sigma_2(L), \mathfrak{s}_1)$ and $d(\Sigma_2(L), \mathfrak{s}_2)$ agree. Coupled with Equation \eqref{eq:LRS}, we then get that this value agrees with the signature of $L$.

For the reader's convenience, we now give a detailed overview of the proof of Theorem \ref{thm:main}, which takes the largest part of the paper.

\subsection{Immersed curves, symmetries, and an outline of the proof of Theorem \ref{thm:main}}
\label{sec:proof_immersedcurves}
If $L$ is a 2-component link, then there are two spin structures $\mathfrak{s_1}$ and $\mathfrak{s_2}$ on $\Sigma_2(L)$. Work of Lin-Ruberman-Saveliev (Equation \eqref{eq:LRS}) shows that the sum of the $d$-invariants of  $\Sigma_2(L)$ at $\mathfrak{s_1}$ and $\mathfrak{s_2}$ agrees with the sum of the signatures of $L$ at the two quasi-orientations on $L$. If we prove that the difference is also the same, then the statement of Theorem \ref{thm:main} follows.

When $L$ is arborescent, $\Sigma_2(L)$ is a graph manifold and can be represented by a plumbing forest $\Gamma$. 
We use the notation $Y_\Gamma$ for the plumbed 3-manifold associated with $\Gamma$, so that $\Sigma_2(L) = Y_\Gamma$. 
Choosing a distinguished vertex $v$ of $\Gamma$ gives a rooted plumbing tree $(\Gamma, v)$, which determines a manifold with torus boundary $M_{\Gamma, v}$ by removing a neighborhood of a regular fiber of the $S^1$-bundle associated with the vertex $v$ in the construction of $Y_\Gamma$. This construction of $M_{\Gamma, v}$ determines a preferred parametrization of the boundary, making $M_{\Gamma, v}$ a bordered manifold. We will see that we can always choose $v$ so that $M_{\Gamma, v}$ is a $\Z_2$-homology solid torus. Since $Y_\Gamma$ is the Dehn filling of $M_{\Gamma, v}$ along a particular slope, we can compute the difference of the two spin $d$-invariants of $Y_\Gamma$ and the difference of the two signatures of $L$ using two invariants for the rooted plumbing tree $(\Gamma, v)$: $\Delta_{sym}$ which was introduced above, and $\Delta \bar\mu$ which is described below. Then an inductive argument on the size of $(\Gamma, v)$ shows that these differences always agree up to a factor of $-\frac 1 4$.

To compute the difference in the signatures of $L$, we relate the signatures to the $\bar\mu$ invariants of $\Sigma_2(L)$. The $\bar\mu$ invariant, introduced by Neumann \cite{Neumann-invt} and Siebenmann \cite{Siebenmann}, is an invariant of a closed graph manifold along with a spin structure. For arborescent links, the $\bar \mu$ invariants for $\Sigma_2(L)$ agree with the signatures for $L$ \cite[Theorem 5]{Saveliev}. Given a rooted plumbing tree $(\Gamma, v)$ for which $M_{\Gamma, v}$ is a $\Z_2$-homology solid torus, we define a related invariant $\Delta \bar\mu(\Gamma, v)$, which has the property that when the filling $Y_\Gamma$ has two spin structures $\Delta \bar\mu(\Gamma, v)$ agrees with the difference in the two $\bar \mu$ invariants of $Y_\Gamma$.

To compute the difference between the $d$-invariants in the spin structures on $\Sigma_2(L) = Y_\Gamma$ we use the fact that the relative $\mathbb{Q}$-grading on $\HFhat(Y_\Gamma)$ can be computed from the bordered Floer invariant for $M_{\Gamma, v}$, as well as the fact that when $Y_\Gamma$ is an L-space the $d$-invariants are simply the Maslov gradings of the unique generator in each of the two self-conjugate \spinc structures of $Y_\Gamma$. We give a geometric computation of this grading difference using the immersed curve representation of the bordered Floer invariant for $M_{\Gamma, v}$; this generalizes earlier work in  \cite[Theorem 5]{HRW-companion}. Recall that the bordered invariant for $M_{\Gamma, v}$ takes the form of an immersed multicurve in $\partial M_{\Gamma, v} \setminus \set{\star}$. When $M_{\Gamma, v}$ is a $\Z_2$-homology solid torus this invariant has a distinguished curve $\gamma$, and this curve is fixed by the elliptic involution on $\partial M_{\Gamma, v}$. Here we mean that the curve is fixed up to homotopy, but if the curve is placed in a standard position it will be fixed setwise by the involution. In this standard position, we can choose two qualitatively different fixed points of the symmetry, which we call $x_0$ and $x_1$. If we consider the lift $\tilde\gamma$ of $\gamma$ to the cover $\R^2 \setminus \Z^2$, then invariance under the elliptic involution takes the form of a rotational symmetry of $\pi$, and rotation about either a lift $\tilde x_0$ of $x_0$ or a lift $\tilde x_1$ of $x_1$ fixes the curve and takes lifts of $x_0$ to lifts of $x_0$ and lifts of $x_1$ to lifts of $x_1$.  We define $\Deltadb(\Gamma, v)$ to be twice the (signed) area bounded by the portion of $\tilde\gamma$ from $\tilde x_0$ and $\tilde x_1$ and the straight line joining these two points; see Figure \ref{fig:intro-example} for an example. In the case that $Y_\Gamma$ is an L-space with two spin structures, we will show that $\Deltadb(\Gamma, v)$ agrees with the difference between the $d$-invariants of $Y_\Gamma$ in the spin structures.

The core of our argument is an inductive proof that $\Deltadb(\Gamma, v) = -\frac 1 4 \Delta\bar\mu(\Gamma, v)$ for all suitable rooted plumbing trees that define $\Z_2$-homology solid tori. This relies on the fact that a rooted plumbing tree can be constructed from smaller rooted trees using three simple operations called twist, extend, and merge; we compute the effect of each of these operations on $\Deltadb$ and $\Delta\bar\mu$ and show that the effects are the same. Since $\Deltadb(\Gamma, v)$ is the difference between the  spin $d$-invariants when $Y_\Gamma$ is an L-space with two spin structures, Theorem \ref{thm:main} follows. We conclude by noting that the equivalence between $d$-invariants and $\bar\mu$ invariants was already known for definite plumbing graphs \cite{Stipsicz}; our argument does not use the definiteness of the plumbing.

\begin{remark}
It is worth noting that although we are most interested in rooted trees $(\Gamma, v)$ for which $Y_\Gamma$ has two spin structures and $Y_\Gamma$ is an L-space, neither property is necessary to show $\Deltadb(\Gamma, v) = -\frac 1 4 \Delta\bar\mu(\Gamma, v)$ and these properties may not hold at intermediate steps of the induction. The first property is needed to relate $\Deltadb(\Gamma, v)$ or $\Delta\bar\mu(\Gamma, v)$ to differences in Maslov gradings or $\bar\mu$ invariants of the filling $Y_\Gamma$, which does not make sense if the filling has only one spin structure. The second property is required to guarantee that the grading difference between generators associated with $x_0$ and $x_1$ agrees with the difference $\Delta d$ between the spin $d$-invariants for $Y_\Gamma$. This is not true in general (see Example \ref{ex:running-example-Delta-sym-not-grading-diff}), and understanding the relationship between $\Deltadb$ and $\Delta d$ may shed light on the extent to which Theorem \ref{thm:main} fails when $\Sigma(L)$ is not an L-space.
\end{remark}

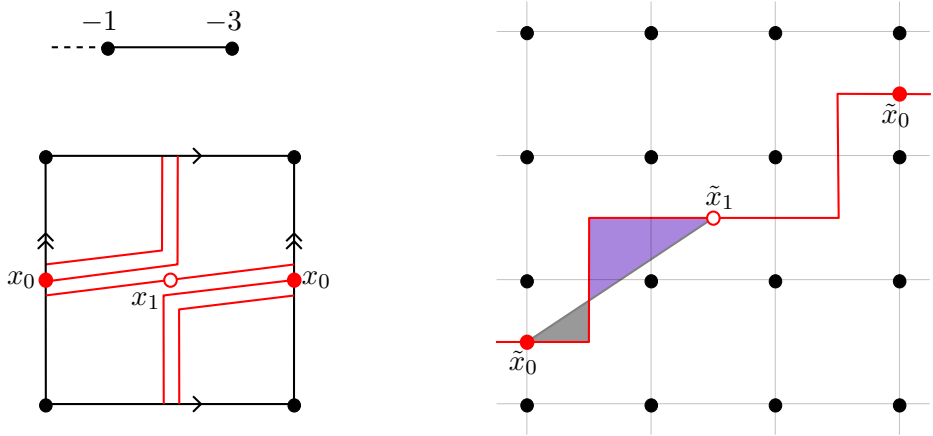
\begin{figure}
\labellist

\pinlabel {$-1$} at 60 400
\pinlabel {$-3$} at 180 400

\pinlabel {$x_0$} at -15 150
\pinlabel {$x_1$} at 105 130
\pinlabel {$x_0$} at 270 150

\pinlabel {$\tilde x_0$} at 470 70
\pinlabel {$\tilde x_1$} at 660 230
\pinlabel {$\tilde x_0$} at 830 310

\endlabellist
\includegraphics[scale = .4]{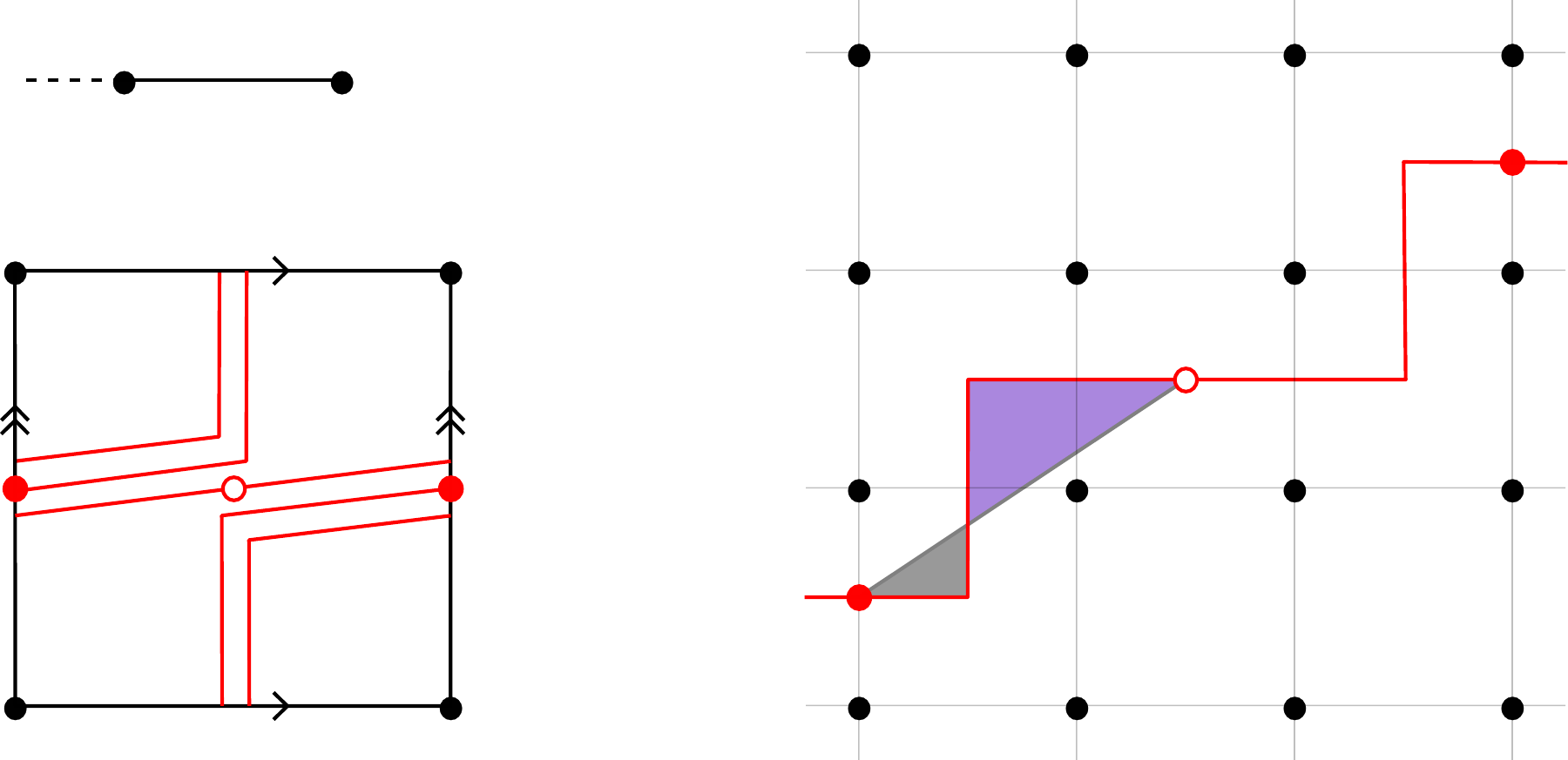}
\caption{The immersed curve bordered Floer invariant associated with the rooted plumbing tree in the top left is shown in the punctured torus (bottom left) and lifted to $\R^2\setminus \Z^2$ (right). Note that we assume the curve is rectilinear lying on the horizontal lines with half-integer $y$-coordinate and vertical lines with half-integer $x$-coordinate, as on the right, but the curve shown in the torus is perturbed slightly from this position for clarity.  The fixed points $x_0$ and $x_1$ of the elliptic involution, and their lifts, are indicated. Twice the area of the shaded region, with the purple region counting negatively, gives $\Deltadb = -\frac 1 2$.}
\label{fig:intro-example}
\end{figure}

\subsection{Organization}
In Section \ref{sec:bordered} we review the necessary background in bordered Floer homology, and also prove a new grading lemma generalizing an immersed curve computation of grading differences to generators in different \spinc structures in Section \ref{sec:gradings}. In Section \ref{sec:delta-sym} we define $\Deltadb$ and study its behavior under the three operations on plumbing trees, while the same is done for $\Delta \bar\mu$ in Section \ref{mubarsection}. Section \ref{sec:main-proof} puts the previous results together to prove Theorem \ref{thm:main}. Finally, in Section \ref{sec:involutions} we use a separate argument argument to prove Theorem \ref{thm:d-invariantinvol} and demonstrate it with an example.

\subsection{Acknowledgments}

The authors would like to thank Adam Levine, Gordana Mati\'c, Andr\'as N\'emethi, Arunima Ray, Nikolai Saveliev, Chris Scaduto, Andr\'{a}s Stipsicz, and Matt Stoffregen for helpful conversations and correspondence. This research was conducted at the following institutions: the Max Planck Institute for Mathematics, ICERM, the Alfr\'{e}d R\'{e}nyi Institute of Mathematics, CIRGET, Princeton University, Duke University, and NC State University. The authors wish to thank all of these places for their hospitality.
JH was supported by NSF grant DMS-2105501.
MM was partially supported by NKFIH grant OTKA K146401.
BW was supported by NSF grant DMS-2213027.

%% file: sections/bordered.tex

\section{Bordered Floer homology}\label{sec:bordered}

In order to prove Theorem \ref{thm:main}, we will need to compute the relative $d$-invariants associated with the spin structures for rational homology sphere graph manifolds that have two spin structures. This will be accomplished using bordered Floer homology, and we now review the relevant parts of this machinery.

\subsection{Bordered invariants for manifolds with torus boundary}
Bordered Floer homology is an extension of Heegaard Floer invariants to manifolds with boundary defined by Lipshitz, Ozsv{\'a}th, and Thurston in \cite{LOT}. For manifolds with parametrized torus boundary, the invariant traditionally takes the form of a homotopy equivalence class of $\Ainfty$-modules or type D structures over a particular algebra $\Alg_{T^2}$. It was shown in \cite{HRW} that this data is equivalent to a decorated collection of (homotopy classes of) immersed curves in the punctured torus. We will refer to a collection of immersed curves as a multicurve. The multicurves can carry two types of decorations---a local system on each curve and a grading decoration recording relative grading differences between curves---but neither decoration will be relevant in the present paper since for the manifolds we consider we will only be interested in a single component of the multicurve and this component is always equipped with the trivial local system. 

The immersed curves can be viewed as living in the boundary of the manifold (minus a fixed basepoint in the boundary). More precisely, given a manifold $M$ with torus boundary we define $T_M$ to be the complement of $\{0\}$ in $H_1(\partial M;\R)/H_1(\partial M;\Z)$ and note that $T_M \cong \partial M \setminus z$ where $z$ is a fixed basepoint in $\partial M$. The bordered Floer invariants are represented by a decorated collection of immersed closed curves in $T_M$, defined up to homotopy of curves, which we denote by $\HFhat(M)$. We remark that this invariant, unlike the $\Ainfty$-modules or type D structures defined in \cite{LOT}, depends only on $M$ and not on a choice of parametrization of $\partial M$. That said, if we fix a set of parametrizing curves ${\alpha, \beta}$ for $\partial M$, then we can represent the immersed curves more conveniently. In particular, a parametrization of $\partial M$ allows us to identify $T_M$ with $\left( \R^2 / \Z^2 \right) \setminus \{0\}$ and thus to draw pictures of the immersed curves in a consistent way; our convention is that this identification takes $\alpha$ to the positive vertical direction and $\beta$ to the positive horizontal direction.

With a choice of parametrization, we can also record an immersed curve as a word in the letters $\{\alpha^{\pm 1}, \beta^{\pm 1} \}$, defined up to cyclic permutation. For example, the oriented immersed curve in the punctured torus $T_M$ shown on the left of Figure \ref{fig:immersed-curve-example} can be represented by the word $\beta \alpha \beta \alpha \beta^{-1} \alpha^{-2} \beta^{-1} \alpha \beta \alpha$ by starting at the indicated dot and following the curve in the direction of its orientation. We can think of an immersed curve as (a smoothing of) a concatenation of length one horizontal and vertical segments; the curve in the figure has been drawn in a way to make this correspondence clear. Then we read off an $\alpha$ (resp. $\alpha^{-1}$) for each vertical segment traversed moving upwards (resp.\ downwards) and a $\beta$ (resp. $\beta^{-1}$) for each horizontal segment traversed moving rightwards (resp.\ leftwards); equivalently, the word comes from recording the sequence of signed intersections of the curve with the edges of the square, where intersections with the top/bottom of the square contribute $\alpha$ or $\alpha^{-1}$ depending on the sign of the intersection, and intersections with the left/right sides of the square contribute $\beta$ or $\beta^{-1}$. Because the starting point on the immersed curve was arbitrary, this curve could just as well be encoded by the word  $\alpha \beta \alpha \beta^{-1} \alpha^{-2} \beta^{-1} \alpha \beta \alpha \beta$, or any other cyclic permutation of this word. We will refer to an equivalence class of words up to cyclic permutation as a \emph{cyclic word} and denote it by enclosing the word in parentheses, so this oriented immersed curve is represented by the cyclic word $(\beta \alpha \beta \alpha \beta^{-1} \alpha^{-2} \beta^{-1} \alpha \beta \alpha)$.  We let $\HFhat(M, \alpha, \beta)$ denote the collection of cyclic words representing the collection of immersed curves $\HFhat(M)$ in this way; note that $\HFhat(M, \alpha, \beta)$ does not record the local system or grading decorations on $\HFhat(M)$ and thus may lose some information, but we will be ignoring these decorations. By abuse of notation we will sometimes refer to the immersed curves $\HFhat(M)$ and the cyclic words $\HFhat(M, \alpha, \beta)$ interchangeably when the parametrization is clear.

\begin{remark}
In this paper we will ignore the orientation on the immersed curves $\HFhat(M)$ (the orientations encode certain relative grading information we will not need). Note that reversing the orientation on an immersed curve has the effect of replacing the corresponding (cyclic) word by its inverse; for example, for the curve in Figure \ref{fig:immersed-curve-example} starting at the dot but following the curve in the opposite direction yields the word $\alpha^{-1} \beta^{-1} \alpha^{-1} \beta \alpha^2 \beta \alpha^{-1} \beta^{-1} \alpha^{-1} \beta^{-1}$. Thus for our purposes the relevant invariants are collections of cyclic words up to taking inverses.
\end{remark}

\input{figures/1.tex}

While $\HFhat(M)$ is defined in the punctured torus $T_M$, it is often convenient to work in certain covering spaces of $T_M$. Let $\widetilde T_M$ denote the covering space $H_1(\partial M;\R) \setminus H_1(\partial M; \Z)$, which given a parametrization $\{\alpha, \beta\}$ of $\partial M$ we identify with $\R^2 \setminus \Z^2$ using the usual convention that $\alpha$ is in the positive vertical direction and $\beta$ is in the positive horizontal direction. We then consider the intermediate covering space $\overline T_M = \widetilde T_M / \langle \lambda \rangle$ where $\lambda$ is the homological longitude (that is, a generator of the kernel of the inclusion $i_*: H_1(\partial M; \Z) \to H_1(M; \Z)$). We remark that if $b_1(M) = 1$ then the multicurve $\HFhat(M)$ restricted to any \spinc structure of $M$ represents the homology class $\lambda$ \cite[Corollary 6.6]{HRW}, so $\lambda$ can be easily determined from the immersed curve invariant $\HFhat(M)$.

Let $p$ denote the covering map $\overline{T}_M \to T_M$. The covering space $\overline{T}_M$ allows us to encode further information about \spinc structures; in particular, given a \spinc structure $\s$ in $\Spinc(M)$ bordered Floer homology gives a collection of decorated immersed closed curves $\HFhat(M;\s)$ in $\overline T_M$ such that
$$\HFhat(M) = \bigcup_{\s \in \Spinc(M)} p(\HFhat(M; \s)).$$
The multicurve $\HFhat(M; \s)$ is defined only up to an overall shift by a deck transformation of the covering map $p$; it follows that $\HFhat(M; \s)$ is determined by $ p (\HFhat(M; \s))$ when the latter consists of a single immersed curve 
, though in general the choice of lift from $T_M$ to $\overline T_M$ contains additional information. Figure \ref{fig:immersed-curve-example} shows an immersed curve in $T_M$ along with lifts to both $\overline T_M$ and $\widetilde T_M$\footnote{Strictly speaking, the multicurve we consider in $\widetilde T_M$ is not a lift of the multicurve $\HFhat(M)$ in $T_M$ in the case that $\HFhat(M)$ contains homologically trivial components, though we will call it a lift by abuse of terminology. More precisely, we mean a preimage of the lift of $\HFhat(M)$ to $\overline T_M$ under the covering map $\widetilde T_M \to \overline T_M$. The preimage of a homologically essential curve in $\overline T_M$ is a single non-compact periodic curve in $\widetilde T_M$, but the preimage of a homologically inessential curve in $\overline T_M$ is infinitely many copies of a closed curve. The resulting multicurve in $\widetilde T_M$ is periodic, invariant under translation by $\lambda$.}.

Bordered Floer invariants satisfy a symmetry under conjugation of spin$^c$ structures. Let $r: T_M \to T_M$ be the elliptic involution, which takes $\alpha$ to $-\alpha$ and $\beta$ to $-\beta$, and let $\bar r$ be a lift of this map to $\overline T_M$. It was shown in \cite{HRW-companion} that conjugating a \spinc structure has the effect of reparameterizing the boundary by the elliptic involution.

\begin{proposition}[Theorem 37 of \cite{HRW-companion}]
For any manifold with torus boundary $M$ and any \spinc structure $\s$ we have
$$\HFhat(M; \bar \s) \simeq \bar r \left( \HFhat(M; \s) \right)$$
where $\bar \s$ denotes the \spinc structure conjugate to $\s$.
\end{proposition}
\noindent In particular, if $\s$ is a self-conjugate \spinc structure on $M$ then $\HFhat(M; \s)$ is fixed by $\bar r$. 

\subsection{Alternative notations}

In the course of some computations, we will also use a shorthand notation for cyclic words that appears in \cite{HW:loop-calculus}; we will refer to this as \emph{loop calculus notation}. Roughly speaking, loop calculus notation comes from cutting a reduced cyclic word into subwords along instances of $\beta$ or $\beta^{-1}$, including each $\beta$ with the subword following it and each $\beta^{-1}$ with the subword preceding it. The resulting subwords take the form of either $\beta \alpha^k \beta^{-1}$, $\alpha^k$, $\beta \alpha^k$, or $ \alpha^k \beta^{-1}$, where $k$ is any nonzero integer in the first two cases and any integer in the second two cases. These subwords are recorded by the letters $a_k, b_k, c_k,$ and $\overline{c_{-k}}$, respectively, which we will call \emph{loop calculus letters} to avoid confusion with the letters ($\alpha$ or $\beta$) in the original cyclic word. When the subscript is not relevant, we refer to the loop calculus letters as being of type $a$, $b$, $c$, or $\bar c$. Note that the subwords for letters of type $c$ and $\bar c$ are related; in general we use bars to indicate reading the corresponding subword backwards and inverting each letter. Note that $\overline{a_k} = a_{-k}$ and $\overline{b_k} = b_{-k}$, so we do not need to consider $\bar a$ and $\bar b$ as separate letter types, but there is no $\ell$ for which $\overline{c_k} = c_\ell$.  We can now write a cyclic word in the letters $\alpha$ and $\beta$ in a more compact form by writing it as a cyclic word of loop calculus letters; note that a subword that ends (resp. does not end) with $\beta^{-1}$ must be followed by a subword that does not start (resp. starts) with $\beta$, which imposes constraints on the sequence of loop calculus letters---for example, a type $\bar c$ letter can not immediately follow a type $c$ letter. As an example, the cyclic word $(\beta \alpha \beta \alpha \beta^{-1} \alpha^{-2} \beta^{-1} \alpha \beta \alpha)$ representing the immersed curve in Figure \ref{fig:immersed-curve-example} can be written in loop calculus notation as the cyclic word $(c_{1} a_1 \overline{c_2} b_1 c_1)$. 

In terms of immersed curves, loop calculus notation corresponds to cutting an immersed curve (viewed as a smoothing of a sequence of horizontal and vertical segments) near the left end of each horizontal segment. Each of the resulting curve segments is represented by a loop calculus letter, where the types $a$, $b$, $c$ or $\bar{c}$ encode the direction the curve travels at the two ends of a segment and the subscript encodes the vertical movement of the curve along that segment; this correspondence is summarized in Figure \ref{fig:loop-calculus-curve-segments}. As an example, the reader is invited to try reading off the loop calculus cyclic word $(c_{1} a_1 \overline{c_2} b_1 c_1)$ directly from the immersed curve in Figure \ref{fig:immersed-curve-example}. Loop calculus notation was a precursor to the immersed curve formulation of bordered Floer invariants, but it is still convenient in some situations; in particular, a simple algorithm for computing bordered Floer invariants of many graph manifold rational homology spheres is described using this notation in \cite[Section 6]{HW:loop-calculus}. 
\input{figures/2.tex}

For readers more familiar with the notation of \cite{LOT}, we remark that it is easy to recover the type D structure $\CFD(M, \alpha, \beta)$ from the immersed curve representatives for $(M, \alpha, \beta)$. If an immersed curve is expressed as a cyclic word in $\alpha$ and $\beta$, each $\alpha$ or $\alpha^{-1}$ gives a generator of $\CFD$ with idempotent $\iota_1$, each $\beta$ or $\beta^{-1}$ gives a generator of $\CFD$ with idempotent $\iota_0$, and consecutive letters in a cyclic word are connected by an arrow (representing a term in the differential) as shown in Table \ref{table:curves-to-CFD}. For example, the immersed curve in Figure \ref{fig:immersed-curve-example} corresponds to the type D structure shown below:

\begin{center}
\begin{tikzpicture}[scale = 1.5]
\node(a) at (0,0) {$\bullet$};
\node(b) at (1,0) {$\circ$};
\node(c) at (2,0) {$\bullet$};
\node(d) at (3,0) {$\circ$};
\node(e) at (4,0) {$\bullet$};
\node(f) at (5,0) {$\circ$};
\node(g) at (6,0) {$\circ$};
\node(h) at (7,0) {$\bullet$};
\node(i) at (8,0) {$\circ$};
\node(j) at (9,0) {$\bullet$};
\node(k) at (10,0) {$\circ$};

\draw[->] (a) to node[above] {$\rho_3$} (b);
\draw[<-] (b) to node[above] {$\rho_{1}$} (c);
\draw[->] (c) to node[above] {$\rho_3$} (d);
\draw[->] (d) to node[above] {$\rho_2$} (e);
\draw[->] (e) to node[above] {$\rho_{1}$} (f);
\draw[<-] (f) to node[above] {$\rho_{23}$} (g);
\draw[<-] (g) to node[above] {$\rho_3$} (h);
\draw[->] (h) to node[above] {$\rho_{123}$} (i);
\draw[<-] (i) to node[above] {$\rho_1$} (j);
\draw[->] (j) to node[above] {$\rho_3$} (k);
\draw[<-, bend left = 20] (k) to node[above] {$\rho_{1}$} (a);
\end{tikzpicture}
\end{center}

\noindent Conversely, a type D structure $\CFD$ over the torus algebra can immediately be represented by a collection of cyclic words in $\{\alpha^{\pm 1}, \beta^{\pm 1} \}$, and thus by a collection of immersed curves, using Table \ref{table:curves-to-CFD}, provided that the basis of $\CFD$ is such that every generator is attached to exactly two arrows. The main content of \cite{HRW} is that it is always possible to choose such a basis for a type D structure over the torus algebra, or to choose a basis that nearly has this property in some precise sense (this latter case gives rise to immersed curves decorated by non-trivial local systems).

\begin{table}
\begin{tabular}{r|cccccc}
\hline
Pair of letters & $\beta^{-1} \alpha^{-1}$ & $\beta^{-1} \beta^{-1}$ & $\beta^{-1} \alpha$ & $\alpha \beta^{-1}$ & $\alpha \alpha$ & $\beta \alpha$ \\
Arrow & $\bullet \overset{\rho_{1}}{\longrightarrow} \circ$ & $\bullet \overset{\rho_{12}}{\longrightarrow} \bullet$ & $\bullet \overset{\rho_{123}}{\longrightarrow} \circ$ & $\circ \overset{\rho_{2}}{\longrightarrow} \bullet$& $\circ \overset{\rho_{23}}{\longrightarrow} \circ$ & $\bullet \overset{\rho_{3}}{\longrightarrow} \circ$ \\

\hline
Pair of letters & $\alpha\beta$ & $\beta\beta$ & $\alpha^{-1} \beta$ & $\beta \alpha^{-1}$ & $\alpha^{-1} \alpha^{-1}$ & $\alpha^{-1} \beta^{-1}$ \\
Arrow & $\circ \overset{\rho_{1}}{\longleftarrow} \bullet$ & $\bullet \overset{\rho_{12}}{\longleftarrow} \bullet$ & $\circ \overset{\rho_{123}}{\longleftarrow} \bullet$ & $\bullet \overset{\rho_{2}}{\longleftarrow} \circ$& $\circ \overset{\rho_{23}}{\longleftarrow} \circ$ & $\circ \overset{\rho_{3}}{\longleftarrow} \bullet$ \\

\hline
\end{tabular}
\caption{Arrows in $\CFD(M, \alpha, \beta)$ corresponding to pairs of consecutive letters in $\HFhat(M, \alpha, \beta)$.}
\label{table:curves-to-CFD}
\end{table}

\subsection{Pairing}

An important feature of bordered Floer homology is a pairing theorem that recovers $\HFhat(M_1 \cup M_2)$ given the bordered invariants of two manifolds with torus boundary $M_1$ and $M_2$. In the language of immersed curves, the pairing theorem is stated in terms of the intersection Floer homology of the two immersed multicurves in the punctured torus. More precisely, let $\phi: \partial M_1 \to \partial M_2$ be an orientation-reversing gluing map by which the two manifolds are identified. For $i \in \{1,2\}$, we choose basepoints $z_i \in \partial M_i$ such that $\phi(z_1) = z_2$ and let $\HFhat(M_i)$ denote the bordered invariant of $M_i$, which is a collection of (decorated) immersed curves in $\partial M_i \setminus z_i$. We now have that $\phi(\HFhat(M_1))$ and $\HFhat(M_2)$ are both decorated multcurves in the punctured torus $\partial M_2 \setminus z_2$. The pairing theorem asserts that
\begin{equation}\label{eq:pairing}
\HFhat(M_1 \cup_\phi M_2) \cong HF( \HFhat(M_2), \phi(\HFhat(M_1)) ),\
\end{equation}
where $HF$ on the right denotes intersection Floer homology in the punctured torus. Since intersection Floer homology is invariant under homotopy of the immersed curves, we can usually arrange that the curves intersect minimally; in this case the Floer complex has no differentials and we simply have that $\HFhat(M_1 \cup_\phi M_2)$ is a vector space generated by the intersections of $\HFhat(M_2)$ and $\phi(\HFhat(M_1))$. The one exception to this is when a curve in $\HFhat(M_2)$ is homotopic to a curve in $\phi(\HFhat(M_1))$; in this case admissibility considerations require a non-minimal intersection.  We remark, however, that this exception is not relevant in this paper since a homologically nontrivial curve in $\HFhat(M_2)$ is never homotopic to a curve in $\phi(\HFhat(M_1))$ if $M_1 \cup_\phi M_2$ is a rational homology sphere.

A refined version of the bordered  pairing theorem recovers the \spinc decomposition on a closed three manifold $Y = M_1 \cup_\phi M_2$, as described in \cite[Section 6.3]{HRW}.  For this, we work in the covering space $\overline T_Y = \widetilde T_{M_2} / \langle \lambda_2, \phi_*(\lambda_1) \rangle$, where $\lambda_i$ is the homological longitude of $M_i$, which is covered by both $\overline T_{M_2}$ and $\overline T_{M_1}$. We let $ \overline p_i$ denote the projection from $\overline T_{M_i}$ to $\overline T_Y$. Projection gives a surjective map
$$\pi: \Spinc(Y) \to \Spinc(M_1) \times \Spinc(M_2),$$
and for $\s_i \in \Spinc(M_i)$ the set $\pi^{-1}(\s_1 \times \s_2)$ is a torsor over $H_Y = H_1(\partial M_2)/ \langle \lambda_2, \phi_*(\lambda_1) \rangle$. The refined pairing theorem says that for each $\s_i \in \Spinc(M_i)$ and each $\s \in \pi^{-1}(\s_1 \times \s_2)$, there is an element $\alpha_{\s}$ of $H_Y$ such that $\HFhat(Y; \s)$ is given by the intersection Floer homology of $ \overline p_2(\HFhat(M_2; \s_2))$ and $ \alpha_{\s} \cdot \overline p_1( \HFhat(M_1; \s_1 ))$ in $\overline T_Y$. In other words, we consider the intersection Floer homology of a fixed lift of $\HFhat(M_2)$ from $T_{M_2}$ to $\overline T_Y$ and $m$ different lifts of $\phi(\HFhat(M_1))$ from $T_{M_2}$ to $\overline T_Y$, where $m = |H_Y|$, so that every intersection point in $T_{M_2}$ between $\HFhat(M_2)$ and $\phi(\HFhat(M_1))$ lifts to an intersection point in $\overline T_Y$. Then two intersection points $x$ and $y$ in $\overline T_Y$ give rise to generators in the same \spinc structure of $Y$ if and only if $x$ and $y$ both lie on the same lift of $\phi(\HFhat(M_1))$. 

In fact, we do not require the full pairing theorem in this paper. The only gluing that will be needed is the special case when $M_1$ is the solid torus $D^2 \times S^1$ and the gluing map takes the meridian $\partial D^2 \times \{pt\}$ to the parametrizing curve $\beta$ in $M_2$.  In this case $M_1 \cup_\phi M_2$ is the Dehn filling of $M_2$ along $\beta$, which we will also denote $M_2(\beta)$. Recall that our convention is to represent $\partial M_2$ as a square with opposite sides identified such that $\beta$ is horizontal and $\alpha$ is vertical; in this special case of filling $M_2$ along $\beta$, $\dim\HFhat(M_1 \cup_\phi M_2)$ counts the minimal intersection of $\HFhat(M_2)$ with a horizontal curve homotopic to $\beta$. An example of such a pairing, where $\HFhat(M_2)$ is the curve from Figure \ref{fig:immersed-curve-example}, is shown in Figure \ref{fig:pairing-example}.

\input{figures/3.tex}

To recover the \spinc decomposition of $\HFhat(M_1 \cup M_2)$ in the special case that $\phi(\HFhat(M_1))$ is homotopic to $\beta$ in $T_{M_2}$, we consider the covering space $\overline T_Y = \overline T_{M_2} / \langle \phi_*( \lambda_1) \rangle$. Since $\lambda_1$ is given by the homology class of $\HFhat(M_1)$, we have $\overline T_Y = \overline T_{M_2} / \langle \beta \rangle$. Suppose $\lambda_2 = m[\alpha] + n[\beta]$; we will assume that $m \neq 0$, since otherwise $Y$ is not a rational homology sphere. If we identify $\widetilde T_{M_2}$ with $\R^2 \setminus \Z^2$ in the usual way, we have that $\overline T_Y$ is $\R^2 / \langle (1,0), (0, m)\rangle$ with the integer lattice points removed; in other words $\overline T_Y$ may be viewed as an $m \times 1$ rectangle (with $m$ punctures) with opposite sides identified. There are $m$ distinct lifts in $\overline T_Y$ of the horizontal line $\beta_A$ representing $\beta$ in $T_{M_2}$ (see Figure \ref{fig:pairing-example}), and for any $\s \in \Spinc(M_2)$ pairing any one of these lifts with $\HFhat(M_2; \s)$ gives $\HFhat(Y; \s')$ in one of the $m$ \spinc structures $\s'$ of $Y$ that restrict to $\s$ on $M_2$.
It is often convenient to lift further to $\widetilde T_{M_2}$; note that each lift of $\beta_A$ to $\overline T_Y$ lifts to infinitely many horizontal lines in $\widetilde T_{M_2}$, but
since the lifts of the curves $\HFhat(M_2)$ to $\widetilde T_{M_2}$ are periodic
we can pick any one of these lifts and use it to compute $\HFhat(Y)$ in the given \spinc structure. Continuing the example of the curve from Figures \ref{fig:immersed-curve-example}, Figure \ref{fig:pairing-example} also shows the pairing in $\overline T_Y$ and $\widetilde T_{M_2}$. There are two \spinc structures on $Y$, and the rank of $\HFhat(Y)$ is 1 in one \spinc structure and 3 in the other \spinc structure.

\subsection{Preferred forms for immersed curves}\label{sec:preferred-forms}

In the special case that $M_1 \cup M_2$ is Dehn filling $M_2$ along $\beta$, it is not difficult to extract $\dim \HFhat(M_1 \cup_\phi M_2)$ directly from the cyclic words representing the immersed curves in $\HFhat(M_2)$. To explain this, we will now discuss a few special representatives of the homotopy classes of immersed curves in $\HFhat(M_2)$. By assumption $\phi( \HFhat(M_1) )$ is homotopic to $\beta$ in $T_{M_2} \simeq (\R^2 / \Z^2) \setminus (0,0)$. We will consider two specific representatives of $[\beta]$ in $T_{M_2}$: we let $\beta_{sym}$ denote the horizontal curve $[0,1] \times \left\{\frac 1 2\right\}$ and, fixing a small $\epsilon > 0$, we let $\beta_A$ denote the curve $[0,1] \times \left\{\frac 1 2 - \epsilon\right\}$. To represent the collection of homotopy classes of immersed curves in $\HFhat(M_2)$ as explicit curves in $T_{M_2}$, we first construct them from the collection of cyclic words $\HFhat(M_2, \alpha, \beta)$ by concatenating length one horizontal and vertical segments, with each segment lying on $[0,1] \times \left\{ \frac 1 2 \right\}$ or $\left\{ \frac 1 2 \right\} \times [0,1]$, and smoothing the corners is a standard way. We will say an immersed curve is in \emph{rectilinear position} if it has this form. Note that curves in rectilinear position do not have transverse self-intersection, making them difficult to draw, so for convenience we often perturb the curve to achieve transverse self-intersection, allowing the horizontal and vertical segments to lie anywhere in an $\epsilon$ neighborhood of the original segments; we say the resulting curve is in \emph{transverse rectilinear position}. For example, the curves in Figure \ref{fig:pairing-example} are in transverse rectilinear position. 

If $\HFhat(M_2)$ is in (transverse) rectilinear position and we use $\beta_A$ to represent the image $\phi( \HFhat(M_1) )$, as in Figure \ref{fig:pairing-example}, then there is clearly one intersection point for each vertical segment in $\HFhat(M_2)$ or equivalently for each copy of $\alpha$ or $\alpha^{-1}$ in $\HFhat(M_2, \alpha, \beta)$. In fact, following the discussion in \cite{HRW}, when the curves are in this position there is an isomorphism at the chain level between the Lagrangian Floer chain complex and the box tensor product of the type A structure represented by $\beta_A$ with the type D structure represented by $\HFhat(M_2)$. More precisely, the pairing position described in \cite{HRW} can be obtained from rectilinear position by pushing horizontal and vertical segments of the multicurve curve representing the type D structure upward and rightward so that the endpoints of segments lie in the top right quadrant of the square and by pushing the curve representing the type A structure downward toward the lower left quadrant, though it is clear that only the latter translation suffices\footnote{This may seem to be the mirror image of the convention described in \cite{HRW}, but this is because we choose to work in $\partial M_2$ while \cite{HRW} describes the pairing in $\partial M_1$.}. Recall that each length one horizontal or vertical segment of the multicurve $\HFhat(M_2)$ corresponds to a generator of $\CFD(M_2,\alpha, \beta)$ with idempotent $\iota_0$ or $\iota_1$, respectively, and the curve $\beta_A$ corresponds to a type A structure with a single generator in the idempotent $\iota_1$; thus the fact that Floer homology picks out the vertical segments of $\HFhat(M_2)$ corresponds to the fact that the unique generator of the type A structure pairs with each $\iota_1$-generator of the type D structure.

It is useful that the pairing position described above gives a direct connection to the box tensor product of type A and type D modules, but this position has some disadvantages. The first drawback is that the curves are not in minimal position, so intersection points do not correspond directly to generators of Floer homology. The second drawback is that these diagrams are not symmetric; we will see that the curves involved are symmetric up to homotopy under the elliptic involution, but the representatives we have chosen of the homotopy class are not themselves fixed by that involution. Our arguments rely on the elliptic involution symmetry of $\HFhat(M_2)$, so we prefer to use representatives of both $\beta$ and $\HFhat(M_2)$ that respect this symmetry. We will address both these issues by using the representative $\beta_{sym}$ of $[\beta]$, which is fixed by the involution, and perturbing our curves representing $\HFhat(M_2)$ further into a form that we call \emph{symmetric $\beta$-pairing position}.

Before defining symmetric $\beta$-pairing position, we observe that any cyclic word $w$ can be written in the form
$$(\alpha^{s_1} \beta^{k_1} \alpha^{s_2} \beta^{k_2} \cdots \alpha^{s_\ell} \beta^{k_\ell}),  $$
where each $s_i \in \{\pm 1\}$ and each $k_i$ can be any integer. Note that the powers $k_i$ may be zero, though since we only allow reduced words we do not allow $k_i$ to be zero if the exponents on $\alpha$ before and after $\beta^{k_i}$ have opposite sign. Here $\ell$ is the total number of instances of $\alpha$ or $\alpha^{-1}$ in $w$; we may assume that $\ell > 0$, since if a cyclic word in $\HFhat(M_2)$ consists only of a power of $\beta$ then $M_2(\beta)$ is not a rational homology sphere. 
Given an immersed curve in rectilinear position, writing the corresponding cyclic word in the above form corresponds to viewing the decomposition into horizontal and vertical segments differently. Each instance of $\alpha^{s_i}$ represents a length one vertical segment as usual, but each instance of $\beta^{k_i}$ represents the concatenation of all the horizontal segments between two successive vertical segments. We will refer to this concatenation of length one horizontal segments as a \emph{maximal horizontal segment} of the immersed curve, and when necessary to avoid confusion we will refer to the length one segments discussed previously as \emph{unit segments}. Thus writing a cyclic word in the form above corresponds to viewing the immersed curve as alternating between unit vertical segments and maximal horizontal segments, noting that maximal horizontal segments may have length zero.

To put the curve in symmetric $\beta$-pairing position, we now move the top endpoint of each unit vertical segment down by $\epsilon/2$ and the bottom endpoint of each vertical segment up by $\epsilon/2$, and we slide each half of each maximal horizontal segment upward or downward by $\epsilon/2$ according to how the endpoint of the adjacent vertical segment moves. If the two halves of a horizontal segment move in opposite directions, we connect them by a short vertical piece of curve of length $\epsilon$. We remark that any maximal horizontal segment of length zero becomes, counterintuitively, just a vertical piece of length $\epsilon$ but we continue to refer to this as a (length zero) maximal horizontal segment. The resulting curve, after smoothing corners in a standard way, is said to be in symmetric $\beta$-pairing position. It is clear that when $\HFhat(M_2)$ is in symmetric $\beta$-pairing position it intersects $\beta_{sym}$ minimally and transversally, and that both curves are symmetric under the elliptic involution. Figure \ref{fig:symmetric-pairing-position} shows the curve from Figures \ref{fig:immersed-curve-example} and \ref{fig:pairing-example} in symmetric $\beta$-pairing position and paired with (lifts of) $\beta_{sym}$ in $\overline{T}_Y$ and $\widetilde{T}_{M_2}$. We remark that it is not in general possible to perturb these curves to achieve transverse self-intersection while preserving symmetry; in figures we perturb the curves slightly for clarity, but they should be understood to be in (likely non-transverse) symmetric position. 

\input{figures/4.tex}

When $\HFhat(M_2)$ is in symmetric $\beta$-pairing position, intersections with $\beta_{sym}$ are in one-to-one correspondence with the length $\epsilon$ vertical pieces that are inserted at the midpoint of some maximal horizontal segments. These in turn come from maximal horizontal segments for which the surrounding vertical segments are both oriented upward or both oriented downward, since this causes the two ends of the horizontal segment to be perturbed in opposite directions. It follows that, in terms of cyclic words, the rank of $\HFhat(M_1 \cup_\phi M_2)$ is obtained by counting instances of $\alpha \beta^k \alpha$ and $\alpha^{-1} \beta^k \alpha^{-1}$ in $\HFhat(M_2,\alpha, \beta)$  for all values of $k$.

\subsection{Computing Heegaard Floer invariants of graph manifolds}\label{sec:graph-mfds}


\begin{figure}[ht]
\fontsize{9pt}{13pt}
\scalebox{1}{%
    \def\svgwidth{\columnwidth}
    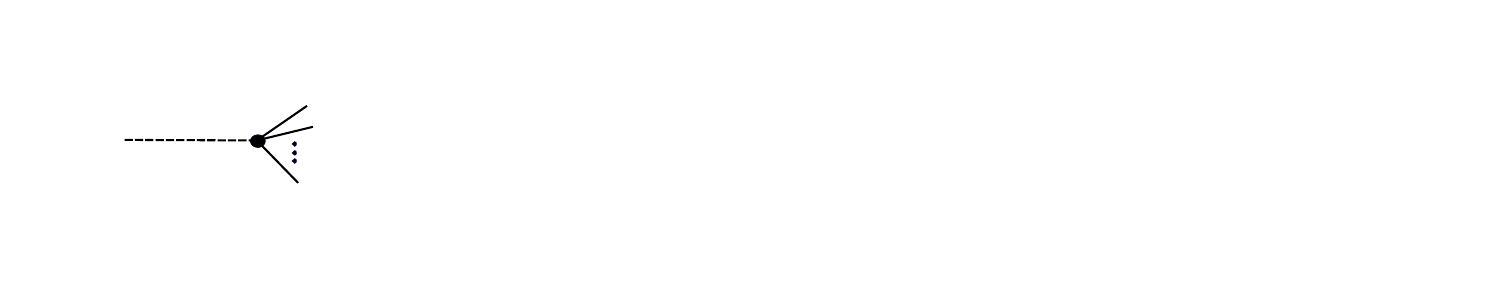
}
\caption{The three elementary plumbing tree moves. Here $\Gamma_1$ and $\Gamma_2$ are arbitrary rooted plumbing trees whose roots have weights $n_1$ and $n_2$. The gray circles labeled by $\Gamma_1$ and $\Gamma_2$ represent (rooted) subtrees of a tree that agree with $\Gamma_1$ or $\Gamma_2$ except that the weight on the distinguished vertex may be different as indicated.}
\label{fig:plumbing-tree-moves}
\end{figure}

We are interested in computing $\HFhat$ of the graph manifold associated with a plumbing tree $\Gamma$, which we will denote $Y_\Gamma$. We will do this using bordered Floer homology, following the algorithm first developed in \cite{Hanselman:graph-mfds} and later simplified in loop calculus notation in \cite[Section 6]{HW:loop-calculus}. 

Consider a rooted plumbing tree $(\Gamma,v)$, that is, a plumbing tree $\Gamma$ with a distinguished vertex $v$ (the distinguished vertex is also referred to as the root). This determines a graph manifold with torus boundary, which we denote by $M_{\Gamma,v}$; the boundary occurs at the distinguished vertex and comes from removing a neighborhood of a regular fiber of the $S^1$-bundle associated with this vertex in the construction of a closed graph manifold from the plumbing tree. By convention, we will always fix parametrizing curves $\{\alpha, \beta\}$ for $\partial M_{\Gamma,v}$ where $\alpha$ is the fiber of the $S^1$-bundle associated with the root and $\beta$ is a curve lying in the base surface of that bundle. It is sufficient to restrict to connected plumbing trees, since disjoint unions of trees correspond to connected sums of graph manifolds. Recall that different plumbing trees may represent the same 3-manifold, and these trees are related by a set of moves that are described in \cite{Neumann}; in particular, valence one or two vertices (excluding the distinguished vertex) with weight $0$ or $\pm 1$ can be removed by Moves R1, R3, or R6. We will say that a (rooted) plumbing tree is \textit{reduced} if it has no valence one or two non-distinguished vertices with weight $0$ or $\pm 1$, and we will assume unless otherwise stated that plumbing trees are reduced.


The simplest rooted plumbing tree is the tree with no edges and a single vertex of weight zero, which is necessarily the root. Any other rooted plumbing tree $(\Gamma,v)$ can be constructed inductively from copies of this simple tree using a few basic operations: the twist operation $\mathcal{T}^m$ adds $m$ to the weight of the root, the extend operation $\mathcal{E}$ attaches a new zero weighted vertex to the root and makes the new vertex the root on the resulting tree, and the merge operation $\mathcal{M}$ combines two rooted plumbing trees by identifying the roots and adding their weights. We call these the three elementary operations on plumbing trees; see Figure \ref{fig:plumbing-tree-moves}. Each of these operations has a corresponding effect on the bordered Floer invariants of the corresponding bordered graph manifolds; we use $\twistloop^m$, $\extendloop$, and $\mergeloop$ to denote the operations on the cyclic word representations of the bordered invariants that are induced by $\mathcal{T}^m$, $\mathcal{E}$, and $\mathcal{M}$, respectively. Note that in the base case in which $\Gamma$ has a single vertex, the manifold $M_{\Gamma, v}$ is a solid torus with meridian $\beta$ and its bordered invariant is given by a single immersed curve homotopic to $\beta$, represented by the cyclic word $(\beta)$. Thus the bordered invariants for any $M_{\Gamma, v}$ can be obtained from copies of $(\beta)$ by applying the operations $\twistloop^m$, $\extendloop$, and $\mergeloop$.

The moves $\mathcal{T}^m$ and $\mathcal{E}$ do not change the underlying manifold with boundary $M_{\Gamma, v}$, they only affect the parametrization of the boundary. In particular, the move $\mathcal{T}^m$ reparametrizes by $m$ Dehn twists about $\alpha$, resulting in new parametrizing curves $\{\alpha', \beta'\} = \{\alpha, \beta + m\alpha\}$, while $\mathcal{E}$ gives the new parametrizing curves $\{\alpha', \beta'\} = \{\beta, -\alpha\}$. The bordered Floer invariants are unchanged as immersed curves in the boundary of the manifold but their representation as cyclic words in $\alpha$ and $\beta$ changes due to these reparametrizations. Thus the operation $\twistloop^m$ replaces instances of $\beta$ with $\beta' (\alpha')^{-m}$ and instances of $\beta^{-1}$ with $(\alpha')^{m} (\beta')^{-1}$, and the operation $\extendloop$ replaces $\beta$ with $\alpha'$, $\alpha$ with $(\beta')^{-1}$, $\beta^{-1}$ with $(\alpha')^{-1}$, and $\alpha^{-1}$ with $\beta'$. After the transformation, we drop the primes and $\alpha$ and $\beta$ now refer to parametrizing curves of the new bordered manifold.

The operation $\twistloop^m$ also has a nice description when the cyclic words are expressed in loop calculus notation: it replaces the letter $c_k$ with the letter $c_{k-m}$ and the letter $\bar c_k$ with the letter $\bar c_{k-m}$ but leaves loop calculus letters of type $a$ or $b$ unchanged. Both operations can also be understood as operations on immersed curves by interpreting cyclic words as homotopy classes of immersed curves in the punctured torus $T_M$. In particular, $\twistloop^m$ changes an immersed curve by applying $m$ negative Dehn twists about $\alpha$, and $\extendloop$ rotates an immersed curve by $90^\circ$ about the puncture.

The merge operation on plumbing trees is the only elementary operation that changes the topology of the associated manifold. The corresponding operation $\mergeloop$ inputs two collections of cyclic words $w_1$ and $w_2$, each representing an immersed multicurve, and produces a new collection of cyclic words $\mergeloop(w_1, w_2)$. The merge operation is symmetric, so $\mergeloop(w_1, w_2) = \mergeloop(w_2, w_1)$. We also have that $\mergeloop(w_1, w_2 \sqcup w'_2 ) = \mergeloop(w_1, w_2) \sqcup \mergeloop(w_1, w'_2)$, so it is enough to describe $\mergeloop(w_1, w_2)$ when $w_1$ and $w_2$ are each a single cyclic word. The merge operation was first described in \cite{HW:loop-calculus} in loop calculus notation in the special case that one of the input cyclic words (say $w_1)$ consists of only type $c$ letters, an assumption which holds in all of the cases we will consider in this paper. In this setting, we find the collection of cyclic words $\mergeloop(w_1, w_2)$ by constructing a toroidal grid in which each row corresponds to a loop calculus letter in $w_1$ and each column corresponds to a loop calculus letter in $w_2$ and filling in the grid with arrows labelled by loop calculus letters as follows:

\begin{center}
\begin{tikzpicture}[scale = 1.5,>=latex]
\scriptsize

\draw[dotted] (0,0) -- (1,0) -- (1,1) -- (0,1) -- (0,0);
\draw[->] (0, 1.2) to node[above]{$a_k$} (1, 1.2);
\draw[->] (-.2, 1) to node[left]{$c_m$} (-.2, 0);
\draw[->, bend right = 45] (.1, .9) to node[below]{$a_k$} (.9,.9);

\draw[dotted] (2,0) -- (3,0) -- (3,1) -- (2,1) -- (2,0);
\draw[->] (2, 1.2) to node[above]{$b_k$} (3, 1.2);
\draw[->] (1.8, 1) to node[left]{$c_m$} (1.8, 0);
\draw[->, bend left = 45] (2.1, .1) to node[above]{$b_k$} (2.9,.1);

\draw[dotted] (4,0) -- (5,0) -- (5,1) -- (4,1) -- (4,0);
\draw[->] (4, 1.2) to node[above]{$c_k$} (5, 1.2);
\draw[->] (3.8, 1) to node[left]{$c_m$} (3.8, 0);
\draw[->] (4.1, .9) to node[above right = -2pt]{$c_{k+m}$} (4.9,.1);

\draw[dotted] (6,0) -- (7,0) -- (7,1) -- (6,1) -- (6,0);
\draw[->] (6, 1.2) to node[above]{$\bar c_k$} (7, 1.2);
\draw[->] (5.8, 1) to node[left]{$c_m$} (5.8, 0);
\draw[->] (6.1, .1) to node[below right = -2pt]{$\bar c_{k+m}$} (6.9,.9);
\end{tikzpicture}
\end{center}

\noindent Remembering that opposite sides of the toroidal grid are identified, the arrows in the grid form one or more directed loops; each loop gives rise to a cyclic word by reading off the arrow labels in order, and this collection of cyclic words is $\mergeloop(w_1, w_2)$. The number of components in $\mergeloop(w_1, w_2)$ is $\gcd(n_1, n_2)$, where $n_1$ is the total number of (type $c$) letters in $w_1$ and $n_2$ is the number of type $c$ letters minus the number of type $\bar c$ letters in $w_2$ (up to inverting $w_2$ if necessary we may assume that $n_2$ is non-negative). When the cyclic words $w_1$ and $w_2$ represent bordered invariants for manifolds, different components of $\mergeloop(w_1, w_2)$ live in different \spinc structures of the corresponding manifold.
We remark that the merge operation is more subtle when both input cyclic words contain letters of type $a$ or $b$, but we will not need to consider this case. 


In \cite{HW:cabling}, it was observed that the merge operation in the case described above admits a convenient geometric description when the cyclic words are interpreted as immersed curves in the punctured torus lifted to $\R^2 \setminus \Z^2$. Let $\gamma_i$ be the immersed curve in the punctured torus corresponding to the word $w_i$, and let $\tilde\gamma_i$ be some lift of $\gamma_i$ to $\R^2 \setminus \Z^2$. For concreteness, we will assume that both curves are in rectilinear position, except for the following perturbation: we translate or tilt each maximal vertical segment by translating each endpoint left or right by $\epsilon$ in the direction of the adjacent horizontal arrow, replacing the maximal vertical segment with the line segment connecting these new endpoints. The assumption that $w_1$ contains only type $c$ segments in its loop calculus notation implies that each maximal vertical segment is adjacent to one horizontal segment on the left and one on the right (see Figure \ref{fig:loop-calculus-curve-segments} for details), so a length $k$ maximal vertical segment of $\tilde\gamma_1$ becomes a line segment of slope $\pm \frac{k}{2\epsilon}$. 
After this perturbation, $\tilde\gamma_1$ moves monotonically rightward and is the graph of some function, which we will denote $f_1$.

We will construct a curve in $\R^2 \setminus \Z^2$ representing a component of $\mergeloop(w_1,w_2)$ by translating each point $(x,y)$ on $\tilde\gamma_2$ upwards by $f_1(x) - \tfrac 1 2$.  We refer to the resulting curve as the \textit{vertical sum} of the curves $\tilde\gamma_2$ and $\tilde\gamma_1$, and denote it $\tilde\gamma_2 +_v \tilde\gamma_1$. Note that outside a neighborhood of the vertical segments, $f_1$ is a piecewise constant function with values in $\Z + \frac 1 2$. It follows that we shift each unit horizontal segment of $\tilde\gamma_2$ by an integral amount that is $\tfrac 1 2$ less than the height of the (unique) horizontal segment of $\tilde\gamma_1$ at the same $x$-coordinate, and we stretch or contract the maximal vertical segments accordingly.


The vertical sum of $\tilde\gamma_2$ and $\tilde\gamma_1$ described above gives one component of $\mergeloop(w_1,w_2)$; to get all components we repeat this operation after shifting $\tilde\gamma_2$ horizontally by different amounts. As before, let $n_i$ be the number of type $c$ letters minus the number of type $\bar c$ letters in $w_i$ in loop calculus notation, which is the same as the signed count of $\beta$ letters in $w_i$. Geometrically, $n_i$ can be interpreted as the net horizontal motion of $\gamma_i$, or equivalently the horizontal period of the periodic curve $\tilde\gamma_i$.
The curve $\tilde\gamma_2 +_v \tilde\gamma_1$ is periodic and is fixed by a translation with horizontal component $n' = \text{lcm}(n_1, n_2)$.
If we fix a lift $\tilde\gamma_2$ of $\gamma_2$, then translating $\tilde\gamma_2$ to the right by $i$ defines other lifts $\tilde\gamma_2^i$; the merge operation produces $\gcd(n_1,n_2)$ components, with one component given by $\tilde\gamma_2^i +_v \tilde\gamma_1$ for each $0 \le i < \gcd(n_1, n_2)$.

%
%
%
%
%
%
%
%

\begin{example}\label{ex:merge1}
If $w_1 = (c_0 c_{-1})$ and $w_2 = (c_0 c_0 c_{-1})$, then using the grid in Figure \ref{fig:VertSumEx1} we find that $\mergeloop(w_1, w_2)$ is given by the single cyclic word $(c_{0} c_{-1} c_{-1} c_{-1} c_{0} c_{-2})$. The same computation is performed as a vertical sum of curves to the right. For each point on $\tilde\gamma_2$ there is a corresponding point of $\tilde\gamma_2 +_v \tilde\gamma_1$ obtained by adding the height of $\tilde\gamma_1$ above the line $y = \tfrac 1 2$ to the $y$-coordinate.

\begin{figure}[ht]
\fontsize{18pt}{24pt}
\begin{tikzpicture}[scale = 1.5,>=latex]

\scriptsize
\draw[->] (.1, 2.2) to node[above]{$c_{0}$} (.9, 2.2);
\draw[->] (1.1, 2.2) to node[above]{$c_{0}$} (1.9, 2.2);
\draw[->] (2.1, 2.2) to node[above]{$c_{-1}$} (2.9, 2.2);

\draw[->] (-.2, 1.9) to node[left]{$c_0$} (-.2, 1.1);
\draw[->] (-.2, .9) to node[left]{$c_{-1}$} (-.2, .1);

\draw[dotted] (0,0) -- (3,0);
\draw[dotted] (0,1) -- (3,1);
\draw[dotted] (0,2) -- (3,2);
\draw[dotted] (0,0) -- (0,2);
\draw[dotted] (1,0) -- (1,2);
\draw[dotted] (2,0) -- (2,2);
\draw[dotted] (3,0) -- (3,2);

\draw[->] (.1, 1.9) to node[above right = -2pt]{$c_{0}$} (.9, 1.1);
\draw[->] (.1, .9) to node[above right = -2pt]{$c_{-1}$} (.9, .1);
\draw[->] (1.1, 1.9) to node[above right = -2pt]{$c_{0}$} (1.9, 1.1);
\draw[->] (1.1, .9) to node[above right = -2pt]{$c_{-1}$} (1.9, .1);
\draw[->] (2.1, 1.9) to node[above right = -2pt]{$c_{-1}$} (2.9, 1.1);
\draw[->] (2.1, .9) to node[above right = -2pt]{$c_{-2}$} (2.9, .1);

\end{tikzpicture}
\hspace{2 cm}
\scalebox{.45}{%
    \def\svgwidth{\columnwidth}
    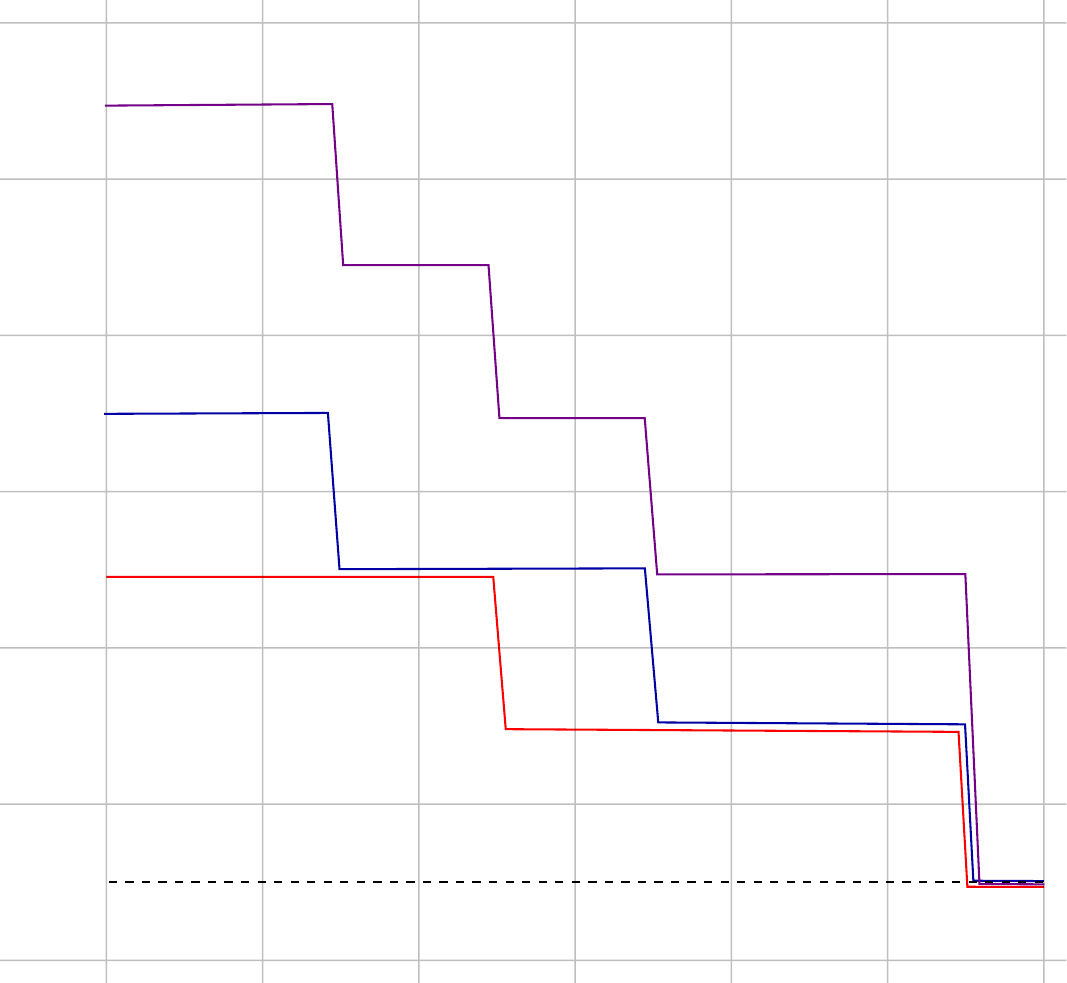
}
\caption{Merge $\mergeloop(w_1, w_2)$ and vertical sum $\tilde\gamma_2 +_v \tilde\gamma_1$ for Example \ref{ex:merge1}}
\label{fig:VertSumEx1}
\end{figure}

\end{example}

\begin{example}\label{ex:merge2}
If $w_1 = (c_0 c_{-1})$ and $w_2 = (a_{-1} b_{-1} c_{-4})$, then using the grid below we find that $\mergeloop(w_1, w_2)$ is given by the single cyclic word $(a_{-1} b_{-1} c_{-4} a_{-1} b_{-1} c_{-5})$. On the right is the corresponding vertical sum of curves. 

\begin{figure}
\begin{tikzpicture}[scale = 1.5,>=latex]

\scriptsize
\draw[->] (.1, 2.2) to node[above]{$a_{-1}$} (.9, 2.2);
\draw[->] (1.1, 2.2) to node[above]{$b_{-1}$} (1.9, 2.2);
\draw[->] (2.1, 2.2) to node[above]{$c_{-4}$} (2.9, 2.2);

\draw[->] (-.2, 1.9) to node[left]{$c_0$} (-.2, 1.1);
\draw[->] (-.2, .9) to node[left]{$c_{-1}$} (-.2, .1);

\draw[dotted] (0,0) -- (3,0);
\draw[dotted] (0,1) -- (3,1);
\draw[dotted] (0,2) -- (3,2);

\draw[dotted] (0,0) -- (0,2);
\draw[dotted] (1,0) -- (1,2);
\draw[dotted] (2,0) -- (2,2);
\draw[dotted] (3,0) -- (3,2);

\draw[->, bend right = 45] (.1, .9) to node[below]{$a_{-1}$} (.9,.9);
\draw[->, bend right = 45] (.1, 1.9) to node[below]{$a_{-1}$} (.9,1.9);

\draw[->, bend left = 45] (1.1, 1.1) to node[above]{$b_{-1}$} (1.9, 1.1);
\draw[->, bend left = 45] (1.1, .1) to node[above]{$b_{-1}$} (1.9, .1);

\draw[->] (2.1, 1.9) to node[above right = -2pt]{$c_{-4}$} (2.9, 1.1);
\draw[->] (2.1, .9) to node[above right = -2pt]{$c_{-5}$} (2.9, .1);
\end{tikzpicture}
\hspace{2 cm}
\scalebox{.2}{%
    \def\svgwidth{\columnwidth}
    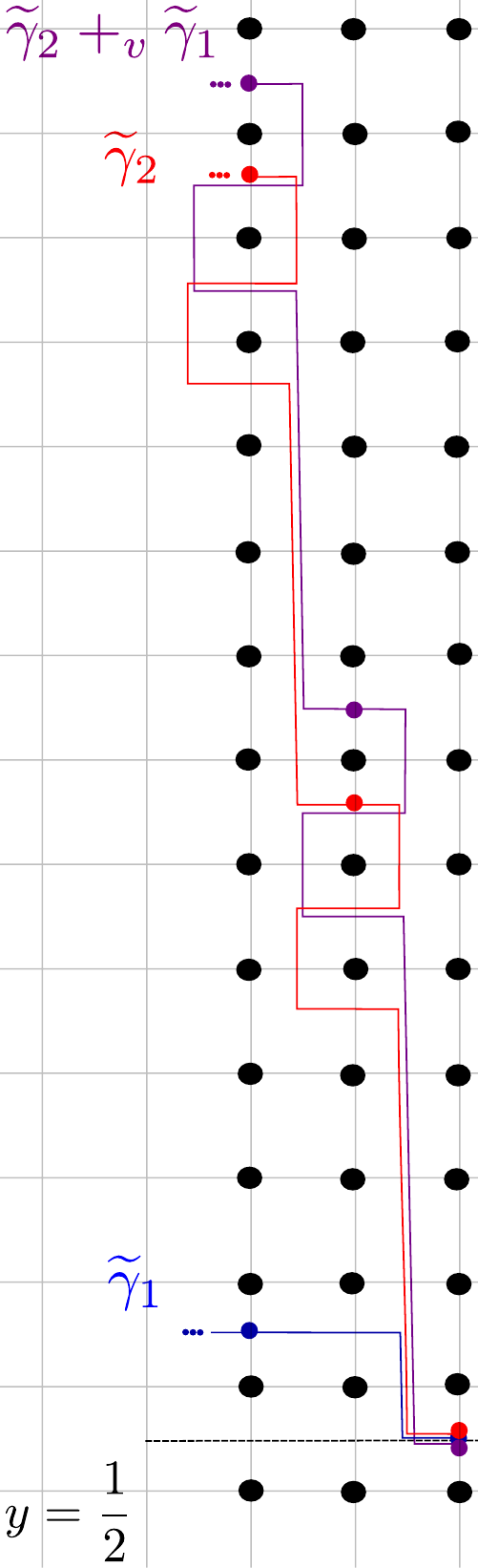
}
\caption{Merge $\mergeloop(w_1, w_2)$ and vertical sum $\tilde\gamma_2 +_v \tilde\gamma_1$ for Example \ref{ex:merge2}}
\label{fig:VertSumEx2}
\end{figure}

\end{example}

%

It will be helpful to keep track of the total number of $\alpha$ or $\beta$ letters in the collection of cyclic words as we build up the immersed curve invariant for a graph manifold. Towards this end, the following relation will be useful:

\begin{lemma}\label{lem:merge-parity}
Let $w_1$ and $w_2$ be cyclic words and assume that $w_1$ contains only type $c$ letters in loop calculus notation. For $i\in\{1,2\},$ let $m_i$ be the signed count of $\alpha$ letters in $w_i$ (that is, the number of $\alpha$ letters minus the number of $\alpha^{-1}$ letters) and let $n_i$ be the signed count of  $\beta$ letters in $w_i$. Let $m'$ and $n'$ denote the signed counts of $\alpha$ letters and $\beta$ letters, respectively, in the collection of cyclic words corresponding to $\mergeloop( w_1, w_2 )$. Then
\begin{align*}
m' &= m_1 n_2 + m_2 n_1, \text{ and } \\
n' &= n_1 n_2.
\end{align*}
\end{lemma}
\begin{proof}
Each $c_k$ piece in $w_i$ contributes $1$ to $n_i$ and $k$ to $m_i$, each $\bar c_k$ piece contributes $-1$ to $n_i$ and $-k$ to $m_i$, and each $a_k$ or $b_k$ piece contributes 0 to $n_i$ and $k$ to $m_i$. It is easy to check that the relations $m' = m_1 n_2 + m_2 n_1$ and $n' = n_1 n_2$ hold if we restrict to the contribution of each square of the grid arising from merging two loop calculus letters, and the result follows from summing over all squares in the grid.
\end{proof}

\subsection{Example computations} \label{sec:example-computations}

Using the operations $\twistloop$, $\extendloop$, and $\mergeloop$ we can compute the bordered Floer invariants of the manifold $M_{\Gamma,v}$ corresponding to a rooted plumbing tree $(\Gamma, v)$. From this we can easily recover $\HFhat$ of $Y_{\Gamma}$, which is obtained from $M_{\Gamma,v}$ by capping off with an appropriately framed solid torus; the framing is such that the meridian of the solid torus glues to $\beta$. By the pairing theorem, $\HFhat(Y_\Gamma)$ is given by the intersection Floer homology in the punctured torus of $\HFhat(M_{\Gamma,v})$ with a simple closed curve homotopic to $\beta$. In this section we illustrate this procedure with two explicit examples.

\begin{example}\label{ex:plumbing-tree-example1}
We will show that the immersed curve in Figure \ref{fig:immersed-curve-example}, which we are using as a running example, is the bordered invariant associated with the plumbing tree
$$\begin{tikzpicture}[scale = .8]
\scriptsize
\tikzstyle{every node}=[draw,circle,fill=white,minimum size=4pt,
                            inner sep=0pt]
\node at (-1.5,0) [circle, fill=black] [label = above: {$-2$}]{};
\node at (-1,0) [circle, fill=black] [label = above: {$-2$}]{};
\node at (-.5,0) [circle, fill=black] [label = above: {$-6$}]{};                          
\node at (0,0) [circle, fill=black] [label = above: {$-1$}]{};
\node at (.5,.4) [circle, fill=black] [label = above right: {$-2$}]{};
\node at (.5,-.4) [circle, fill=black] [label = above right: {$-3$}]{};
\draw[dashed] (-2,0) -- (-1.5, 0);
\draw (-1.5,.0) -- (0, 0);
\draw (.5, .4) -- (0,0) -- (.5,-.4);
\end{tikzpicture}.$$
We begin with the invariant associated with the plumbing tree \begin{tikzpicture}
\scriptsize
\tikzstyle{every node}=[draw,circle,fill=white,minimum size=4pt,
                            inner sep=0pt]
\node at (0,0) [circle, fill=black] [label = above right: {0}]{};
\draw[dashed] (-.5,.0) -- (0, 0);
\end{tikzpicture}
, which is represented by the cyclic word $(\beta)$, and we apply $\twistloop^{-2}$ to get the word $(\beta \alpha \alpha)$ associated with the plumbing tree 
\begin{tikzpicture}
\scriptsize
\tikzstyle{every node}=[draw,circle,fill=white,minimum size=4pt,
                            inner sep=0pt]
\node at (0,0) [circle, fill=black] [label = above right: {$-2$}]{};
\draw[dashed] (-.5,.0) -- (0, 0);
\end{tikzpicture}. 
We get the invariant associated with 
\begin{tikzpicture}
\scriptsize
\tikzstyle{every node}=[draw,circle,fill=white,minimum size=4pt,
                            inner sep=0pt]
\node at (0,0) [circle, fill=black] [label = above : {$0$}]{};
\node at (.5,0) [circle, fill=black] [label = above right: {$-2$}]{};
\draw[dashed] (-.5,.0) -- (0, 0);
\draw (0,0) -- (.5,0);
\end{tikzpicture}
by applying $\extendloop$. This gives $(\alpha \beta^{-1} \beta^{-1})$, which by inverting the word we may write as $(\beta \beta \alpha^{-1})$; in loop calculus notation this is $(c_0 c_{-1})$. A similar computation gives the curve $(\beta^3 \alpha^{-1})$ for the plumbing tree 
\begin{tikzpicture}
\scriptsize
\tikzstyle{every node}=[draw,circle,fill=white,minimum size=4pt,
                            inner sep=0pt]
\node at (0,0) [circle, fill=black] [label = above right: {0}]{};
\node at (.5,0) [circle, fill=black] [label = above right: {$-3$}]{};
\draw[dashed] (-.5,.0) -- (0, 0);
\draw (0,0) -- (.5,0);
\end{tikzpicture}
, which in loop calculus notation is $(c_0 c_0 c_{-1})$. As shown in Example \ref{ex:merge1}, merging these two cyclic words gives the cyclic word  $(c_0 c_{-1} c_{-1} c_{-1} c_0 c_{-2}) = (\beta \beta \alpha^{-1} \beta \alpha^{-1} \beta \alpha^{-1} \beta \beta \alpha^{-1} \alpha^{-1})$ for the plumbing tree
$$\begin{tikzpicture}[scale = .8]
\scriptsize
\tikzstyle{every node}=[draw,circle,fill=white,minimum size=4pt,
                            inner sep=0pt]
\node at (0,0) [circle, fill=black] [label = above: {$0$}]{};
\node at (.5,.4) [circle, fill=black] [label = above right: {$-2$}]{};
\node at (.5,-.4) [circle, fill=black] [label = above right: {$-3$}]{};
\draw[dashed] (-.5,.0) -- (0, 0);
\draw (.5, .4) -- (0,0) -- (.5,-.4);
\end{tikzpicture}.$$ 
We then apply $\twistloop^{-1}$ to get $(c_1 c_0 c_0 c_0 c_1 c_{-1}) = (\beta \alpha \beta \beta \beta \beta \alpha \beta \alpha^{-1})$, followed by $\extendloop$ to get the cyclic word $(\alpha \beta^{-1} \alpha \alpha \alpha \alpha \beta^{-1} \alpha \beta)$ for the plumbing tree
$$\begin{tikzpicture}[scale = .8]
\scriptsize
\tikzstyle{every node}=[draw,circle,fill=white,minimum size=4pt,
                            inner sep=0pt]
\node at (-.5,0) [circle, fill=black] [label = above: {$0$}]{};                          
\node at (0,0) [circle, fill=black] [label = above: {$-1$}]{};
\node at (.5,.4) [circle, fill=black] [label = above right: {$-2$}]{};
\node at (.5,-.4) [circle, fill=black] [label = above right: {$-3$}]{};
\draw[dashed] (-1,.0) -- (-.5, 0);
\draw (-.5,.0) -- (0, 0);
\draw (.5, .4) -- (0,0) -- (.5,-.4);
\end{tikzpicture}.$$
\sloppy Applying $\twistloop^{-6}$ gives $(\alpha \beta^{-1} \alpha^{-1} \alpha^{-1} \beta^{-1} \alpha \beta)$ and then applying $\extendloop$ gives the cyclic word \mbox{$(\beta^{-1} \alpha^{-1} \beta \beta \alpha^{-1} \beta^{-1} \alpha)$} for the plumbing tree
$$\begin{tikzpicture}[scale = .8]
\scriptsize
\tikzstyle{every node}=[draw,circle,fill=white,minimum size=4pt,
                            inner sep=0pt]
\node at (-1,0) [circle, fill=black] [label = above: {$0$}]{};
\node at (-.5,0) [circle, fill=black] [label = above: {$-6$}]{};                          
\node at (0,0) [circle, fill=black] [label = above: {$-1$}]{};
\node at (.5,.4) [circle, fill=black] [label = above right: {$-2$}]{};
\node at (.5,-.4) [circle, fill=black] [label = above right: {$-3$}]{};
\draw[dashed] (-1.5,0) -- (-1, 0);
\draw (-1,.0) -- (0, 0);
\draw (.5, .4) -- (0,0) -- (.5,-.4);
\end{tikzpicture}.$$
Applying $\twistloop^{-2}$ gives $(\beta^{-1} \alpha^{-1} \beta \alpha \alpha \beta \alpha^{-1} \beta^{-1} \alpha^{-1})$ and then applying $\extendloop$ gives the cyclic word $(\alpha^{-1} \beta \alpha \beta^{-1} \beta^{-1} \alpha \beta \alpha^{-1} \beta)$ associated with the plumbing tree
$$\begin{tikzpicture}[scale = .8]
\scriptsize
\tikzstyle{every node}=[draw,circle,fill=white,minimum size=4pt,
                            inner sep=0pt]
\node at (-1.5,0) [circle, fill=black] [label = above: {$0$}]{};
\node at (-1,0) [circle, fill=black] [label = above: {$-2$}]{};
\node at (-.5,0) [circle, fill=black] [label = above: {$-6$}]{};                          
\node at (0,0) [circle, fill=black] [label = above: {$-1$}]{};
\node at (.5,.4) [circle, fill=black] [label = above right: {$-2$}]{};
\node at (.5,-.4) [circle, fill=black] [label = above right: {$-3$}]{};
\draw[dashed] (-2,0) -- (-1.5, 0);
\draw (-1.5,.0) -- (0, 0);
\draw (.5, .4) -- (0,0) -- (.5,-.4);
\end{tikzpicture}.$$
To complete the example, we apply $\twistloop^{-2}$ to obtain the cyclic word $(\beta \alpha \beta^{-1} \alpha^{-1} \alpha^{-1} \beta^{-1} \alpha \beta \alpha \beta \alpha)$. The corresponding immersed curve is shown in Figure \ref{fig:immersed-curve-example}.
\end{example}

\begin{example}\label{ex:plumbing-tree-example2}
As a second example, we will compute $\HFhat(Y_\Gamma)$ where $\Gamma$ is the plumbing tree 
\begin{center}
\begin{tikzpicture}
\scriptsize
\tikzstyle{every node}=[draw,circle,fill=white,minimum size=4pt,
                            inner sep=0pt]
\node at (0,0) [circle, fill=black] {};
\node at (1,0) [circle, fill=black] {};
\node at (-.5,.5) [circle, fill=black] {};
\node at (-.5,-.5) [circle, fill=black] {};
\node at (1.5,.5) [circle, fill=black] {};
\node at (1.5,-.5) [circle, fill=black] {};
\draw (0,0) -- (1,0); \draw (-.5,-.5) -- (0,0) -- (-.5,.5); \draw (1.5,-.5) -- (1,0) -- (1.5,.5);
\put(-19,-1){$-3$}
\put(33,-1){$-1$}
\put(-32,-16){$-2$}
\put(-32,10){$-2$}
\put(46,-16){$-3$}
\put(46,10){$-2$}
\end{tikzpicture} \qquad .
\end{center}
Following the beginning of the previous example, we find that $(\alpha \beta^{-1} \alpha \alpha \alpha \alpha \beta^{-1} \alpha \beta)$ is the cyclic word for the plumbing tree
$$\begin{tikzpicture}[scale = .8]
\scriptsize
\tikzstyle{every node}=[draw,circle,fill=white,minimum size=4pt,
                            inner sep=0pt]
\node at (-.5,0) [circle, fill=black] [label = above: {$0$}]{};                          
\node at (0,0) [circle, fill=black] [label = above: {$-1$}]{};
\node at (.5,.4) [circle, fill=black] [label = above right: {$-2$}]{};
\node at (.5,-.4) [circle, fill=black] [label = above right: {$-3$}]{};
\draw[dashed] (-1,.0) -- (-.5, 0);
\draw (-.5,.0) -- (0, 0);
\draw (.5, .4) -- (0,0) -- (.5,-.4);
\end{tikzpicture}.$$
In loop calculus notation, this is $(\bar c_{-4} b_1 a_1)$, which is equivalent to $(a_{-1} b_{-1} c_{-4})$ by inverting the word. Merging this with the cyclic word $(c_0 c_{-1})$, as in Example \ref{ex:merge2}, gives the cyclic word $(a_{-1} b_{-1} c_{-4} a_{-1} b_{-1} c_{-5})$ associated with the plumbing tree 
$$\begin{tikzpicture}[scale = .8]
\scriptsize
\tikzstyle{every node}=[draw,circle,fill=white,minimum size=4pt,
                            inner sep=0pt]
\node at (-.5,0) [circle, fill=black] [label = above: {$0$}]{};
\node at (-1,-.4) [circle, fill=black] [label = above left: {$-2$}]{};                          
                          
\node at (0,0) [circle, fill=black] [label = above: {$-1$}]{};
\node at (.5,.4) [circle, fill=black] [label = above right: {$-2$}]{};
\node at (.5,-.4) [circle, fill=black] [label = above right: {$-3$}]{};
\draw[dashed] (-1,.4) -- (-.5, 0);
\draw (-1,-.4) -- (-.5,.0) -- (0, 0);
\draw (.5, .4) -- (0,0) -- (.5,-.4);
\end{tikzpicture}.$$
Applying $\twistloop^{-3}$ gives the word $(a_{-1} b_{-1} c_{-1} a_{-1} b_{-1} c_{-2}) = (\beta \alpha^{-1} \beta^{-1} \alpha^{-1} \beta \alpha^{-1} \beta \alpha^{-1} \beta^{-1} \alpha^{-1} \beta \alpha^{-2})$. Applying $\extendloop$ and cyclicly permuting gives the word $(\beta \beta \alpha \beta \alpha^{-1} \beta \alpha \beta \alpha \beta \alpha^{-1} \beta \alpha)$ associated with the plumbing tree
$$\begin{tikzpicture}[scale = .8]
\scriptsize
\tikzstyle{every node}=[draw,circle,fill=white,minimum size=4pt,
                            inner sep=0pt]
\node at (-.5,0) [circle, fill=black] [label = above: {$ -3$}]{};
\node at (-1,-.4) [circle, fill=black] [label = above left: {$-2$}]{};                          
\node at (-1,.4) [circle, fill=black] [label = above: {$0$}]{};                           
\node at (0,0) [circle, fill=black] [label = above: {$-1$}]{};
\node at (.5,.4) [circle, fill=black] [label = above right: {$-2$}]{};
\node at (.5,-.4) [circle, fill=black] [label = above right: {$-3$}]{};
\draw[dashed] (-1,.4) -- (-1.5, .8);
\draw (-1,.4) -- (-.5, 0);

\draw (-1,-.4) -- (-.5,.0) -- (0, 0);
\draw (.5, .4) -- (0,0) -- (.5,-.4);
\end{tikzpicture}.$$
Finally, applying $\twistloop^{-2}$ gives $(\beta \alpha \alpha \beta \alpha \alpha \alpha \beta \alpha \beta \alpha \alpha \alpha \beta \alpha \alpha \alpha \beta \alpha \beta \alpha \alpha \alpha)$, the cyclic word associated with $(\Gamma, v)$. The corresponding immersed curve $\HFhat(M_{\Gamma, v})$ can be seen on the right side of Figure \ref{fig:delta-sym-examples}. $Y_\Gamma$ is obtained from $M_{\Gamma,v}$ by capping off with a solid torus, which is framed so that the meridian is identified with $\beta$. It follows that $\HFhat(Y_\Gamma)$ is given by the (intersection) Floer homology of $\HFhat(M_{\Gamma,v})$ with a simple closed curve homotopic to $\beta$. Since there are no instances of $\alpha^{-1}$ in the cyclic word representing $\HFhat(M_{\Gamma,v})$, $\HFhat(Y_\Gamma)$ has one generator for each $\alpha$ in $\HFhat(M_{\Gamma,v})$ and thus $\dim \HFhat(Y_\Gamma)$ is 16.

\end{example}

\subsection{L-space gluing results}

\bigbreak

By restricting to bordered manifolds that have L-space fillings, we will be able to assume the bordered invariants we consider have some additional structure. In particular, this is the main mechanism by which we will ensure that the simplified version of the merge operation described in Section \ref{sec:graph-mfds} is sufficient. Here we recall some results about L-spaces for toroidal gluing and prove an L-space condition for the merge operation that we will need in our inductive arguments later.

L-spaces formed by toroidal gluing can be classified in terms of the sets of L-space slopes on the two manifolds being glued. Recall that for a manifold $M$ with torus boundary, a \emph{slope} on $\partial M$ is a primitive homology class in $H_1(\partial M)$ modulo ${\pm 1}$. If a pair of parametrizing curves ${\alpha, \beta}$ is fixed for $\partial M$, then the set of slopes may be identified with $\hat\Q = \Q \cup \left\{ \tfrac 1 0 \right\} $, where $[p\alpha + q\beta]$ is identified with $\tfrac{p}{q}$. Let $Sl(M)$ denote the set of slopes on $M$, and let $\mathcal{L}(M) \subset Sl(M)$ denote the set of \emph{L-space slopes}---that is, slopes for which Dehn filling yields an L-space. $\mathcal{L}(M)$ never contains the rational longitude $\lambda$ of $M$ and it is always either empty, a single point, a closed interval, or $\hat\Q \setminus \{\lambda\}$ \cite[Theorem 1.6]{RR}. We call the interior $\mathcal{L}^\circ(M)$ the set of \emph{strict L-space slopes} of $M$. When two manifolds with torus boundary are glued, the sets of strict L-space slopes for each manifold determines whether the resulting manifold is an L-space:

\begin{proposition}[{\cite[Theorem 1.14]{HRW}}] \label{prop:Lspace-gluing-thm}
Suppose $M_1$ and $M_2$ are boundary incompressible. Then $M_1 \cup_\phi M_2$ is an L-space if and only if $\phi_*( \mathcal{L}^\circ(M_1)) \cup \mathcal{L}^\circ(M_2) = Sl(M_2)$, where $\phi_*: Sl(M_1) \to Sl(M_2)$ is the map on slopes induced by $\phi$. In other words, $M_1 \cup_\phi M_2$ is an L-space if and only if every slope on the shared torus is a strict L-space slope on at least one side.
\end{proposition}
In particular, it is clear that for $M_1 \cup M_2$ to be an L-space both $M_1$ and $M_2$ must have non-empty set of strict L-space slopes. We have a term for such manifolds:
\begin{definition}
A manifold with torus boundary $M$ is \emph{Floer simple} if it has more than one L-space Dehn filling (equivalently, if $\mathcal{L}^\circ(M)$ is non-empty).
\end{definition}

It was shown in \cite[Proposition 6]{HRRW} that the Floer simple condition is equivalent to another condition defined in terms of loop calculus notation known as the simple loop type condition.

\begin{definition}
A manifold with torus boundary $M$ is of \emph{simple loop type} if for some choice of parametrization $(\alpha, \beta)$ of $\partial M$ the invariant $\HFhat(M, \alpha, \beta)$ consists of exactly one cyclic word for each spin$^c$ structure of $M$ and these words contain only type $c$ letters in loop calculus notation.
\end{definition}
\noindent In particular, a key consequence of Floer simplicity is that the multicurve $\HFhat(M)$ has one curve for each \spinc structure.

The equivalence between the Floer simple and simple loop type conditions can be stated as follows when considering a particular slope:
\begin{proposition}[{See \cite[Proposition 6]{HRRW} and \cite[Proposition 3.9]{RR}}]
\label{prop:infty-strict-lspace-slope}
For a bordered manifold with torus boundary $(M,\alpha,\beta)$, the $\infty$ slope (that is, the slope $\alpha$) is a strict L-space slope if and only if $\HFhat(M, \alpha, \beta)$ consists of exactly one cyclic word for each spin$^c$ structure of $M$ and, up to reversing the cyclic words, these words contain only type $c$ letters in loop calculus notation.
\end{proposition}

Note that Proposition \ref{prop:Lspace-gluing-thm} requires both $M_1$ and $M_2$ to be boundary incompressible (if either is not, a similar statement holds but with $\mathcal{L}^\circ(M_i)$ replaced with $\mathcal{L}(M_i)$). For this reason, it is helpful to identify when a plumbed 3-manifold is boundary compressible.

\begin{lemma}\label{lem:boundary-compressible}
Given a reduced rooted plumbing tree $(\Gamma, v)$, the graph manifold with boundary $M_{\Gamma,v}$ is boundary compressible if and only if $\Gamma$ is linear (that is, $v$ is a leaf of $\Gamma$ and all other vertices of $\Gamma$ have valence at most two).
\end{lemma}
\begin{proof}
If $v$ is a leaf of $\Gamma$ and all other vertices of $\Gamma$ have valence at most two, then $M_{\Gamma,v}$ is a solid torus and is thus boundary compressible. The converse follows from\cite[Theorem 4.3]{Neumann}, which states that a closed 3-manifold associated with a connected plumbing tree in normal form is prime.  Note that a reduced plumbing tree in our sense is not the same as the normal form used in \cite{Neumann}, but it can be put into normal form without changing the connectedness of the tree \cite[Theorem 4.1]{Neumann}. If $M_{\Gamma,v}$ is boundary compressible then it is a solid torus connect summed with some closed 3-manifold $Y$. We may choose a generic filling slope such that filling $M_{\Gamma,v}$ is a (nontrivial) lens space connect summed with $Y$ and such that this filling is represented by a reduced plumbing tree $\Gamma'$ obtained from $\Gamma$ by attaching a new linear chain of vertices to the distinguished vertex $v$. Since $\Gamma'$ is connected, we must have that $Y_{\Gamma'}$ is prime, so $Y \cong S^3$. It follows that $Y_{\Gamma'}$ is a lens space, but a reduced plumbing tree for a lens space has vertices of valence at most two, so the boundary vertex $v$ in $\Gamma$ is a leaf (i.e. has valence one) and all the other vertices of $\Gamma$ have valence at most two.
\end{proof}

Using the L-space condition for toroidal gluing (Proposition \ref{prop:Lspace-gluing-thm}), we will prove an L-space condition for the merge operation that we will need in our inductive computation of bordered invariants in Section \ref{sec:main-proof}.

\begin{lemma}\label{lem:Lspace-merge-condition}
Consider reduced rooted trees $\Gamma_1$ and $\Gamma_2$, each with more than one vertex, and let $\Gamma_{12} = \mathcal{M}(\Gamma_1, \Gamma_2)$. If $M_{\Gamma_{12}}$ has an L-space slope $r$ not in $\Z \cup \{\infty\}$ (this is in particular true if $M_{\Gamma_{12}}$ is Floer simple), then both $M_{\Gamma_{1}}$ and $M_{\Gamma_{2}}$ are Floer simple and at least one of them has $\infty$ as a strict L-space slope.
\end{lemma}

\begin{proof}
Let $\Gamma_0$ be a linear rooted plumbing tree for an $r$-framed solid torus; since $r \neq \infty$ we can choose $\Gamma_0$ to be reduced, and since $r \not\in \Z$ we have that $\Gamma_0$ has more than one vertex. Attaching $\Gamma_0$ to $\Gamma_{12}$ by identifying the distinguished vertices and adding their weights gives a plumbing tree for the closed graph manifold $M_{\Gamma_{12}}(r)$, which by assumption is an L-space. While this manifold is constructed by gluing a solid torus $M_{\Gamma_0}$ to $M_{\Gamma_{12}}$, we will consider two other ways of cutting this manifold along tori:
$$M_{\Gamma_{12}}(r) = M_{\Gamma_{12}} \cup M_{\Gamma_0} = M_{\Gamma_{02}} \cup M_{\Gamma_1} = M_{\Gamma_{01}} \cup M_{\Gamma_2},$$
where $\Gamma_{0i} = \mathcal{M}(\Gamma_0, \Gamma_i)$ for $i =1,2$. In each gluing, the parametrizing curve $\alpha$ of one side is identified with $\alpha$ of the other, and $\beta$ is identified with $-\beta$. Since $\Gamma_0$, $\Gamma_1$, and $\Gamma_2$ are all reduced with more than one vertex each, $\Gamma_{01}$ and $\Gamma_{02}$ are reduced with distinguished vertices of valence at least two.
By Proposition \ref{lem:boundary-compressible}, $M_{\Gamma_{02}}$ and $M_{\Gamma_{01}}$ are boundary incompressible. It follows from Proposition \ref{prop:Lspace-gluing-thm} that $M_{\Gamma_{1}}$ and $M_{\Gamma_{2}}$ are both Floer simple.

Now suppose $M_{\Gamma_{1}}$ does not have $\infty$ as a strict L-space slope; we will show that $M_{\Gamma_{2}}$ must have  $\infty$ as a strict L-space slope. Since $M_{\Gamma_{1}}$ does not have $\infty$ as a strict L-space slope, its bordered Floer invariant contains type $a$ and $b$ letters in loop calculus notation by Proposition \ref{prop:infty-strict-lspace-slope}. These type $a$ and $b$ letters remain after merging with $M_{\Gamma_0}$, a solid torus whose bordered invariant contains only type $c$ letters in loop calculus notation, so $M_{\Gamma_{01}}$ also does not have $\infty$ as a strict L-space slope. It then follows from Proposition \ref{prop:Lspace-gluing-thm} that $\infty$ is a strict L-space slope of $M_{\Gamma_{2}}$.
\end{proof}

\begin{remark}\label{rmk:Lspace-merge-condition}
If we remove the condition that $\Gamma_1$ and $\Gamma_2$ have more than one vertex, we can realize the twist operation as a degenerate case of the merge operation: in particular, if $\Gamma_2$ is a single vertex with weight $m$ then $\merge(\Gamma_1, \Gamma_2)$ is the same as $\twist^m(\Gamma_1)$. In this case, the statement in Lemma \ref{lem:Lspace-merge-condition} nearly holds except that $M_{\Gamma_1}$ may not be Floer simple. It is clear that $M_{\Gamma_2}$ is Floer simple and has $\infty$ as a strict L-space slope, since it is an integer framed solid torus, and since $M_{\Gamma_1}$ agrees with $M_{\Gamma_{12}}$ up to reparametrization it has at least one L-space slope, but it is Floer simple only if $M_{\Gamma_{12}}$ is.
\end{remark}

%% file: figures/1.tex

\begin{figure}[]
\centering
\begin{tikzpicture}[on top/.style={preaction={draw=white,-,line width=#1}},
on top/.default=4pt]
\def\U{1.5cm}
\def\adj{0.3cm}
\def\s{0.05cm}
\def\u{0.1cm}
\def\S{0.15cm}

\begin{scope}[xshift=-4.5cm]

\begin{scope}
\clip (-\U, -\U) rectangle (\U, \U);
\foreach \x in {-2*\U, 0, 2*\U}
\foreach \y in {-2*\U, 0, 2*\U} {
\begin{scope}[xshift=\x, yshift=\y]
\drawimmersedcurve
\end{scope}
}
\end{scope}

\draw[red, fill=red] (7*\s,0) circle (0.7*\s);
\draw[red] (0.5*\U-\s, -\s) -- (0.5*\U, 0) -- (0.5*\U-\s, \s);

\draw (-\U, -\U) rectangle (\U, \U);
\draw (\U,\U) node {$\bpt$};

\draw[->] (-\U, 0) -- (-\U, 0.5*\U);
\draw[->] (\U, 0) -- (\U, 0.5*\U);
\draw[->>] (0, \U) -- (0.5*\U, \U);
\draw[->>] (0, -\U) -- (0.5*\U, -\U);

\end{scope}

\begin{scope}[yshift=0]
	\begin{scope}
	\clip (-\U+3*\u, -2.25*\U) rectangle (3*\u, 2.25*\U);
		\drawimmersedcurvevar
		\begin{scope}[xshift=-\U, yshift=-2*\U]
		\drawimmersedcurvevar
		\end{scope}
		\begin{scope}[xshift=\U, yshift=2*\U]
		\drawimmersedcurvevar
		\end{scope}
	\end{scope}
	\draw (-\U+2*\S, -2.25*\U) -- (-\U+2*\S, 2.25*\U);
	\draw (2*\S, -2.25*\U) -- (2*\S, 2.25*\U);
	\draw (-0.5*\U,0.5*\U) node {$\bpt$};
	\draw (-0.5*\U,1.5*\U) node {$\bpt$};
	\draw (-0.5*\U,-0.5*\U) node {$\bpt$};
	\draw (-0.5*\U,-1.5*\U) node {$\bpt$};
	\draw[red, fill=red] (-\U+2*\S, -\U) circle (1.4*\s);
	\draw[red, fill=red] (2*\S, \U) circle (1.4*\s);
	\draw[red, ->] (-\U+2*\S, -\U) -- (-0.5*\U, -\U);
	\draw[->] (-\U+2*\S, -0.6*\U) -- (-\U+2*\S, -0.5*\U);
	\draw[->] (2*\S, 2*\U-0.6*\U) -- (2*\S, 2*\U-0.5*\U);
	\draw[->>] (-\U+2*\S, -1.6*\U) -- (-\U+2*\S, -1.5*\U);
	\draw[->>] (2*\S, 2*\U-1.6*\U) -- (2*\S, 2*\U-1.5*\U);
\end{scope}

\begin{scope}[xshift=4cm, xscale=0.5, yscale=0.5]
\draw[step=\U, lightgray, densely dotted, thin, xshift=0.5*\U, yshift=0.5*\U] (-3.75*\U,-3.75*\U) grid (2.75*\U,2.75*\U);
\draw[step=\U, lightgray, thin, xshift=0.5*\U, yshift=0.5*\U] (-3.5*\U,-3.5*\U) grid (2.5*\U,2.5*\U);
	\foreach \x in {-1, 0, 1} {
		\begin{scope}[xshift=\x*\U, yshift=2*\x*\U]
		\drawimmersedcurvevar
		\draw[red, fill=red] (-\U+2*\S, -\U) circle (1.4*\s);
		\end{scope}
	}
	\foreach \x in {-2, -1, ..., 3}
	\foreach \y in {-3, -2, ..., 2} {
		\draw (\x*\U-0.5*\U,\y*\U+0.5*\U) node {\tiny $\bpt$};
	}
	\draw[red, fill=red] (\U+2*\S, 3*\U) circle (1.4*\s);
	\draw[red, densely dotted] (\U+2*\S, 3*\U) -- (\U+6*\S, 3*\U);
	\draw[red, densely dotted] (-2*\U+2*\S, -3*\U) -- (-2*\U-2*\S, -3*\U);
\end{scope}

\end{tikzpicture}
\caption{An oriented immersed curve shown in $T_M$ (left), $\overline T_M$ (middle), and $\widetilde T_M$ (right) for $M$ the exterior of the right-hand trefoil $K$. By drawing these pictures in the plane, we have implicitly made a choice of parametrizing curves $\alpha, \beta$ for $\partial M$: $\alpha = -\mu - 2\lambda$ and $\beta = 2\mu + 3\lambda$, where $\mu$ and $\lambda$ are the meridian and Seifert longitude of $K$. Note that in the middle figure for $\overline T_M$ the opposite sides of the strip are identified with a vertical translation by two units so that the endpoints of the curves in the figure are identified. From the left figure, we get that the immersed curve can be represented by the cyclic word $(\beta \alpha \beta \alpha \beta^{-1} \alpha^{-2} \beta^{-1} \alpha \beta \alpha)$.}
\label{fig:immersed-curve-example}
\end{figure}

%% file: figures/2.tex

\begin{figure}[]
\centering
\begin{tikzpicture}[on top/.style={preaction={draw=white,-,line width=#1}},
on top/.default=4pt]
\def\U{1cm}
\def\adj{0.2cm}
\def\s{0.03cm}
\def\u{0.07cm}
\def\S{0.13cm}

\begin{scope}[xshift=-3*\U]
\draw[<->] (0,-1.5*\U) -- (0, 1.5*\U) node[midway, anchor=east] {$k$};
\end{scope}

\begin{scope}[xshift=-2*\U]
	
	\begin{scope}
	\clip (-0.75*\U, -2*\U) rectangle (0.75*\U, 2*\U);
	\draw[step=\U, lightgray, thin] (-0.5*\U, -1.75*\U) grid (0.5*\U, 1.75*\U);
	
	\draw (0,+\U) node {$\bpt$};
	\draw (0,0.1) node {\Huge $\vdots$};
	\draw (0,-\U) node {$\bpt$};
	
	\draw[red, densely dotted] (-2*\S, -1.5*\U) -- (0, -1.5*\U);
	\draw[red, fill=red] (0, -1.5*\U) circle (\s);
	\draw[red] (0, -1.5*\U) -- (2*\S, -1.5*\U);
	\draw[red] (2*\S, -1.5*\U) .. controls (3*\S, -1.5*\U) and (3*\S, -1.5*\U) .. 		(3*\S, -1.5*\U+\S);
	\draw[red, ->] (3*\S, -1.5*\U+\S) -- (3*\S, 0);
	\begin{scope}[yscale=-1]
	\draw[red, densely dotted] (-2*\S, -1.5*\U) -- (0, -1.5*\U);
	\draw[red, fill=red] (0, -1.5*\U) circle (\s);
	\draw[red] (0, -1.5*\U) -- (2*\S, -1.5*\U);
	\draw[red] (2*\S, -1.5*\U) .. controls (3*\S, -1.5*\U) and (3*\S, -1.5*\U) .. 		(3*\S, -1.5*\U+\S);
	\draw[red] (3*\S, -1.5*\U+\S) -- (3*\S, 0);
	\end{scope}
	\end{scope}
	
	\draw (0,-2*\U) node[] {\textcolor{red}{$a_k$}};
\end{scope}

\begin{scope}[xscale=-1]
	
	\begin{scope}
	\clip (-0.75*\U, -2*\U) rectangle (0.75*\U, 2*\U);
	\draw[step=\U, lightgray, thin] (-0.5*\U, -1.75*\U) grid (0.5*\U, 1.75*\U);
	
	\draw (0,+\U) node {$\bpt$};
	\draw (0,0.1) node {\Huge $\vdots$};
	\draw (0,-\U) node {$\bpt$};
	
	\draw[red, densely dotted] (-2*\S, -1.5*\U) -- (0, -1.5*\U);
	\draw[red, fill=red] (0, -1.5*\U) circle (\s);
	\draw[red] (0, -1.5*\U) -- (2*\S, -1.5*\U);
	\draw[red] (2*\S, -1.5*\U) .. controls (3*\S, -1.5*\U) and (3*\S, -1.5*\U) .. 		(3*\S, -1.5*\U+\S);
	\draw[red, ->] (3*\S, -1.5*\U+\S) -- (3*\S, 0);
	\begin{scope}[yscale=-1]
	\draw[red, densely dotted] (-2*\S, -1.5*\U) -- (0, -1.5*\U);
	\draw[red, fill=red] (0, -1.5*\U) circle (\s);
	\draw[red] (0, -1.5*\U) -- (2*\S, -1.5*\U);
	\draw[red] (2*\S, -1.5*\U) .. controls (3*\S, -1.5*\U) and (3*\S, -1.5*\U) .. 		(3*\S, -1.5*\U+\S);
	\draw[red] (3*\S, -1.5*\U+\S) -- (3*\S, 0);
	\end{scope}
	\end{scope}
	
	\draw (0,-2*\U) node[] {\textcolor{red}{$b_k$}};
\end{scope}

\begin{scope}[xshift=2*\U]
	
	\begin{scope}
	\clip (-0.75*\U, -2*\U) rectangle (1.75*\U, 2*\U);
	\draw[step=\U, lightgray, thin] (-0.5*\U, -1.75*\U) grid (1.5*\U, 1.75*\U);
	
	\draw (0,+\U) node {$\bpt$};
	\draw (0,0.1) node {\Huge $\vdots$};
	\draw (0,-\U) node {$\bpt$};
	\draw (\U,+\U) node {$\bpt$};
	\draw (\U,0.1) node {\Huge $\vdots$};
	\draw (\U,-\U) node {$\bpt$};
	
	\draw[red, densely dotted] (-2*\S, -1.5*\U) -- (0, -1.5*\U);
	\draw[red, fill=red] (0, -1.5*\U) circle (\s);
	\draw[red] (0, -1.5*\U) -- (0.5*\U-\S, -1.5*\U);
	\draw[red] (0.5*\U-\S, -1.5*\U) .. controls (0.5*\U, -1.5*\U) and (0.5*\U, -1.5*\U) .. 		(0.5*\U, -1.5*\U+\S);
	\draw[red, ->] (0.5*\U, -1.5*\U+\S) -- (0.5*\U, 0);
	\begin{scope}[yscale=-1, xscale=-1, xshift=-\U]
	\draw[red, densely dotted] (-2*\S, -1.5*\U) -- (0, -1.5*\U);
	\draw[red, fill=red] (0, -1.5*\U) circle (\s);
	\draw[red] (0, -1.5*\U) -- (0.5*\U-\S, -1.5*\U);
	\draw[red] (0.5*\U-\S, -1.5*\U) .. controls (0.5*\U, -1.5*\U) and (0.5*\U, -1.5*\U) .. 		(0.5*\U, -1.5*\U+\S);
	\draw[red] (0.5*\U, -1.5*\U+\S) -- (0.5*\U, 0);
	\end{scope}
	\end{scope}
	
	\draw (0.5*\U,-2*\U) node[] {\textcolor{red}{$c_k$}};
\end{scope}

\begin{scope}[xshift=5*\U]
	
	\begin{scope}
	\clip (-0.75*\U, -2*\U) rectangle (1.75*\U, 2*\U);
	\draw[step=\U, lightgray, thin] (-0.5*\U, -1.75*\U) grid (1.5*\U, 1.75*\U);
	
	\draw (0,+\U) node {$\bpt$};
	\draw (0,0.1) node {\Huge $\vdots$};
	\draw (0,-\U) node {$\bpt$};
	\draw (\U,+\U) node {$\bpt$};
	\draw (\U,0.1) node {\Huge $\vdots$};
	\draw (\U,-\U) node {$\bpt$};
	
	\draw[red, densely dotted] (-2*\S, -1.5*\U) -- (0, -1.5*\U);
	\draw[red, fill=red] (0, -1.5*\U) circle (\s);
	\draw[red] (0, -1.5*\U) -- (0.5*\U-\S, -1.5*\U);
	\draw[red] (0.5*\U-\S, -1.5*\U) .. controls (0.5*\U, -1.5*\U) and (0.5*\U, -1.5*\U) .. 		(0.5*\U, -1.5*\U+\S);
	\draw[red] (0.5*\U, -1.5*\U+\S) -- (0.5*\U, 0);
	\begin{scope}[yscale=-1, xscale=-1, xshift=-\U]
	\draw[red, densely dotted] (-2*\S, -1.5*\U) -- (0, -1.5*\U);
	\draw[red, fill=red] (0, -1.5*\U) circle (\s);
	\draw[red] (0, -1.5*\U) -- (0.5*\U-\S, -1.5*\U);
	\draw[red] (0.5*\U-\S, -1.5*\U) .. controls (0.5*\U, -1.5*\U) and (0.5*\U, -1.5*\U) .. 		(0.5*\U, -1.5*\U+\S);
	\draw[red, ->] (0.5*\U, -1.5*\U+\S) -- (0.5*\U, 0);
	\end{scope}
	\end{scope}
	
	\draw (0.5*\U,-2*\U) node[] {\textcolor{red}{$\overline{c}_k$}};
\end{scope}

\end{tikzpicture}
\caption{The segments of an immersed curve corresponding to letters in loop calculus notation.}
\label{fig:loop-calculus-curve-segments}
\end{figure}

%% file: figures/3.tex

\begin{figure}[]
\centering
\begin{tikzpicture}[on top/.style={preaction={draw=white,-,line width=#1}},
on top/.default=4pt]
\def\U{1.5cm}
\def\adj{0.3cm}
\def\s{0.05cm}
\def\u{0.1cm}
\def\S{0.15cm}

\begin{scope}[xshift=-4.5cm]

\begin{scope}
\clip (-\U, -\U) rectangle (\U, \U);
\draw[fill=black, opacity=0.1] (-3*\s, -10*\s) .. controls (-3*\s, -8*\s) and (-3*\s, -7*\s) ..
      (-3*\s, -6*\s) .. controls (-3*\s, -6*\s+\s) and (-7*\s+2*\s, -4*\s) ..
      (-7*\s, -4*\s) --
      (-\U, -4*\s) --
      (-\U, -10*\s) --
      cycle;
\draw[fill=black, opacity=0.1] (5*\s, -10*\s) --
      (5*\s, -6*\s) .. controls (5*\s, -5*\s) and (6*\s, -4*\s) ..
      (7*\s, -4*\s) --
      (\U, -4*\s) --
      (\U, -10*\s) --
      cycle;
\foreach \x in {-2*\U, 0, 2*\U}
\foreach \y in {-2*\U, 0, 2*\U} {
\begin{scope}[xshift=\x, yshift=\y]
\drawimmersedcurve
\end{scope}
}

\draw[blue] (-\U, -10*\s) -- (\U, -10*\s);
\draw[blue] (0.5*\U,-10*\s) node[anchor=north] {\scriptsize $\beta_A$};
\end{scope}

\draw[red, fill=red] (7*\s,0) circle (0.7*\s);
\draw[red] (0.5*\U-\s, -\s) -- (0.5*\U, 0) -- (0.5*\U-\s, \s);

\draw (-\U, -\U) rectangle (\U, \U);
\draw (\U,\U) node {$\bpt$};

\draw[->] (-\U, 0) -- (-\U, 0.5*\U);
\draw[->] (\U, 0) -- (\U, 0.5*\U);
\draw[->>] (0, \U) -- (0.5*\U, \U);
\draw[->>] (0, -\U) -- (0.5*\U, -\U);

\end{scope}

\begin{scope}[yshift=0.25*\U]
\begin{scope}[xscale=0.5, yscale=0.5]
	\begin{scope}
	\clip (-2*\U+6*\u, -2.5*\U) rectangle (6*\u, 1.5*\U);
		\begin{scope}[yshift=-2*\U]
		\draw[fill=black, opacity=0.1] (-3*\s, -10*\s) .. controls (-3*\s, -8*\s) and (-3*\s, -7*\s) ..
      		(-3*\s, -6*\s) .. controls (-3*\s, -6*\s+\s) and (-7*\s+2*\s, -4*\s) ..
      		(-7*\s, -4*\s) --
      		(-2*\U, -4*\s) --
      		(-2*\U, -10*\s) --
      		cycle;
		\draw[fill=black, opacity=0.1] (5*\s, -10*\s) --
      		(5*\s, -6*\s) .. controls (5*\s, -5*\s) and (6*\s, -4*\s) ..
      		(7*\s, -4*\s) --
      		(\U, -4*\s) --
      		(\U, -10*\s) --
      		cycle;
		\end{scope}
	\foreach \x in {-2, 0, 2}
	\foreach \y in {-4,0,4} {
	\begin{scope}[xshift=\x*\U, yshift=\y*\U]
	\drawimmersedcurve
	\end{scope}
	}
	\draw[green] (-2*\U, -10*\s) -- (\U, -10*\s);
	\draw[blue] (-2*\U, -2*\U-10*\s) -- (\U, -2*\U-10*\s);
	\draw (-\U,-\U) node {$\bpt$};
	\draw (-\U,\U) node {$\bpt$};
	\end{scope}
	\draw (-2*\U+6*\u, -2.5*\U) rectangle (6*\u, 1.5*\U);
	\draw[green] (0.8*\U, -10*\s) node {\scriptsize $\overline{\beta_A}$};
	\draw[blue] (0.8*\U, -2*\U-10*\s) node {\scriptsize $\overline{\beta_A}$};
	
	\draw[->] (-2*\U+6*\u, -\U) -- (-2*\U+6*\u, -0.5*\U);
	\draw[->] (6*\u, -\U) -- (6*\u, -0.5*\U);
	\draw[->>] (-\U, -2.5*\U) -- (-\U+6*\u, -2.5*\U);
	\draw[->>] (-\U, 1.5*\U) -- (-\U+6*\u, 1.5*\U);
\end{scope}
\end{scope}

\begin{scope}[xshift=4.5cm, xscale=0.5, yscale=0.5]

\draw[step=\U, lightgray, densely dotted, thin, xshift=0.5*\U, yshift=0.5*\U] (-3.75*\U,-3.75*\U) grid (2.75*\U,2.75*\U);
\draw[step=\U, lightgray, thin, xshift=0.5*\U, yshift=0.5*\U] (-3.5*\U,-3.5*\U) grid (2.5*\U,2.5*\U);

	\draw[fill=black, opacity=0.1] (-\U, -\U-3*\S) --
	(-\U, -\U-2*\S) .. controls (-\U, -\U-\S) and (-\U, -\U-\S) ..
	(-\U+\S, -\U-\S) --
	(-\S, -\U-\S) .. controls (0, -\U-\S) and (0, -\U-\S) ..
	(0, -\U-2*\S) --
	(0, -\U-3*\S) --
	cycle;
	
	\foreach \x in {-1, 0, 1} {
		\begin{scope}[xshift=\x*\U, yshift=2*\x*\U]
		\drawimmersedcurvevar
		\draw[red, fill=red] (-\U+2*\S, -\U) circle (1.4*\s);
		\end{scope}
	}
	\foreach \x in {-2, -1, ..., 3}
	\foreach \y in {-3, -2, ..., 2} {
		\draw (\x*\U-0.5*\U,\y*\U+0.5*\U) node {\tiny $\bpt$};
	}
	\draw[red, fill=red] (\U+2*\S, 3*\U) circle (1.4*\s);
	\draw[red, densely dotted] (\U+2*\S, 3*\U) -- (\U+6*\S, 3*\U);
	\draw[red, densely dotted] (-2*\U+2*\S, -3*\U) -- (-2*\U-2*\S, -3*\U);
	
	\draw[green] (-3*\U, -3*\S) -- (3*\U, -3*\S) node[anchor=west] {\scriptsize $\widetilde{\beta_A}$};
	\draw[blue] (-3*\U, -\U-3*\S) -- (3*\U, -\U-3*\S) node[anchor=west] {\scriptsize $\widetilde{\beta_A}$};
\end{scope}

\end{tikzpicture}
\caption{Pairing the immersed curve from Figure \ref{fig:immersed-curve-example}, in (transverse) rectilinear position, with the representative $\beta_A$ of the homology class of $\beta$. The pairing is shown in $T_{M_2}$, $\overline{T}_Y$, and $\widetilde{T}_{M_2}$ from left to right, with two lifts of $\beta_A$ shown in the covering spaces. Note that there is one intersection for each vertical segment in the curve but the curves are not in minimal position; a bigon between two intersection points is shaded. }
\label{fig:pairing-example}
\end{figure}

%% file: figures/4.tex

\begin{figure}[]
\centering
\begin{tikzpicture}[on top/.style={preaction={draw=white,-,line width=#1}},
on top/.default=4pt]
\def\U{1.5cm}
\def\adj{0.3cm}
\def\s{0.05cm}
\def\u{0.1cm}
\def\S{0.15cm}

\begin{scope}[xshift=-4.5cm]
\drawimmersedcurvesym
	\draw (-0.5*\U,-0.5*\U) node {$\bpt$};
	\draw (-0.5*\U,0.5*\U) node {$\bpt$};
	\draw (0.5*\U,-0.5*\U) node {$\bpt$};
	\draw (0.5*\U,0.5*\U) node {$\bpt$};
\end{scope}

\begin{scope}[yshift=0.25*\U]
	\begin{scope}[xshift=0.5, yshift=0.5]
	\begin{scope}
	\clip (-\U+3*\u, -1.25*\U) rectangle (3*\u, 0.75*\U);
	\foreach \x in {-1, 0, 1}
	\foreach \y in {-2,0,2} {
	\begin{scope}[xshift=\x*\U, yshift=\y*\U]
	\drawimmersedcurvesym
	\end{scope}
	}
	\draw[green] (-2*\U, 0) -- (\U, 0);
	\draw[blue] (-2*\U, -\U) -- (\U, -\U);
	\end{scope}
	\draw (-0.5*\U,-0.5*\U) node {$\bpt$};
	\draw (-0.5*\U,0.5*\U) node {$\bpt$};
	\draw (-\U+3*\u, -1.25*\U) rectangle (3*\u, 0.75*\U);
	\draw[green] (0.5*\U, 0) node {\scriptsize $\overline{\beta_{sym}}$};
	\draw[blue] (0.5*\U, -\U) node {\scriptsize $\overline{\beta_{sym}}$};
	\draw[->] (-\U+3*\u, -0.5*\U) -- (-\U+3*\u, -0.25*\U);
	\draw[->] (3*\u, -0.5*\U) -- (3*\u, -0.25*\U);
	\draw[->>] (-0.5*\U, -1.25*\U) -- (-0.5*\U+3*\u, -1.25*\U);
	\draw[->>] (-0.5*\U, 0.75*\U) -- (-0.5*\U+3*\u, 0.75*\U);
	\end{scope}
\end{scope}

\begin{scope}[xshift=4.5cm, xscale=0.5, yscale=0.5]

\draw[step=\U, lightgray, densely dotted, thin, xshift=0.5*\U, yshift=0.5*\U] (-3.75*\U,-3.75*\U) grid (2.75*\U,2.75*\U);
\draw[step=\U, lightgray, thin, xshift=0.5*\U, yshift=0.5*\U] (-3.5*\U,-3.5*\U) grid (2.5*\U,2.5*\U);
	
	\begin{scope}
	\clip (-3*\U, -3.3*\U) rectangle (3*\U, 3.3*\U);
	\foreach \x in {-2,-1, 0, 1, 2} {
		\begin{scope}[xshift=\x*\U, yshift=2*\x*\U]
		\drawimmersedcurvesym
		\end{scope}
	}
	\foreach \x in {-2, -1, ..., 3}
	\foreach \y in {-3, -2, ..., 2} {
		\draw (\x*\U-0.5*\U,\y*\U+0.5*\U) node {\tiny $\bpt$};
	}
	\end{scope}
	\draw[green] (-3*\U, 0) -- (3*\U, 0) node[anchor=west] {\scriptsize $\widetilde{\beta_{sym}}$};
	\draw[blue] (-3*\U, -\U) -- (3*\U, -\U) node[anchor=west] {\scriptsize $\widetilde{\beta_{sym}}$};
\end{scope}

\end{tikzpicture}
\caption{The pairing diagram from Figure \ref{fig:pairing-example} in $\overline T_Y$ (left) and $\widetilde T_{M_2}$ (right), with the curve $\HFhat(M_2)$ in symmetric $\beta$-pairing position. The curve $\HFhat(M_2)$ is red but we highlight the horizontal segments by coloring them violet. The curves are in minimal position, and we see by counting intersection points that $\dim \HFhat(Y; \s)$ is 1 for the \spinc stucture corresponding to the blue lift of $\beta_{sym}$ curve and 3 for the \spinc structure corresponding to the green lift of $\beta_{sym}$. }
\label{fig:symmetric-pairing-position}
\end{figure}

%% file: figures/plumbing-tree-moves-inkscape-version.pdf_tex
\begingroup%
  \makeatletter%
  \providecommand\color[2][]{%
    \errmessage{(Inkscape) Color is used for the text in Inkscape, but the package 'color.sty' is not loaded}%
    \renewcommand\color[2][]{}%
  }%
  \providecommand\transparent[1]{%
    \errmessage{(Inkscape) Transparency is used (non-zero) for the text in Inkscape, but the package 'transparent.sty' is not loaded}%
    \renewcommand\transparent[1]{}%
  }%
  \providecommand\rotatebox[2]{#2}%
  \ifx\svgwidth\undefined%
    \setlength{\unitlength}{714.64952844bp}%
    \ifx\svgscale\undefined%
      \relax%
    \else%
      \setlength{\unitlength}{\unitlength * \real{\svgscale}}%
    \fi%
  \else%
    \setlength{\unitlength}{\svgwidth}%
  \fi%
  \global\let\svgwidth\undefined%
  \global\let\svgscale\undefined%
  \makeatother%
  \begin{picture}(1,0.20677747)%
    \put(0,0){\includegraphics[width=\unitlength,page=1]{plumbing-tree-moves-inkscape-version.pdf}}%
    \put(0.22092752,0.11792453){\color[rgb]{0,0,0}\makebox(0,0)[lb]{\smash{$\Gamma_1$}}}%
    \put(0.11548704,0.12636184){\color[rgb]{0,0,0}\makebox(0,0)[lb]{\smash{$n_1 +m$}}}%
    \put(0.14862585,0.00840459){\color[rgb]{0,0,0}\makebox(0,0)[lb]{\smash{$\mathcal{T}^m(\Gamma_1)$}}}%
    \put(0,0){\includegraphics[width=\unitlength,page=2]{plumbing-tree-moves-inkscape-version.pdf}}%
    \put(0.57061809,0.11266659){\color[rgb]{0,0,0}\makebox(0,0)[lb]{\smash{$\Gamma_1$}}}%
    \put(0.49246342,0.12530171){\color[rgb]{0,0,0}\makebox(0,0)[lb]{\smash{$n_1$}}}%
    \put(0.4941185,0.00734445){\color[rgb]{0,0,0}\makebox(0,0)[lb]{\smash{$\mathcal{E}(\Gamma_1)$}}}%
    \put(0,0){\includegraphics[width=\unitlength,page=3]{plumbing-tree-moves-inkscape-version.pdf}}%
    \put(0.41519534,0.1240341){\color[rgb]{0,0,0}\makebox(0,0)[lb]{\smash{$0$}}}%
    \put(0,0){\includegraphics[width=\unitlength,page=4]{plumbing-tree-moves-inkscape-version.pdf}}%
    \put(0.87453441,0.17236765){\color[rgb]{0,0,0}\makebox(0,0)[lb]{\smash{$\Gamma_1$}}}%
    \put(0.75142311,0.14231563){\color[rgb]{0,0,0}\makebox(0,0)[lb]{\smash{$n_1 + n_2$}}}%
    \put(0.7956477,0.00717323){\color[rgb]{0,0,0}\makebox(0,0)[lb]{\smash{$\mathcal{M}(\Gamma_1, \Gamma_2)$}}}%
    \put(0,0){\includegraphics[width=\unitlength,page=5]{plumbing-tree-moves-inkscape-version.pdf}}%
    \put(0.86023169,0.06313321){\color[rgb]{0,0,0}\makebox(0,0)[lb]{\smash{$\Gamma_2$}}}%
  \end{picture}%
\endgroup%

%% file: figures/FigVertSumEx1.pdf_tex
\begingroup%
  \makeatletter%
  \providecommand\color[2][]{%
    \errmessage{(Inkscape) Color is used for the text in Inkscape, but the package 'color.sty' is not loaded}%
    \renewcommand\color[2][]{}%
  }%
  \providecommand\transparent[1]{%
    \errmessage{(Inkscape) Transparency is used (non-zero) for the text in Inkscape, but the package 'transparent.sty' is not loaded}%
    \renewcommand\transparent[1]{}%
  }%
  \providecommand\rotatebox[2]{#2}%
  \ifx\svgwidth\undefined%
    \setlength{\unitlength}{511.81673786bp}%
    \ifx\svgscale\undefined%
      \relax%
    \else%
      \setlength{\unitlength}{\unitlength * \real{\svgscale}}%
    \fi%
  \else%
    \setlength{\unitlength}{\svgwidth}%
  \fi%
  \global\let\svgwidth\undefined%
  \global\let\svgscale\undefined%
  \makeatother%
  \begin{picture}(1,0.92107273)%
    \put(0,0){\includegraphics[width=\unitlength,page=1]{FigVertSumEx1.pdf}}%
    \put(0.02553825,0.55431998){\color[rgb]{0,0,0}\makebox(0,0)[lb]{\smash{$\textcolor{blue}{\widetilde{\gamma}_1}$}}}%
    \put(0.02461165,0.34823634){\color[rgb]{0,0,0}\makebox(0,0)[lb]{\smash{$\textcolor{red}{\widetilde{\gamma}_2}$}}}%
    \put(0.01402467,0.84802165){\color[rgb]{0,0,0}\makebox(0,0)[lb]{\smash{$\textcolor{violet}{\widetilde{\gamma}_2 +_v \widetilde{\gamma}_1}$}}}%
    \put(0.00376393,0.05711508){\color[rgb]{0,0,0}\makebox(0,0)[lb]{\smash{$y=\frac{1}{2}$}}}%
    \put(0,0){\includegraphics[width=\unitlength,page=2]{FigVertSumEx1.pdf}}%
  \end{picture}%
\endgroup%

%% file: figures/FigVertSumEx2.pdf_tex
\begingroup%
  \makeatletter%
  \providecommand\color[2][]{%
    \errmessage{(Inkscape) Color is used for the text in Inkscape, but the package 'color.sty' is not loaded}%
    \renewcommand\color[2][]{}%
  }%
  \providecommand\transparent[1]{%
    \errmessage{(Inkscape) Transparency is used (non-zero) for the text in Inkscape, but the package 'transparent.sty' is not loaded}%
    \renewcommand\transparent[1]{}%
  }%
  \providecommand\rotatebox[2]{#2}%
  \ifx\svgwidth\undefined%
    \setlength{\unitlength}{240.46044537bp}%
    \ifx\svgscale\undefined%
      \relax%
    \else%
      \setlength{\unitlength}{\unitlength * \real{\svgscale}}%
    \fi%
  \else%
    \setlength{\unitlength}{\svgwidth}%
  \fi%
  \global\let\svgwidth\undefined%
  \global\let\svgscale\undefined%
  \makeatother%
  \begin{picture}(1,3.27589579)%
    \put(0,0){\includegraphics[width=\unitlength,page=1]{FigVertSumEx2.pdf}}%
  \end{picture}%
\endgroup%

%% file: sections/gradings.tex
\subsection{Gradings}\label{sec:gradings}

We now review gradings in bordered Floer homology, which we will need to compute $d$-invariants. We begin by reviewing the grading package in the traditional formulation of bordered Floer homology in Section \ref{sec:gradings-traditional}, before discussing a geometric interpretation of grading differences in Sections \ref{sec:grading-curves-same-spinc} and \ref{sec:rational-grading-curves}. We note that the material in Section \ref{sec:gradings-traditional} will not be needed outside of the proof of Proposition \ref{prop:rational-grading-lemma}, and that Proposition \ref{prop:rational-grading-lemma} is more general than we need in our arguments. The trusting reader may wish to skip this subsection entirely, noting only the statements of Corollaries \ref{cor:symmetry-same-grading} and \ref{cor:grading-difference-for-fixed-points}.

\subsubsection{Gradings in bordered Floer homology} \label{sec:gradings-traditional}

For a closed three manifold $Y$ and a torsion \spinc structure $\s$, $\HFhat(Y; \s)$ admits an absolute $\Q$-grading. In particular, when $Y$ is a rational homology sphere $\HFhat(Y)$ admits an absolute $\Q$-grading across all \spinc structures. When bordered Floer homology is used to compute $\HFhat$ of $Y = M_1 \cup_\phi M_2$ we are able to recover this grading, though only up to an overall shift, giving a relative $\Q$-grading on $\HFhat(Y)$. We sketch here how this relative grading is computed; see \cite[Section 11.1]{LOT} for more details.

For a manifold $M$ with torus boundary, the type D structure $\CFD(M)$ defined in \cite{LOT} is equipped with a relative grading $\gr_D$ by a quotient of a non-commutative group $G$. The elements of $G$ can be represented by triples $(m;a,b)$ where $m,a,b \in \frac 1 2 \Z$ and $a+b \in \Z$; the first component is called the Maslov component and the remaining two components are the spin$^c$ components. The group operation for this group is given by
$$(m_1; a_1, b_1) \cdot (m_2; a_2, b_2) = (m_1 + m_2 + \left|\begin{smallmatrix} a_1 & b_1 \\ a_2 & b_2 \end{smallmatrix}\right|; a_1+a_2, b_1+b_2).$$
After fixing a distinguished generator $x^D_0$ of $\CFD(M)$, which we declare to have grading $(0;0,0)$, the grading $\gr_D$ takes values in $G$ modulo right multiplication by $P(x^D_0)$, where $P(x^D_0)$ is the set of gradings on the periodic domains at $x^D_0$. Note that for rational homology tori $P(x^D_0)$ is cyclic. The grading $\gr_D$ satisfies $\gr_D(\partial x) = \lambda^{-1} \gr_D(x)$, where $\lambda = (1;0,0)$, and $\gr_D(\rho_I \otimes x) = \gr(\rho_I) \cdot \gr_D(1\otimes x)$, where the grading $\gr$ on the torus algebra is determined by
$$\gr(\rho_1) = (-\frac 1 2; \frac 1 2, -\frac 1 2), \quad \gr(\rho_2) = (-\frac 1 2; \frac 1 2, \frac 1 2), \quad \gr(\rho_3) = (-\frac 1 2; -\frac 1 2, \frac 1 2)$$
and the fact that $\gr(\rho_I \rho_J) = \gr(\rho_I)\gr(\rho_J)$. These rules determine the grading on all generators that can be connected to $x^D_0$ by a sequence of operations. In practice it is possible to extract a generator of $P(x^D_0)$ from $\CFD$ by computing the total change in grading along a loop starting and ending at $x^D_0$, which must be an element of $P(x^D_0)$. The grading $\gr_A$ on $\CFA(M)$ can be obtained from $\gr_D$ (for corresponding generators or periodic domains) by switching the two spin$^c$ components and multiplying the spin$^c$ components by $-1$. We let $x^A_0$ denote the generator of $\CFA$ corresponding to $x^D_0$ and let $P(x^A_0)$ denote the corresponding set of periodic domain gradings; $\gr_A$ takes values in $G$ modulo left multiplication by $P(x^A_0)$. In this paper we will primarily be concerned with the case that $(M_1, \alpha_1, \beta_1)$ is $(D^2\times S^1, m, \ell)$, in which case $\CFA(M_1, \alpha_1, \beta_1)$ has a single generator $x_0^A$ of grading $(0;0,0)$ and $P(x_0^A)$ is generated by $(-\frac 1 2; 1, 0)$.

Given two bordered manifolds $(M_1, \alpha_1, \beta_1)$ and $(M_2, \alpha_2, \beta_2)$ whose boundaries are glued via a map $\phi$ taking $\alpha_1$ to $\beta_2$ and $\beta_1$ to $\alpha_2$, the pairing theorem in \cite{LOT} gives that
$$\HFhat(M_1 \cup M_2) \cong H_*\left(\CFA(M_1, \alpha_1, \beta_1) \boxtimes \CFD(M_2, \alpha_1, \beta_2) \right).$$
The box tensor product admits a grading $\gr^\boxtimes$ defined by
$$\gr^\boxtimes(x^A\otimes x^D) = \gr_A(x^A) \cdot \gr_D(x^D),$$
which takes values in $P(x^A_0) \backslash G / P(x^D_0)$ where $x^A_0$ and $x^D_0$ are distinguished generators in $\CFA(M_1,\alpha_1,\beta_1)$ and $\CFD(M_2, \alpha_2, \beta_2)$, respectively. Let $P_A$ and $P_D$ denote generators of $P(x^A_0)$ and $P(x^D_0)$. If $M_1 \cup_\phi M_2$ is a rational homology sphere then the spin$^c$ components of $P_A$ and $P_D$ are linearly independent vectors in $\left(\frac 1 2 \Z\right)^2$ and span a lattice; we can uniquely represent each element of $P(x^A_0) \backslash G / P(x^D_0)$ by an element of $G$ whose spin$^c$ components lie in a fundamental domain of the lattice. In this form $\gr^\boxtimes$ recovers both the spin$^c$ decomposition of $\HFhat(M_1 \cup M_2)$ and the relative Maslov grading in each spin$^c$ structure: generators in the same spin$^c$ structure have the same spin$^c$ components, and for generators in the same spin$^c$ structure as $x^A_0 \otimes x^D_0$ the Maslov component of $\gr^\boxtimes$ gives the relative Maslov grading. In other words, to compute the relative Maslov grading $\gr$ of $x^A \otimes x^D$ in the same spin$^c$ structure as $x^A_0 \otimes x^D_0$, we compute $\gr_A(x^A) \cdot \gr_D(x^D)$, multiply on the left by an appropriate integral power of $P_A$ and on the right by an appropriate integral power of $P_D$ such that the resulting spin$^c$ components are zero, and then read off the Maslov component. For a rational homology sphere, it was shown in \cite{LOTQGradings} that this same procedure also recovers the relative $\mathbb{Q}$-grading between different spin$^c$ structures if we allow fractional powers of $P_A$ and $P_D$, where we define $(m;a,b)^r$ to be $(rm; ra, rb)$ for any rational $r$. We will be particularly interested in the rational grading differences between different \spinc structures.


\subsubsection{Relative grading within a \spinc structure via immersed curves}\label{sec:grading-curves-same-spinc}
Fortunately the relative grading on $\HFhat(M_1 \cup_{\phi} M_2)$ admits a nice description in terms of immersed curves. We begin by considering the grading difference between two generators in the same \spinc structure of $Y$, for which the grading difference is integral. By Theorem 5 of \cite{HRW-companion}, the relative grading on $\HFhat(Y)$ in each spin$^c$ structure agrees with the standard relative grading for Floer homology of immersed curves under the isomorphism in Equation \eqref{eq:pairing}; we will now review the relevant relative grading on the Floer homology of curves. We restrict to the special case that both generators correspond to intersection points of a single component of $\phi(\HFhat(M_1))$ and a single component of $\HFhat(M_2)$; the general case is similar but requires the grading decorations on multicurves that we will not discuss.

Before describing the relative grading, we must introduce some notation. Let $x$ and $y$ be intersection points of $\phi(\HFhat(M_1, \s_1))$ and $\HFhat(M_2, \s_2)$ in $T_{M_2}$ that lie on the same component of $\phi(\HFhat(M_1, \s_1))$ and on the same component of $\HFhat(M_2, \s_2)$. Consider paths $P_1$ from $y$ to $x$ in $\phi(\HFhat(M_1, \s_1))$ and $P_2$ from $x$ to $y$ in $\HFhat(M_2, \s_2)$ and consider the closed path $P = P_1 \cdot P_2$ (that is, the composition of the paths formed by $P_1$ followed by $P_2$). The generators $x$ and $y$ are in the same \spinc component if and only if $P_1$ and $P_2$ can be chosen such that $P$ is nullhomologous in the torus $\partial M_2$; this ensures that $P$ lifts to a closed path in $\widetilde T_{M_2}$. We define the rotation numbers $\rot(P_i)$ of the path $P_i$ to be the total rotation of the tangent vector when traversing the path $P_i$, measured in radians divided by $\pi$. We also define $\theta(x)$ and $\theta(y)$ to be the angles between $\phi(\HFhat(M_1))$ and $\HFhat(M_2)$ at $x$ and $y$, respectively, specifically the angle traced out by counterclockwise rotation from $\phi(\HFhat(M_1))$ to $\HFhat(M_2)$ (again measured in radians divided by $\pi$). We then define
$$\rot(P) = \rot(P_1) + \theta(x) + \rot(P_2) - \theta(y).$$
Note that if we arrange the curves to intersect orthogonally then we simply have
$$\rot(P) = \rot(P_1)+\rot(P_2).$$
With this notation in place we can define the grading difference between the two generators to be
$$\gr(y) - \gr(x) = 2w(P) - \rot(P)$$ 
where $w(P)$ is the winding number of $P$ around the puncture in $T_{M_2}$. To summarize, combining these definitions with \cite[Theorem 5]{HRW-companion} gives the following:

\begin{proposition}\label{prop:old-grading-lemma}
Let $x$ and $y$ be generators of $\HFhat(M_1 \cup_\phi M_2; \s)$ for some $\s$ such that $x$ and $y$ are represented by intersection points on the same component of $\phi(\HFhat(M_1, \s_1))$ and $\HFhat(M_2, \s_2)$, and choose $P_1$, $P_2$, and $P$ as above such that $P$ is nullhomologous in the torus $\partial M_2$. Then the grading difference is given by
$$\gr(y) - \gr(x) = 2w(P) - \rot(P).$$ 
\end{proposition}
\noindent  In particular, if $P$ bounds a bigon from $x$ to $y$ then $\gr(y) - \gr(x) = 2w(P) - 1$.

\subsubsection{Rational grading differences via immersed curves}\label{sec:rational-grading-curves}

The goal of the rest of this section is to extend the geometric interpretation of the integral grading differences in Proposition \ref{prop:old-grading-lemma} to describe the rational grading difference between two generators of $\HFhat(M_1 \cup_\phi M_2)$ in different \spinc structures of a rational homology sphere $M_1 \cup_\phi M_2$ that restrict to the same \spinc structures on $M_1$ and $M_2$. For simplicity we will only do this in the case that $M_1$ is a solid torus whose meridian is identified with $\beta$ in $\partial M_2$; we also continue to restrict to the case of generators lying on the same component of each immersed multicurve, and further assume that this component is homologically nontrivial. We let $\gamma$ be the component of $\HFhat(M_2)$ containing the two points, so that we are interested in computing relative gradings for the (intersection) Floer homology of $\gamma$ with $\beta$. Much of this subsection is devoted to setting up notation to formally state and prove the result, but we begin with an informal statement: to compute the rational grading difference $\gr(y) - \gr(x)$ we construct a three sided path in the plane $\widetilde T_{M_2}$ for which one side is the curve $\gamma$ repeated some number $\ell$ times, one side is the subpath $\rho_{x,y}$ of $\gamma$ from $x$ to $y$ repeated some number $n$ times, and the third side is a horizontal straight line; the grading difference is twice the signed area enclosed by this path (assuming the curves are in rectilinear position) minus twice the area of a triangle determined by the three vertices of this path minus the net rotation along the path, divided by $n$.

We will assume that $\gamma$ is in symmetric $\beta$-pairing position. The result will be stated in terms of symmetric pairing position; that is, we consider the Floer chain complex of $\gamma$ with $\beta_{sym}$, the particular representative of $[\beta]$ given by the horizontal curve $[0,1] \times \left\{ \frac 1 2 \right\}$. However, in the proof we will need to consider the Floer chain complex of $\gamma$ with $\beta_A$, the representative of $[\beta]$ given by $[0,1] \times \left\{ \frac 1 2 - \epsilon\right\}$, since this Floer complex directly relates to the box tensor product of type A and D modules. We now introduce some notation to relate the generators of these two complexes. Recall that generators of $\HF(\gamma, \beta_{sym})$ are in one-to-one correspondence with intersections between $\gamma$ and $\beta_{sym}$ (since the curves are in minimal position) and these occur at the midpoints of maximal horizontal segments in $\gamma$ for which exactly one of the two adjacent vertical segments lies below the maximal horizontal segment in the cover $\widetilde T_{M_2}$.
Given such a maximal horizontal segment midpoint $x$ on $\gamma$, we let $x^D$ denote the intersection of $\gamma$ with $\beta_A$ that lies on the adjacent unit vertical segment that is below the maximal horizontal segment (see Figure \ref{fig:labeling-generators}). Recall that each unit vertical segment of $\gamma$ corresponds to an $\iota_1$-generator of the type D structure represented by $\gamma$. It follows that $x^D$, lying on a vertical segment of $\gamma$, also corresponds to some $\iota_1$-generator of the type D structure; we will use $\bar x^D$ to denote the this generator of the type D structure. By slight abuse of notation, we will use $x$ and $x^D$ to refer both to points on $\gamma$ and to the corresponding generators of $\CF(\gamma, \beta_{sym})$ and $\CF(\gamma, \beta_A)$, respectively.

%
%
%

Because symmetric $\beta$-pairing position lies within an $(\epsilon/2)$-neighborhood of rectilinear position, it follows from the discussion in Section \ref{sec:preferred-forms} that the Floer chain complex $\CF(\gamma, \beta_A)$ is isomorphic to the box tensor product of the type A module corresponding to $\beta$ and the type D module corresponding to $\gamma$. Under this isomorphism, the generator $x^D$ of $\CF(\gamma, \beta_A)$ corresponds to the generator $\bar x_0^A \otimes \bar x^D$ in the box tensor product, where $\bar x_0^A$ is the unique generator of the type A module associated with $\beta$.
Note that $\beta_{sym}$ can be continuously deformed into $\beta_A$ in such a way that the intersection with $\gamma$ moves continuously from $x$ to $x^D$, so the corresponding generators $x$ of $\CF(\gamma, \beta_{sym})$ and $x^D$ of $\CF(\gamma, \beta_A)$ are related. In particular they have the same grading, so $\gr(x) = \gr(x^D) = \gr(\bar x_0^A \otimes \bar x^D)$.

\begin{figure}[ht]
\fontsize{18pt}{24pt}
\scalebox{.6}{%
    \def\svgwidth{\columnwidth}
    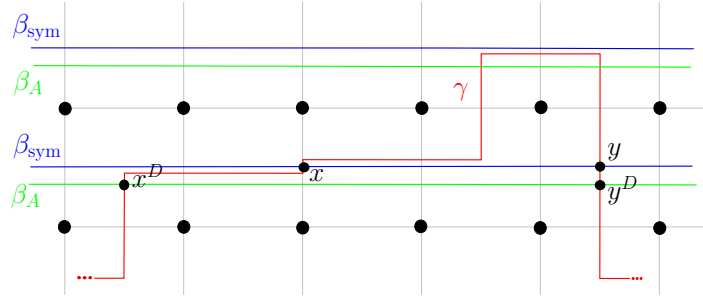
}
\caption{Two intersection points $x$ and $y$ of $\gamma$ with $\beta_{sym}$ and the corresponding intersections $x^D$ and $y^D$ of $\gamma$ with $\beta_A$.}
\label{fig:labeling-generators}
\end{figure}

%

Fixing one intersection $x$ of $\gamma$ with $\beta_{sym}$ we can view $\gamma$ as a closed path starting and ending at $x$; we'll denote this path by $\gamma_x$. Let $\langle b_\gamma, a_\gamma \rangle$ be the vector representing $[\gamma]$ in terms of the  basis $\{\beta, \alpha\}$ for $H_1(\partial M_2)$, so that the path $\gamma$ lifts to a path in $\widetilde T_{M_2} \cong \R^2 \setminus \Z^2$ that moves $b_\gamma$ units to the right and $a_\gamma$ units up. We will also consider the path $\gamma_{x^D}$ which traverses the curve $\gamma$ starting and ending at $x^D$; this has the same net movement when lifted to the plane.
For any other intersection $y$ of $\gamma$ with $\beta_{sym}$, let $\rho_{x,y}$
denote the subpath of $\gamma$ starting at $x$ and ending at $y$, and let $\rho_{x^D,y^D}$ denote the subpath of $\gamma_{x^D}$ starting at $x^D$ and ending at $y^D$. Note that $\rho_{x,y}$ consists of unit vertical length separated by maximal horizontal segments, with half of a maximal horizontal segment on either end; $\rho_{x^D,y^D}$ on the other hand has a full maximal horizontal segment on one end and ends with a unit vertical segment on the other end. See Figure \ref{fig:grading-lemma-repeated-paths} for an example of $\gamma_x$, $\gamma_{x^D}$, $\rho_{x,y}$, and $\rho_{x^D,y^D}$.

Let $\langle b_{x,y}, a_{x,y} \rangle$ denote the vector corresponding to a lift of the path $\rho_{x,y}$ to $\widetilde T_{M_2}$, where $a_{x,y} \in \Z$ and $b_{x,y} \in \frac 1 2 \Z$. We also let $\langle b_{x^D,y^D}, a_{x^D,y^D} \rangle$ denote the vector corresponding to a lift of $\rho_{x^D,y^D}$, noting that $a_{x^D,y^D} = a_{x,y}$ but $b_{x^D,y^D}$ differs from $b_{x,y}$ if the maximal horizontal segments containing $x$ and $y$ have different lengths.
As described in \cite[Remark 34]{HRW-companion}, the two spin$^c$ components of $\gr_D$ of generators correspond to the coordinates of (the midpoints of) the corresponding unit segments of an immersed curve in rectilinear position, with the first spin$^c$ component given by the $x$-coordinate and the second spin$^c$ component given by the opposite of the $y$-coordinate. It follows that if $\gr_D(\bar x^D) = (0;0,0)$, then $\gr_D(\bar y^D)= (m_{x,y}; b_{x^D,y^D}, -a_{x^D,y^D})$ for some rational $m_{x,y}$. Similarly, by computing the grading change along the closed path $\gamma_{x^D}$ we find that $\gr_D(\bar x^D) = (m_{\gamma}; b_\gamma, -a_\gamma)$ for some rational $m_{\gamma}$, and thus this must be a power of $P_D$.

By assumption $\gamma$ is homologically nontrivial in $\partial M_2$, so $\langle b_\gamma, a_\gamma \rangle \neq \vec 0$. In fact, we may assume that $a_\gamma \neq 0$ since $[\gamma]$ is a multiple of the homological longitude and this is not a multiple of $\beta$ since $M_2(\beta)$ is a rational homology sphere. It follows that we can choose integers $n \neq 0$ and $\ell$ so that $n a_{x,y} = \ell a_\gamma$. We can construct a path $\widetilde{\gamma_x^\ell}$ that starts at a lift of $x$ by concatenating $\ell$ copies of the path $\gamma_x$ in $T_{M_2}$ and lifting the resulting path to $\widetilde{T}_{M_2}$. We similarly define a path $\widetilde{\gamma_{x^D}^\ell}$ by lifting the concatenation of $\ell$ copies of $\gamma_{x^D}$. We also define a path $\widetilde{\rho_{x,y}^n}$ by lifting the concatenation of $n$ copies of $\rho_{x,y}$, but since $\rho_{x,y}$ starts and ends at different points some care is needed to define this concatenation. In particular, we must translate each copy of $\rho_{x,y}$ horizontally if necessary so that its initial point coincides with the previous copy's final point; note that the points $x$ and $y$ are midpoints of maximal horizontal segments so may appear at either coordinates $(0, \tfrac 1 2)$ or $(\tfrac 1 2, \tfrac 1 2)$ in $T_{M_2}$, and each copy of $\rho_{x,y}$ is translated either an integer or half integer amount. Due to horizontal translations by half integer amounts the concatenated path may pass through the marked point in $T_{M_2}$, but the marked points are not relevant in the present construction so this is not a problem. We have chosen $n$ and $\ell$ so that if the lifts $\widetilde{\rho_{x,y}^n}$ and $\widetilde{\gamma_x^\ell}$ begin at the same point their terminal points also occur at the same height and are separated horizontally by some distance $k$; in particular, $k = n b_{x,y} - \ell b_\gamma \in \frac 1 2 \Z$. We will also consider the path $\widetilde{\rho_{x^D,y^D}^n}$ defined similarly using $\rho_{x^D,y^D}$ instead of $\rho_{x,y}$, noting that the terminal points of $\widetilde{\gamma_{x^D}^\ell}$ and $\widetilde{ \rho_{x^D,y^D}^n}$ still lie at the same height but are separated by a different horizontal distance $k^D =n b_{x^D,y^D} - \ell b_\gamma \in \Z$. See Figure \ref{fig:grading-lemma-repeated-paths} for an example of $\widetilde{ \gamma^\ell_x}$, $\widetilde{\rho^n_{x,y}}$, $\widetilde{ \gamma^\ell_{x^D}}$, and $\widetilde{\rho^n_{x^D,y^D}}$.

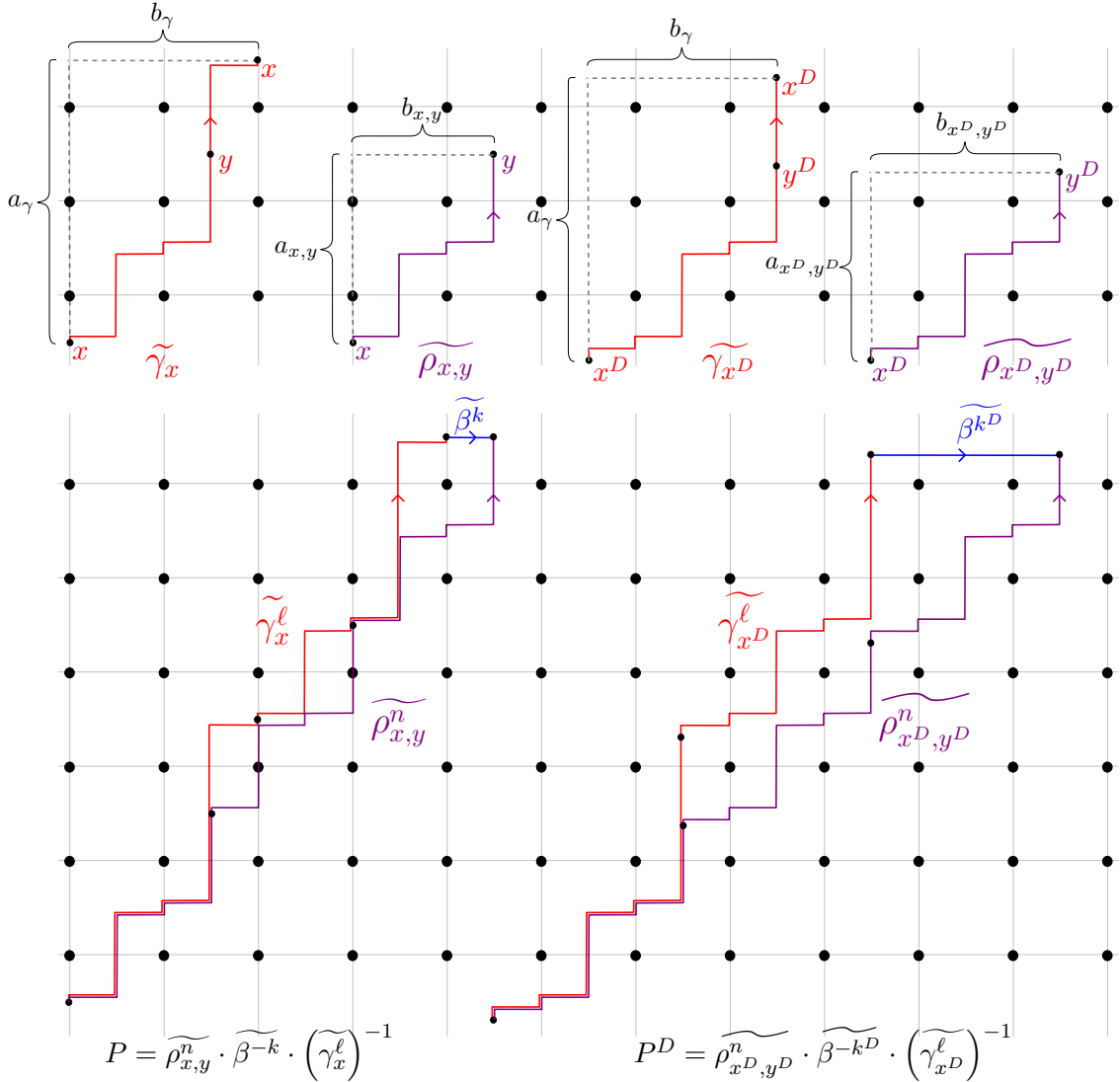
\begin{figure}
\vspace{3mm}
\labellist

\small
\pinlabel {$a_\gamma$} at -20 1060
\pinlabel {$b_\gamma$} at 160 1300

\pinlabel {$a_{x,y}$} at 330 1000
\pinlabel {$b_{x,y}$} at 490 1180

\pinlabel {$a_\gamma$} at 640 1040
\pinlabel {$b_\gamma$} at 820 1280

\pinlabel {$a_{x^D,y^D}$} at 975 980
\pinlabel {$b_{x^D,y^D}$} at 1190 1160

\normalsize
\pinlabel { $P =  \widetilde{\rho_{x,y}^n} \cdot \widetilde{\beta^{-k}} \cdot \left(\widetilde{\gamma_x^\ell}\right)^{-1}$} at 270 -20
\pinlabel { $P^D =  \widetilde{\rho_{x^D,y^D}^n} \cdot \widetilde{\beta^{-k^D}} \cdot \left(\widetilde{\gamma_{x^D}^\ell}\right)^{-1}$} at 1000 -20

\color{red}
\pinlabel {$x$} at 55 870
\pinlabel {$y$} at 240 1110
\pinlabel {$x$} at 295 1230
\pinlabel {\Large $ \widetilde{\gamma_x}$} at 160 860

\pinlabel {$x^D$} at 730 850
\pinlabel {$y^D$} at 970 1100
\pinlabel {$x^D$} at 970 1215
\pinlabel {\Large $ \widetilde{\gamma_{x^D}}$} at 880 860

\pinlabel {\Large $\widetilde{\gamma^\ell_x }$} at 300 530
\pinlabel {\Large $\widetilde{\gamma^\ell_{x^D} }$} at 900 530

\color{InkscapePurple}
\pinlabel {$x$} at 415 870
\pinlabel {$y$} at 600 1110
\pinlabel {\Large $ \widetilde{\rho_{x,y}}$} at 520 860

\pinlabel {$x^D$} at 1085 850
\pinlabel {$y^D$} at 1330 1095
\pinlabel {\Large $ \widetilde{\rho_{x^D,y^D}}$} at 1260 860

\pinlabel {\Large $\widetilde{\rho^n_{x,y} }$} at 460 400
\pinlabel {\Large $\widetilde{\rho^n_{x^D,y^D} }$} at 1130 400

\color{blue}
\pinlabel {$ \widetilde{\beta^k}$} at 550 790
\pinlabel {$ \widetilde{\beta^{k^D}}$} at 1200 780

\endlabellist

\includegraphics[scale = .3]{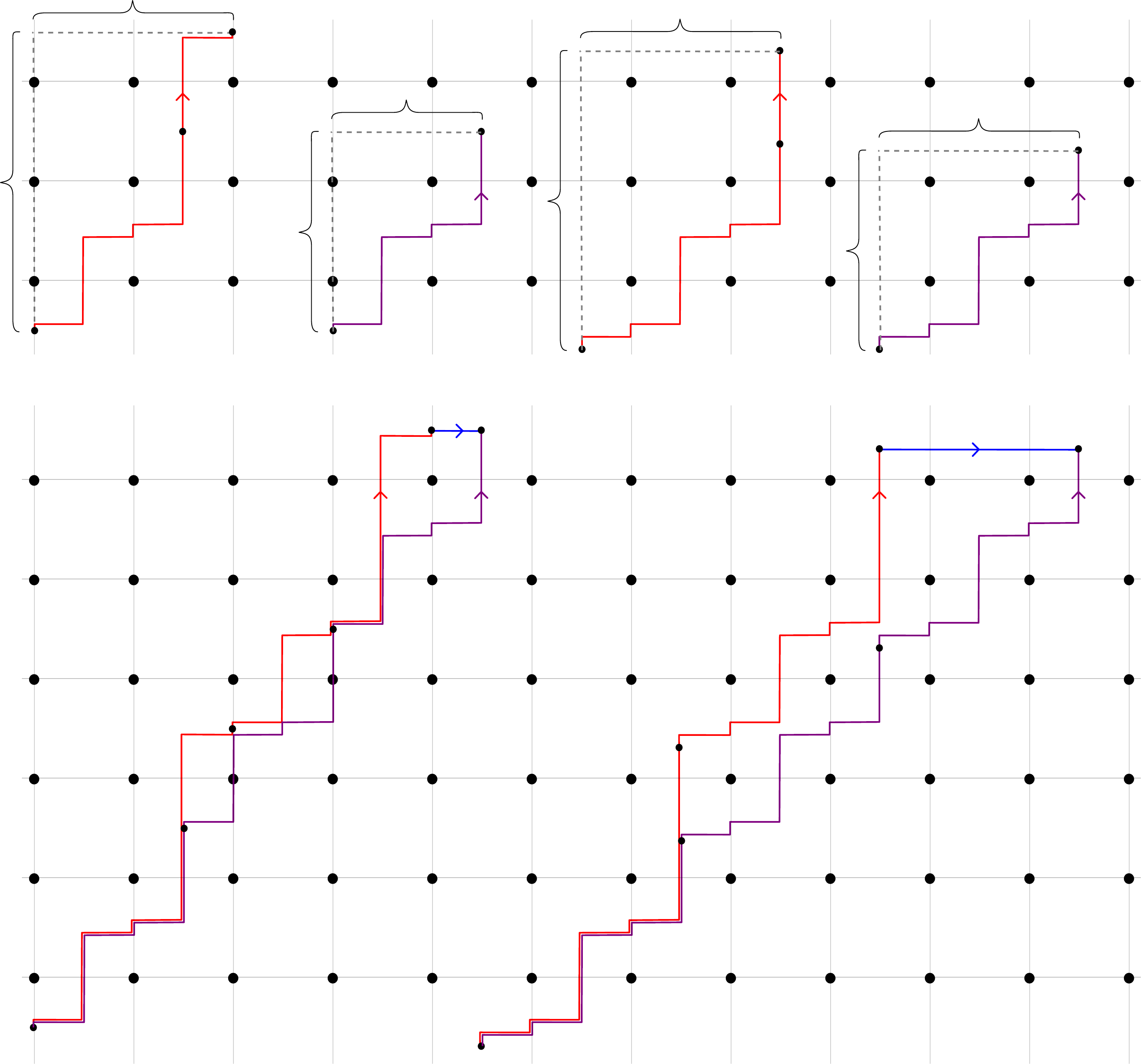}
\vspace{3mm}
\caption{The top shows an example of a path $\widetilde{ \gamma_x}$ and the subpath $\widetilde{\rho_{x,y}}$, along with the variants $\widetilde{ \gamma_{x^D}}$ and $\widetilde{\rho_{x^D,y^D}}$. Below are the repeated paths $\widetilde{ \gamma^\ell_x}$ and $\widetilde{\rho^n_{x,y}}$, where in this example $\ell = 2$ and $n = 3$, which combine with a lift of a length $k = \frac 1 2$ path along $\beta$ to form a closed path $P$; there is also an analogous closed path $P^D$ using the type D variants of the paths. The areas enclosed by these paths are $\text{Area}(P) = 2$ and $\text{Area}(P^D) = 5$.
}
\label{fig:grading-lemma-repeated-paths}
\end{figure}

Given the paths above, let $P$ denote the closed path $\widetilde{ \rho_{x,y}^n } \cdot \widetilde{\beta^{-k}} \cdot \left(\widetilde{ \gamma_x^\ell}\right)^{-1}$, where $\widetilde{\beta^{-k}}$ is a lift of a path on $\beta_{sym}$ that starts at $x$ and moves leftward by $k$ units (or rightward by $-k$ units). We will consider the area bounded by $P$, with the area of each region enclosed counted with multiplicity given by the winding number of $P$ around that region; we denote this by $\text{Area}(P)$. Note that $P$ consists of maximal horizontal segments of integer or half integer length that lie on (or within an $\epsilon$-neighborhood of) horizontal lines with $y$-coordinates in $\Z + \frac 1 2$ and unit vertical segments that lie on vertical lines with $x$-coordinates in $\frac 1 2 \Z $.  When computing the area, we ignore the perturbation of the horizontal segments from rectilinear position (or, equivalently, we take the limit as $\epsilon \to 0$), so that $\text{Area}(P) \in \frac 1 2 \Z$. In the course of proving Proposition \ref{prop:rational-grading-lemma} we will also consider the closed path $P^D = \widetilde{ \rho_{x^D,y^D}^n } \cdot \widetilde{\beta^{-k^D}} \cdot \left(\widetilde{ \gamma_{x^D}^\ell} \right)^{-1}$, defined analogously to $P$ using the the type D variations of the paths. Note that since the path $\rho_{x^D, y^D}$ moves rightward by $b_{x^D,y^D} \in \Z$, successive copies of $\rho_{x^D, y^D}$ do not need to be translated horizontally before concatenation. As a result, $P^D$ is disjoint from the punctures and moreover the area $\text{Area}(P^D)$ is an integer and agrees with the sum of the winding numbers of $P^D$ about each puncture. See Figure \ref{fig:grading-lemma-repeated-paths} for an example of $P$ and $P^D$.

Rotation numbers for smooth paths are defined as before, and the rotation number for a piecewise smooth path for which adjacent smooth segments are parallel at their boundary is the sum of the rotation numbers of the smooth segments. We define
$$\rot(P) = \rot(\widetilde{ \rho_{x,y}^n }) - \rot(\widetilde{ \gamma_x^\ell}) = n \rot( \rho_{x,y} ) - \ell \rot( \gamma_x ),$$
and similarly for $\rot(P^D)$. We do not include terms for rotation at the corners in the definition of $\rot(P)$, but note that the paths $\widetilde{ \gamma_x^\ell}$ and $\widetilde{ \rho_{x,y}^n }$ coincide near their initial points and are both perpendicular to $\widetilde{\beta^k}$ at their terminal points.


\begin{proposition} \label{prop:rational-grading-lemma}
Let $\gamma$ be an immersed curve in $T_{M_2}$ not homologous to a multiple of $\beta$ and assumed to be in symmetric $\beta$-pairing position, and let $x$ and $y$ be two generators of $HF(\gamma, \beta_{sym})$. Constructing the path $P = \widetilde{ \rho_{x,y}^n } \cdot \widetilde{\beta^{-k}} \cdot \left(\widetilde{ \gamma_x^\ell}\right)^{-1}$ as described above, the grading difference between $x$ and $y$ is determined by
$$\gr(y) - \gr(x) = \frac{2\textup{Area}(P) - \rot(P) - (n-1)k a_{x, y}}{n} $$

\end{proposition}

\begin{proof}[Proof of Proposition \ref{prop:rational-grading-lemma}]
We fix the relative gradings on the type D structure corresponding to $\gamma$ so that $\gr_D(\bar x^D) = (0;0,0)$ and recall that the type A structure associated with $\beta_{sym}$ has a single generator $\bar x^A_0$ with $\gr_A(\bar x^A_0) = (0;0,0)$. We have that $\gr(x) = \gr(x^D) = \gr(\bar x_0^A \otimes \bar x^D) = 0$ and our goal is to compute $\gr(y) = \gr(y^D) = \gr(\bar x_0^A \otimes \bar y^D)$. 
This can be obtained from $\gr^\boxtimes( \bar x_0^A \otimes \bar y^D ) = \gr_D(\bar y^D)$. Following the discussion above, we have $\gr_D(\bar y^D) = (m_{x,y}; b^D_{x,y}, -a^D_{x,y})$, where $\langle b^D_{x,y}, a^D_{x,y}\rangle$ is the vector corresponding to the path $\rho_{x,y}^D$. 
We take $P_A$ to be $(-\frac 1 2; 1, 0)$, and we let $P_\gamma$ denote $(m_{\gamma}; b_{\gamma}, -a_\gamma)$, which is the power of $P_D$ corresponding to the grading change when traversing the path $\gamma_{x^D}$.
  We compute $\gr(y)$ by taking the Maslov component of $(m_{x,y}; b_{x^D,y^D}, -a_{x^D,y^D})$ after multiplying on the left by some rational power of $P_A$ and on the right by some rational power of $P_D$ chosen so that we end up with spin$^c$ components that are zero. In this case we use the product
$$P_A^{-\frac{k^D}{n}}  (m_{x,y}; b_{x^D,y^D}, -a_{x^D,y^D}) P_\gamma^{-\frac{\ell}{n}} = (\frac{k^D}{2n} ; -\frac{k^D}{n}, 0) (m_{x,y}; b_{x^D,y^D}, -a_{x^D,y^D}) (-\frac{\ell m_\gamma}{n}; -\frac {\ell b_\gamma}{n}, \frac{\ell a_\gamma}{n})$$
The second spin$^c$ component is zero since $a_{x,y} = a_{x^D,y^D}$ and $\ell$ and $n$ were chosen so that $n a_{x,y} = \ell a_\gamma$. The first spin$^c$ component is zero since $k^D$ is precisely the difference $n b_{x^D,y^D} - \ell b_\gamma$. Thus
\begin{equation}\label{eq:grading-of-y}
\gr(y) = m_{x,y} + \frac{k^D}{2n} - \frac{\ell m_\gamma}{n} + \frac{k^D a_{x^D,y^D}}{n}.
\end{equation}
Having computed $\gr(y)$ in terms of the unknown gradings $(m_{x,y}; b_{x^D,y^D}, -a_{x^D,y^D})$ and $(m_\gamma; b_\gamma, -a_\gamma)$; our strategy will be to relate this information to areas and rotation numbers by considering a different tensor product.

Let $\bf{D}_\gamma$ be a hypothetical type D structure whose immersed curve representative, when lifted to $\widetilde{T}_{M_2}$, contains the path $\widetilde{ \gamma_{x^D}^\ell}$; recall that $\widetilde{ \gamma_{x^D}^\ell}$ is formed by lifting the concatenation of $\ell$ copies of $\gamma_{x^D}$.  Let $x^D_{0}$ denote the beginning of the path $\widetilde{ \gamma_{x^D}^\ell}$ and let $x^D_{i}$ denote the lift of the end of the $i$th copy of $\gamma_{x^D}$ in the concatenation, with $x^D_\ell$ giving the endpoint of $\widetilde{ \gamma_{x^D}^\ell}$; for each point, adding a bar to the notation refers to the corresponding generator of $\bf{D}_\gamma$. Similarly, let $\bf{D}_\rho$ be any type D structure that contains a portion described by the path $\widetilde{ \rho_{x^D,y^D}^n }$. We remark that this path will only be an immersed curve segment if the path $\rho_{x^D,y^D}$, which is vertical at its endpoints, is oriented the same direction at both endpoints. In this case, as with $\bf{D}_\gamma$, we can take $\bf{D}_\rho$ to be any type D structure whose immersed curve representative lifted to $\widetilde{T}_{M_2}$ contains the curve segment $\widetilde{ \rho_{x^D,y^D}^n }$. If $\rho_{x^D,y^D}$ is oriented in opposite directions at its endpoints, then $\widetilde{ \rho_{x^D,y^D}^n }$ will not be a smooth curve segment, but it will be a piecewise smooth path for which the two smooth segments are tangent when they share an endpoint, so it lies on an immersed train track of the form described in \cite{HRW}. Such train tracks encode type D structures, albeit with a non-preferred basis, and we choose $\bf{D}_\rho$ to be a type D structure that is represented by a train track containing $\widetilde{ \rho_{x^D,y^D}^n }$. While we usually simplify a train track to an immersed curves, all the relevant geometric constructions work for suitable immersed train tracks as well. In particular, the pairing theorem still relates Floer homology to the box tensor product, and the grading difference formula Prop \ref{prop:old-grading-lemma} still applies (where the rotation number of a piecewise smooth path in a train track is the sum of the rotation number of its smooth pieces); see \cite[Section 2]{HRW} for details. We let $y^D_0$ denote the starting point of the path $\widetilde{ \rho_{x^D,y^D}^n }$ and let $y^D_i$ denote the end of the $i$th lift of $\rho_{x^D,y^D}$ along the path, using bars as usual to denote the corresponding generators of $\bf{D}_\rho$. Note that the points $x^D_0$ and $y^D_0$ coincide in $\widetilde{T}_{M_2}$. Finally, let $\bf{A}$ be any type A structure whose corresponding immersed curve $\delta$, when lifted to the curve $\tilde \delta$ in $\widetilde{T}_{M_2}$, passes through the three points $x^D_0 = y^D_0$, $x^D_\ell$ and $y^D_n$ and connects the last two of these points (which must lie at the same height) by a horizontal line segment. 
Note that each of these three points lies on a unit vertical segment of $\widetilde{ \gamma_{x^D}^\ell}$ or $\widetilde{ \rho_{x^D,y^D}^n }$; we assume that $\tilde\delta$  is in rectilinear position, translated slightly downward and leftward into pairing position, so that each of these points lies on a unit horizontal segment of $\tilde\delta$ corresponding to a generator of $\bf{A}$. We let $x_0^A$, $x_\gamma^A$, and $x_\rho^A$ denote points on the lift $\tilde\delta$ corresponding to the points $x^D_0 = y^D_0$, $x^D_\ell$ and $y^D_n$, respectively, and use bars to indicate the corresponding generators of $\bf{A}$. We let $\widetilde{\delta_\gamma}$ and $\widetilde{\delta_\rho}$ denote the paths in $\tilde\delta$ from $x_0^A$ to $x_\gamma^A$ and $x_\rho^A$, respectively.
See Figure \ref{fig:grading-lemma-area} for an example.


\begin{figure}
\labellist

\pinlabel {\small $k^D$} at 800 805

\color{red}
\pinlabel {$\widetilde{\gamma^\ell_{x^D} }$} at 520 520
\pinlabel {\small $x_0^D$} at 165 -5
\pinlabel {\small $x_1^D$} at 405 370
\pinlabel {\small $x_\ell^D$} at 650 715

\color{InkscapePurple}
\pinlabel { $\widetilde{\rho^n_{x^D,y^D} }$} at 740 400
\pinlabel {\small $y_0^D$} at 220 5
\pinlabel {\small $y_1^D$} at 460 255
\pinlabel {\small $y_2^D$} at 700 495
\pinlabel {\small $y_n^D$} at 940 720

\color{blue}
\pinlabel {$ \widetilde{\delta}$} at 20 400
\pinlabel {\small $x_0^A$} at 165 45
\pinlabel {\small $x_\rho^A$} at 940 765
\pinlabel {\small $x_\gamma^A$} at 650 765

\endlabellist
\includegraphics[scale = .3]{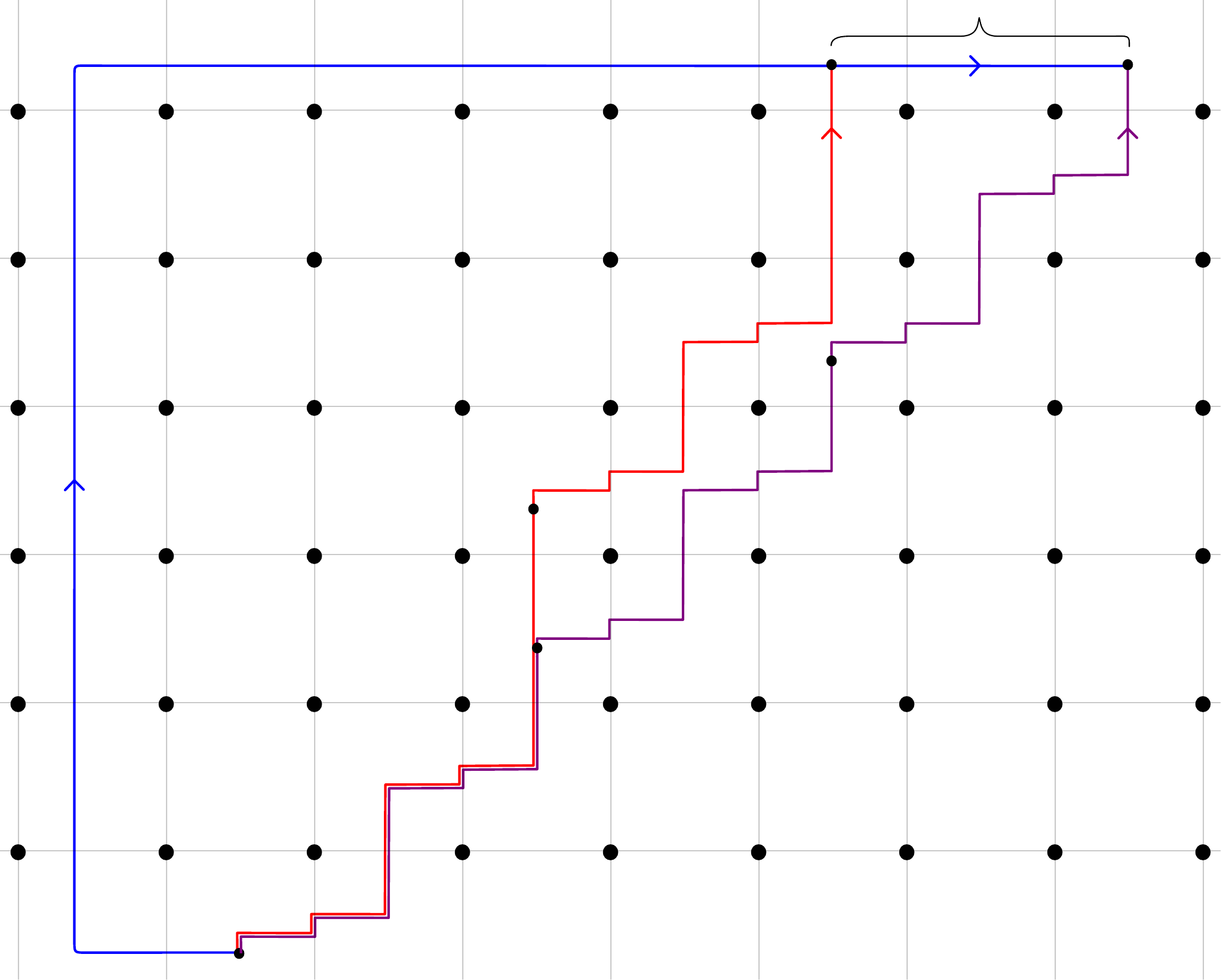}
\caption{The paths $\widetilde{\gamma_{x^D}^\ell}$ and $\widetilde{ \rho_{x^D, y^D}^n}$ in $\widetilde{T}_{M_2}$ from Figure \ref{fig:grading-lemma-repeated-paths}, along with an immersed curve $\tilde\delta$ that intersect both paths at their endpoints. These paths correspond to portions of type D structures ${\bf D}_\gamma$ and ${\bf D}_\rho$, respectively, and $\tilde\delta$ represents part of a type A structure ${\bf A}$. To compute $\gr_D( y^D_n) = \left( \gr_D( y^D) \right)^n$, we consider the gradings of the various intersections of these three curves as generators of the appropriate Floer homologies.}
\label{fig:grading-lemma-area}
\end{figure}

We fix relative gradings on $\bf{D}_\gamma$, $\bf{D}_\rho$, and $\bf{A}$ so that $\gr_D(\bar x^D_0)$, $\gr_D(\bar y^D_0)$, and $\gr_A(\bar x_0^A)$ are all $(0;0,0)$. Since the path from $y^D_0$ to $y^D_1$ is just a lift of $\rho_{x^D,y^D}$, it is clear that $\gr_D(\bar y^D_1) = \gr_D(\bar y^D) = (m_{x,y}; b_{x^D,y^D}, -a_{x^D,y^D})$. Moreover, since this path repeats we have that $\gr_D(\bar y^D_i) = \gr_D(\bar y^D)^i$ and in particular $\gr_D(\bar y^D_n) = (n m_{x,y}; n b_{x^D,y^D}, -n a_{x^D,y^D})$. Similarly, $\gr_D(\bar x^D_i) = P_\gamma^i$ and in particular $\gr_D(\bar x^D_\ell) = (\ell m_\gamma; \ell b_\gamma, -\ell a_\gamma)$.
To relate the gradings of generators of $\bf{A}$ to the coordinates of the corresponding unit horizontal and vertical segments in the immersed curve for $\bf{A}$, note that immersed curves representing type A structures are reflections across a line of slope $-1$ of the immersed curves representing the corresponding type D structure. Since the reflection swaps the $x$ and $y$-coordinates and negates both and changing from type A to type D structures swaps and negates the two \spinc components of the grading, we can check that the coordinate of a segment on $\widetilde\delta$ correspond to the spin$^c$ components in $\gr_A$ of the corresponding generator of $\bf{A}$ with a minus sign on the $x$-coordinate.
 It follows that $\gr_A(\bar x^A_\gamma) = (m_A; -\ell b_\gamma, \ell a_\gamma)$ for some rational $m_A$, and from this we see that $\gr^\boxtimes( \bar x^A_\gamma \otimes \bar x^D_\ell) = (\ell m_\gamma + m_A; 0, 0)$ and so 
$$\gr( \bar x^A_\gamma \otimes \bar x^D_\ell) = \ell m_\gamma + m_A.$$
 To compute the grading of $\bar x^A_\rho$, we observe that $x^A_\rho$ is $k^D = n b_{x^D,y^D} - \ell b_{\gamma}$ units to the right of $x^A_\gamma$ and connected to $x^A_\gamma$ by $\vert k^D \vert$ unit horizontal segments on the $\widetilde\delta$ curve. Each of these unit horizontal segments, with generators $u$ and $v$ on the left and right end respectively, corresponds to a type A operation $m_3(v, \rho_2, \rho_1) = u$; this is the same operation appearing in the type A module corresponding to $\beta$ and gives the grading relationship $\gr_A(u) = \gr_A(v) \cdot P_A$.
 From this we compute that
$$\gr_A( \bar x^A_\rho )  = \gr_A( \bar x^A_\gamma) \cdot P_A^{-k^D} = (m_A; -\ell b_\gamma, \ell a_\gamma)(\frac {k^D}{2}; -k^D, 0) = (m_A + \frac{k^D}{2} + k^D \ell a_\gamma; -\ell b_\gamma - k^D, \ell a_\gamma), $$
$$\gr^\boxtimes(\bar x^A_\rho \otimes \bar y^D_n) = (m_A + \frac{k^D}{2} + k^D\ell a_\gamma; -\ell b_\gamma - k^D, \ell a_\gamma) \cdot (n m_{x,y}; n b_{x^D,y^D}, -n a_{x^D,y^D}) $$
$$ \hspace{3 cm} = (m_A + \frac{k^D}{2} + k^D \ell a_\gamma + n m_{x,y}; 0, 0),$$
$$\gr(\bar x^A_\rho \otimes \bar y^D_n) = m_A + \frac{k^D}{2} + k^D \ell a_\gamma + nm_{x,y}.$$
We now consider the difference
$$\gr(\bar x^A_\rho \otimes \bar y^D_n) - \gr( \bar x^A_\gamma \otimes \bar x^D_\ell) = \frac{k^D}{2} + k^D \ell a_\gamma - \ell m_\gamma + nm_{x,y} = \frac{k^D}{2} + k^D n a_{x^D,y^D} - \ell m_\gamma + nm_{x,y}$$
and note from \eqref{eq:grading-of-y} that
$$n \gr(y) = n m_{x,y} + \frac{k^D}{2} - \ell m_\gamma + k^D a_{x^D,y^D} = \left[\gr(\bar x^A_\rho \otimes \bar y^D_n) - \gr( \bar x^A_\gamma \otimes \bar x^D_\ell) \right] - (n-1)k^D a_{x^D,y^D} .$$
Since $\bar x_\rho^A \otimes \bar y^D_n$ lies in the same spin$^c$ structure as $\bar x_0^A \otimes \bar y^D_0$, we can use Proposition \ref{prop:old-grading-lemma} to compute
$$\gr(\bar x^A_\rho \otimes \bar y^D_n) = 2 w\left(\widetilde{\rho_{x^D,y^D}^n} \cdot \widetilde{\delta_\rho}^{-1}\right) - \rot(\widetilde{\rho_{x^D,y^D}^n}) + \rot(\widetilde{\delta_\rho}).$$
Similarly we find that
$$\gr(\bar x^A_\gamma \otimes \bar x^D_\ell) = 2 w\left(\widetilde{\gamma_{x^D}^\ell} \cdot \widetilde{\delta_\gamma}^{-1}\right) - \rot(\widetilde{\gamma_{x^D}^\ell}) + \rot(\widetilde{\delta_\gamma}). $$
Because the curves are in an arbitrarily small neighborhood of rectilinear position, the winding number of a closed path is the same as the signed area enclosed by that path. When we subtract these two gradings the areas cancel except for the area $\text{Area}(P^D)$ enclosed by the path $P^D = \widetilde{ \rho_{x^D,y^D}^n } \cdot \widetilde{\beta^{-k^D}} \cdot \left(\widetilde{ \gamma_{x^D}^\ell}\right)^{-1}$. The rotation numbers of $\widetilde{\delta_\rho}$ and $\widetilde{\delta_\gamma}$ also cancel; these agree since $\tilde\delta$ has no rotation between $x^A_\gamma$ and $x^A_\rho$. We thus have
\begin{equation}\label{eq:n-times-grading-difference}
n \gr(y) = 2\text{Area}(P^D) - \rot(P^D) - (n-1)k^D a_{x^D, y^D} .
\end{equation}

It only remains to relate the rotation numbers and areas bounded by the path $P^D$ and the path $P$ used in the statement of the proposition. It is clear that $\rot(P^D) = \rot(P)$. For the area, we  notice that the path $\widetilde{ \rho_{x^D,y^D}^n }$ can be obtained from $\widetilde{ \rho_{x,y}^n }$ by changing the length of the horizontal segment at the beginning and end of each copy of $\rho_{x,y}$. More precisely, if $h_x$ and $h_y$ denote the length of the maximal horizontal segment of $\gamma$ containing $x$ and $y$ respectively (recorded as negative if the path moves leftward on the horizontal segment), then we obtain $\rho_{x^D,y^D}$ from $\rho_{x,y}$ by adding a horizontal segment of length $\frac{h_x}{2}$ to the beginning and removing a horizontal segment of length $\frac{h_y}{2}$ from the end. It follows that in addition to moving the starting point leftward by $\frac{h_x}{2}$, the path $\widetilde{ \rho_{x^D,y^D}^n }$ is obtained from $\widetilde{ \rho_{x,y}^n }$ by translating all the vertical segments in the $i$th copy of $\rho_{x,y}$ leftward by $(i-1) \frac{h_y - h_x}{2}$. The path $\widetilde{ \gamma_{x^D}^\ell}$ is the same as $\widetilde{ \gamma_{x}^\ell}$ except for its initial and final horizontal segments; we move the starting point leftward by $\frac{h_x}{2}$ units, add a $\frac{h_x}{2}$ to the length of the horizontal segment at the beginning of $\widetilde{ \gamma_{x^D}^\ell}$ and remove the length $\frac{h_x}{2}$ horizontal segment at the end of $\widetilde{ \gamma_{x^D}^\ell}$. From this we can see that
$$\text{Area}(P) = \text{Area}(P^D) + a_{x, y} \frac{h_y - h_x}{2}  \sum_{i = 1}^n (i-1) = \text{Area}(P^D) + a_{x, y} \frac{h_y - h_x}{2} \frac{n(n-1)}{2}.$$
Recall that $a_{x^D, y^D}= a_{x,y}$. We also have that $k - k^D = n \frac{h_y - h_x}{2}$, so \eqref{eq:n-times-grading-difference} becomes
$$n \gr(y) = 2\text{Area}(P) - \rot(P) - (n-1)k a_{x, y} .$$
The result follows after dividing by $n$.
\end{proof}

\begin{remark}
Proposition \ref{prop:rational-grading-lemma} contains Proposition \ref{prop:old-grading-lemma} (in the case that $M_1$ is a solid torus whose meridian glues to $\beta$) as a special case, since if $x$ and $y$ are in the same \spinc structure then $a_{x,y} = \ell = 0$, and we can take $n = 1$.
\end{remark}

\begin{remark}\label{rmk:rational-grading-lemma-adjusted-path}
The path $P$ used in the statement of the Proposition \ref{prop:rational-grading-lemma} could be defined differently, making notation more cumbersome but simplifying the path. Note that the first portion of $\widetilde{ \rho_{x,y}^n }$, following one lift of $\rho_{x,y}$, agrees with the first portion $\widetilde{ \gamma_x^\ell}$, so removing the relevant portions at the beginning and end of $P$ gives a different closed path $P'$ that encloses the same area and has the same rotation number. This path has three sides formed from $(n-1)$ lifts of $\rho_{x,y}$, $\ell$ lifts of $\gamma_x$ with a copy of $\rho_{x,y}$ removed, and a length $k$ path along a lift of $\beta_{sym}$. Note that $(n-1)k a_{x,y}$ can be interpreted as twice the area of the triangle in $\R^2$ with straight line edges whose three corners are the same as those of $P'$.
On the other hand, for a cleaner statement at the expense of defining yet another variation on the path $P$, the term $(n-1) k a_{x,y}$ could be wrapped into the term $2 \text{Area}(P)$ by defining a path $P''$ constructed in the same way as $P$ except that we add a horizontal segment moving leftward by $\frac k n$ to the end of $\rho_{x,y}$ before iterating; note that this has the effect of shifting the $i$th copy of $\rho_{x,y}$ in $\widetilde{ \rho_{x,y}^n }$ leftward by $\frac{(i-1)k}{n}$ units. After this shift the endpoints of $\widetilde{ \gamma_x^\ell}$ and $\widetilde{ \rho_{x,y}^n }$ agree, so no horizontal component along $\beta$ is required when constructing $P''$. We can check that $\rot(P'') = \rot(P')$ and $\text{Area}(P'') = \text{Area}(P') - a_{x,y} \frac k n \frac{n(n-1)}{2}$, and so
$$\gr(y) - \gr(x) = \frac{2 \text{Area}(P'') - \rot(P'')}{n}.$$
\end{remark}

\begin{example}\label{ex:grading-example}
To demonstrate Proposition \ref{prop:rational-grading-lemma}, we return to the running example in Figure \ref{fig:symmetric-pairing-position} and compute the relative gradings of $\HF(\gamma, \beta)$. Let $y_1$, $y_2$, and $y_3$ denote the three generators in the \spinc structure for which the rank is three, as shown in Figure \ref{fig:grading-example}, and let $x$ denote the sole generator in the other \spinc structure. We declare that $\gr(x) = 0$ and compute the gradings of each $y_i$ relative to this. The paths $\rho_{x,y_i}$ are shown in Figure \ref{fig:grading-example}. In each case, $a_{x,y_i} = 1$ and $a_\gamma = 2$ so we can take $n = 2$ and $\ell = 1$ when constructing the path $P$ appearing in Proposition \ref{prop:rational-grading-lemma}, which we denote by $P_i$. We remove the initial copy of $\rho_{x,y_i}$ from the beginning and end of the path, as in Remark \ref{rmk:rational-grading-lemma-adjusted-path}, to give paths $P'_i$; these are shown in Figure \ref{fig:grading-example}. Note that since $\rot(\gamma_x) = 0$, we have that $\rot(P'_i) = 2 \rot(\rho_{x,y_i})$. For $y_1$ we have that $k = 1$, $\text{Area}(P'_1) = 1$, and $\rot(\rho_{x,y_1}) = 0$, so $\gr(y_1) = \frac{2 - 0 - 1}{2} = \frac{1}{2}$. For $y_2$ we have that $k = 0$, $\text{Area}(P'_2) = \frac{5}{2}$, and $\rot(\rho_{x,y_2}) = 1$, so $\gr(y_2) = \frac{5 - 2 - 0}{2} = \frac{3}{2}$. For $y_3$ we have that $k = -1$, $\text{Area}(P'_3) = 0$, and $\rot(\rho_{x,y_3}) = 0$, so $\gr(y_3) = \frac{0 - 0 +1}{2} = \frac{1}{2}$.
\end{example}
\input{figures/9.tex}

We observe that in the previous example $\gr(y_1) = \gr(y_3)$. This fact can be seen more clearly by using the symmetry of the curve $\gamma$ under rotation about $x$, which takes $y_1$ to $y_3$. If we choose a different path $\gamma_x$ when computing $\gr(y_3)$, namely the same path followed in the opposite direction, then the paths $\gamma_x$ and $\rho_{x,y_3}$ used in the computation of $\gr(y_3)$ are rotations about $x$ of the paths $\gamma_x$ and $\rho_{x,y_1}$ used in the computation of $\gr(y_1)$. Since the geometric relationship between these curves is preserved by that rotation, it is clear the gradings should agree. This observation generalizes to the following:
\begin{corollary}\label{cor:symmetry-same-grading}
Suppose $\gamma$ is in symmetric $\beta$-pairing position and $x$ is an intersection point of $\gamma$ with $\beta_{sym}$ such that $\gamma$ is fixed setwise by rotation by $\pi$ about $x$. If $y$ and $y'$ are two intersection points of $\gamma$ with $\beta_{sym}$ that are swapped by rotation about $x$, then $\gr(y) = \gr(y')$.
\end{corollary}

\begin{proof}
This follows immediately from the symmetry of the curve and the geometric description of the grading difference between $y$ or $y'$ and $x$. The paths used to construct $P$ when computing $\gr(y) - \gr(x)$ can be chosen to be the rotation about $x$ of those used to construct $P$ when computing $\gr(y') - \gr(x)$, so the area and rotation number agree for these paths. We have that $a_{x,y} = -a_{x,y'}$ but also the value of $k$ is negated for $y'$ relative to $y$, and these sign changes cancel.
\end{proof}

\begin{remark}
This symmetry in gradings of $\HFhat(M_1 \cup_\phi M_2)$ is not readily apparent if the bordered computation is done algebraically rather than using immersed curves. The type A module $\CFA(M_1)$ has a single generator with idempotent $\iota_1$, so the generators of $\widehat{ \mathit{CF} }(M_1 \cup_\phi M_2) = \CFA(M_1) \boxtimes \CFD(M_2)$ are in one-to-one correspondence with $\iota_1$-generators of $\CFD(M_2)$. But given a point $x$ about which the immersed curve is symmetric, $\CFD(M_2)$ is not symmetric in an obvious sense about the corresponding generator $\bar x^D$. In particular, the gradings $\gr_D(\bar y^D)$ and $\gr_D(\overline{y'}^D)$ are not clearly related. Geometrically, this corresponds to the fact that the immersed curve is not symmetric about $x^D$ and the path from $x^D$ to $y^D$ is not a rotation of the path from $x^D$ to $(y')^D$. The fact that the different gradings $\gr_D(\bar y^D)$ and $\gr_D(\overline{y'}^D)$ give the same result after tensoring with $\gr_A(\bar x_0^A)$ and multiplying by powers of $P_A$ and $P_D$ is nontrivial. By allowing us to perturb curves to a symmetric position, the immersed curve computation elucidates this symmetry in the gradings.
\end{remark}

Proposition \ref{prop:rational-grading-lemma} is more general than we will need in this paper. Apart from the symmetry observation in \ref{cor:symmetry-same-grading}, we will only need to use Proposition \ref{prop:rational-grading-lemma} in the following special case, namely the case that $\gamma$ is symmetric under a rotation by $\pi$ and $x$ and $y$ are both fixed points of this symmetry.

\begin{corollary}\label{cor:grading-difference-for-fixed-points}
Suppose $\gamma$ is in symmetric $\beta$-pairing position and $x$ is an intersection point of $\gamma$ with $\beta_{sym}$ such that $\gamma$ is fixed setwise by rotation by $\pi$ about $x$. Suppose $y$ is another intersection point of $\gamma$ with $\beta_{sym}$ that is exactly halfway along $\gamma$, so that the map from $\gamma$ to itself induced by the rotation by $\pi$ about $x$ fixes $y$. Then viewing $x$ and $y$ as generators of $\HF(\gamma, \beta_{sym})$ we have
$$\gr(y) - \gr(x) = \text{Area}(P') - \rot({\rho_{x,y}}),$$
where $\rho_{x,y}$ is a path from $x$ to $y$ along $\gamma_x$, $\rho_{y,x}$ is the complementary path in $\gamma_x$ from $y$ to $x$ (continuing in the same direction, so that $\rho_{x,y} \circ \rho_{y,x}$ gives the path $\gamma_x$), and $P' = \widetilde{\rho_{x,y}} \cdot \left(\widetilde{\rho_{y,x}} \right)^{-1}$.\end{corollary}

\begin{proof}
By the symmetry of the curve, rotation by $\pi$ about $x$ takes $\rho_{x,y}$ to $-\rho_{y,x}$. It follows that $a_\gamma = 2 a_{x,y}$ and $b_\gamma = 2 b_{x,y}$, so we can take $n = 2$ and $\ell = 1$ when defining the path $P$ in Proposition \ref{prop:rational-grading-lemma} and we have that $k = 0$.
$P'$ is the path $P$ from Proposition \ref{prop:rational-grading-lemma} with a copy of $\rho_{x,y}$ removed at the beginning and end, as in Remark \ref{rmk:rational-grading-lemma-adjusted-path}. Symmetry also implies that $\rot(\rho_{y,x}) = -\rot(\rho_{x,y})$ and so $\rot(P') = 2\rot{\rho_{x,y}}$. Proposition \ref{prop:rational-grading-lemma} then gives
$$\gr(y) - \gr(x) = \frac{2\text{Area}(P') - \rot(P')}{2} = \text{Area}(P') - \rot(\rho_{x,y}) . $$
\end{proof}

%% file: figures/FigLabellingGenerators.pdf_tex
\begingroup%
  \makeatletter%
  \providecommand\color[2][]{%
    \errmessage{(Inkscape) Color is used for the text in Inkscape, but the package 'color.sty' is not loaded}%
    \renewcommand\color[2][]{}%
  }%
  \providecommand\transparent[1]{%
    \errmessage{(Inkscape) Transparency is used (non-zero) for the text in Inkscape, but the package 'transparent.sty' is not loaded}%
    \renewcommand\transparent[1]{}%
  }%
  \providecommand\rotatebox[2]{#2}%
  \ifx\svgwidth\undefined%
    \setlength{\unitlength}{710.11654351bp}%
    \ifx\svgscale\undefined%
      \relax%
    \else%
      \setlength{\unitlength}{\unitlength * \real{\svgscale}}%
    \fi%
  \else%
    \setlength{\unitlength}{\svgwidth}%
  \fi%
  \global\let\svgwidth\undefined%
  \global\let\svgscale\undefined%
  \makeatother%
  \begin{picture}(1,0.43457045)%
    \put(0,0){\includegraphics[width=\unitlength,page=1]{FigLabellingGenerators.pdf}}%
    \put(0.00754403,0.39078092){\color[rgb]{0,0,0}\makebox(0,0)[lb]{\smash{$\textcolor{blue}{\beta_{\textrm{sym}}}$}}}%
    \put(0,0){\includegraphics[width=\unitlength,page=2]{FigLabellingGenerators.pdf}}%
    \put(0.0089559,0.31005986){\color[rgb]{0,0,0}\makebox(0,0)[lb]{\smash{$\textcolor{green}{\beta_A}$}}}%
    \put(0.00829843,0.22104017){\color[rgb]{0,0,0}\makebox(0,0)[lb]{\smash{$\textcolor{blue}{\beta_{\textrm{sym}}}$}}}%
    \put(0.00518389,0.14937196){\color[rgb]{0,0,0}\makebox(0,0)[lb]{\smash{$\textcolor{green}{\beta_A}$}}}%
    \put(0.4292555,0.1803024){\color[rgb]{0,0,0}\makebox(0,0)[lb]{\smash{$x$}}}%
    \put(0.17728477,0.17200395){\color[rgb]{0,0,0}\makebox(0,0)[lb]{\smash{$x^D$}}}%
    \put(0.85323018,0.21726813){\color[rgb]{0,0,0}\makebox(0,0)[lb]{\smash{$y$}}}%
    \put(0.85323012,0.15540711){\color[rgb]{0,0,0}\makebox(0,0)[lb]{\smash{$y^D$}}}%
    \put(0.63369884,0.30327012){\color[rgb]{0,0,0}\makebox(0,0)[lb]{\smash{$\textcolor{red}{\gamma}$}}}%
    \put(0,0){\includegraphics[width=\unitlength,page=3]{FigLabellingGenerators.pdf}}%
  \end{picture}%
\endgroup%

%% file: figures/9.tex

\begin{figure}[h]
\centering
\begin{tikzpicture}[on top/.style={preaction={draw=white,-,line width=#1}},
on top/.default=4pt]
\def\U{1.5cm}
\def\u{0.05cm}
\def\radius{0.03cm}
\def\labelshift{0.12cm}
\newcommand{\localcolor}{InkscapePurple}

\begin{scope}[xscale=0.7, yscale=0.7]

\draw[step=\U, lightgray, densely dotted, thin, xshift=0.5*\U, yshift=0.5*\U]
(-2.75*\U,-2.75*\U) grid (1.75*\U,1.75*\U);
\draw[step=\U, lightgray, thin, xshift=0.5*\U, yshift=0.5*\U]
(-2.5*\U,-2.5*\U) grid (1.5*\U,1.5*\U);

	\begin{scope}
	\clip (-2.25*\U, -2.25*\U) rectangle (2.25*\U, 2.25*\U);
	\foreach \i in {-1,1} {
	\begin{scope}[xshift=\i*\U, yshift=2*\i*\U, densely dotted]
		\foreach \x in {1,-1} {
	   	   \begin{scope}[xscale=\x, yscale=\x]
		   \renewcommand{\localcolor}{red}
	      		\draw[\localcolor, rounded corners=0.5*\u]
			(-0.5*\U, -\U) -- (-0.5*\U, -\U+\u) -- (\u,-\U+\u) -- (\u,-\U+2*\u);
	      		\draw[\localcolor] (\u,-\U+2*\u) -- (\u,-2*\u);
	      		\draw[\localcolor, rounded corners=0.5*\u]
			(\u, -2*\u) -- (\u, -1*\u) -- (0.5*\U-0.5*\u, -1*\u) -- (0.5*\U-0.5*\u, 2*\u) -- (\U-2*\u, 2*\u) -- (\U-2*\u, 3*\u);
	      		\draw[\localcolor] (\U-2*\u, 3*\u) -- (\U-2*\u, \U-3*\u);
	      		\draw[\localcolor, rounded corners=0.5*\u]
			(\U-2*\u, \U-3*\u) -- (\U-2*\u, \U-2*\u) -- (0, \U-2*\u) -- (0, \U-3*\u);
	      		\draw[\localcolor] (0, \U-3*\u) -- (0,\u);
	      		\draw[\localcolor] (0, \u) -- (0,0);
	   	   \end{scope}
		}
	\end{scope}
		\foreach \x in {1,-1} {
	   	   \begin{scope}[xscale=\x, yscale=\x]
		   \renewcommand{\localcolor}{red}
	      		\draw[\localcolor, rounded corners=0.5*\u]
			(-0.5*\U, -\U) -- (-0.5*\U, -\U+\u) -- (\u,-\U+\u) -- (\u,-\U+2*\u);
	      		\draw[\localcolor] (\u,-\U+2*\u) -- (\u,-2*\u);
	      		\draw[\localcolor, rounded corners=0.5*\u]
			(\u, -2*\u) -- (\u, -1*\u) -- (0.5*\U-0.5*\u, -1*\u) -- (0.5*\U-0.5*\u, 2*\u) -- (\U-2*\u, 2*\u) -- (\U-2*\u, 3*\u);
	      		\draw[\localcolor] (\U-2*\u, 3*\u) -- (\U-2*\u, \U-3*\u);
	      		\draw[\localcolor, rounded corners=0.5*\u]
			(\U-2*\u, \U-3*\u) -- (\U-2*\u, \U-2*\u) -- (0, \U-2*\u) -- (0, \U-3*\u);
	      		\draw[\localcolor] (0, \U-3*\u) -- (0,\u);
	      		\draw[\localcolor] (0, \u) -- (0,0);
	   	   \end{scope}
		}
	}
	\end{scope}
	
	\foreach \x in {-1, 0, ..., 2}
	\foreach \y in {-2, ..., 1} {
		\draw (\x*\U-0.5*\U,\y*\U+0.5*\U) node {\tiny $\bpt$};
	}
	\draw[green] (-2*\U, 0) -- (2*\U, 0);
	\draw[blue] (-2*\U, -\U) -- (2*\U, -\U);
	
	\draw[fill=black] (-0.5*\U, -\U) circle (\radius)
		++ (\labelshift,-2*\labelshift) node {\tiny $x$};
	\draw[fill=black] (0.5*\U, \U) circle (\radius)
		++ (\labelshift,-2*\labelshift) node {\tiny $x$};
	\draw[fill=black] (-0.5*\U+0.5*\u,0) circle (\radius)
		++ (-2.5*\labelshift,1.5*\labelshift) node {\tiny $y_3$};
	\draw[fill=black] (0,0) circle (\radius)
		++ (2.5*\labelshift,1.5*\labelshift) node {\tiny $y_2$};
	\draw[fill=black] (0.5*\U-0.5*\u,0) circle (\radius)
		++ (2.5*\labelshift,-1.5*\labelshift) node {\tiny $y_1$};
		
	\draw[red] (0, \U)
		++ (-3*\labelshift,1.5*\labelshift) node {\small $\widetilde{\gamma_x}$};
	\draw[blue] (2*\U, -\U)
		++ (2*\labelshift,-4*\labelshift) node {\small $\widetilde{\beta_{sym}}$};
	\draw[green] (2*\U, 0)
		++ (2*\labelshift,3*\labelshift) node {\small $\widetilde{\beta_{sym}}$};
		
	\draw[red, ->] (-\U+2*\u, -0.5*\U-\u) -- (-\U+2*\u, -0.5*\U);

\begin{scope}[xshift=4*\U]

\draw[step=\U, lightgray, densely dotted, thin, xshift=0.5*\U, yshift=0.5*\U]
(-1.75*\U,-1.75*\U) grid (7.75*\U,1.75*\U);
\draw[step=\U, lightgray, thin, xshift=0.5*\U, yshift=0.5*\U]
(-1.5*\U,-1.5*\U) grid (7.5*\U,1.5*\U);
\foreach \x in {0, ..., 8}
\foreach \y in {-1, 0, 1} {
	\draw (\x*\U-0.5*\U,\y*\U+0.5*\U) node {\tiny $\bpt$};
}

		\renewcommand{\localcolor}{red}
		
	      		\draw[\localcolor, rounded corners=0.5*\u]
			(0.5*\U-0.5*\u, 0) -- (0.5*\U-0.5*\u, 2*\u) -- (\U-2*\u, 2*\u) -- (\U-2*\u, 3*\u);
	      		\draw[\localcolor] (\U-2*\u, 3*\u) -- (\U-2*\u, \U-3*\u);
	      		\draw[\localcolor, rounded corners=0.5*\u]
			(\U-2*\u, \U-3*\u) -- (\U-2*\u, \U-2*\u) -- (0, \U-2*\u) -- (0, \U-3*\u);
	      		\draw[\localcolor] (0, \U-3*\u) -- (0,\u);
	      		\draw[\localcolor] (0, \u) -- (0,0);
			
			\begin{scope}[xscale=-1, yscale=-1]
	      		\draw[\localcolor, rounded corners=0.5*\u]
			(-0.5*\U, -\U) -- (-0.5*\U, -\U+\u) -- (\u,-\U+\u) -- (\u,-\U+2*\u);
	      		\draw[\localcolor] (\u,-\U+2*\u) -- (\u,-2*\u);
	      		\draw[\localcolor, rounded corners=0.5*\u]
			(\u, -2*\u) -- (\u, -1*\u) -- (0.5*\U-0.5*\u, -1*\u) -- (0.5*\U-0.5*\u, 2*\u) -- (\U-2*\u, 2*\u) -- (\U-2*\u, 3*\u);
	      		\draw[\localcolor] (\U-2*\u, 3*\u) -- (\U-2*\u, \U-3*\u);
	      		\draw[\localcolor, rounded corners=0.5*\u]
			(\U-2*\u, \U-3*\u) -- (\U-2*\u, \U-2*\u) -- (0, \U-2*\u) -- (0, \U-3*\u);
	      		\draw[\localcolor] (0, \U-3*\u) -- (0,\u);
	      		\draw[\localcolor] (0, \u) -- (0,0);
	   	   	\end{scope}
			
		\renewcommand{\localcolor}{InkscapePurple}
		\begin{scope}[xshift=\U-0.5*\u, yshift=\U]
			\draw[\localcolor, rounded corners=0.5*\u]
			(-0.5*\U, -\U) -- (-0.5*\U, -\U+\u) -- (\u,-\U+\u) -- (\u,-\U+2*\u);
	      		\draw[\localcolor] (\u,-\U+2*\u) -- (\u,-2*\u);
	      		\draw[\localcolor, rounded corners=0.5*\u]
			(\u, -2*\u) -- (\u, -1*\u) -- (0.5*\U-0.5*\u, -1*\u) -- (0.5*\U-0.5*\u, 0);
			\draw[\localcolor, ->] (\u, -0.25*\U-\u) -- (\u, -0.25*\U);
		\end{scope}
		
		\draw[blue] (0.5*\U, \U) -- (1.5*\U-\u,\U);
				
		\draw[red, ->] (-\U+2*\u, -0.5*\U-\u) -- (-\U+2*\u, -0.5*\U);
		\draw[blue, ->] (0.5*\U, \U) -- (0.75*\U-\u,\U);
		
		\draw[fill=black] (0.5*\U, \U) circle (\radius)
			++ (\labelshift,-2*\labelshift) node {\tiny $x$};
		\draw[fill=black] (0.5*\U-0.5*\u,0) circle (\radius)
			++ (2.5*\labelshift,-1.5*\labelshift) node {\tiny $y_1$};
		\draw[fill=black] (1.5*\U-\u,\U) circle (\radius);
		
		\draw[red] (0, \U)
			++ (-3*\labelshift,3*\labelshift) node {\small $\widetilde{\gamma_x} - \rho_{x, y_1}$};
		\draw[blue] (\U, \U)
			++ (0,3*\labelshift) node {\small $\widetilde{\beta^k}$};
		\draw[InkscapePurple] (\U, 0.25*\U)
			++ (5*\labelshift,0) node {\small $\rho_{x, y_1}$};
		
		\draw (0,-1.5*\U) node {\small $\mathrm{Area}(P_1')=1$};

	\begin{scope}[xshift=3*\U]
		\renewcommand{\localcolor}{red}
			\begin{scope}[xscale=-1, yscale=-1]
	      		\draw[\localcolor, rounded corners=0.5*\u]
			(-0.5*\U, -\U) -- (-0.5*\U, -\U+\u) -- (\u,-\U+\u) -- (\u,-\U+2*\u);
	      		\draw[\localcolor] (\u,-\U+2*\u) -- (\u,-2*\u);
	      		\draw[\localcolor, rounded corners=0.5*\u]
			(\u, -2*\u) -- (\u, -1*\u) -- (0.5*\U-0.5*\u, -1*\u) -- (0.5*\U-0.5*\u, 2*\u) -- (\U-2*\u, 2*\u) -- (\U-2*\u, 3*\u);
	      		\draw[\localcolor] (\U-2*\u, 3*\u) -- (\U-2*\u, \U-3*\u);
	      		\draw[\localcolor, rounded corners=0.5*\u]
			(\U-2*\u, \U-3*\u) -- (\U-2*\u, \U-2*\u) -- (0, \U-2*\u) -- (0, \U-3*\u);
	      		\draw[\localcolor] (0, \U-3*\u) -- (0,\u);
	      		\draw[\localcolor] (0, \u) -- (0,0);
			\end{scope}
			
		\renewcommand{\localcolor}{InkscapePurple}
		\begin{scope}[xshift=0.5*\U, yshift=\U]
			\draw[\localcolor, rounded corners=0.5*\u]
			(-0.5*\U, -\U) -- (-0.5*\U, -\U+\u) -- (\u,-\U+\u) -- (\u,-\U+2*\u);
	      		\draw[\localcolor] (\u,-\U+2*\u) -- (\u,-2*\u);
	      		\draw[\localcolor, rounded corners=0.5*\u]
			(\u, -2*\u) -- (\u, -1*\u) -- (0.5*\U-0.5*\u, -1*\u) -- (0.5*\U-0.5*\u, 2*\u) -- (\U-2*\u, 2*\u) -- (\U-2*\u, 3*\u);
	      		\draw[\localcolor] (\U-2*\u, 3*\u) -- (\U-2*\u, \U-3*\u);
	      		\draw[\localcolor, rounded corners=0.5*\u]
			(\U-2*\u, \U-3*\u) -- (\U-2*\u, \U-2*\u) -- (0, \U-2*\u) -- (0, \U-3*\u);
	      		\draw[\localcolor] (0, \U-3*\u) -- (0,\u);
	      		\draw[\localcolor] (0, \u) -- (0,0);
			\draw[\localcolor, ->] (\u, -0.25*\U-\u) -- (\u, -0.25*\U);
		\end{scope}
				
		\draw[red, ->] (-\U+2*\u, -0.5*\U-\u) -- (-\U+2*\u, -0.5*\U);
		
		\draw[fill=black] (0.5*\U, \U) circle (\radius)
			++ (1.5*\labelshift,1.5*\labelshift) node {\tiny $x$};
		\draw[fill=black] (0,0) circle (\radius)
			++ (2.5*\labelshift,-1.5*\labelshift) node {\tiny $y_2$};
			
		\draw[red] (0, \U)
			++ (-3*\labelshift,3*\labelshift) node {\small $\widetilde{\gamma_x} - \rho_{x, y_2}$};
		\draw[InkscapePurple] (0.5*\U, 0.25*\U)
			++ (5*\labelshift,0) node {\small $\rho_{x, y_2}$};
			
		\draw (0,-1.5*\U) node {\small $\mathrm{Area}(P_2')=\frac52$};
	
	\end{scope}

	\begin{scope}[xshift=6*\U]
		\renewcommand{\localcolor}{red}
		\begin{scope}[xshift=1.5*\u, yshift=\U-\u]
	      		\draw[\localcolor, rounded corners=0.5*\u]
			(-0.5*\U-\u, -\U+\u) -- (\u,-\U+\u) -- (\u,-\U+2*\u);
	      		\draw[\localcolor] (\u,-\U+2*\u) -- (\u,-2*\u);
	      		\draw[\localcolor, rounded corners=0.5*\u]
			(\u, -2*\u) -- (\u, -1*\u) -- (0.5*\U-1.5*\u, -1*\u) -- (0.5*\U-1.5*\u, \u);
		\end{scope}
		
		\renewcommand{\localcolor}{InkscapePurple}
		\begin{scope}[xshift=0.5*\u, yshift=\U]
			\draw[\localcolor, rounded corners=0.5*\u]
			(-0.5*\U, -\U) -- (-0.5*\U, -\U+\u) -- (\u,-\U+\u) -- (\u,-\U+2*\u);
	      		\draw[\localcolor] (\u,-\U+2*\u) -- (\u,-2*\u);
	      		\draw[\localcolor, rounded corners=0.5*\u]
			(\u, -2*\u) -- (\u, -1*\u) -- (0.5*\U-0.5*\u, -1*\u) -- (0.5*\U-0.5*\u, 2*\u) -- (\U-2*\u, 2*\u) -- (\U-2*\u, 3*\u);
	      		\draw[\localcolor] (\U-2*\u, 3*\u) -- (\U-2*\u, \U-3*\u);
	      		\draw[\localcolor, rounded corners=0.5*\u]
			(\U-2*\u, \U-3*\u) -- (\U-2*\u, \U-2*\u) -- (0, \U-2*\u) -- (0, \U-3*\u);
	      		\draw[\localcolor] (0, \U-3*\u) -- (0,\u);
	      		\draw[\localcolor] (0, \u) -- (0,0);
			\begin{scope}[xscale=-1, yscale=-1]
	      			\draw[\localcolor, rounded corners=0.5*\u]
				(0.5*\U-0.5*\u, 0) -- (0.5*\U-0.5*\u, 2*\u) -- (\U-2*\u, 2*\u) -- (\U-2*\u, 3*\u);
	      			\draw[\localcolor] (\U-2*\u, 3*\u) -- (\U-2*\u, \U-3*\u);
	      			\draw[\localcolor, rounded corners=0.5*\u]
				(\U-2*\u, \U-3*\u) -- (\U-2*\u, \U-2*\u) -- (0, \U-2*\u) -- (0, \U-3*\u);
	      			\draw[\localcolor] (0, \U-3*\u) -- (0,\u);
	      			\draw[\localcolor] (0, \u) -- (0,0);
			\end{scope}
			\draw[\localcolor, ->] (\U-2*\u, 0.75*\U-\u) -- (\U-2*\u, 0.75*\U);
		\end{scope}

		\draw[blue] (0.5*\U, \U) -- (-0.5*\U+\u,\U);
				
		\draw[red, ->] (2.5*\u, 0.75*\U-\u) -- (2.5*\u, 0.75*\U);
		\draw[blue, ->] (0.5*\U, \U) -- (-0.25*\U+\u,\U);
		
		\draw[fill=black] (0.5*\U, \U) circle (\radius)
			++ (\labelshift,-2*\labelshift) node {\tiny $x$};
	\draw[fill=black] (-0.5*\U+0.5*\u,0) circle (\radius)
		++ (-2.5*\labelshift,-1.5*\labelshift) node {\tiny $y_3$};
		\draw[fill=black] (-0.5*\U+\u,\U) circle (\radius);
		
		\draw[red] (0, 0)
			++ (11*\labelshift,0) node {\small $\widetilde{\gamma_x} - \rho_{x, y_3}$};
		\draw[blue] (0, \U)
			++ (-3*\labelshift,3*\labelshift) node {\small $\widetilde{\beta^k}$};
		\draw[InkscapePurple] (\U, 1.25*\U)
			++ (5*\labelshift,0) node {\small $\rho_{x, y_1}$};
		
		\draw (0,-1.5*\U) node {\small $\mathrm{Area}(P_3')=0$};

	\end{scope}

\end{scope}

\end{scope}

\end{tikzpicture}
\caption{Paths used to compute gradings in Example \ref{ex:grading-example}.
The path $P'$ is obtained by concatenating a shift of $\rho_{x,y}$, the inverse of $\widetilde{\beta^k}$, and the inverse of $\widetilde{\gamma_x} - \rho_{x,y}$.}
\label{fig:grading-example}
\end{figure}
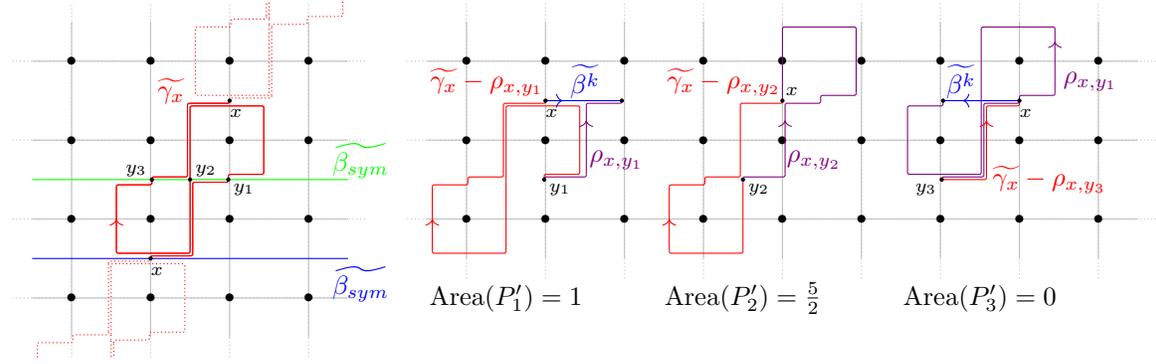

%% file: sections/delta-d.tex

\section{An invariant from symmetric curves}\label{sec:delta-sym}

\subsection{\texorpdfstring{$\Deltadb$ for symmetric curves}{Delta\_sym for symmetric curves}}
In this section we extract a numerical invariant $\Deltadb$ from the immersed curve invariant for certain bordered $\Z_2$-homology solid tori $M$; in the case that a certain filling of $M$ is an L-space we will see that $\Deltadb$ agrees with the difference between the $d$-invariants of that filling in its two spin structures. To begin, we define $\Deltadb(\gamma)$ for an immersed curve $\gamma$ in the punctured torus that is fixed by the elliptic involution and that has an odd number of $\alpha^{\pm 1}$ letters or an odd number of $\beta^{\pm 1}$ letters in the corresponding cyclic word $w_\gamma$. We assume that $\gamma$ is in rectilinear position and we consider a pair of antipodal fixed points $x_0$ and $x_1$ of the elliptic involution; that is, we choose points $x_0$ and $x_1$ such that rotation by $\pi$ about $x_0$ is a symmetry of $\gamma$ and fixes $x_1$. Note that $x_0$ and $x_1$ must lie at either a midpoint or an endpoint of a unit horizontal or vertical segment. Let $\gamma_{01}$ denote a path in $\gamma$ from $x_0$ to $x_1$, let $\tilde\gamma_{01}$ denote a lift to the plane, and let $\ell_{01}$ be the straight line segment from the initial point of $\tilde\gamma_{01}$ to the terminal point of $\tilde\gamma_{01}$. We define the \emph{deviation from the diagonal of $\tilde\gamma_{01}$} to be the (signed) area enclosed between $\tilde\gamma_{01}$ and $\ell_{01}$, in which the area of each region of $\R^2 \setminus (\tilde\gamma_{01}\cup \ell_{01})$ is counted with multiplicity given by the winding number of the path $\gamma_{01} \cdot \ell_{01} ^{-1}$ around the region. 

We will define $\Deltadb(\gamma)$ to be twice the deviation from the diagonal of $\tilde\gamma_{01}$ for a path $\gamma_{01}$ as described above. We first observe that this is well defined once we select the points $x_0$ and $x_1$; that is, it does not depend on the choice of path from $x_0$ to $x_1$. There are two obvious paths from $x_0$ to $x_1$ each covering a different half of the curve $\gamma_0$, but rotation by $\pi$ about $x_0$ takes one of these paths to the other, so their deviation from the diagonal is clearly the same. We could also consider paths that differ from these by adding full loops around $\gamma$, but the deviation from the diagonal of $\gamma$ (as a path from $x_0$ to itself) is zero since by symmetry the contribution of the portion from $x_0$ to $x_1$ cancels the contribution of the portion from $x_1$ to $x_0$.
We also note that the pair of points $\{x_0, x_1\}$ is effectively unique in the sense that if there is another such pair $\{x'_0, x'_1\}$ then there must be a translational symmetry of $\tilde\gamma$ taking $\{x_0, x_1\}$ to $\{x'_0, x'_1\}$; choosing this other pair of points (and possibly reindexing them) would produce the same value for $\Deltadb(\gamma)$.
%
However, switching the role of $x_0$ and $x_1$ changes $\Deltadb(\gamma)$ by a sign, since the deviation from the diagonal of $\tilde\gamma_{10}$, the lift of a path from $x_1$ to $x_0$, is the opposite of the deviation from the diagonal of $\tilde\gamma_{01}$. Thus as described so far $\Deltadb(\gamma)$ is well defined only up to sign.

We can remove this sign indeterminacy in the definition of $\Deltadb(\gamma)$ by setting a convention for labelling the pair of fixed points. Here we will use the assumption that the word $w_\gamma$ has an odd number of $\alpha^{\pm 1}$ letters or an odd number of $\beta^{\pm 1}$ letters. If $w_\gamma$ has an odd number of $\alpha^{\pm 1}$ letters, then one of the two fixed points falls at the midpoint of a unit vertical segment (i.e. has integral $y$-coordinate) and the other falls on a maximal horizontal segment (i.e. has half-integer $y$-coordinate); in this case we let $x_0$ be the point with half-integer $y$-coordinate and $x_1$ be the point with integer $y$-coordinate. If $w_\gamma$ has an odd number of $\beta^{\pm 1}$ letters, then one of the two fixed points has integral $x$-coordinate and the other has half-integer $x$-coordinate; in this case we let $x_0$ be the point with integer $x$-coordinate and $x_1$ be the point with half-integer $x$-coordinate. See Figure \ref{fig:delta-sym-examples} for examples of $x_0$ and $x_1$. Note that the convention above is consistent in the case that $w_\gamma$ has both an odd number of $\alpha^{\pm 1}$ letters and an odd number of $\beta^{\pm 1}$ letters, since no point on $\gamma$ can have integer $x$- and $y$-coordinates. By setting this convention, $\Deltadb(\gamma)$ is well-defined.
\begin{definition}\label{def:delta-sym-curve}
If $\gamma$ is a homotopy class of immersed curves in the punctured torus that is fixed by the elliptic involution and whose cyclic word $w_\gamma$ has an odd number of $\alpha^{\pm}$ letters or an odd number of $\beta^{\pm}$ letters, then we define $\Deltadb(\gamma)$ to be twice the deviation from the diagonal of $\tilde\gamma_{01}$, where $x_0$, $x_1$, and $\tilde\gamma_{01}$ are as described above.
\end{definition}

\begin{example}\label{ex:delta-sym-example1}
The computation of $\Deltadb(\gamma)$ for the immersed curve $\gamma$ represented by the cyclic word $(\beta \alpha \beta \alpha \beta^{-1} \alpha^{-2} \beta^{-1} \alpha \beta \alpha)$, which is our running example appearing in Figures \ref{fig:immersed-curve-example}, \ref{fig:pairing-example}, and \ref{fig:symmetric-pairing-position} and arising from the plumbing tree in Example \ref{ex:plumbing-tree-example1}, is shown on the left side of Figure \ref{fig:delta-sym-examples}. The point $x_0$ has integer $x$-coordinate and the point $x_1$ has half-integer $x$-coordinate. The deviation from the diagonal of $\tilde \gamma_{01}$ is $\frac 5 4$, the area of the shaded region in the figure. Thus $\Deltadb(\gamma) = \frac 5 2$.
\end{example}

\begin{example}\label{ex:delta-sym-example2}
The computation of $\Deltadb(\gamma)$ for the immersed curve $\gamma$ represented by the cyclic word $(\beta \alpha^2 \beta \alpha^3 \beta \alpha \beta \alpha^3 \beta \alpha^3 \beta \alpha \beta \alpha^3)$, which arises from the plumbing tree in Example \ref{ex:plumbing-tree-example2}, is shown on the right side of Figure \ref{fig:delta-sym-examples}. The deviation from the diagonal of $\tilde \gamma_{01}$ is the signed area of the shaded region in the figure, with purple regions counting negatively. The easiest way to compute this area is to find the area of the right triangle formed by $\ell_{01}$ and the dotted lines in the figure, which is $14$, and subtract the area between $\tilde\gamma_{01}$ and the dotted line, also 14. Thus the deviation from the diagonal is 0 and $\Deltadb(\gamma) = 0$.
\end{example}


\input{figures/13.tex}

\subsection{\texorpdfstring{$\Deltadb$ for $\Z_2$-homology solid tori}{Delta\_sym for Z\_2 homology solid tori}}

We now turn to 3-manifolds $M$ with torus boundary that are $\Z_2$-homology solid tori, which equivalently means they have some Dehn filling that is a $\Z_2$-homology sphere. We first observe that such an $M$ has a single self-conjugate \spinc-structure. 
 
\begin{lemma}
If $M$ is a $\Z_2$-homology solid torus, then $M$ has a single self-conjugate \spinc structure $\s_0$.
\end{lemma}
\begin{proof}
Let $Y$ be a $\Z_2$-homology sphere Dehn filling of $M$.
Since $H^1(Y; \Z/2\Z)=0$, $Y$ is a $\Q HS^3$. By the long exact sequence associated to $0 \to \Z \to \Z \to \Z /2\Z \to 0$, $H^2(Y;\Z)$ has no 2-torsion. This fact, together with the Mayer-Vietoris long exact sequence for $Y = M \cup (S^1 \times D^2)$, implies that $H^2(M;\Z)$ has no 2-torsion either; therefore $|H^2(M;\Z)|$ is odd.

Since $|H^2(M;\Z)|$ is odd, $M$ has an odd number of self-conjugate \spinc structures. In particular it has at least one, which we denote by $\s_0$. 
It is straightforward to check that the set $\Spinc_{sc}(W)$ of self-conjugate \spinc structures of $W$ corresponds to
\[
\Spinc_{sc}(W) =
\{
\s_0 + \alpha \,|\, 2\alpha = 0 \mbox{ in } H^2(W;\Z)
\},
\]
and is therefore a torsor over the 2-torsion subgroup $\mathcal{T}_2(H^2(W;\Z))$. Since $\mathcal{T}_2(H^2(M;\Z)) = 0$, $\s_0$ must be the only self-conjugate \spinc structure of $M$.
\end{proof}

Given a boundary parametrization $(\alpha, \beta)$ on a $\Z_2$-homology solid torus $M$, we can detect which Dehn fillings of $M$ are $\Z_2$-homology spheres by counting $\alpha$ and $\beta$ letters in the cyclic words representing the bordered invariant $\HFhat(M, \alpha, \beta)$.

\begin{lemma}\label{lem:parity-of-alpha-beta}
If $(M,\alpha,\beta)$ is a bordered $\Z_2$-homology solid torus then at least one of the fillings $M(\alpha)$ and $M(\beta)$ is a $\Z_2$-homology sphere. Moreover, $M(\alpha)$ is a $\Z_2$-homology sphere if and only if the number of $\beta$ or $\beta^{-1}$ letters in the cyclic words representing $\HFhat(M,\alpha,\beta)$ is odd and $M(\beta)$ is a $\Z_2$-homology sphere if and only if the number of $\alpha$ or $\alpha^{-1}$ letters in the cyclic words representing $\HFhat(M,\alpha,\beta)$ is odd.
\end{lemma}
\begin{proof}
We first show that a 3-manifold $Y$ is a $\Z_2$-homology sphere if and only if the total rank of $\HFhat(Y)$ is odd. If $Y$ is a $\Z_2$-homology sphere then it has a single spin structure. Since $Y$ is a rational homology sphere as well, it follows that it has a single self-conjugate \spinc structure. Other \spinc structures come in pairs and $\HFhat(Y)$ has odd rank in each \spinc structure, so the total rank of $\HFhat(Y)$ is odd. Conversely, if $Y$ is not a $\Z_2$-homology sphere, then it has an even number of spin structures. It also has an even number of self-conjugate \spinc structures, unless $Y$ is not a rational homology sphere in which case the rank of $\HFhat(Y)$ is even in every \spinc structure.

The argument above shows that $M(\alpha)$ is a $\Z_2$-homology sphere if and only if the total rank of $\HFhat(M(\alpha))$ is odd.  It is clear from the pairing theorem that the total rank of $\HFhat(M(\alpha))$ is congruent modulo 2 to the total number of $\beta$ or $\beta^{-1}$ letters in the cyclic words representing $\HFhat(M, \alpha, \beta)$. Similarly, $M(\beta)$ is a $\Z_2$-homology sphere if and only if $\HFhat(M, \alpha, \beta)$ contains an odd number of $\alpha$ and $\alpha^{-1}$ letters. Finally, note that if neither $M(\alpha)$ nor $M(\beta)$ is a $\Z_2$-homology sphere, then $\HFhat(M,\alpha,\beta)$ contains both an even number of $\alpha^{\pm 1}$ letters and an even number of $\beta^{\pm 1}$ letters. But this property will remain true for any reparametrization of $M$, since any reparametrization can be realized as a sequence of Dehn twists about $\alpha$ or $\beta$, which have the effect of increasing or decreasing the number of $\alpha^{\pm 1}$ letters by the number of $\beta^{\pm 1}$ letters or vice-versa. It follows that no Dehn filling of $M$ is a $\Z_2$-homology sphere, contradicting the fact that $M$ is a $\Z_2$-homology solid torus.
\end{proof}

It will be helpful to keep track of which Dehn fillings of $(M, \alpha, \beta)$ are $\Z_2$-homology spheres, so we introduce the following terminology:
\begin{definition}
We say that $(M, \alpha, \beta)$ is a $\Z_2$-homology solid torus of \emph{type $\alpha$}, \emph{type $\beta$}, or \emph{type $\alpha\beta$} depending on whether the $\alpha$ filling, $\beta$ filling, or both produce a $\Z_2$-homology sphere.
\end{definition}

We will restrict to manifolds $M$ that have at least one L-space filling. The reason is that this assumption ensures $\HFhat(M)$ has exactly one homologically nontrivial curve in each \spinc structure. To see this, note that any straight line of any slope other than the rational longitude intersects any lift of a homologically nontrivial curve in $\HFhat(M)$ to the plane at least once. If there is more than one such curve in any \spinc structure $\s$ of $M$ then for any filling and any \spinc structure on the filling that restricts to $\s$ there is more than one intersection point. We now observe that for a bordered $\Z_2$-homology solid torus $M$ with at least one L-space filling there is a distinguished curve, $\gamma$, which is the unique homologically nontrivial curve in $\HFhat(M, \s_0)$ for $\s_0$ the unique self-conjugate \spinc structure on $M$. Since $\s_0$ is self-conjugate, the conjugation symmetry of $\HFhat(M)$ implies that $\gamma$ is fixed by the elliptic involution. Let $a_\gamma$ and $b_\gamma$ denote the constants such that $[\gamma] = a_\gamma [\alpha] + b_\gamma [\beta]$ in $H_1(T_M; \Z)$.

\begin{lemma} \label{lem:hst-types}
Let $(M,\alpha,\beta)$ be a bordered $\Z_2$-homology solid torus with at least one L-space filling, and let $\gamma$ be the distinguished component of $\HFhat(M,\alpha, \beta)$ as described above. Then at least one of $a_\gamma$ or $b_\gamma$ is odd, and  $M(\alpha)$ (resp. $M(\beta)$) is a $\Z_2$-homology sphere if and only if $b_\gamma$ (resp $a_\gamma$) is odd.
\end{lemma}
\begin{proof}
By Lemma \ref{lem:parity-of-alpha-beta} we just need to show the parity of $a_\gamma$ (resp. $b_\gamma)$ is the same as the parity of the total number of $\alpha^{\pm 1}$ (resp. $\beta^{\pm 1}$) letters in the cyclic words representing $\HFhat(M,\alpha,\beta)$. This is true because homologically trivial curves in $\HFhat(M, \alpha, \beta; \s_0)$ have an even number of $\alpha^{\pm 1}$ letters and an even number of $\beta^{\pm 1}$ letters, and \spinc structures not equal to $\s_0$ come in pairs contributing an even number of either type of letter.
\end{proof}

We can now define $\Deltadb$ for a bordered $\Z_2$-homology solid torus with an L-space filling.

\begin{definition}
Let $(M, \alpha, \beta)$ be a bordered manifold such that $M$ is a $\Z_2$-homology solid torus with an L-space filling. We define $\Deltadb(M, \alpha, \beta)$ to be $\Deltadb(\gamma)$, where $\gamma$ is the unique homologically nontrivial curve in $\HFhat(M, \alpha, \beta; \s_0)$ for $\s_0$ the unique self-conjugate \spinc structure on $M$.
\end{definition}

As an example, the bordered manifolds associated with the rooted plumbing trees in Examples \ref{ex:plumbing-tree-example1} and \ref{ex:plumbing-tree-example1} are both $\Z_2$-homology solid tori and $\Deltadb$ for these manifolds is given by Examples \ref{ex:delta-sym-example1} and \ref{ex:delta-sym-example2}.

\begin{remark}
It is not necessary to restrict to manifolds with an L-space filling to define $\Deltadb$. We only use that the multicurve $\HFhat(M, \alpha, \beta;\s_0)$ has a single homologically nontrivial component, a strictly weaker condition than $M$ having an L-space filling. Even this assumption could be removed by defining $\Deltadb(M, \alpha, \beta)$ to be an appropriate weighted sum of $\Deltadb(\gamma_i)$ over the homologically nontrivial components $\gamma_i$ of $\HFhat(M, \alpha, \beta; \s_0)$. But in what follows we need to restrict to manifolds with L-space fillings anyway to ensure that we can use the merge operation in its simple form.
\end{remark}

It will be helpful to observe that when $\Deltadb (M, \alpha, \beta)$ is defined, the coordinates of the symmetric points $x_0$ and $x_1$ (in the torus viewed as $\R^2 / \Z^2$) determine and are determined by the $\alpha, \beta$ type of $(M, \alpha, \beta)$ as follows:
\begin{prop} \label{prop:x0x1-coordinates-by-type}
Let $M$ be a $\Z_2$-homology solid torus for which $\Deltadb$ is defined, and let $x_0$ and $x_1$ be the fixed points of the elliptic involution on the distinguished component of $\HFhat(M)$.
\begin{itemize}
\item $M$ has type $\alpha$ if and only if $x_0$ is at $(0, \tfrac 1 2)$ and $x_1$ is at $(\tfrac 1 2, \tfrac 1 2)$;
\item $M$ has type $\beta$ if and only if $x_0$ is at $(\tfrac 1 2, \tfrac 1 2)$ and $x_1$ is at $(\tfrac 1 2, 0)$;
\item $M$ has type $\alpha \beta$ if and only if $x_0$ is at $(0, \tfrac 1 2)$ and $x_1$ is at $(\tfrac 1 2, 0)$.
\end{itemize}
\end{prop}

\begin{proof}
First note that the coordinates of $x_0$ and $x_1$ must be either $(0,\tfrac 1 2)$, $(\tfrac 1 2, 0)$, or $(\tfrac 1 2, \tfrac 1 2)$; this follows from the following facts: each coordinate is in $\tfrac 1 2 \Z$ since each point lies on the midpoint or endpoint of a unit horizontal or vertical segment, and neither point can coincide with the puncture at $(0,0)$. If $M$ is type $\alpha$, then by the proof of Lemma \ref{lem:parity-of-alpha-beta} there are an odd number of $\beta$ segments in $\gamma$ and an even number of $\alpha$ segments; it follows that in this case $x_0$ and $x_1$ have the same $y$-coordinate and different $x$-coordinates. The pair of points must have coordinates $(0, \tfrac 1 2)$ and $(\tfrac 1 2, \tfrac 1 2)$, and by the convention defined immediately before Definition \ref{def:delta-sym-curve}, $x_0$ is the point with integer $x$-coordinate.  The cases that $M$ has type $\beta$ or type $\alpha\beta$ are similar and left to the reader.
\end{proof}

\subsection{\texorpdfstring{Effect of plumbing tree moves on $\Deltadb$}{Effect of plumbing tree moves on Delta\_sym}}

We are particularly interested in the invariant $\Deltadb(M,\alpha, \beta)$ when $(M, \alpha, \beta)$ is the bordered manifold associated with a rooted plumbing tree $(\Gamma, v)$; in this case we will also use the notation $\Deltadb(\Gamma, v)$ for $\Deltadb(M, \alpha, \beta)$. For the remainder of this section we will explore how $\Deltadb(\Gamma, v)$ changes under the elementary operations $\extend$, $\twist^{\pm 1}$ and $\merge$ on rooted plumbing trees.

\begin{proposition}[Effect of extend move on $\Deltadb$]\label{prop:delta-d-extend}
Let $(\Gamma, v)$ be a rooted tree such that $M_{\Gamma, v}$ is a $\Z_2$-homology solid torus with an L-space filling and suppose that $(\Gamma, v) = \extend(\Gamma', v')$. Then $\Deltadb(\Gamma', v')$ is defined and 
$$\Deltadb(\Gamma, v) = - \Deltadb(\Gamma', v').$$
\end{proposition}
\begin{proof}
It is clear that $\Deltadb(\Gamma', v')$ is defined since $(\Gamma', v')$ defines the same manifold as $(\Gamma, v)$ just with a different parametrization. Let $\gamma$ be the distinguished curve component of the bordered invariant for $(\Gamma, v)$, and let $\gamma'$ be the distinguished curve component associated with $(\Gamma', v')$. Since the extend operation $\extend$ is just a reparametrizing taking $\beta$ to $\alpha$ and $\alpha$ to $-\beta$, $\gamma'$ is a 90 degree rotation of $\gamma$. It is clear that the deviation from the diagonal of $\tilde\gamma_{01}$ is the same as the deviation from the diagonal of a rotation of $\tilde\gamma_{01}$, so $\Deltadb$ is unchanged up to sign. However, the two symmetric points $x_0$ and $x_1$ on $\gamma$ are labeled such that either $x_0$ has an integral $x$-coordinate and $x_1$ has a half-integral $x$-coordinate, or $x_0$ has a half-integral $y$-coordinate and $x_1$ has an integral $y$-coordinate. These conditions are switched after rotation by 90 degrees, so the images of $x_0$ and $x_1$ on $\gamma$ under the rotation are $x'_1$ and $x'_0$, respectively. It follows that the deviation from the diagonal of $\tilde\gamma'_{01}$ is the same as the deviation from the diagonal of $\tilde\gamma_{10}$, and so $\Deltadb(\Gamma, v) = - \Deltadb(\Gamma', v')$.
\end{proof}

\begin{proposition}[Effect of merge move on $\Deltadb$]\label{prop:delta-d-merge}
Let $(\Gamma, v)$ be a rooted tree such that $M_{\Gamma, v}$ is a $\Z_2$-homology solid torus for which some filling along a slope $r$ not in $\Z \cup \{\infty\}$ is an L-space, and suppose $(\Gamma, v) = \merge( (\Gamma', v'), (\Gamma'', v''))$ where $\Gamma'$ and $\Gamma''$ each have more than one vertex. Then $M_{\Gamma', v'}$ and $M_{\Gamma'', v''}$ are both Floer simple, they are both $\Z_2$-homology solid tori, and they are not both of type $\beta$. Moreover, $\Deltadb$ for $(\Gamma, v)$ can be obtained from $\Deltadb$ for $(\Gamma', v')$ and $(\Gamma'', v'')$ as follows:

\begin{itemize}
\item if $M_{\Gamma', v'}$ has  type $\beta$ and $M_{\Gamma'', v''}$ has type $\alpha$, then $\Deltadb(\Gamma, v) = \Deltadb(\Gamma', v')$
\item if $M_{\Gamma', v'}$ has  type $\beta$ and $M_{\Gamma'', v''}$ has type $\alpha\beta$, then $\Deltadb(\Gamma, v) = -\Deltadb(\Gamma', v')$
\item if $M_{\Gamma'', v''}$ has  type $\beta$ and $M_{\Gamma', v'}$ has type $\alpha$, then $\Deltadb(\Gamma, v) = \Deltadb(\Gamma'', v'')$
\item if $M_{\Gamma'', v''}$ has  type $\beta$ and $M_{\Gamma', v'}$ has type $\alpha\beta$, then $\Deltadb(\Gamma, v) = -\Deltadb(\Gamma'', v'')$
\item if neither $M_{\Gamma', v'}$ nor $M_{\Gamma'', v''}$ has type $\beta$, then $\Deltadb(\Gamma, v) = \Deltadb(\Gamma', v') + \Deltadb(\Gamma'', v'')$.
\end{itemize}
\end{proposition}

%
%
\begin{proof}
The fact that both plumbing tree inputs to the merge operation must correspond to Floer simple manifolds follows from Lemma \ref{lem:Lspace-merge-condition}.
Moreover, $\infty$ is a strict $L$-space slope for at least one of the two inputs, hence one of the two input words consists in calculus notation only of $c$ letters, up to orientation. To see that both inputs determine $\Z_2$-homology solid tori, we use Lemma \ref{lem:merge-parity} and observe that if the cyclic words associated with either input have an even number of $\alpha$ letters and an even number of $\beta$ letters then the same is true of the cyclic words associated with $(\Gamma, v)$, contradicting the assumption that $M_{\Gamma, v}$ is a $\Z_2$-homology solid torus. Moreover, if both inputs are $\Z_2$-homology solid tori of type $\beta$ then the corresponding collections of cyclic words for both inputs have even numbers of $\beta$ letters and once again the cyclic words associated with $(\Gamma, v)$ would have an even number of $\alpha$ letters and an even number of $\beta$ letters, a contradiction.

To compute $\Deltadb(\Gamma, v)$ from $\Deltadb(\Gamma', v')$ and $\Deltadb(\Gamma'', v'')$ we must relate the distinguished symmetric component $\gamma$ of the merged invariant $\HFhat(M_{\Gamma, v})$ to the distinguished symmetric components of $\HFhat(M_{\Gamma', v'})$ and $\HFhat(M_{\Gamma'', v''})$, which we denote by $\gamma'$ and $\gamma''$. At least one of these curves has only letters of type $c$ in loop calculus notation; we will assume without loss of generality that $M_{\Gamma', v'}$ has this property. We may now use the simplified version of the merge operation described in Section \ref{sec:graph-mfds}, and its geometric realization. Recall that the components of $\HFhat(M_{\Gamma, v})$ come from the vertical sum of various horizontal shifts of components of $\HFhat(M_{\Gamma', v'})$ and $\HFhat(M_{\Gamma'', v''})$ after lifting these curves to $\R^2 \setminus \Z^2$. We claim that $\tilde\gamma$ is the vertical sum of $\tilde\gamma'$ with a particular horizontal shift of $\tilde\gamma''$, namely one for which some lifts $\tilde x'_1$ and $\tilde x''_1$ of $x'_1$ and $x''_1$ have the same $x$-coordinate. Here $\{x'_0, x'_1\}$ and $\{x''_0, x''_1\}$ are the pairs of fixed points of the elliptic involution symmetry on $\gamma'$ and $\gamma''$, respectively, labeled using the usual convention. By Proposition \ref{prop:x0x1-coordinates-by-type}, $x'_1$ and $x''_1$ always have $x$-coordinate in $\Z + \tfrac 1 2$, so it is always possible to choose a lift $\tilde\gamma''$ with this property (in contrast, the $x$-coordinate of $x'_0$ or $x''_0$ may be an integer or half-integer depending on the type of the corresponding manifold). Let $\tilde\gamma$ denote this particular vertical sum of $\tilde\gamma''$ and $\tilde\gamma'$. It is clear that $\tilde\gamma$ is symmetric under rotation by $\pi$ about the point corresponding to $\tilde x'_1$ and $\tilde x''_1$.
Moreover, the component of $\HFhat(M_{\Gamma, v})$ arising from the lift of $\gamma''$ that is shifted rightward from $\tilde\gamma''$ by $k$ units is the same as the rotation by $\pi$ of the component arising from shifting $\tilde\gamma''$ leftward by $k$ units. It follows that all components of $\HFhat(M_{\Gamma, v})$ coming from summing $\tilde\gamma'$ with other shifts of $\tilde\gamma''$ come in pairs which are rotations of each other.
Similarly, all components of the merged invariant $\HFhat(M_{\Gamma, v})$ arising from components of $\HFhat(M_{\Gamma', v'})$ other than $\gamma'$ come in pairs that differ by a rotation, since the components of $\HFhat(M_{\Gamma', v'})$ other than $\gamma'$ come in such pairs. The same is true for components of the merged invariant $\HFhat(M_{\Gamma, v})$ arising from components of $\HFhat(M_{\Gamma'', v''})$  that are not in the unique self-conjugate \spinc structure. Thus all homologically nontrivial components of $\HFhat(M_{\Gamma, v})$ other than the one specified above come in pairs. From this it is straightforward to see that $\tilde\gamma$ must be (a lift of) the distinguished symmetric curve in $\HFhat(M_{\Gamma, v})$.

Let $n$, $n'$, and $n''$ denote the signed count of $\beta$ letters in the cyclic words representing $\gamma$, $\gamma'$, and $\gamma''$, respectively. Note that $n$ is the least common multiple of $n'$ and $n''$, since the merge of the words representing $\gamma'$ and $\gamma''$ has a signed total of $n'n''$ beta letters (this follows from Lemma \ref{lem:merge-parity}), which is then evenly divided amongst the $gcd(n',n'')$-many components of the merge of the words representing $\gamma'$ and $\gamma''$. Let $m'$ and $m''$ be the relatively prime integers such that $n' m' = n'' m'' = \text{lcm}(n', n'')$.
When we construct $\tilde \gamma$ by vertically adding $\tilde\gamma'$ to $\tilde\gamma''$, the curve $\tilde\gamma$ repeats after $m''$ periods of $\tilde\gamma''$. More precisely, we will consider the path $\gamma''_{[m'']}$ consisting of $m''$ laps around $\gamma''$, starting and ending at $x''_1$, and the path $\gamma''_{[m''/2]}$ which consists of $m''/2$ full laps around $\gamma''$ starting at $x''_1$ and ending at $x''_1$ or $x''_0$ depending on whether $m''$ is even or odd. In a similar way, we define $\gamma'_{[m']}$ and $\gamma'_{[m'/2]}$, and use $\tilde\gamma'_{[m']}$, $\tilde\gamma'_{[m'/2]}$, $\tilde\gamma''_{[m'']}$ and $\tilde\gamma''_{[m''/2]}$ to denote lifts of theses curves to $\R^2 \setminus \Z^2$. To ensure that the vertical sum $\tilde\gamma''_{[m'']} +_v \tilde\gamma'_{[m']}$ gives $\tilde\gamma$, we fix the lift $\tilde\gamma''_{[m'']}$ so that its starting point (a lift of $x''_1$) has the same $x$-coordinate as a lift of $x'_1$ on $\tilde\gamma'$; note that the portion of $\tilde\gamma'$ starting and ending at the same $x$-coordinates as $\tilde\gamma''_{[m'']}$ is a lift $\tilde\gamma'_{[m']}$ of $\gamma'_{[m']}$. 
Now the initial point and the midpoint of the vertical sum $\tilde\gamma''_{[m'']} +_v \tilde\gamma'_{[m']}$ are the two symmetry points, which for now we call $\tilde x_i$ and $\tilde x_j$ respectively since we do not yet know which is a lift of $x_0$ and which is a lift of $x_1$.
(They are distinct because the fact that not both of $M_{\Gamma',v'}$ and $M_{\Gamma'',v''}$ are of type $\beta$ forces at least one between $m'$ and $m''$ to be odd, and from here one can use Proposition \ref{prop:x0x1-coordinates-by-type} to show that the initial point and the midpoint have different type.)
We let $\tilde \gamma_{ij}$ denote this path along $\tilde\gamma$ from $\tilde x_i$ to $\tilde x_j$.

The indices $i$ and $j$ depend on the $\alpha$, $\beta$ types of the input $\Z_2$-homology solid tori. Proposition \ref{prop:x0x1-coordinates-by-type} determines whether the $x$-coordinates and $y$-coordinates of the initial points and midpoints of $\tilde\gamma''_{[m'']}$ and $\tilde\gamma'_{[m']}$ are in $\Z$ or $\Z+\tfrac 1 2$, and using the fact that the vertical sum adds $y$-coordinates and then subtracts $\tfrac 1 2$, we can determine the same information for $\tilde x_i$ and $\tilde x_j$. We can check that $\tilde x_i = \tilde x_1$ and $\tilde x_j = \tilde x_0$ unless one input is of type $\beta$ and the other is of type $\alpha \beta$, in which case $\tilde x_i = \tilde x_0$ and $\tilde x_j = \tilde x_1$. Some examples of the vertical sums described here are shown in Figure \ref{fig:effect-of-merge-examples}.

It is easy to see that the deviation from the diagonal for curves is additive under vertical summing. From this we get that the deviation from the diagonal of $\tilde\gamma_{ij}$ is the sum of the deviation from the diagonal of $\tilde\gamma'_{[m'/2]}$ and of $\tilde\gamma''_{[m''/2]}$.  The deviation from the diagonal of $\tilde\gamma'_{[m'/2]}$ is zero by symmetry if $m'$ is even, since then $\tilde\gamma'_{[m'/2]}$ is a lift of a full number of laps around $\gamma'$ and the deviation from the diagonal from $\tilde x_0$ to $\tilde x_1$ cancels that from $\tilde x_1$ to $\tilde x_0$. Note that $m'$ is even precisely when $M_{\Gamma'', v''}$ has type $\beta$ (so that $n''$ is even) and $M_{\Gamma', v'}$ does not (since $M_{\Gamma', v'}$ and $M_{\Gamma'', v''}$ can't both be type $\beta$). If on the other hand $m'$ is odd, then the deviation from the diagonal of $\tilde\gamma'_{[m'/2]}$ is the same as the deviation from the diagonal of $\tilde\gamma'_{10}$ since these paths agree up to adding full loops around $\tilde\gamma'$. Similarly, the deviation from the diagonal of $\tilde\gamma''_{[m''/2]}$ is zero if $m''$ is even and agrees with the deviation from the diagonal of $\tilde\gamma''_{10}$ if $m''$ is odd.

To summarize, we have that the deviation from the diagonal of $\tilde\gamma_{10}$ is the sum of the deviations from the diagonal of $\tilde\gamma'_{10}$ and of $\tilde\gamma''_{10}$, where either of these contributions is replaced by zero if the $\Z_2$-homology solid torus corresponding to the other curve has type $\beta$, and there is a negative sign if one input has type $\beta$ and the other has type $\alpha\beta$. Since $\Deltadb(\gamma)$ is $-2$ times the deviation from the diagonal of $\tilde\gamma_{10}$, the result follows.
\end{proof}

\begin{figure}
\labellist

\color{red}
\pinlabel {$\widetilde\gamma''_{[\frac{m''}{2}]}$} at -50 310
\pinlabel {\scriptsize $(m'' = 1)$} at -50 230

\pinlabel {$\widetilde\gamma''_{[\frac{m''}{2}]}$} at 720 310
\pinlabel {\scriptsize $(m'' = 2)$} at 720 240

\color{blue}
\pinlabel {$\widetilde\gamma'_{[\frac{m'}{2}]}$} at -50 550
\pinlabel {\scriptsize $(m' = 3)$} at -50 470

\pinlabel {$\widetilde\gamma'_{[\frac{m'}{2}]}$} at 720 550
\pinlabel {\scriptsize $(m' = 3)$} at 720 470

\color{InkscapePurple}
\pinlabel {$\widetilde\gamma''_{[\frac{m''}{2}]} +_v \widetilde\gamma'_{[\frac{m'}{2}]}$} at 20 810

\pinlabel {$\widetilde\gamma''_{[\frac{m''}{2}]} +_v \widetilde\gamma'_{[\frac{m'}{2}]}$} at 720 810

\endlabellist
\hspace{6mm}
\includegraphics[scale = .25]{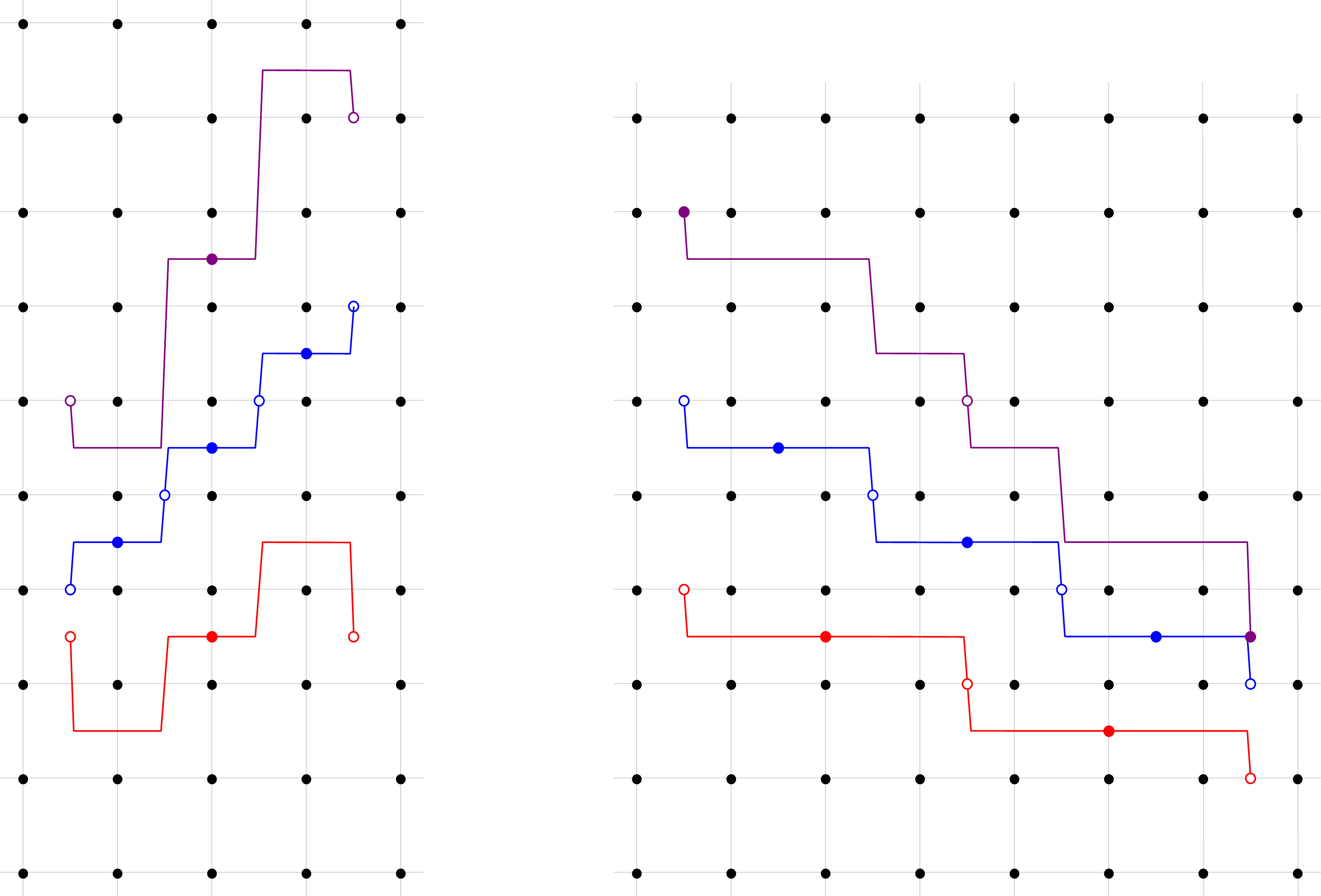}
\caption{Constructing one period of $\tilde\gamma$ from $m'$ periods of $\tilde\gamma'$ and $m''$ periods of $\tilde\gamma''$. Solid dots represent lifts of the points $x_0$, $x'_0$, or $x''_0$, while open dots represent lifts of $x_1$, $x'_1$, or $x''_1$. On the left, $m'$ and $m''$ are both odd. We see that in this case the first half $\tilde\gamma_{ij}$ of the resulting curve is $\tilde\gamma_{10}$. Furthermore, the first halves of $\tilde\gamma'_{[m']}$ and $\tilde\gamma''_{[m'']}$ have deviation from the diagonal equal to that of $\tilde\gamma'_{10}$ and $\tilde\gamma''_{10}$, respectively, and the sum of these is the deviation from the diagonal of $\tilde\gamma_{10}$. On the right, $m'$ is odd while $m''$ is even, so the deviation from the diagonal of the first half of $\tilde\gamma'_{[m']}$ agrees with that of $\tilde\gamma'_{10}$, but the deviation from the diagonal of the first half of $\tilde\gamma''_{[m'']}$ is zero. Moreover, in this case $M_{\Gamma', v'}$ has type $\beta$ and $M_{\Gamma'', v''}$ has type $\alpha\beta$, and we see that the first half $\tilde\gamma_{ij}$ is $\tilde\gamma_{01}$ rather than $\tilde\gamma_{10}$, resulting in an extra minus sign in the $\Deltadb$ of the resulting curve.}
\label{fig:effect-of-merge-examples}
\end{figure}

\begin{proposition}[Effect of twist move on $\Deltadb$]\label{prop:delta-d-twist}
Let $(\Gamma, v)$ be a rooted tree such that $M_{\Gamma, v}$ is a $\Z_2$-homology solid torus with an L-space filling, and suppose $(\Gamma, v) = \twist^{\pm 1}(\Gamma', v')$. Note that $M_{\Gamma', v'}$ is also a $\Z_2$-homology solid torus with an L-space filling, so $\Deltadb(\Gamma', v')$ is defined. Then we have that
\begin{itemize}
	\item if $M_{\Gamma', v'}$ has type $\beta$ then $\Deltadb(\Gamma, v) = -\Deltadb(\Gamma', v')$.
	\item if $M_{\Gamma', v'}$ has type $\alpha$ or type $\alpha\beta$ then $\Deltadb(\Gamma, v) = \Deltadb(\Gamma', v') \mp \frac 1 4$.
\end{itemize} 
\end{proposition}
\begin{proof}
This is nearly a special case of Proposition \ref{prop:delta-d-merge} if we remove the hypothesis that $\Gamma''$ has more than one vertex, since applying $\twist^{\pm 1}$ has the same effect as merging with the rooted tree $\Gamma_{\twist^\pm}$ whose single vertex $v$ has weight $\pm 1$. The proof of Proposition \ref{prop:delta-d-merge} goes through unchanged in this setting except that $M_{\Gamma', v'}$ is no longer guaranteed to be Floer simple (see Remark \ref{rmk:Lspace-merge-condition}). The immersed curve for the corresponding manifold $M_{\Gamma_{\twist^\pm}, v}$ is represented by the cyclic word $(\beta \alpha^{\mp 1})$ and this manifold is a $\Z_2$-homology solid torus of type $\alpha\beta$. We can check that $\Deltadb(\Gamma_{\twist^\pm}, v)= \mp \frac 1 4$. If $M_{\Gamma', v'}$ has type $\beta$ then $\Deltadb(\Gamma, v) = -\Deltadb(\Gamma', v')$ by the second bullet point in Proposition \ref{prop:delta-d-merge}, and if $M_{\Gamma', v'}$ has type $\alpha$ or $\alpha\beta$ then $\Deltadb(\Gamma, v) = \Deltadb(\Gamma', v') \mp \frac 1 4$ by the last bullet point in Proposition \ref{prop:delta-d-merge}.
\end{proof}

\subsection{\texorpdfstring{$\Deltadb$ and grading differences}{Delta\_sym and grading differences}}
The invariant $\Deltadb$ we have defined for a bordered $\Z_2$-homology solid torus $(M,\alpha,\beta)$ is closely related to, though not always the same as, the difference between the gradings in two distinguished generators of $\HFhat$ for a certain filling of $M$. We will focus on the case that $(M,\alpha,\beta)$ has type $\alpha$, which means that Dehn filling along $\beta$ is not a $\Z_2$-homology sphere. By Proposition \ref{prop:x0x1-coordinates-by-type}, the symmetric points $x_0$ and $x_1$ on $\gamma$ both have half-integer $y$-coordinates, which implies that they both occur at midpoints of maximal horizontal segments of $\gamma$. In particular, $x_0$ and $x_1$ determine two distinguished generators of the Floer homology of $\HFhat(M, \alpha, \beta)$ with $\beta_{sym}$, which by abuse of notation we also denote $x_0$ and $x_1$.

The difference in grading between these two generators is given by Corollary \ref{cor:grading-difference-for-fixed-points} as the sum of an area term and a rotation term:
$$\gr(x_1) - \gr(x_0) = \text{Area}(P') - \rot(\gamma_{01})$$
where $P'$ is a lift of $\gamma_{01}$ followed by a lift of $-\gamma_{10}$ translated to match the endpoints.

\begin{lemma}\label{lem:Delta-sym-is-area-term}
If $(M,\alpha,\beta)$ is a bordered $\Z_2$-homology solid torus of type $\alpha$, and $x_0$ and $x_1$ are the generators of $\HFhat(M(\beta))$ arising from the symmetric points $x_0$ and $x_1$ on $\gamma$, then $\Deltadb(M,\alpha,\beta)$ is the area term of $\gr(x_1) - \gr(x_0)$.
\end{lemma}
\begin{proof}
By the rotational symmetry of $\gamma$, $-\gamma_{10}$ is the same as the rotation by $\pi$ about $x_1$ of $\gamma_{01}$. In the cover, the relevant translation of $-\tilde\gamma_{10}$ is the rotation of $\tilde\gamma_{01}$ by $\pi$ about the midpoint of the line segment between the endpoints of $\tilde\gamma_{01}$. It follows that the area enclosed by $P$ is twice the area enclosed by $\tilde\gamma_{01}$ and this line segment, and this by definition is $\Deltadb(M,\alpha,\beta)$.
\end{proof}

In the case that $M(\beta)$ is an L-space, we can show that $\Deltadb(M,\alpha,\beta)$ is precisely $\gr(x_1) - \gr(x_0)$ and moreover that this is the difference in the $d$-invariants between the two spin structures of $M(\beta)$.

\begin{lemma}\label{lem:Delta-sym-is-Delta-d}
Let $(M,\alpha,\beta)$ be a bordered $\Z_2$-homology solid torus of type $\alpha$, and suppose that the lift $\tilde \gamma$ of the distinguished symmetric immersed curve $\gamma$ associated with $(M,\alpha,\beta)$ is embedded in $\R^2 \setminus \Z^2$. If $M(\beta)$ is an L-space, then the two spin structures on $M(\beta)$ can be labeled $\s_0$ and $\s_1$ such that
$$\Deltadb(M,\alpha,\beta) = d(M(\beta); \s_1) - d(M(\beta); \s_0).$$
\end{lemma}
\begin{proof}
We assume that $\tilde\gamma$ is in symmetric $\beta$-pairing position. Because $M(\beta)$ is an L-space, $\tilde\gamma$ intersects each horizontal line of half-integer height exactly once. Because $\tilde\gamma$ is embedded the net rotation along an arc connecting the intersections with horizontal lines at successive heights must be zero. It follows that $\rot(\gamma_{01}) = 0$, and so by 
Lemma \ref{lem:Delta-sym-is-area-term} we have $\Deltadb(M,\alpha,\beta) = \gr(x_1) - \gr(x_0)$.

It remains to show that this grading difference is the difference in the $d$-invariants of the two spin structures. Because $M(\beta)$ is an L-space, the $d$-invariant of $M(\beta)$ in a given \spinc structure is simply the grading of the unique generator of $\HFhat(M(\beta))$ in that \spinc structure. Letting $\s_{x_0}$ and $\s_{x_1}$ be the \spinc structures of $M(\beta)$ containing the generators $x_0$ and $x_1$, we have that 
$$\Deltadb(M,\alpha,\beta) = \gr(x_1) - \gr(x_0) = d(M(\beta); \s_{x_1}) - d(M(\beta); \s_{x_0}).$$
We expect that $\s_{x_0}$ and $\s_{x_1}$ are always the spin structures, but we will not prove this. Instead we argue that all other $d$-invariants for $M(\beta)$ other than those associated with $\s_{x_0}$ and $\s_{x_1}$ come in pairs; this is clearly true for \spinc structures on $M(\beta)$ that do not restrict to the unique self-conjugate \spinc structure on $M$. For those that do, the corresponding generators of $\HFhat(M(\beta))$ come from intersection points of the distinguished symmetric curve $\gamma$ with $\beta_{sym}$ and Corollary \ref{cor:symmetry-same-grading} implies that each generator has the same grading as its rotation about $x_0$. Thus either $d(M(\beta); \s_{x_0}) = d(M(\beta); \s_{x_1})$ or $d(M(\beta); \s_{x_0})$ and $d(M(\beta); \s_{x_1})$ are the only two $d$-invariants that occur an odd number of times for $\HFhat(M(\beta))$. Similarly, since $d(M(\beta); \s) = d(M(\beta); \bar\s)$ for any \spinc structure $\s$ with conjugate $\bar\s$, and since the spin structures of the rational homology sphere $M(\beta)$ are the same as the self-conjugate \spinc structures, either the two spin $d$-invariants are equal or they are the only two $d$-invariants that occur an odd number of times. It follows that if $d(M(\beta); \s_{x_0}) \neq d(M(\beta); \s_{x_1})$ then $\s_{x_0}$ and $\s_{x_1}$ must be the spin structures on $M(\beta)$ and we let $\s_0 = \s_{x_0}$ and $\s_1 = \s_{x_1}$. If instead $d(M(\beta); \s_{x_0}) = d(M(\beta); \s_{x_1})$ then we can let $\s_0$ and $\s_1$ be the two spin structures (which are not necessarily $\s_{x_0}$ and $\s_{x_1}$) indexed in either order. In either case
\[d(M(\beta); \s_{x_1}) - d(M(\beta); \s_{x_0}) = d(M(\beta); \s_1) - d(M(\beta); \s_0). \hfill\qedhere\]
\end{proof}

We remark that when $M(\beta)$ is not an L-space $\Deltadb(M,\alpha,\beta)$ can fail to agree with $d(M(\beta); \s_1) - d(M(\beta); \s_0)$ in two ways. First, if $\rot(\gamma_{01})$ is nonzero, $\Deltadb(M,\alpha,\beta)$ will not agree with the difference between the grading of the distinguished generators $x_1$ and $x_0$. Second, the difference between these distinguished generators may not give the difference between the $d$-invariants. Both of these failures occur with our running example:

\begin{example}\label{ex:running-example-Delta-sym-not-grading-diff}
Consider the manifold $(M, \alpha, \beta)$ associated with the rooted tree in Example \ref{ex:plumbing-tree-example1}. In Example \ref{ex:delta-sym-example1} we computed that $\Deltadb(M, \alpha, \beta) = \frac 5 2$. The gradings for the $\beta$-filling of this manifold were computed in Example \ref{ex:grading-example}. The symmetric points $x_0$ and $x_1$ correspond to the generators $x$ and $y_2$ as labeled in Example \ref{ex:grading-example}, so we have that $\gr(x_1) - \gr(x_0) = \frac 3 2 = \Deltadb(M, \alpha, \beta) - \rot(\gamma_{01})$. We can deduce the $d$-invariants from $\HFhat$ in this case and we find that the $d$-invariant in the \spinc structure containing the generators $y_1$, $y_2$, and $y_3$ is $\gr(y_1) = \gr(y_3)$, and so the difference between the two $d$-invariants is $\frac 1 2$.
\end{example}

%% file: figures/13.tex

\begin{figure}[]
\centering
\begin{tikzpicture}[on top/.style={preaction={draw=white,-,line width=#1}},
on top/.default=4pt]
\newcommand{\drawcurveleft}{
\draw[red, rounded corners=0.5*\u] (-0.5*\U, -\U) -- (\u,-\U) -- (\u,0) -- (\U,0) -- (\U, \U-\u) -- (0,\U-\u) -- (0,0);
}
\newcommand{\drawcurveright}{
\draw[red, rounded corners=0.5*\v] (0,0) -- (0,-\V) -- (-\V, -\V) -- (-\V, -4*\V) -- (-2*\V, -4*\V) -- (-2*\V, -5*\V) -- (-3*\V, -5*\V) -- (-3*\V, -8*\V) -- (-3.5*\V, -8*\V);
}
\def\U{1.2cm}
\def\u{0.07cm}
\def\V{0.7cm}
\def\v{0.05cm}
\def\radiussympoint{0.05cm}

\begin{scope}[yshift=-\U]
\draw[step=\U, lightgray, densely dotted, thin, xshift=-0.5*\U, yshift=-0.5*\U] (-0.75*\U,-0.75*\U) grid (4.75*\U,3.75*\U);
\draw[step=\U, lightgray, thin, xshift=-0.5*\U, yshift=-0.5*\U] (-0.5*\U,-0.5*\U) grid (4.5*\U,3.5*\U);
	
	\foreach \x in {0,1} {
	\begin{scope}[xshift=\x*\U, yshift=2*\x*\U]
      		\foreach \y in {-1, 1}{
		\begin{scope}[xscale=\y, yscale=\y]
		\drawcurveleft
		\end{scope}
		}	
	\end{scope}
	}
	\draw[densely dotted, red] (-1*\U, -\U) -- (-0.5*\U, -\U);
	\draw[red] (-0.75*\U, -\U) -- (-0.5*\U, -\U);
	\draw[densely dotted, red] (1.5*\U, 3*\U) -- (2*\U, 3*\U);
	\draw[red] (1.5*\U, 3*\U) -- (1.75*\U, 3*\U);	
	
	\foreach \x in {0,1,2} {
	\draw[fill] (-0.5*\U + \x*\U, -\U + 2*\x*\U) node[anchor=south west] {\small $x_0$} circle (\radiussympoint);
	}
	\foreach \x in {0,1} {
	\draw[fill=white] (\x*\U, 2*\x*\U) node[anchor=south west] {\small $x_1$} circle (\radiussympoint);
	}
	
	\begin{scope}[xshift=3*\U]
	\draw[red, fill=black, fill opacity=0.2, rounded corners=0.5*\u] (-0.5*\U, -\U) -- (\u,-\U) -- (\u,0) -- (\U,0) -- (\U, \U-\u) -- (0,\U-\u) -- (0,0);
	\draw (-0.5*\U, -\U) -- (0,0);
		\draw[fill] (-0.5*\U, -\U) node[anchor=east] {\small $x_0$} circle (\radiussympoint);
		\draw[fill=white] (0,0) node[anchor=east] {\small $x_1$} circle (\radiussympoint);
		\draw[red] (0.5*\U, \U) node[anchor=south] {\small $\tilde \gamma_{01}$};
		\draw (-0.4*\U, -0.5*\U) node[anchor=south] {\small $\ell_{01}$};
	\end{scope}
	
	\foreach \x in {0, 1, ..., 4}
	\foreach \y in {-1, 0, ..., 2} {
		\draw (\x*\U-0.5*\U,\y*\U+0.5*\U) node {\tiny $\bpt$};
	}
\end{scope}

\begin{scope}[xshift=10cm]
\draw[step=\V, lightgray, densely dotted, thin, xshift=-0.5*\V, yshift=-0.5*\V] (-3.75*\V,-8.75*\V) grid (4.75*\V,9.75*\V);
\draw[step=\V, lightgray, thin, xshift=-0.5*\V, yshift=-0.5*\V] (-3.5*\V,-8.5*\V) grid (4.5*\V,9.5*\V);
	
      	\foreach \y in {-1, 1}{
	\begin{scope}[xscale=\y, yscale=\y]
	\drawcurveright
	\end{scope}
	}
	\draw[densely dotted, red] (-4*\V, -8*\V) -- (-3.5*\V, -8*\V);
	\draw[densely dotted, red] (4*\V, 8*\V) -- (3.5*\V, 8*\V);
	
	\foreach \x in {0,1} {
	\draw[fill] (-3.5*\V + 7*\x*\V, -8*\V + 16*\x*\V) node[anchor=south east] {\small $x_0$} circle (\radiussympoint);
	}
	\draw[fill=white] (0,0) node[anchor=east] {\small $x_1$} circle (\radiussympoint);
	
	\begin{scope}[xshift=3*\V]
		\begin{scope}
		\clip (0,0) -- (0,-\V) -- (-\V, -\V) -- (-\V, -4*\V) -- (-2*\V, -4*\V) -- (-2*\V, -5*\V) -- (-3*\V, -5*\V) -- (-3*\V, -8*\V) -- (-3.5*\V, -8*\V) -- (0,-8*\V);
		\draw[fill=violet, fill opacity=0.2, rounded corners=0.5*\v] (0,0) -- (-3.5*\V, -8*\V) -- (-3.5*\V,0);
		\end{scope}
		\begin{scope}
		\clip (0,0) -- (0,-\V) -- (-\V, -\V) -- (-\V, -4*\V) -- (-2*\V, -4*\V) -- (-2*\V, -5*\V) -- (-3*\V, -5*\V) -- (-3*\V, -8*\V) -- (-3.5*\V, -8*\V) -- (-3.5*\V,0);
		\draw[fill=black, fill opacity=0.2, rounded corners=0.5*\v] (0,0) -- (-3.5*\V, -8*\V) -- (0,-8*\V);
		\end{scope}
	\draw[dashed] (-3.5*\V, -8*\V) -- (0,-8*\V) -- (0,0);
	\drawcurveright
		\draw[fill] (-3.5*\V, -8*\V) node[anchor=south east] {\small $x_0$} circle (\radiussympoint);
		\draw[fill=white] (0,0) node[anchor=east] {\small $x_1$} circle (\radiussympoint);
		\draw[red] (-3*\V, -6*\V) node[anchor=east] {\small $\tilde \gamma_{01}$};
		\draw (-1.5*\V, -3.5*\V) node[anchor=south east] {\small $\ell_{01}$};
	\end{scope}
	
	\foreach \x in {-3, -2, ..., 4}
	\foreach \y in {-9, -8, ..., 8} {
		\draw (\x*\V-0.5*\V,\y*\V+0.5*\V) node {\tiny $\bpt$};
	}

\end{scope}

\end{tikzpicture}
\caption{The computations of $\Deltadb(\gamma)$ from Examples \ref{ex:delta-sym-example1} (left) and \ref{ex:delta-sym-example2} (right). In each case, a lift of $\gamma$ to the plane is shown with (lifts of) the two fixed points $x_0$ and $x_1$ of the elliptic involution symmetry. Separately, the segment $\tilde\gamma_{01}$ from a lift of $x_0$ to a lift of $x_1$ is shown and the region between this curve and the straight line $\ell_{01}$ connecting the endpoints is shaded. Twice the signed area of the shaded region gives $\Deltadb(\gamma)$.}
\label{fig:delta-sym-examples}
\end{figure}

%% file: sections/mu-bar.tex

\section{An invariant from relative Wu sets}\label{mubarsection}

In this section we study rooted plumbing trees $(\Gamma, v)$ whose associated 3-manifolds are $\mathbb{Z}_2$-homology solid tori and define an (integer-valued) invariant $\Delta \overline{\mu}$ of $(\Gamma, v)$. 

Given a Wu set $S$ on a plumbing tree $\Gamma$, Neumann \cite{Neumann-invt} and Siebenmann \cite{Siebenmann} independently defined an integer $\bar\mu$ that gives rise to an invariant $\bar\mu(Y, \mathfrak{s})$ of plumbed 3-manifolds $Y$ equipped with spin structures $\mathfrak{s}$. The invariant $\bar\mu(Y, \mathfrak{s})$ satisfies the following properties: When $Y$ is an integer homology sphere (with its unique spin structure), $\bar\mu(Y, \mathfrak{s})$ reduces (modulo 16) to the Rokhlin invariant $\mu(Y) \in \set{0,8} \subset 16\Z$. If $Y$ is the double branched cover $\Sigma_2(L)$ of a plumbing link $L$, then $\bar\mu(Y, \mathfrak{s})$ agrees with the signature of $L$ at the quasi-orientation of $L$ that corresponds to $\mathfrak{s}$ \cite[Theorem 5]{Saveliev}. In Section \ref{sec:Wu_old} we will review the construction of $\bar\mu$.

The rest of the section is outlined as follows: In Section \ref{relativeWusets} we generalize the notion of Wu sets to the setting of rooted trees. In Section \ref{deltamubar} we define our invariant $\Delta \overline{\mu}$ and show in Lemma \ref{lem:relative-to-absolute-mu-bar} that $\Delta \overline{\mu}$ can be thought of as a relative version of $\bar\mu$. In Section \ref{deltamubar} we study the behavior of $\Delta \overline{\mu}$ under the twist, extend, and merge operations and use this to relate $\Delta \overline{\mu}$ to the invariant $\Delta_{sym}$ from Section \ref{sec:delta-sym} in Proposition \ref{prop:main-thm-rooted-trees}.

\subsection{\texorpdfstring{Wu sets and the Neumann-Siebenmann invariant $\bar\mu$}{Wu sets and the Neumann-Siebenmann invariant mu-bar}}
\label{sec:Wu_old}

Given a plumbing tree $\Gamma$ with vertex set $V(\Gamma)$, one can define an intersection form $Q_{\Gamma}(\cdot, \cdot)$ on $\Z \langle V(\Gamma) \rangle$ by setting
\[
Q_{\Gamma}(v,w) =
\begin{cases}
n_{w} & \mbox{if $v = w$} \\
1 & \mbox{if $v$ and $w$ are adjacent} \\
0 & \mbox{otherwise}
\end{cases}
\]
on $V(\Gamma)$ and extending linearly. A \textit{Wu set} for $\Gamma$ is a subset $S \subset V(\Gamma)$ that satisfies
\[
Q_{\Gamma}(S, w) \equiv n_w \pmod 2  \qquad \forall w \in V(\Gamma).
\]
Here we're thinking of $S$ as an element of $\Z \langle V(\Gamma) \rangle$ by taking the sum of the vertices in $S$ with coefficient 1 for each term.

Given any tree $\Gamma$, the set of Wu sets on $\Gamma$ is in 1-1 correspondence with the set of spin structures on $Y_{\Gamma}$, the plumbing 3-manifold associated with $\Gamma$; for details, see \cite[Proposition 5.7.11]{Gompf-Stipsicz} (where Wu sets are described as characteristic sublinks). We note that Wu sets can be defined for more general plumbing graphs than weighted trees, but we restrict our attention to weighted trees, since this is the case we need.

A simple induction argument using the fact that each tree has a leaf shows that if $S$ is a Wu set for a weighted tree $\Gamma$, then no two adjacent vertices can belong to $S$ \cite{Stipsicz}.
Then to every subset $S \subset V(\Gamma)$ we can associate an \textit{embedded} surface $\Sigma_S$ in the plumbing 4-manifold $X_{\Gamma}$ by taking the union of the zero sections of the $D^2$-bundles over $S^2$ that are associated with the vertices in $S$.
Note that if $\Gamma$ is not a tree because it has cycles, then there may be two adjacent vertices in the same Wu set $S$, and therefore $\Sigma_S$ would be just immersed, not embedded.

The Neumann-Siebenmann invariant $\overline{\mu}$ of a weighted tree $\Gamma$ is then the function that assigns to each Wu set $S$ the following integer:
\[
\overline{\mu}(\Gamma, S) = \sigma(X_{\Gamma}) - [\Sigma_S]^2.
\]
Here $\sigma(X_{\Gamma})$ is the signature of $X_{\Gamma}$ and $[\Sigma_{S}]^2$ is the self-intersection of $\Sigma_{S}$. We remark that there are different conventions in the literature regarding the definition of $\overline{\mu}(\Gamma, S)$ -- for example the definition in \cite{Saveliev} differs from ours by a factor of $\frac18$.

If $S_0$ and $S_1$ are Wu sets on a plumbing tree $\Gamma$, then 
\begin{equation}
\label{eq:mu-mu}
\overline{\mu}(\Gamma, S_1) - \overline{\mu}(\Gamma, S_0) =  [\Sigma_{S_0}]^2 - [\Sigma_{S_1}]^2 = Q_{\Gamma}(S_0, S_0) - Q_{\Gamma}(S_1, S_1) = \sum\limits_{v \in S_0} n_v - \sum\limits_{v \in S_1} n_v,
\end{equation}
where $n_v$ is the weight on vertex $v$, and in the last equality we used the fact that a Wu set on a tree does not contain any two adjacent vertices.
This will be important for us in Section \ref{deltamubar}.

\subsection{Relative Wu sets for rooted plumbing trees}\label{relativeWusets}

In this subsection, we generalize the notion of a Wu set to the setting of rooted plumbing trees. We will take the same notation as in the previous subsection.

\begin{definition}
Let $(\Gamma, v)$ be a (connected) rooted plumbing tree.
A \emph{relative Wu set} for $(\Gamma, v)$ is a subset $S \subset V(\Gamma)$ such that
\[
Q_{\Gamma}(S, w) \equiv n_w \pmod 2  \qquad \forall w \in V(\Gamma) \setminus \set v.
\]
When there's no confusion, we will denote $Q_{\Gamma}(S, w)$ by $S \cdot w$. We say that $S$ is:
\begin{itemize}
\item \emph{type 0} if $v \notin S$, and \emph{type 1} if $v \in S$;
\item \emph{balanced} if $Q_\Gamma(S, v) \equiv n_v \pmod 2$, and \emph{unbalanced} otherwise.
\end{itemize}
\end{definition}

\noindent Balanced relative Wu sets for $(\Gamma, v)$ coincide with regular Wu sets for $\Gamma$.

We will denote the set of relative Wu sets for $(\Gamma, v)$ by $\Wu(\Gamma, v)$. Depending on whether a relative Wu set is type 0 or 1, balanced or unbalanced, we get a partition:
\[
\Wu(\Gamma, v) =
\Wu_b^0(\Gamma, v) \sqcup
\Wu_b^1(\Gamma, v) \sqcup
\Wu_u^0(\Gamma, v) \sqcup
\Wu_u^1(\Gamma, v).
\]
It is helpful to keep track of the counts of relative Wu sets of each type, so we will define 
$$\vec n(\Gamma, v) := \left(\# Wu_b^0(\Gamma, v), \#Wu_b^1(\Gamma, v), \#Wu_u^0(\Gamma, v), \#Wu_u^1(\Gamma, v)\right).$$
We will call $\vec n(\Gamma, v)$ the \emph{Wu type} of $(\Gamma, v)$.
\begin{example}
\label{ex:relative-Wu-sets-v_0}
Let $(\set{v_0}, v_0)$ be the graph consisting of a single 0-weighted vertex $v_0$, which is also the terminal vertex. Then there are exactly two relative Wu sets, both balanced, one of type 0 and one of type 1. In other words, $\vec n(\Gamma, v) = (1,1,0,0)$. More precisely,
\begin{align*}
\Wu_b^0(\Gamma, v) &= \set{\emptyset} &
\Wu_b^1(\Gamma, v) &= \set{\{v_0\}} \\
\Wu_u^0(\Gamma, v) &= \emptyset  &
\Wu_u^1(\Gamma, v) &= \emptyset.
\end{align*}
\end{example}

\subsection{Effect of plumbing tree moves on relative Wu sets}

In this subsection we explain how the twist, extend, and merge operations from Section \ref{sec:graph-mfds} performed on rooted plumbing trees $(\Gamma, v)$ affect the relative Wu sets for $(\Gamma, v)$. 


\subsubsection{Twist move}

Recall that a rooted plumbing tree $(\Gamma, v)$ is obtained from $(\Gamma', v)$ by a \emph{($\pm 1$)-twist} move $\twist^{\pm 1}$  if $V(\Gamma) = V(\Gamma')$, $E(\Gamma) = E(\Gamma')$, and the weight of the terminal vertex $v$ changes by $\pm 1$, while the other weights remain the same:
\[
n_w = 
\begin{cases}
n'_w & \mbox{if $w \neq v$} \\
n'_w \pm 1 & \mbox{if $w=v$}.
\end{cases}
\]

\begin{lemma}
\label{lem:twist}
Suppose $(\Gamma, v)$ is obtained from $(\Gamma', v)$ by a ($\pm 1$)-twist move. Then
\begin{align*}
\Wu_b^0(\Gamma, v) &= \Wu_u^0(\Gamma', v) &
\Wu_b^1(\Gamma, v) &= \Wu_b^1(\Gamma', v)  \\
\Wu_u^0(\Gamma, v) &= \Wu_b^0(\Gamma', v) &
\Wu_u^1(\Gamma, v) &= \Wu_u^1(\Gamma', v).
\end{align*}
In particular, the twist operation takes a plumbing tree with Wu type $(a, b, c, d)$ to a plumbing tree with Wu type $(c, b, a, d)$.
\end{lemma}

\begin{proof}
From the definition of a relative Wu set, it is immediate to check that a subset $S \subset V(\Gamma') = V(\Gamma)$ is a relative Wu set for $(\Gamma, v)$ if and only if it is a relative Wu set for $(\Gamma', v)$. It's clear that a relative Wu set for $(\Gamma, v)$ is of type 0 (respectively type 1) if and only if is of type 0 (respectively type 1) for $(\Gamma', v)$. Hence it's enough to show that the twisting operation reverses balanced and unbalanced for type 0 relative Wu sets, but preserves balanced and unbalanced for type 1 relative Wu sets.
This is true because $n_v \equiv n'_v + 1 \pmod 2$, and $Q_{\Gamma}(S,v) \equiv Q_{\Gamma'}(S,v) \pmod 2$ if and only if $S$ has type 0.
\end{proof}

\subsubsection{Extend move}

Recall that a rooted plumbing tree $(\Gamma, v)$ is obtained from $(\Gamma',v')$ by an \emph{extend} move $\extend$ if $V(\Gamma) = V(\Gamma') \cup \set{v}$, $E(\Gamma) = E(\Gamma') \cup \set{e_{(v,v')}}$, where $e_{(v,v')}$ is a single edge between $v$ and $v'$, and the weight $n_v$ of the new vertex $v$ is 0:
\[
n_w = 
\begin{cases}
n'_w & \mbox{if $w \neq v$} \\
0 & \mbox{if $w=v$}.
\end{cases}
\]

\begin{lemma}\label{lem:extend}
Suppose $(\Gamma, v)$ is obtained from $(\Gamma', v')$ by an extend move. Then
\begin{align*}
\Wu_b^0(\Gamma, v) &= \Wu_b^0(\Gamma', v') &
\Wu_b^1(\Gamma, v) &= \{ S \cup \{v \} \,\vert\, S \in \Wu_u^0(\Gamma', v') \} \\
\Wu_u^0(\Gamma, v) &= \Wu_b^1(\Gamma', v') &
\Wu_u^1(\Gamma, v) &= \{ S \cup \{v \} \,\vert\, S \in \Wu_u^1(\Gamma', v') \}.
\end{align*}
In particular, the extend operation takes a plumbing tree with Wu type $(a, b, c, d)$ to a plumbing tree with Wu type $(a, c, b, d)$.
\end{lemma}

\begin{proof}
Consider the map $f \colon \Wu(\Gamma', v') \to \Wu(\Gamma,v)$ given by
\begin{equation}
\label{eq:extend-Wu}
f(S) =
\begin{cases}
S & \mbox{if $Q_{\Gamma'}(S, v') \equiv n'_{v'} \pmod 2$}\\
S \cup \set{v} & \mbox{if $Q_{\Gamma'}(S, v') \not\equiv n'_{v'} \pmod 2$}
\end{cases}
\end{equation}
We first check that $f(S)$ is a relative Wu set for $(\Gamma,v)$, namely that $Q_{\Gamma}(f(S), w) \equiv n_w$ for each $w \in \Gamma \setminus \set{v}$. The only non-trivial vertex to check is $v'$. The choice of whether or not to include $v$ in $f(S)$ is done precisely to satisfy the condition that $Q_{\Gamma}(f(S), v') \equiv n_{v'}$ (note that it is immediate to check that exactly one of the choices will yield a relative Wu set). Hence $f$ is well-defined. The map $f$ has an inverse, given by restricting a Wu set from $\Gamma$ to $\Gamma'$, so $f$ is bijective. 

Finally observe that $f(S)$ has type 0 if and only if $S$ is balanced, and $f(S)$ is balanced if and only if $S$ has type 0.
The first claim follows immediately from Equation \eqref{eq:extend-Wu}, and the second claim follows from the fact that
\[
Q_{\Gamma}(f(S), v) - n_{v} = Q_{\Gamma}(f(S), v) - 0 = 
\begin{cases}
0 & \mbox{if $v' \notin S$}\\
1 & \mbox{if $v' \in S$}
\end{cases}
\]
since $v'$ is the only vertex of $\Gamma$ adjacent to $v$ and $n_{v} = 0$.
\end{proof}

\subsubsection{Merge move}
Recall that a rooted plumbing tree $(\Gamma, v)$ is obtained from rooted plumbing trees $(\Gamma', v')$ and $(\Gamma'', v'')$ by a \emph{merge} move $\merge$ if $V(\Gamma) = V(\Gamma') \cup V(\Gamma'') / (v' \sim v'')$, $E(\Gamma) = E(\Gamma') \cup E(\Gamma'')$, the terminal vertex $v$ is the merged vertex $v' = v''$, and the weight $n_v$ of $v$ is $n'_{v'} + n''_{v''}$:
\[
n_w = 
\begin{cases}
n'_w & \mbox{if $w\in V(\Gamma') \setminus \set{v'}$} \\
n''_w & \mbox{if $w\in V(\Gamma'') \setminus \set{v''}$} \\
n'_{v'} + n''_{v''} & \mbox{if $w=v$}.
\end{cases}
\]

\begin{lemma}\label{lem:merge}
Suppose $(\Gamma, v)$ is obtained from $(\Gamma', v')$ and $(\Gamma'', v'')$ by a merge move. Then the relative Wu sets for $(\Gamma, v)$ can be identified with pairs of relative Wu sets for $(\Gamma', v')$ and $(\Gamma'', v'')$:
\begin{align*}
\Wu_b^0(\Gamma, v) & \leftrightarrow (\Wu_b^0(\Gamma', v') \times \Wu_b^0(\Gamma'', v'')) \cup (\Wu_u^0(\Gamma', v') \times \Wu_u^0(\Gamma'', v'')) \\
\Wu_b^1(\Gamma, v) & \leftrightarrow (\Wu_b^1(\Gamma', v') \times \Wu_b^1(\Gamma'', v'')) \cup (\Wu_u^1(\Gamma', v') \times \Wu_u^1(\Gamma'', v'')) \\
\Wu_u^0(\Gamma, v) & \leftrightarrow (\Wu_b^0(\Gamma', v') \times \Wu_u^0(\Gamma'', v'')) \cup (\Wu_u^0(\Gamma', v') \times \Wu_b^0(\Gamma'', v'')) \\
\Wu_u^1(\Gamma, v) & \leftrightarrow (\Wu_b^1(\Gamma', v') \times \Wu_u^1(\Gamma'', v'')) \cup (\Wu_u^1(\Gamma', v') \times \Wu_b^1(\Gamma'', v'')).
\end{align*}
In particular, merge takes plumbing trees with Wu types $(a', b', c', d')$ and {$(a'', b'', c'', d'')$} to a plumbing tree with Wu type $(a'a'' + c'c'', b'b'' + d'd'', a'c'' + c' a'', b' d'' + d' b'')$.
\end{lemma}

\begin{proof}
We will show that the function
\begin{align*}
f \colon \Wu(\Gamma, v) &\to \Wu(\Gamma', v') \times \Wu(\Gamma'', v'')\\
S &\mapsto (S \cap V(\Gamma'), S \cap V(\Gamma''))
\end{align*} 
gives the above identifications. From the definition of a relative Wu set, it is immediate to check that if $S$ is a relative Wu set, then $S \cap V(\Gamma')$ and $S \cap V(\Gamma'')$ are also relative Wu sets, so the map $f$ is well-defined. Moreover, since $V(\Gamma) = V(\Gamma') \cup V(\Gamma'')$, the map $f$ is injective.

We now check that the image of $f$ consists of all pairs $(S', S'')$ that are either both of type $0$ or both of type $1$. One direction is clear --- if $S$ is type 0 (resp.\ type 1), then both $S \cap V(\Gamma')$ and $S \cap V(\Gamma'')$ are type 0 (resp.\ type 1). Conversely, every pair $(S', S'')$ of relative Wu sets can be glued to form a relative Wu set $S = S' \cup S''$, since for all $w \neq v$ the equation $Q_{\Gamma}(S, w) \equiv n_{w}$ follows from the corresponding equation for $\Gamma'$ or $\Gamma''$. Hence we have that $f(S) = (S', S'')$ if and only if $S'$ and $S''$ are both of type 0 or both of type 1.

Lastly, we look at how $f$ affects the property of being balanced or unbalanced. If we denote $f(S) = (S', S'')$, then it is immediate to check that
\[
Q_{\Gamma}(S,v) = Q_{\Gamma'}(S',v') + Q_{\Gamma''}(S'',v''),
\]
since $n_{v} = n'_{v'} + n''_{v''}$. It follows that $S$ is balanced (i.e.\ $Q_{\Gamma}(S,v) \equiv n_{v} \pmod 2$) if and only if $S'$ and $S''$ are both balanced or both unbalanced, which concludes the proof.
\end{proof}

Using the above lemmas, we get the following property of relative Wu sets.

\begin{lemma}
\label{lem:TwoNon-emptyWuSet}
Any rooted weighted tree $(\Gamma, v)$ has Wu type $(k,k,0,0)$, $(k,0,k,0)$, or $(0,k,k,0)$ where $k$ is a power of two.
\end{lemma}

\begin{proof}
We'll think of $(\Gamma, v)$ as being built inductively using a series of $(\pm 1)$-twist, extend, and merge moves. The base case is $(\{v_0\}, v_0)$, which has Wu type $(1,1,0,0)$ as shown in Example \ref{ex:relative-Wu-sets-v_0}. The twist and extend operations permute the first three coordinates of the Wu type, so they preserve the fact that the Wu type is one of the options listed. If the inputs to the merge operation have one of the Wu types listed, then by Lemma \ref{lem:merge} the Wu type of the output is determined by the Wu types of the inputs as indicated in the following table:
$$\begin{array}{c|ccc}
 & (k,k,0,0) & (k,0,k,0) & (0,k,k,0) \\
 \hline
 (\ell,\ell,0,0) & (k\ell, k\ell, 0, 0) & (k\ell, 0, k\ell, 0) & (0, k\ell, k\ell, 0)\\
 (\ell,0,\ell,0) & (k\ell, 0, k\ell, 0) & (2 k\ell, 0, 2 k\ell, 0) & (k\ell, 0, k\ell, 0)\\
 (0,\ell,\ell,0) & (0, k\ell, k\ell, 0) & (k\ell, 0, k\ell, 0) &  (k\ell, k\ell, 0, 0)
 \end{array}$$
In particular, the output of the merge operation has one of the desired Wu types.
\end{proof}


Since the Wu sets for $\Gamma$ correspond to the balanced Wu sets for $(\Gamma, v)$, if $(\Gamma, v)$ has Wu type $(k,k,0,0)$, $(k,0,k,0)$, or $(0,k,k,0)$ then $\Gamma$ has either $k$ or $2k$ Wu sets, and hence the plumbing 3-manifold $Y_\Gamma$ has either $k$ or $2k$ spin structures. 
Now think of $Y_\Gamma$ as the $\beta$-filling of the bordered 3-manifold $(M_{\Gamma, v}, \alpha, \beta)$. 
Because the twist and extend operations preserve the fact that a plumbing tree has Wu type $(k,k,0,0)$, $(k,0,k,0)$, or $(0,k,k,0)$, it follows that any Dehn filling of $(M_{\Gamma, v}, \alpha, \beta)$ has either $k$ or $2k$ spin structures. In particular, if $(M_{\Gamma, v}, \alpha, \beta)$ has some filling with a single spin structure (equivalently if $M_{\Gamma, v}$ is a $\Z_2$-homology solid torus), then $k$ must be 1, that is, $(\Gamma, v)$ has Wu type $(1,1,0,0)$, $(1,0,1,0)$, or $(0,1,1,0)$, which implies that $(\Gamma, v)$ has two relative Wu sets.

The Wu type determines which of the $\alpha$, $\beta$ fillings of a $\Z_2$-homology solid torus are $\Z_2$-homology spheres. Recall that the $\beta$-filling of $M_{\Gamma, v}$ is $Y_\Gamma$ and the $\alpha$-filling of $M_{\Gamma, v}$ is $Y_{\Gamma'}$ where $(\Gamma, v) = \extend( \Gamma', v')$. One can check that $M_{\Gamma, v}$ is a $\Z_2$-homology solid torus of type $\alpha$, $\beta$, or $\alpha\beta$ if and only if $(\Gamma, v)$ has Wu type $(1,1,0,0)$, $(1,0,1,0)$, or $(0,1,1,0)$, respectively.

\subsection{\texorpdfstring{$\Delta \overline{\mu}$ and plumbing tree moves}{Delta mu-bar and plumbing tree moves}}\label{deltamubar}
We now define an invariant $\Delta \bar\mu$ of rooted plumbing trees $(\Gamma, v)$ whose associated 3-manifolds are $\Z_2$-homology solid tori. Given any relative Wu set $S$ of $(\Gamma, v)$, we have the integer $\sum_{w \in S} n_w$. The invariant $\Delta \bar\mu$ is the difference in $\sum_{w \in S} n_w$ between the two relative Wu sets of $(\Gamma, v)$. To make this well-defined, we just need to fix a convention for the ordering of the two relative Wu sets.

\begin{definition}
Let $(\Gamma, v)$ be a rooted plumbing tree such that $M_{\Gamma, v}$ is a $\Z_2$-homology solid torus. If $(\Gamma, v)$ has Wu type $(1,1,0,0)$ or $(0,1,1,0)$, let $S_1$ denote the relative Wu set in $\Wu_b^1(\Gamma, v)$, and if $(\Gamma, v)$ has Wu type $(1,0,1,0)$, let $S_1$ denote the relative Wu set in $\Wu_b^0(\Gamma, v)$; in any case let $S_0$ denote the other relative Wu set of $(\Gamma, v )$. We define $\Delta\bar\mu(\Gamma, v)$ to be $\sum_{w \in S_1} n_w - \sum_{w \in S_0} n_w$.

\end{definition}

The motivation for this definition is the following:
\begin{lemma}\label{lem:relative-to-absolute-mu-bar}
If $(\Gamma, v)$ has Wu type $(1,1,0,0)$, so that $Y_\Gamma$ has two spin structures, then $\Delta\bar\mu(\Gamma, v)$ gives the difference between the $\bar\mu$-invariants for the two spin structures of $Y_\Gamma$.
\end{lemma}
\begin{proof}
This follows immediately from Equation \eqref{eq:mu-mu}.
\end{proof}

For the remainder of this section we will explore how $\Delta\bar\mu(\Gamma, v)$ changes under the elementary operations $\extend$, $\twist^{\pm 1}$, and $\merge$ on rooted plumbing trees.

\begin{lemma}\label{lem:delta-mu-extend}
Let $(\Gamma, v)$ be a plumbing tree for which $M_{\Gamma, v}$ is a $\Z_2$-homology solid torus, and suppose $(\Gamma, v) = \extend(\Gamma', v')$. Then $\Delta\bar\mu( \Gamma, v)= -\Delta\bar\mu(\Gamma', v')$.
\end{lemma}
\begin{proof}
First note that $M_{\Gamma', v'}$, which is only a reparametrization of $M_{\Gamma, v}$, is still a $\Z_2$-homology solid torus, so $\Delta\bar\mu(\Gamma', v')$ is defined. By Lemma \ref{lem:extend}, the relative Wu sets for $(\Gamma, v)$ restricted to $V(\Gamma)\setminus\{v\}$ are the same subsets of vertices as the relative Wu sets for $(\Gamma', v')$. Since the weights on the vertices in $V(\Gamma)\setminus\{v\}$ are the same as the weights on the corresponding vertices in $\Gamma'$, and the weight on $v$ is $0$, each relative Wu set $S_i$ for $(\Gamma, v)$ corresponds to a relative Wu set $S'_j$ for $(\Gamma', v')$ with $\Sigma_{w\in S_i}  n_w = \Sigma_{w' \in S'_j} n'_{w'}$. It only remains to check the indexing of the relative Wu sets before and after the extend move. It is straightforward to check that $S_0$ for $(\Gamma, v)$ corresponds to $S'_1$ for $(\Gamma', v')$, and so
$$\Delta\bar\mu(\Gamma, v) = \sum_{w \in S_1} n_w - \sum_{w\in S_0} n_w =\sum_{w' \in S'_0} n'_{w'} -  \sum_{w' \in S'_1} n'_{w'} = -\Delta\bar\mu(\Gamma', v').$$
\end{proof}

\begin{lemma}\label{lem:delta-mu-twist}
Let $(\Gamma, v)$ be a plumbing tree for which $M_{\Gamma, v}$ is a $\Z_2$-homology solid torus, and suppose $(\Gamma, v) = \twist^{\pm 1}(\Gamma', v)$.
If $M_{\Gamma, v}$ is a $\Z_2$-homology solid torus of type $\beta$, then $\Delta\bar\mu( \Gamma, v)= -\Delta\bar\mu(\Gamma', v)$. If $M_{\Gamma, v}$ is a $\Z_2$-homology solid torus of type $\alpha$ or type $\alpha\beta$, then $\Delta\bar\mu( \Gamma, v)= \Delta\bar\mu(\Gamma', v) \pm 1$.
\end{lemma}
\begin{proof}
It is again clear that $M_{\Gamma', v}$ is still a $\Z_2$-homology solid torus, so $\Delta\bar\mu(\Gamma', v)$ is defined. If $M_{\Gamma, v}$ is of type $\beta$, then $(\Gamma, v)$ has Wu type $(1,0,1,0)$, and the same is true for $(\Gamma', v)$ since $\twist^{\pm 1}$ interchanges the relative Wu sets in $\Wu_b^0$ and $\Wu_u^0$. We get that $S_0 = S'_1$ and $S_1 = S'_0$ as subsets of $V(\Gamma)$. Moreover, no weights have changed for the vertices in any of these Wu sets because none of these Wu sets contain $v$. It follows that
$$\Delta\bar\mu(\Gamma, v) = \sum_{w\in S_1} n_w - \sum_{w\in S_0} n_w =  \sum_{w'\in S'_0} n'_{w'} - \sum_{w'\in S'_1} n'_{w'} = -\Delta\bar\mu(\Gamma', v').$$

If $M_{\Gamma, v}$ is of type $\alpha$ or $\alpha \beta$, then $(\Gamma, v)$ has Wu type $(1,1,0,0)$ or $(0,1,1,0)$ respectively, and the same is true for $(\Gamma', v)$. In both cases $S_1 \in \Wu_b^1(\Gamma, v)$ and $S'_1 \in \Wu_b^1(\Gamma', v)$ correspond to the same subset of vertices that contains $v$, while $S_0$ and $S'_0$ correspond to the same subset of vertices that does not contain $v$. Since $\twist^{\pm 1}$ changes the weight on $v$ by $\pm 1$, it follows that 
$$\Delta\bar\mu(\Gamma, v) = \sum_{w\in S_1} n_w - \sum_{w \in S_0} n_w = \left(\left(\sum_{w' \in S'_1} n'_{w'}\right) \pm 1\right) - \sum_{w' \in S'_0} n'_{w'} = \Delta\bar\mu(\Gamma', v) \pm 1.$$
\end{proof}

\begin{lemma}\label{lem:delta-mu-merge}
Let $(\Gamma, v)$ be a plumbing tree for which $M_{\Gamma, v}$ is a $\Z_2$-homology solid torus, and suppose $(\Gamma, v) = \merge( (\Gamma', v'), (\Gamma'', v'') )$. Then $M_{\Gamma', v'}$ and  $M_{\Gamma'', v''}$ are both $\Z_2$-homology solid tori, and they are not both of type $\beta$---that is, $(\Gamma', v')$ and $(\Gamma'', v'')$ are not both of Wu type $\vec n = (1,0,1,0)$.
Moreover:

\begin{itemize}
\item if $\vec n(\Gamma', v') = (1,0,1,0)$, then
\[
\Delta\bar\mu(\Gamma, v) = 
\begin{cases}
\Delta\bar\mu(\Gamma', v'), & \text{if }  \vec n(\Gamma'', v'') = (1,1,0,0) \\
- \Delta\bar\mu(\Gamma', v'), & \text{if } \vec n(\Gamma'', v'') = (0,1,1,0)
\end{cases}
\]
\item if $\vec n(\Gamma'', v'') = (1,0,1,0)$, then
\[
\Delta\bar\mu(\Gamma, v) = 
\begin{cases}
\Delta\bar\mu(\Gamma'', v''), & \text{if }  \vec n(\Gamma', v') = (1,1,0,0) \\
- \Delta\bar\mu(\Gamma'', v''), & \text{if } \vec n(\Gamma', v') = (0,1,1,0)
\end{cases}
\]
\item if neither $(\Gamma', v')$ nor $(\Gamma'', v'')$ has $\vec n= (1,0,1,0)$, then $\Delta\bar\mu(\Gamma, v) = \Delta\bar\mu(\Gamma', v') + \Delta\bar\mu(\Gamma'', v'')$.
\end{itemize}

\end{lemma}

\begin{proof}
The first claim that both inputs to the merge operation determine $\Z_2$-homology solid tori and they are not both of type $\beta$ (equivalently, the plumbing trees do not both have Wu type $(1,0,1,0)$) follows easily from the table in the proof of Lemma \ref{lem:TwoNon-emptyWuSet}. Next let $S_i$, $S'_i$, and $S''_i$ be the relative Wu sets for $(\Gamma, v)$, $(\Gamma', v')$, and $(\Gamma'', v'')$, respectively, for $i \in \{0,1\}$.

If both inputs have Wu type $(1,1,0,0)$ or both have Wu type $(0,1,1,0)$, then by Lemma \ref{lem:merge} we see that $(\Gamma, v)$ has Wu type $(1,1,0,0)$ and that $S_0 = S'_0 \cup S''_0$ and $S_1 = S'_1 \cup S''_1$. It follows that
\begin{align*}
\Delta\bar\mu(\Gamma, v) =\sum_{w\in S_1} n_w -  \sum_{w\in S_0}n_w &= \left(\sum_{w'\in S'_1} n'_{w'} + \sum_{w''\in S''_1} n''_{w''}\right)  - \left(\sum_{w' \in S'_0} n'_{w'} + \sum_{w''\in S''_0} n''_{w''}\right) \\
&= \Delta\bar\mu(\Gamma', v') + \Delta\bar\mu(\Gamma'', v'').
\end{align*}

If one input to the merge operation has Wu type $(1,1,0,0)$ and the other has Wu type $(0,1,1,0)$, then by Lemma \ref{lem:merge} we see that $(\Gamma, v)$ has Wu type $(0,1,1,0)$ and once again that $S_0 = S'_0 \cup S''_0$ and $S_1 = S'_1 \cup S''_1$. As in the previous case, it follows that $\Delta\bar\mu(\Gamma, v) = \Delta\bar\mu(\Gamma', v') + \Delta\bar\mu(\Gamma', v')$.

If $(\Gamma', v')$ has Wu type $(1,0,1,0)$ and $(\Gamma'', v'')$ has Wu type $(1,1,0,0)$, then by Lemma \ref{lem:merge} we see that $(\Gamma, v)$ has Wu type $(1,0,1,0)$. In this case, $S_0 = S'_0 \cup S''_0$ and $S_1 = S'_1 \cup S''_0$,
\begin{align*}
\Delta\bar\mu(\Gamma, v) = \sum_{w\in S_1} n_w - \sum_{w\in S_0} n_w &= \left(\sum_{w'\in S'_1} n'_{w'} + \sum_{w''\in S''_0} n''_{w''}\right) - \left(\sum_{w'\in S'_0} n'_{w'} + \sum_{w''\in S''_0}n''_{w''}\right) \\
&= \sum_{w'\in S'_1} n'_{w'} - \sum_{w'\in S'_0} n'_{w'} = \Delta\bar\mu(\Gamma', v').
\end{align*}
If $(\Gamma', v')$ has Wu type $(1,0,1,0)$ and $(\Gamma'', v'')$ has Wu type $(0,1,1,0)$, the reasoning is the same as above, except that in this case $S_1 = S'_0 \cup S''_0$ and $S_0 = S'_1 \cup S''_0$, so
\[
\Delta\bar\mu(\Gamma, v) = - \Delta\bar\mu(\Gamma', v').
\]
Finally, a similar argument shows that if $(\Gamma'', v'')$ has Wu type $(1,0,1,0)$, then $\Delta\bar\mu(\Gamma, v) = \pm \Delta\bar\mu(\Gamma'', v'')$, depending on whether $(\Gamma', v')$ has Wu type $(1,1,0,0)$ or $(0,1,1,0)$. \end{proof} 

{\color{brown}
} 

%% file: sections/main_proof.tex

\section{Proof of Theorem \ref{thm:main}}\label{sec:main-proof}

We can now prove Theorem \ref{thm:main}. We first show that $\Deltadb$ agrees with a constant multiple of $\Delta\bar\mu$ for rooted plumbing trees before relating these two invariants to $\Delta d$ and $\Delta \sigma$ respectively. We restrict to rooted plumbing trees $(\Gamma,v)$ for which some filling of the 3-manifold $M_{\Gamma, v}$ is an L-space, and we will place a restriction on which filling is an L-space.
Recall that the purpose of this restriction is that it ensures we may use the simplified merge operation described in Section \ref{sec:graph-mfds}.
To state this restriction, note that given a rooted tree $(\Gamma, v)$, we can define a new rooted tree $(\Gamma_* , v_*)$ by removing $v$ if it is a leaf, in which case the vertex adjacent to $v$ becomes the new root, repeating this process until the resulting root $v_*$ is not a leaf of $\Gamma_*$, and changing the weight of $v_*$ to 0. Note that when $(\Gamma, v)$ is constructed inductively by applying twist, extend and merge operations, $(\Gamma_*, v_*)$ is the rooted tree obtained immediately after the last merge operation, or in the case that $(\Gamma, v)$ is a linear tree and no merge operations are needed in its construction $(\Gamma_*, v_*)$ is the tree with a single 0-weighted vertex.

Since $M_{\Gamma, v}$ and $M_{\Gamma_*, v_*}$ are the same 3-manifold with different boundary parameterizations, any slope $r$ for $M_{\Gamma, v}$ determines a slope $r_*$ for $M_{\Gamma_*, v_*}$. We will assume that $M_{\Gamma, v}$ has an L-space filling along some slope $r$ for which $r_*$ is not in $\Z \cup \{\infty\}$. This condition is automatically satisfied if $M_{\Gamma, v}$ is Floer simple, since then there is an interval of L-space fillings. It also holds if $Y_\Gamma$ is an L-space, $v$ is a leaf, and $\Gamma$ is reduced, in which case we take $r_*$ to be the slope of $M_{\Gamma_*, v_*}$ corresponding to the filling $Y_\Gamma$.  To see that $r_*$ is not in $\Z\cup\{\infty\}$, let $a_n, a_{n-1}, \ldots, a_1$ be the weights on the vertices removed when constructing $\Gamma_*$ from $\Gamma$ (note that at least one vertex is removed since $v$ is a leaf) and let $a_0$ be the weight on $v_*$ in $\Gamma$, so that $\Gamma$ is obtained from $\Gamma_*$ by applying the operations
$$\twist^{a_n} \circ \extend \circ \twist^{a_{n-1}} \circ \extend \circ \cdots \circ \twist^{a_1}\circ \extend \circ \twist^{a_0}.$$
The filling slope that gives $Y_\Gamma$ with respect to the parametrization for $\partial M_{\Gamma, v}$ is 0, and for any rooted plumbing tree the boundary reparametrization given by the operation $\twist^{a_i} \circ \extend$ takes the slope  $-\frac{1}{a_i + r}$ to $r$, 
so it can be checked that $r_*$ is given by the continued fraction
$$a_0 -\frac{1}{a_1 - \frac{1}{a_2 - \frac{1}{ \ddots - \frac{1}{ a_n}}}}.$$
Since $\Gamma$ is reduced the $a_i$ are not $0$ or $\pm 1$ for $i>0$, and it follows that $r_*$ is not in $\Z\cup\{\infty\}$.

\begin{proposition}\label{prop:main-thm-rooted-trees}
If $(\Gamma, v)$ is a rooted plumbing for which $M_{\Gamma, v}$ is a $\Z_2$-homology solid torus that has some L-space filling along a slope $r$ for which $r_*$ is not in $\Z\cup\{\infty\}$, then $\Deltadb(\Gamma, v) = -\tfrac 1 4 \Delta\bar\mu(\Gamma, v)$.
\end{proposition}
\begin{proof}
We proceed by induction on the number of vertices in $\Gamma$ and, for trees with the same number of vertices, on the absolute value of the weight of $v$. For the base case we consider the plumbing tree $\Gamma$ with a single $0$-weighted vertex $v$; in this case it is straightforward to check that $\Deltadb(\Gamma, v) = \Delta\bar\mu(\Gamma, v) = 0$. We now proceed to the inductive step, in which we consider three cases:

\emph{Case 1: The weight on $v$ is nonzero}. Let $\Gamma'$ be the weighted tree obtained from $\Gamma$ by decreasing the weight on $v$ by one if the weight is positive or increasing it by one if the weight is negative, so that $\Gamma'$ has a terminal weight that is lower than $\Gamma$'s in absolute value and $(\Gamma, v) = \twist^{\pm 1}(\Gamma', v)$. By the inductive hypothesis, the result holds for $(\Gamma', v)$. By Proposition \ref{prop:delta-d-twist} and Lemma \ref{lem:delta-mu-twist}, we have that
$$\Deltadb(\Gamma, v) = -\Deltadb(\Gamma', v) = \tfrac 1 4 \Delta\bar\mu(\Gamma', v) = -\tfrac 1 4 \Delta\bar\mu(\Gamma, v)$$
if $M_\Gamma$ is of type $\beta$, and 
$$\Deltadb(\Gamma, v) = \Deltadb(\Gamma', v) \mp \tfrac 1 4 = -\tfrac 1 4 \left( \Delta\bar\mu(\Gamma', v) \pm 1 \right) = -\tfrac 1 4 \Delta\bar\mu(\Gamma, v)$$
if $M_\Gamma$ is of type $\alpha$ or $\alpha\beta$.

\emph{Case 2: The weight on $v$ is zero and $v$ is a leaf.} Let $(\Gamma', v')$ be the rooted tree obtained from $(\Gamma, v)$ by removing the vertex $v$, letting $v'$ be the vertex that was adjacent to $v$ in $\Gamma$, so that  $(\Gamma, v) = \extend (\Gamma', v')$. By the inductive hypothesis, the result holds for $(\Gamma', v')$. The result for $(\Gamma, v)$ then follows from Proposition \ref{prop:delta-d-extend} and Lemma \ref{lem:delta-mu-extend}.

\emph{Case 3: The weight on $v$ is zero and $v$ has valence at least two.} We can choose two rooted weighted trees $(\Gamma', v')$ and $(\Gamma'', v'')$, each with strictly fewer vertices than $\Gamma$, so that $(\Gamma, v) = \merge\left( (\Gamma', v') , (\Gamma'', v'') \right)$. By Proposition \ref{prop:delta-d-merge}, both $M_{\Gamma', v'}$ and $M_{\Gamma'', v''}$ are Floer simple $\Z_2$-homology solid tori,
so the result holds for $(\Gamma', v')$ and $(\Gamma'', v'')$ by the inductive hypothesis. The result for $(\Gamma, v)$ then follows from Proposition \ref{prop:delta-d-merge} and Lemma \ref{lem:delta-mu-merge}.
\end{proof}

We next observe that for any graph manifold $Y$ with two spin structures arising from a plumbing tree $\Gamma$, we can choose a leaf $v$ of $\Gamma$ so that the rooted tree $(\Gamma, v)$ defines a $\Z_2$-homology solid torus of type $\alpha$.

\begin{lemma}\label{lem:leaf-in-one-Wu-set}
If $\Gamma$ is a plumbing tree such that the associated 3-manifold $Y_\Gamma$ has two spin structures, then there is a leaf $v$ such that $M_{\Gamma, v}$ is a $\Z_2$-homology solid torus of type $\alpha$.
\end{lemma}
\begin{proof}
A Wu set on a tree $\Gamma$ is determined by its restriction to the leaves of $\Gamma$ (see \cite{Stipsicz}), so if there are two distinct Wu sets we can chose some leaf $v$ that is in exactly one of the two Wu sets. The Wu sets of $\Gamma$ are precisely the balanced Wu sets of $(\Gamma, v)$, and since there is one that contains $v$ and one that does not contain $v$ Lemma \ref{lem:TwoNon-emptyWuSet}
implies that $(\Gamma, v)$ has Wu type $(1,1,0,0)$. It follows that $M_{\Gamma, v}$ is a $\Z_2$-homology solid torus of type $\alpha$.
\end{proof}

We are now ready to prove the main theorem.

\begin{proof}[Proof of Theorem \ref{thm:main}]
Let $L$ be a 2-component arborescent link whose branched double cover $\Sigma_2(L)$ is an L-space. Since $L$ is arborescent, it can be represented by a (connected) plumbing tree $\Gamma$ and the graph manifold represented by this plumbing tree is $\Sigma_2(L)$. We note that $\Gamma$ need not be reduced; that is, it may contain valence one or two vertices with weight in $\{-1, 0, 1\}$. We can remove these to obtain a reduced plumbing tree $\Gamma'$ using the relevant moves from Neumann's calculus for plumbed manifolds (specifically moves R1, R3 and R6) \cite{Neumann}, which do not change the corresponding graph manifold but may change the link associated with the plumbing tree. It is straightforward to check that blowing down a $\pm 1$-weighted valence one or two vertex or collapsing a 0-weighted valence two vertex does not affect the corresponding link. Removing a 0-weighted leaf may result in a disconnected graph; in this case the link associated with the plumbing tree before removing the 0-weighted leaf is a connected sum of the links associated with the components of the resulting graph. Thus if $\Gamma'$ is disconnected, the link $L$ is a connected sum of the links corresponding to the components of $\Gamma'$, and $\Sigma_2(L)$ is the connected sum of the 3-manifolds associated with the components of $\Gamma'$. Exactly one of the components of $\Gamma'$, call it $\Gamma_1$, must give rise to a 2-component link and a 3-manifold with two spin structures, while all other components define knots and $\Z_2$-homology spheres. Since both signatures and $d$-invariants are additive under connected sums, it is clear that the result holds for $L$ if it holds for the link $L_1$ corresponding to $\Gamma_1$. 

It follows that it is sufficient to prove the result in the case that the plumbing tree $\Gamma$ representing $L$ is both connected and reduced, as we will assume for the remainder of the proof.

By Lemma \ref{lem:leaf-in-one-Wu-set}, we can choose a leaf $v$ of $\Gamma$ so that $M_{\Gamma, v}$ is a $\Z_2$-homology solid torus of type $\alpha$. Let $(\Gamma_*, v_*)$ be obtained from $(\Gamma, v)$ by removing $v$ if it is a leaf and repeating until the root is not a leaf. Since we assume $\Gamma$ is reduced, $Y_\Gamma$ is the filling of $M_{\Gamma_*, v_*}$ along some slope $r_*$ not in $\Z \cup \{\infty\}$.
Since $Y_\Gamma$ is an L-space, Proposition \ref{prop:main-thm-rooted-trees} applies and says that $\Deltadb(\Gamma, v) = -\tfrac 1 4 \Delta\bar\mu(\Gamma, v)$.

From the proof of Proposition \ref{prop:main-thm-rooted-trees}, note that $(\Gamma_*, v_*)$ is the merge of two trees for which the corresponding manifolds are Floer simple and for which at least one has $\infty$ as a strict L-space slope. Since the manifolds being merged are Floer simple, their immersed curves are embedded when lifted to $\R^2 \setminus \Z^2$. Since one curve has $\infty$ as a strict L-space the geometric interpretation of the merge operation is valid and we see that the immersed curves for $(\Gamma_*, v_*)$ (and therefore also those for $(\Gamma, v)$) are also embedded in the covering space $\R^2 \setminus \Z^2$. Lemma \ref{lem:Delta-sym-is-Delta-d} now implies that $d(Y_\Gamma; \s_1) - d(Y_\Gamma; \s_0) =  \Deltadb(\Gamma, v)$ up to sign. 

Because $M_{\Gamma, v}$ is a $\Z_2$-homology solid torus of type $\alpha$, it has Wu type $(1,1,0,0)$. By Lemma \ref{lem:relative-to-absolute-mu-bar}, $\Delta\bar\mu(\Gamma, v)$ agrees with the difference between the $\bar\mu$-invariants for the two spin structures on $Y_\Gamma$. These in turn agree with the signatures of $L$ with the two quasi-orientations by \cite[Theorem 5]{Saveliev}.
 And so the difference between the two $d$-invariants is $-\frac 1 4$ times the difference between the signatures. But the sum of the two $d$-invariants is $-\frac 1 4$ times the sum of the two signatures by \cite[Theorem A]{LRS}, so the result follows.
\end{proof}

%% file: sections/involutions.tex
\section{Links with involutions}
\label{sec:involutions}
In this section we prove Theorem \ref{thm:d-invariantinvol} and use it to give a new example of a 2-component link $L$ for which $d(\Sigma_2(L), \mathfrak{s}_i) = - \frac{1}{4} \sigma(L, \omega_{\mathfrak{s}_i})$ for $i \in \{1, 2\}$.

\subsection{Quasi-orientations and spin structures}

\input{sections/quasi-orientations.tex}

\subsection{Proof of Theorem \ref{thm:d-invariantinvol}}

With these observations in place we are ready to prove Theorem \ref{thm:d-invariantinvol}.

\begin{proof}[Proof of Theorem \ref{thm:d-invariantinvol}]
By Equation \ref{eq:LRS},
\begin{equation*}
d(\Sigma_2(L), \s_1) + d(\Sigma_2(L), \s_2)= - \frac{1}{4} (\sigma(L, \omega_1) +  \sigma(L, \omega_2)).
\end{equation*}
Properties (1) and (2) of the involution $\iota$ imply that $\iota$ reverses the orientation of the component of $L$ that contains the two fixed points, but preserves the orientation of the other component. Thus, $\iota$ swaps the two quasi-orientations of $L$, and
\[
\sigma(L, \omega_1) =  \sigma(L, \omega_2)
\]
since $(L, \omega_1)$ and $(L, \omega_2)$ are ambiently diffeomorphic.

From Proposition \ref{prop:naturality_of_s_omega}, we know that $\iota^*(\s_1) = \s_2$, so by invariance of $d$-invariants under diffeomorphism
\[
d(\Sigma_2(L), \s_1) = d(\Sigma_2(L), \s_2).
\]

Hence, $d(\Sigma_2(L), \mathfrak{s}_i) = - \frac{1}{4} \sigma(L, \omega_{\mathfrak{s}_i})$ for $i \in \{1, 2\}$. 
\end{proof}

\subsection{An example}
In this section we use Theorem \ref{thm:d-invariantinvol} to give an example of a 2-component link $L$ for which $d(\Sigma_2(L), \mathfrak{s}_i) = - \frac{1}{4} \sigma(L, \omega_{\mathfrak{s}_i})$ for $i \in \{1, 2\}$. The link $L$ is shown in Figure  \ref{fig:linkwithinvo}, as well as the involution $\iota$ that satisfies the conditions of Theorem \ref{thm:d-invariantinvol}. We remark that, since this example happens to be arborescent, the equality $d(\Sigma_2(L), \mathfrak{s}_i) = - \frac{1}{4} \sigma(L, \omega_{\mathfrak{s}_i})$ also follows from Theorem \ref{thm:main} along with the observation that $\Delta\bar\mu = 0$. This example is new in the sense that all previous examples (known to the authors) are quasi-alternating  \cite{LiscaOwensQuasi} or can be represented by plumbing trees with definite intersection forms \cite{Saveliev, Stipsicz}, but $L$ has neither property. We expect that we can produce infinitely many new examples by varying some of the coefficients in Figure \ref{fig:linkwithinvo}. We note that finding examples that satisfy properties (1) and (2) below is easy; the challenge lies in verifying properties (3) and (4).

\begin{figure}
\centering
  \includegraphics[scale=.25]{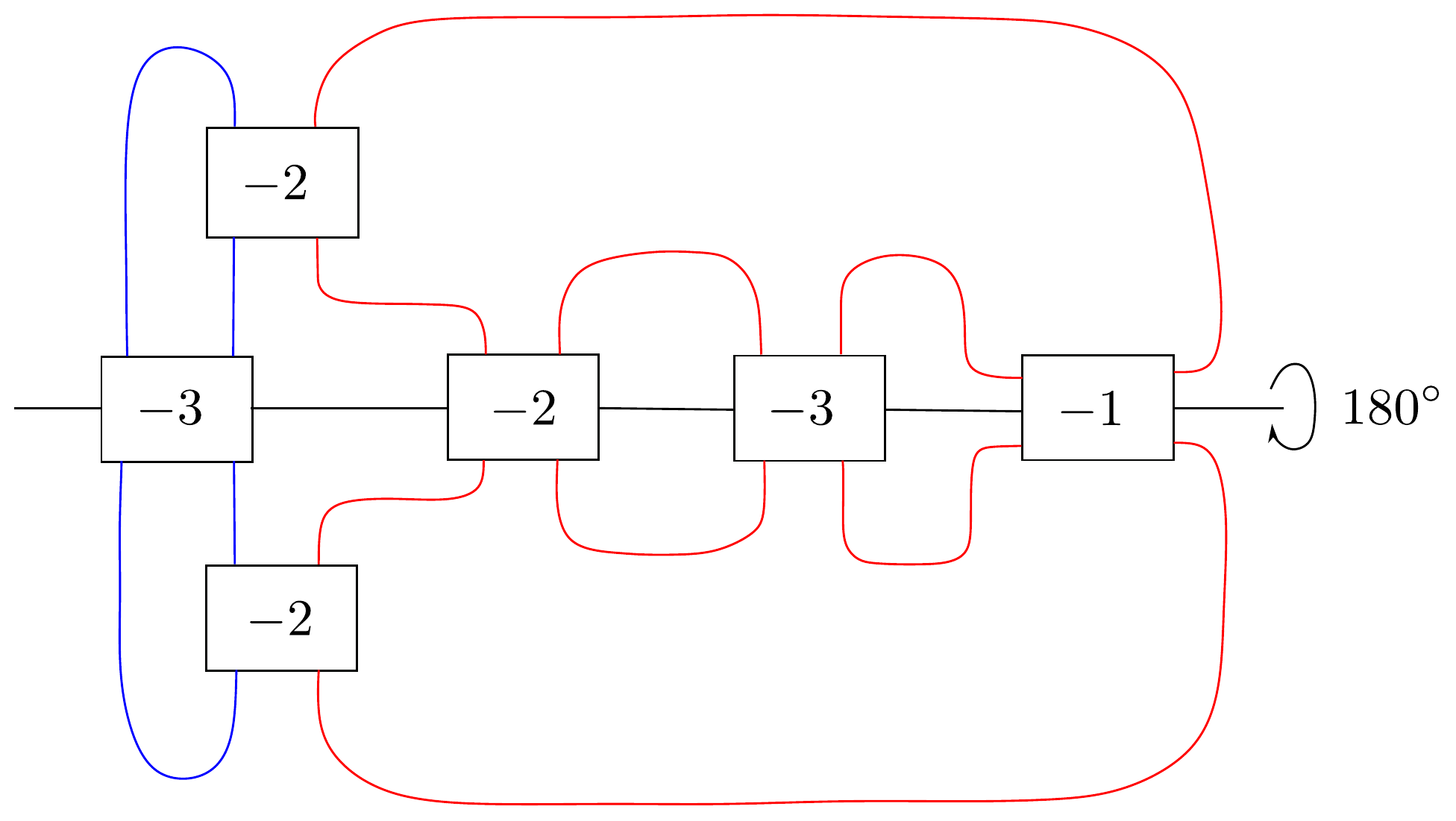}%
  \caption{A two-component link $L \subset S^3$ with an involution $\iota: S^3 \rightarrow S^3$ that is induced by $180^{\circ}$ rotation about the horizontal axis. The integers in the boxes represent the number of positive half twists. The fixed point set of $\iota$ is the circle that comes from identifying the ends of the axis; it intersects $L$ in two points on the red component.}
 \label{fig:linkwithinvo}
\end{figure}

\begin{figure}
\centering
  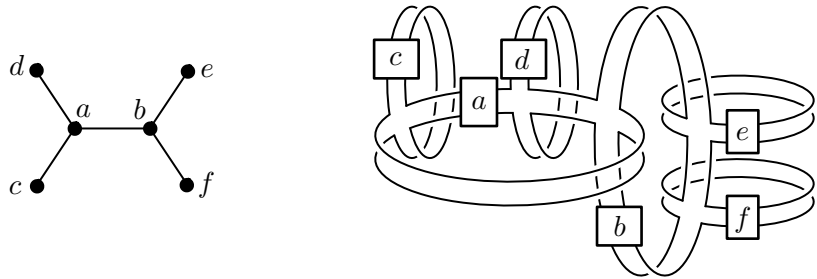%
  \caption{A weighted tree $\Gamma(a,b,c,d,e,f)$ and the corresponding plumbing type link $L(a,b,c,d,e,f)$. The integers in the boxes indicate the number of positive half twists.}
 \label{fig:tree+link}
\end{figure}

\begin{proposition}
The 2-component link $L$ in Figure \ref{fig:linkwithinvo} has the following properties:
\begin{enumerate}
\item\label{it:L_admits_involution} $L$ admits an involution $\iota$ that satisfies the conditions of Theorem  \ref{thm:d-invariantinvol};
\item\label{it:BDC_L-space} $\Sigma_2(L)$ is an $L$-space;
\item\label{it:Kh_thick} $L$ is $\Kh$-thick, hence it is not quasi-alternating;
\item\label{it:no_definite_plumbing} $\Sigma_2(L)$ is not the boundary of a definite plumbing.
\end{enumerate}
\end{proposition}

\begin{proof}

It's clear from Figure \ref{fig:linkwithinvo} that $L$ admits an involution $\iota$ of $S^3$ that satisfies the conditions of Theorem  \ref{thm:d-invariantinvol}.

For Property \eqref{it:BDC_L-space}, we note that  $L$ admits an alternate description as the plumbing (or arborescent) link $L(-3,-1,-2,-2,-2,-3)$ in Figure \ref{fig:tree+link}. To see this we just cancel the unnecessary crossings in $L(-3,-1,-2,-2,-2,-3)$; this process is explained in greater detail in \cite[Figure 1.8]{G:arborescent}. Then the double branched cover $\Sigma_2(L)$ of $L$ can be thought of as the graph manifold associated with the plumbing tree $\Gamma(-3,-1,-2,-2,-2,-3)$. The fact that $\Sigma_2(L)$ is an L-space follows from Example \ref{ex:plumbing-tree-example2} in Section \ref{sec:example-computations}.


Regarding Property \eqref{it:Kh_thick}, from a PD code for $L$ obtained using SnapPy \cite{SnapPy}, we computed the Khovanov homology of $L$ (over $\Z$) using KnotJob \cite{KnotJob}. The results of the computation are shown in Table \ref{tab:Kh-thick}. The cell colored in pink shows that $L$ is $\Kh$-thick.

Lastly, we show Property \eqref{it:no_definite_plumbing}. Let $\Gamma$ denote the weighted tree $\Gamma(-3,-1,-2,-2,-2,-3)$ in Figure \ref{fig:tree+link}. If $\Gamma$ were in Neumann normal form \cite{Neumann}, then we could argue as follows: if we suppose that $\Sigma_2(L)$ is the boundary of some negative definite plumbing $\Gamma'$, then, by repeatedly blowing down $\Gamma'$ we would reach%
\footnote{
\label{fn:Neumann}
The fact that blowing down is the only operation needed in our case follows from closely examining the proof of \cite[Theorem 4.1]{Neumann}.
If we apply the procedure explained there to our case, then we do not need to use Steps 5 and 6, because the shape of our plumbing tree $\Gamma$ (the H-shaped graph in Figure \ref{fig:tree+link}) is different from the possible outcomes in Steps 5-6. Thus, in our case Neumann's algorithm must stop at Step 4 at the latest.
We can then show that Step 4 is not used either, otherwise we would need a $0$-framed vertex of genus $-1$, which, in our case, can be created only in Step 3 (because the algorithm starts with $\Gamma'$, a negative definite tree); however, the $0$-framed vertices of genus $-1$ created in Step 3 are leaves with an adjacent vertex of valence $2$, which is not the case of the input of Step 4.
Thus, in our case the algorithm stops after Steps 1-3: the two operations from these steps that can be performed on trees are $(-1)$-framed blow-downs and $\mathbb{RP}^2$-absorptions, but the latter ones create vertices of genus $-1$ that cannot be cancelled and would also be present in the Neumann normal form $\Gamma$. It follows that the procedure consists of blow-downs only.
}
the Neumann normal form $\Gamma$ for $\Sigma(L)$.
Then, the Neumann normal form $\Gamma$ would still be negative definite, because blowing down preserves negative definiteness, but this would be a contradiction because $\Gamma$ is indefinite, as one can explicitly check.

Unfortunately, this simple line of argument does not work on the nose: $\Gamma$ is not in Neumann normal form, at least according to Neumann's original definition, because part of $\Gamma$ can be absorbed into a plumbing with negative, i.e.\ non-orientable, genus.
A workaround suggested to us by Andr\'as N\'emethi is to take a regular 2-fold cover $Y_2$ of $\Sigma_2(L)$, hoping that its Neumann normal form is indefinite, and then apply the above argument to show that $Y_2$ is not the boundary of a negative definite plumbing; then by a theorem of Stein neither is $\Sigma_2(L)$.
Another possible workaround would be to use a different notion of Neumann normal form, such as in \cite{Pedersen}, but since the literature is not as extensive there we prefer to stick with the double cover reasoning.

We start by constructing a regular double cover of $\Sigma_2(L)$. Recall that a regular double cover of $\Sigma_2(L)$ is determined by an index-2 subgroup of $\pi_1(\Sigma_2(L))$, or equivalently a surjective map $\pi_1(\Sigma_2(L)) \to \Z_2$, which necessarily factors through $H_1(\Sigma_2(L))$. Since the weighted tree $\Gamma = \Gamma(-3,-1,-2,-2,-2,-3)$ in Figure \ref{fig:tree+link} gives a surgery diagram of $\Sigma_2(L)$, if we let $X$ denote the 2-handlebody induced by $\Gamma$ then the intersection matrix $Q_X$ for $X$ presents $H_1(\Sigma_2(L))$. More explicitly, $H_1(\Sigma_2(L))$ is the free abelian group generated by the meridians of the link components in the surgery diagram, modulo the relations given by the columns of $Q_X$.

Let $\mu_a, \mu_b, \ldots, \mu_f$ denote the meridians of the link components in the surgery diagram for $\Sigma_2(L)$. Then with the choice of values for $a, b, \ldots, f$, the map
\[
\begin{aligned}
\tilde\chi \colon \Z^6 &\to \Z_2\\
\mu_a &\mapsto 0\\
\mu_b &\mapsto 0\\
\mu_c &\mapsto 1\\
\mu_d &\mapsto 1\\
\mu_e &\mapsto 0\\
\mu_f &\mapsto 0\\
\end{aligned}
\]
sends all the relations given by the columns of $Q_X$ to $0$. Hence $\tilde\chi$ descends to a map on the quotient $\chi \colon H_1(\Sigma_2(L)) \cong \Z_{16} \to \Z_2$. We define $Y_2$ to be the regular double cover of $\Sigma_2(L)$ that is associated with the map $\chi$.

To find a surgery presentation for $Y_2$, we first make a handleslide so that the character $\chi$ evaluates non-trivially on only one meridian, as shown in Figure \ref{fig:character-slide}. After an isotopy, the new surgery diagram of $\Sigma_2(L)$ looks as in Figure \ref{fig:surgeryS2L}: we denote the red unknot, on whose meridian $\chi$ evaluates non-trivially, by $\gamma$, and the $(-2)$-framed circle it links by $\delta$. 
We note that the computation of $\chi$ on the meridians $\mu_{\gamma}$ and $\mu_\delta$ is done by tracing the change-of-basis induced by the handleslide: $\mu_{c} = \mu_{\gamma}$ and $\mu_{d} = \mu_\delta+\mu_\gamma$, from which we deduce $\mu_{\gamma} = \mu_c$ and $\mu_\delta = \mu_d - \mu_c$.

\begin{figure}
\centering
  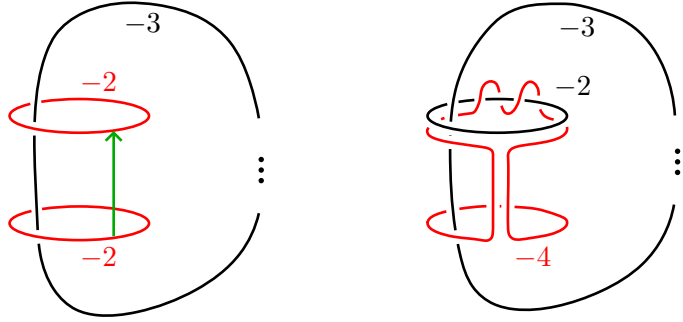%
  \caption{The figure on the left shows a part of the surgery diagram of $\Sigma_2(L) $ determined by the plumbing tree in Figure \ref{fig:tree+link}. The three link components are for the weights $a=-3$, $c=-2$ (bottom), and $d=-2$ (top). We perform a handleslide to get to the diagram on the right.
For both diagrams, the character $\chi$ sends the meridians of the red curves to $1$ and the other meridians to $0$.}
 \label{fig:character-slide}
\end{figure}

\begin{figure}
\centering
  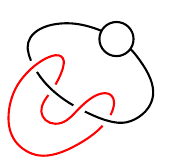  \caption{A portion of the surgery diagram for $\Sigma_2(L)$ after handleslide and isotopy.}
 \label{fig:surgeryS2L}
\end{figure}

Next view $\Sigma_2(L)$ as the union of a solid torus $T$ whose core curve corresponds to $\gamma$ and its complement $E$. Then the double cover $Y_2$ of $\Sigma_2(L) $ is the double covers of these two pieces glued along their common torus boundary. The double cover $\widetilde T$ of $T$ is a solid torus (since its boundary is connected), so we can describe $\widetilde T$ as surgery on a framed curve. The curve is given as follows: lift the $(-4)$-framed curve $\gamma$ to $\widetilde T$; we get a $(-2)$-framed curve, because twice the meridian of $\gamma$ lifts to a single meridian upstairs, and this is the curve we want.

\begin{figure}
\centering
  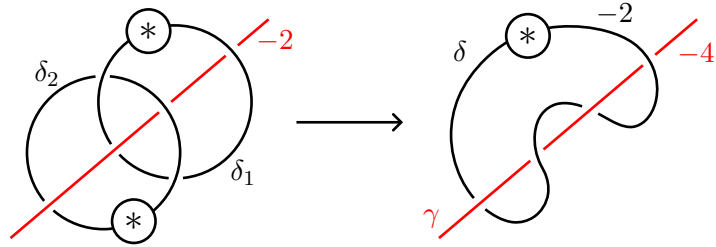%
  \caption{The figure on the left shows a portion of a surgery diagram for the double cover $Y_2$ of $\Sigma_2(L)$, while the figure on the right shows a portion of a surgery diagram for $\Sigma_2(L)$. Note that the curve $\delta$ lifts to two curves $\delta_1$ and $\delta_2$ which are linked.}
 \label{fig:surgeryDC1}
\end{figure}

As for the complement $E$, we construct its double cover $\widetilde{E}$ by unrolling $E$ around the circle $\gamma$: each Dehn surgery performed on a link component that is split from $\gamma$ lifts to two copies of the same Dehn surgery. The curve $\delta$ also lifts to two disjoint curves $\delta_1$ and $\delta_2$, since the linking number of $\delta$ with $\gamma$ is always even (for any choice of orientation), but the curves $\delta_1$ and $\delta_2$ are linked (as in Figure \ref{fig:surgeryDC1}). To find the framings $f_1$ and $f_2$ on $\delta_1$ and $\delta_2$, observe that $\delta$ is an unknot in the surgery diagram in Figure \ref{fig:surgeryDC1}, so it bounds a disc $D$, which intersects $\gamma$ in two points. Note that $D$ defines the $0$-framing for $\delta$, i.e.\ $\lk(\partial D, \partial D^+) = 0$, where the $+$ denotes a pushoff in the positive direction normal to $D$ (for a choice of orientation on $\delta$). The double cover of $D$ branched over the two intersection points with $\gamma$ is an annulus $A$ with boundary $\delta_1 \cup \delta_2$. We'll orient $A$ so that its boundary orientation agrees with the induced orientation from $\delta$.
By taking the pushoff in the positive direction normal to $A$, we obtain framings $f_1$ and $f_2$ for the boundary components $\delta_1$ and $\delta_2$ respectively. (We remark that $A$ exists only in $S^3$ where the surgery diagram is drawn, not in the manifold $Y_2$.) Since $\partial A$ is null-homologous, we have that $\lk(\partial A, \partial A^+) = 0$, which can be expanded as:
\[
\lk(\delta_1, \delta_1^+) + \lk(\delta_2, \delta_2^+) + 2\cdot\lk(\delta_1,\delta_2) = 0.
\]

\begin{figure}
\centering
  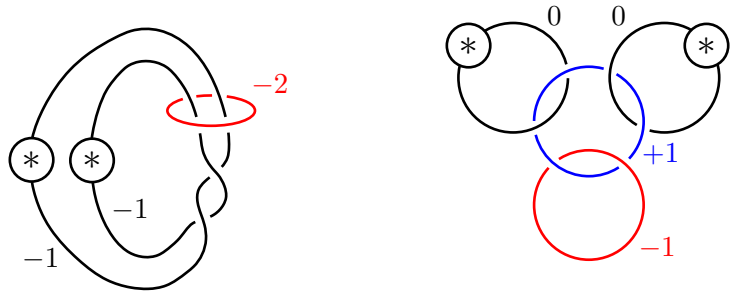%
  \caption{The left figure shows the surgery diagram from Figure \ref{fig:surgeryDC1}, now complete with framings. The right figure shows a diagram obtained by doing a blow-up of the diagram on the left.}
 \label{fig:surgeryDC2}
\end{figure}

\noindent From Figure \ref{fig:surgeryDC1} we get that $\lk(\delta_1, \delta_2)=-1$, and by symmetry we have that $\lk(\delta_1, \delta_1^+) = \lk(\delta_2, \delta_2^+)$. Thus $f_1 = f_2 = +1$, which implies that a framing $k$ on $\delta$ lifts to the framing $k+1$ on $\delta_1$ and on $\delta_2$. It follows that the surgery diagram for $Y_2$ is as in Figure \ref{fig:surgeryDC2} on the left. After blowing up, we can redraw it as in Figure \ref{fig:surgeryDC2} on the right, which is the surgery diagram associated to the weighted graph $\Gamma_2$ in Figure \ref{fig:graphDC} on the left. A sequence of plumbing calculus moves bring it to the weighted graph $\Gamma_2'$ in Figure \ref{fig:graphDC} on the right, which is in Neumann normal form.

\begin{figure}
\centering
  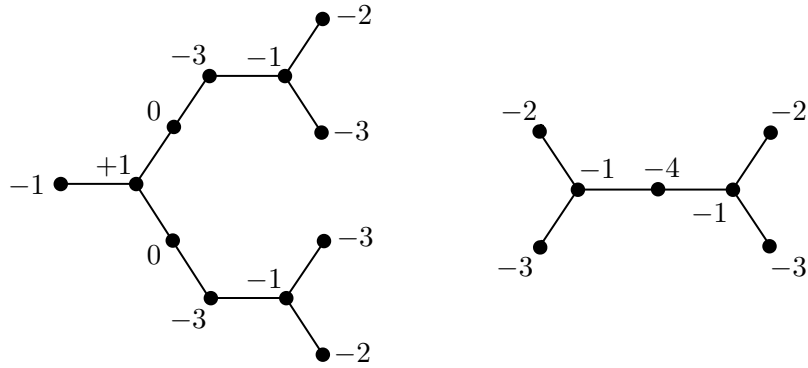%
  \caption{The weighted tree $\Gamma_2$ associated to the surgery diagram in Figure \ref{fig:surgeryDC2} on the right, and a weighted tree $\Gamma_2'$ obtained from it by plumbing calculus.}
 \label{fig:graphDC}
\end{figure}

We are now in a position to show that $\Sigma_2(L)$ does not bound a negative definite plumbing.
If it were, then by a theorem of Stein \cite{Stein} its double cover $Y_2$ would also bound a negative definite plumbing (see also \cite[12.3.3]{Nemethi}). We could then blow it down until we reach its Neumann normal form, which is $\Gamma_2'$, and this must be negative definite because we can follow the same argument as in the footnote on page \pageref{fn:Neumann}. However, $\Gamma_2'$ is indefinite, for example because the determinant of its intersection form is $+8$, so there must be an odd (hence positive) number of positive eigenvalues.

\begin{figure}
\centering
  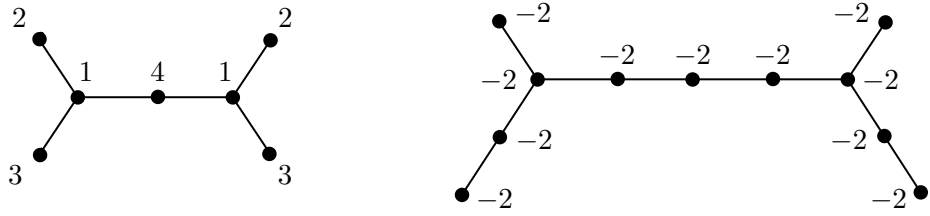%
  \caption{The left hand side shows the weighted tree $-\Gamma_2'$ obtained by changing all the signs in the graph $\Gamma_2'$ from Figure \ref{fig:graphDC}. After performing blow-ups and blow-downs we reach the Neumann normal form on the right hand side.}
 \label{fig:graphDC2}
\end{figure}

Finally, we show that $\Sigma_2(L)$ does not bound a \emph{positive} definite plumbing.
Using Stein's theorem as above, it suffices to show that $-Y_2$ (the double cover with opposite orientation) does not bound a negative definite plumbing. We start from the plumbing graph $-\Gamma_2'$ for $-Y_2$ obtained by changing the signs of all weights of $\Gamma_2'$, and we do blow-ups and blow-downs until we reach the Neumann normal form, which is shown in Figure \ref{fig:graphDC2} on the right. This is indefinite, since the only negative definite, connected weighted trees with all weights equal to $-2$ are of type ADE.
Since the Neumann normal form is indefinite, $-Y_2$ does not bound a negative definite plumbing, by the same argument as before.
\end{proof}

{\small
\newcommand{\Rone}{\mathbb{Z}}
\newcommand{\Rmor}[1]{\mathbb{Z}^{#1}}
\newcommand{\Tone}[1]{\mathbb{Z}_{#1}}
\newcommand{\Tmor}[2]{(\mathbb{Z}_{#1})^{#2}}
\newcommand{\Zero}{$0$}
\begin{center}
\begin{table}
\setlength\extrarowheight{2pt}
\begin{tabular}{|c||c|c|c|c|c|c|c|c|}
\hline
\backslashbox{\!$q$\!}{\!$h$\!} & $-2$ & $-1$ & $0$ & $1$ & $2$ & $3$ & $4$ & $5$ \\
\hline
\hline
$10$  &   &   &   &   &   &   &   & $ \Rone $ \\
\hline
$8$  &   &   &   &   &   &   & $ \Rone $ & $ \Tone{2} $ \\
\hline
$6$  &   &   &   &   &   & $ \Rone $ & $ \Rone \oplus \Tone{2} $ &   \\
\hline
$4$  &   &   &   &   & $ \Rmor{2} $ & $ \Rone \oplus \Tone{2} $ &   &   \\
\hline
$2$  &   &   & $ \Rone $ \cellcolor{pink} & $ \Rone $ & $ \Rone \oplus \Tmor{2}{2} $ &   &   &   \\
\hline
$0$  &   &   & $ \Rmor{3} $ & $ \Rmor{2} \oplus \Tone{2} $ &   &   &   &   \\
\hline
$-2$  &   & $ \Rone $ & $ \Rmor{2} \oplus \Tone{2} $ &   &   &   &   &   \\
\hline
$-4$  &   & $ \Rone \oplus \Tone{2} $ &   &   &   &   &   &   \\
\hline
$-6$  & $ \Rone $ &   &   &   &   &   &   &   \\
\hline
\end{tabular}
	\vspace{10pt}
\caption{The Khovanov homology with $\Z$ coefficients of the link $L$ from Figure \ref{fig:linkwithinvo}.}
\label{tab:Kh-thick}
\end{table}
\end{center}
}

%% file: sections/quasi-orientations.tex


The key to the proof of Theorem \ref{thm:d-invariantinvol} is keeping track of how the involution acts on both quasi-orientations and spin structures. Before giving the proof, we recall quasi-orientations on links $L \subset S^3$ and their relationship to spin structures on the double branched covers $\Sigma_2(L)$, following \cite[Section 2.2]{Turaev}.

Given any link $L$ in $S^3$, a \textit{quasi-orientation} of $L$ is a pair $\{o, -o\}$ consisting of an orientation $o$ of $L$ together with its opposite orientation $-o$. We denote the set of quasi-orientations of $L$ by $QO(L)$; its cardinality is $2^{n-1}$, where $n$ is the number of components of $L$. In \cite[Section 2.2]{Turaev}, Turaev constructed a bijection between $QO(L)$ and the set $\Spin(\Sigma_2(L))$ of spin structures on $\Sigma_2(L)$:
\begin{align*}
\s_{\bullet} \colon QO(L) &\to \Spin(\Sigma_2(L))\\
\omega &\mapsto \s_{\omega}
\end{align*}
We briefly review this below.

Let $p: \Sigma_2(L) \rightarrow S^3$ be the 2-fold branched covering map.
Set $Y := \Sigma_2(L) \setminus p^{-1}(L)$. Let $\mathfrak{s}$ be the canonical spin structure on $Y$, obtained by taking the unique spin structure on $S^3$, restricting it to the complement $S^3 \setminus L$, and then lifting it to $Y$.
Given any orientation $o$ of $L$, we can define an element
\[
h(o) \in H^1(Y; \Z/2\Z) \cong \mathrm{Hom}(H_1(Y, \mathbb{Z}), \Z/2\Z)
\]
by requiring that for any oriented loop $\ell \subset Y$
\[
\langle h(o), [\ell] \rangle \equiv \frac{1}{2} \lk(p(\ell), L_o) \pmod2.
\]
Since $\textrm{Spin}(Y)$ is an affine space over $H^1(Y; \Z/2\Z)$, this gives us a map from the set $O(L)$ of orientations of $L$ to the set $\mathrm{Spin}(Y)$ of spin structures on $Y$: $o \mapsto \mathfrak{s} + h(o)$. 

For every orientation $o$, we have the following from \cite[Section 2.2]{Turaev}:
\begin{itemize}
\item $\mathfrak{s} + h(o)$ extends to a unique spin structure on $\Sigma_2(L)$; 
\item $\mathfrak{s} + h(o)=\mathfrak{s} + h(\widetilde{o})$ if and only if $\widetilde{o} = -o$, i.e., $\widetilde{o}$ is the opposite orientation of $o$;
\item every spin structure on $\Sigma_2(L)$ is $\mathfrak{s} + h(o)$ for some orientation $o$.
\end{itemize}
This tells us that the map $O(L) \rightarrow \mathrm{Spin}(Y)$ descends to a bijection $\s_\bullet: QO(L) \rightarrow \mathrm{Spin}(\Sigma_2(L))$ given by $\omega \mapsto \s_\omega := \mathfrak{s} + h(o)$, where $\omega=\{\pm o\}$.

Next we show that $\s_\bullet$ is natural with respect to diffeomorphisms. Let $L$ and $L'$ be links in $S^3$ with a diffeomorphism $\varphi \colon (S^3, L) \to (S^3, L')$. There is an induced diffeomorphism $\Phi: \Sigma_2(L) \rightarrow \Sigma_2(L')$ that makes the following diagram commute:
\begin{center}
\begin{tikzcd}
\Sigma_2(L) \arrow{d}{p} \arrow{r}{\Phi} & \Sigma_2(L') \arrow{d}{p'} \\
S^3 \arrow{r}{\varphi} & S^3
\end{tikzcd}
\end{center}
Here $p: \Sigma_2(L) \rightarrow S^3$ and $p': \Sigma_2(L') \rightarrow S^3$ are the double branched covering maps from above.

\begin{lemma}
\label{lem:invol1}
Let $L$, $L'$, $p$, $p'$, $\varphi$ and $\Phi$ be as above.
Let $Y =  \Sigma_2(L) \setminus p^{-1}(L)$ with canonical spin structure $\s$, and let $Y' =  \Sigma_2(L') \setminus p^{-1}(L')$ with canonical spin structure $\s'$. Then $\Phi^*(\s') = \s$.
\end{lemma}
\begin{proof}
Let $\s_0$ (resp.\ $\s_0'$) denote the unique spin structure on $S^3$ restricted to $S^3 \setminus L$ (resp.\ $S^3 \setminus L'$).
From the commutative diagram
\begin{center}
\begin{tikzcd}
\mathrm{Spin}(S^3 \setminus L)
& \mathrm{Spin}(S^3 \setminus L')   \arrow[swap]{l}{\varphi^*} \\
\mathrm{Spin}(S^3) \arrow{u}{}
& |[]| \mathrm{Spin}(S^3) \arrow{u}{} \arrow{l}{\varphi^*}
\end{tikzcd}
\end{center}
we get that $\s_0 = \varphi^*(\s_0')$.

Now consider the 2-fold covering maps $p: Y \rightarrow S^3 \setminus L$ and $p': Y' \rightarrow S^3 \setminus L'$. By definition, $\s = p^*(\s_0)$ and $\s'=(p')^*(\s_0')$. 
The commutativity of the diagram
\begin{center}
\begin{tikzcd}
\Spin(Y) & \Spin(Y') \arrow[swap]{l}{\Phi^*} \\
\Spin(S^3 \setminus L) \arrow{u}{p^*} & \Spin(S^3 \setminus L') \arrow{u}{(p')^*} \arrow[swap]{l}{\varphi^*}
\end{tikzcd}
\end{center}
implies that $\Phi^*(\s') = (\Phi^* \circ (p')^*)(\s_0') = (p^* \circ \varphi^*)(\s_0') = p^*(\s_0) = \s$, as required.
\end{proof}

With this in mind, we can prove the following result.

\begin{proposition}
\label{prop:naturality_of_s_omega}
Let $\varphi \colon (S^3, L) \to (S^3, L')$ be a diffeomorphism that preserves the orientation of $S^3$.
Then the following diagram commutes:
\begin{center}
\begin{tikzcd}
QO(L) \arrow{d}{\s_{\bullet}} & QO(L') \arrow{d}{\s_{\bullet}} \arrow[swap]{l}{\varphi^*} \\
\Spin(\Sigma_2(L)) & \Spin(\Sigma_2(L')) \arrow[swap]{l}{\Phi^*}
\end{tikzcd}
\end{center}
where $\varphi^*$ denotes the pullback of quasi-orientations via the map $\varphi$, and $\Phi^*$ denotes the pullback of spin structures through the map $\Phi$.

More explicitly, for every $\omega' \in QO(L')$, we have
\begin{equation}
\label{eq:s_omega_commutes}
\s_{\varphi^*(\omega')} = \Phi^*(\s_{\omega'}).
\end{equation}
\end{proposition}

\begin{proof}
We show that the two sides of Equation \eqref{eq:s_omega_commutes} coincide.
Let
\[
p \colon \Sigma^2(L) \to S^3 \qquad \mbox{and} \qquad p' \colon \Sigma^2(L') \to S^3
\]
denote the projection maps, and let $Y := \Sigma^2(L) \setminus p^{-1}(L)$ and $Y' := \Sigma^2(L') \setminus (p')^{-1}(L')$.

Fix a quasi-orientation $\omega' = \{\pm o'\}$ on $L'$.
We claim that it is enough to prove that the induced map on cohomology
\[
\Phi^* \colon H^1(Y'; \Z/2\Z) \to H^1(Y; \Z/2\Z)
\]
sends $h(o')$ to $h(\varphi^*(o'))$, i.e.\
\begin{equation}
\label{eq:h_o_commutes}
\Phi^*(h(o')) = h(\varphi^*(o')).
\end{equation}
Assuming Equation \eqref{eq:h_o_commutes}, Equation \eqref{eq:s_omega_commutes} (and hence Proposition \ref{prop:naturality_of_s_omega}) follows from
\begin{align*}
\Phi^*(\s_{\omega'}) &= \Phi^*(\s' + h(o'))\\
&= \Phi^*(\s') + \Phi^*(h(o')) \\
&= \s + \Phi^*(h(o')) \\
&= \s + h(\varphi^*(o')) \\
&= \s_{\varphi^*(\omega')}
\end{align*}
where we used the naturality of the $H^1(\cdot;\Z/2\Z)$-action on spin structures in the second equality, Lemma \ref{lem:invol1} in the third equality, and Equation \eqref{eq:h_o_commutes} in the fourth one.

To prove Equation \eqref{eq:h_o_commutes}, we simply evaluate both sides, which sit in
\[
H^1(Y; \Z/2\Z) \cong \mathrm{Hom}(H_1(Y, \Z), \Z/2\Z),
\]
on a generic element of $H_1(Y;\Z)$. Let $\ell$ be a link in $Y$ representing some homology class $[\ell] \in H_1(Y;\Z)$. We compute
\begin{align*}
\langle \Phi^*(h(o')), [\ell] \rangle &= \langle h(o'), [\Phi(\ell)] \rangle \\
&\equiv \frac12 \cdot \lk( (p' \circ \Phi) (\ell), L'_{o'}) \pmod2
\end{align*}
and 
\begin{align*}
\langle h(\varphi^*(o')), [\ell] \rangle &\equiv \frac12 \cdot \lk( p(\ell), L_{\varphi^*(o')}) \pmod2 \\
&\equiv \frac12 \cdot \lk( (\varphi \circ p) (\ell), L'_{o'}) \pmod2,
\end{align*}
where in the second computation we used the fact that the diffeomorphism $\varphi$ of $S^3$ preserves the linking number.
Since $p' \circ \Phi = \varphi \circ p$, the two computations yield the same result, thus showing that Equation \eqref{eq:h_o_commutes} holds.
\end{proof}

%% file: figures/tree+link.pdf_tex
\begingroup%
  \makeatletter%
  \providecommand\color[2][]{%
    \errmessage{(Inkscape) Color is used for the text in Inkscape, but the package 'color.sty' is not loaded}%
    \renewcommand\color[2][]{}%
  }%
  \providecommand\transparent[1]{%
    \errmessage{(Inkscape) Transparency is used (non-zero) for the text in Inkscape, but the package 'transparent.sty' is not loaded}%
    \renewcommand\transparent[1]{}%
  }%
  \providecommand\rotatebox[2]{#2}%
  \newcommand*\fsize{\dimexpr\f@size pt\relax}%
  \newcommand*\lineheight[1]{\fontsize{\fsize}{#1\fsize}\selectfont}%
  \ifx\svgwidth\undefined%
    \setlength{\unitlength}{304.46761484bp}%
    \ifx\svgscale\undefined%
      \relax%
    \else%
      \setlength{\unitlength}{\unitlength * \real{\svgscale}}%
    \fi%
  \else%
    \setlength{\unitlength}{\svgwidth}%
  \fi%
  \global\let\svgwidth\undefined%
  \global\let\svgscale\undefined%
  \makeatother%
  \begin{picture}(1,0.34601292)%
    \lineheight{1}%
    \setlength\tabcolsep{0pt}%
    \put(0,0){\includegraphics[width=\unitlength,page=1]{figures/tree+link.pdf}}%
    \put(0.08004701,0.20335113){\color[rgb]{0,0,0}\makebox(0,0)[lt]{\lineheight{1.25}\smash{\begin{tabular}[t]{l}$a$\end{tabular}}}}%
    \put(0.15139849,0.20078871){\color[rgb]{0,0,0}\makebox(0,0)[lt]{\lineheight{1.25}\smash{\begin{tabular}[t]{l}$b$\end{tabular}}}}%
    \put(0.23471824,0.25263185){\color[rgb]{0,0,0}\makebox(0,0)[lt]{\lineheight{1.25}\smash{\begin{tabular}[t]{l}$e$\end{tabular}}}}%
    \put(-0.00214612,0.25548162){\color[rgb]{0,0,0}\makebox(0,0)[lt]{\lineheight{1.25}\smash{\begin{tabular}[t]{l}$d$\end{tabular}}}}%
    \put(0.23123007,0.10942615){\color[rgb]{0,0,0}\makebox(0,0)[lt]{\lineheight{1.25}\smash{\begin{tabular}[t]{l}$f$\end{tabular}}}}%
    \put(-0.00268489,0.10758299){\color[rgb]{0,0,0}\makebox(0,0)[lt]{\lineheight{1.25}\smash{\begin{tabular}[t]{l}$c$\end{tabular}}}}%
    \put(0.57002146,0.21298534){\color[rgb]{0,0,0}\makebox(0,0)[lt]{\lineheight{1.25}\smash{\begin{tabular}[t]{l}$a$\end{tabular}}}}%
    \put(0.89719697,0.17461078){\color[rgb]{0,0,0}\makebox(0,0)[lt]{\lineheight{1.25}\smash{\begin{tabular}[t]{l}$e$\end{tabular}}}}%
    \put(0.89498995,0.06611232){\color[rgb]{0,0,0}\makebox(0,0)[lt]{\lineheight{1.25}\smash{\begin{tabular}[t]{l}$f$\end{tabular}}}}%
    \put(0.7457747,0.05567539){\color[rgb]{0,0,0}\makebox(0,0)[lt]{\lineheight{1.25}\smash{\begin{tabular}[t]{l}$b$\end{tabular}}}}%
    \put(0.46820516,0.2667518){\color[rgb]{0,0,0}\makebox(0,0)[lt]{\lineheight{1.25}\smash{\begin{tabular}[t]{l}$c$\end{tabular}}}}%
    \put(0.62309269,0.26259368){\color[rgb]{0,0,0}\makebox(0,0)[lt]{\lineheight{1.25}\smash{\begin{tabular}[t]{l}$d$\end{tabular}}}}%
  \end{picture}%
\endgroup%

%% file: figures/character-slide.pdf_tex
\begingroup%
  \makeatletter%
  \providecommand\color[2][]{%
    \errmessage{(Inkscape) Color is used for the text in Inkscape, but the package 'color.sty' is not loaded}%
    \renewcommand\color[2][]{}%
  }%
  \providecommand\transparent[1]{%
    \errmessage{(Inkscape) Transparency is used (non-zero) for the text in Inkscape, but the package 'transparent.sty' is not loaded}%
    \renewcommand\transparent[1]{}%
  }%
  \providecommand\rotatebox[2]{#2}%
  \newcommand*\fsize{\dimexpr\f@size pt\relax}%
  \newcommand*\lineheight[1]{\fontsize{\fsize}{#1\fsize}\selectfont}%
  \ifx\svgwidth\undefined%
    \setlength{\unitlength}{254.55997016bp}%
    \ifx\svgscale\undefined%
      \relax%
    \else%
      \setlength{\unitlength}{\unitlength * \real{\svgscale}}%
    \fi%
  \else%
    \setlength{\unitlength}{\svgwidth}%
  \fi%
  \global\let\svgwidth\undefined%
  \global\let\svgscale\undefined%
  \makeatother%
  \begin{picture}(1,0.46737255)%
    \lineheight{1}%
    \setlength\tabcolsep{0pt}%
    \put(0,0){\includegraphics[width=\unitlength,page=1]{figures/character-slide.pdf}}%
    \put(0.36234038,0.19797415){\color[rgb]{0,0,0}\makebox(0,0)[lt]{\lineheight{1.25}\smash{\begin{tabular}[t]{l}\Huge $\vdots$\end{tabular}}}}%
    \put(0.9786779,0.21013859){\color[rgb]{0,0,0}\makebox(0,0)[lt]{\lineheight{1.25}\smash{\begin{tabular}[t]{l}\Huge $\vdots$\end{tabular}}}}%
    \put(0.10639987,0.33486632){\color[rgb]{0,0,0}\makebox(0,0)[lt]{\lineheight{1.25}\smash{\begin{tabular}[t]{l}\textcolor{red}{$-2$}\end{tabular}}}}%
    \put(0.10639987,0.07559724){\color[rgb]{0,0,0}\makebox(0,0)[lt]{\lineheight{1.25}\smash{\begin{tabular}[t]{l}\textcolor{red}{$-2$}\end{tabular}}}}%
    \put(0.74909289,0.07559716){\color[rgb]{0,0,0}\makebox(0,0)[lt]{\lineheight{1.25}\smash{\begin{tabular}[t]{l}\textcolor{red}{$-4$}\end{tabular}}}}%
    \put(0.80815086,0.3283485){\color[rgb]{0,0,0}\makebox(0,0)[lt]{\lineheight{1.25}\smash{\begin{tabular}[t]{l}$-2$\end{tabular}}}}%
    \put(0.81456318,0.41636924){\color[rgb]{0,0,0}\makebox(0,0)[lt]{\lineheight{1.25}\smash{\begin{tabular}[t]{l}$-3$\end{tabular}}}}%
    \put(0.1715629,0.42177897){\color[rgb]{0,0,0}\makebox(0,0)[lt]{\lineheight{1.25}\smash{\begin{tabular}[t]{l}$-3$\end{tabular}}}}%
  \end{picture}%
\endgroup%

%% file: figures/surgeryS2L.pdf_tex
\begingroup%
  \makeatletter%
  \providecommand\color[2][]{%
    \errmessage{(Inkscape) Color is used for the text in Inkscape, but the package 'color.sty' is not loaded}%
    \renewcommand\color[2][]{}%
  }%
  \providecommand\transparent[1]{%
    \errmessage{(Inkscape) Transparency is used (non-zero) for the text in Inkscape, but the package 'transparent.sty' is not loaded}%
    \renewcommand\transparent[1]{}%
  }%
  \providecommand\rotatebox[2]{#2}%
  \newcommand*\fsize{\dimexpr\f@size pt\relax}%
  \newcommand*\lineheight[1]{\fontsize{\fsize}{#1\fsize}\selectfont}%
  \ifx\svgwidth\undefined%
    \setlength{\unitlength}{80.75256648bp}%
    \ifx\svgscale\undefined%
      \relax%
    \else%
      \setlength{\unitlength}{\unitlength * \real{\svgscale}}%
    \fi%
  \else%
    \setlength{\unitlength}{\svgwidth}%
  \fi%
  \global\let\svgwidth\undefined%
  \global\let\svgscale\undefined%
  \makeatother%
  \begin{picture}(1,0.93319626)%
    \lineheight{1}%
    \setlength\tabcolsep{0pt}%
    \put(0,0){\includegraphics[width=\unitlength,page=1]{figures/surgeryS2L.pdf}}%
    \put(0.51414934,0.02166774){\color[rgb]{0,0,0}\makebox(0,0)[lt]{\lineheight{1.25}\smash{\begin{tabular}[t]{l}\textcolor{red}{$-4$}\end{tabular}}}}%
    \put(-0.00109536,0.45245079){\color[rgb]{0,0,0}\makebox(0,0)[lt]{\lineheight{1.25}\smash{\begin{tabular}[t]{l}\textcolor{red}{$\gamma$}\end{tabular}}}}%
    \put(0.3284009,0.82493344){\color[rgb]{0,0,0}\makebox(0,0)[lt]{\lineheight{1.25}\smash{\begin{tabular}[t]{l}$-2$\end{tabular}}}}%
    \put(0.64713759,0.65911985){\color[rgb]{0,0,0}\makebox(0,0)[lt]{\lineheight{1.25}\smash{\begin{tabular}[t]{l}\Large$*$\end{tabular}}}}%
    \put(0.87482731,0.57114194){\color[rgb]{0,0,0}\makebox(0,0)[lt]{\lineheight{1.25}\smash{\begin{tabular}[t]{l}$\delta$\end{tabular}}}}%
  \end{picture}%
\endgroup%

%% file: figures/surgeryDC1.pdf_tex
\begingroup%
  \makeatletter%
  \providecommand\color[2][]{%
    \errmessage{(Inkscape) Color is used for the text in Inkscape, but the package 'color.sty' is not loaded}%
    \renewcommand\color[2][]{}%
  }%
  \providecommand\transparent[1]{%
    \errmessage{(Inkscape) Transparency is used (non-zero) for the text in Inkscape, but the package 'transparent.sty' is not loaded}%
    \renewcommand\transparent[1]{}%
  }%
  \providecommand\rotatebox[2]{#2}%
  \newcommand*\fsize{\dimexpr\f@size pt\relax}%
  \newcommand*\lineheight[1]{\fontsize{\fsize}{#1\fsize}\selectfont}%
  \ifx\svgwidth\undefined%
    \setlength{\unitlength}{266.30792332bp}%
    \ifx\svgscale\undefined%
      \relax%
    \else%
      \setlength{\unitlength}{\unitlength * \real{\svgscale}}%
    \fi%
  \else%
    \setlength{\unitlength}{\svgwidth}%
  \fi%
  \global\let\svgwidth\undefined%
  \global\let\svgscale\undefined%
  \makeatother%
  \begin{picture}(1,0.3482959)%
    \lineheight{1}%
    \setlength\tabcolsep{0pt}%
    \put(0,0){\includegraphics[width=\unitlength,page=1]{figures/surgeryDC1.pdf}}%
    \put(0.58511623,0.031955){\color[rgb]{0,0,0}\makebox(0,0)[lt]{\lineheight{1.25}\smash{\begin{tabular}[t]{l}\textcolor{red}{$\gamma$}\end{tabular}}}}%
    \put(0.83009474,0.31546747){\color[rgb]{0,0,0}\makebox(0,0)[lt]{\lineheight{1.25}\smash{\begin{tabular}[t]{l}$-2$\end{tabular}}}}%
    \put(0.62762555,0.25930261){\color[rgb]{0,0,0}\makebox(0,0)[lt]{\lineheight{1.25}\smash{\begin{tabular}[t]{l}$\delta$\end{tabular}}}}%
    \put(0.31269449,0.09047485){\color[rgb]{0,0,0}\makebox(0,0)[lt]{\lineheight{1.25}\smash{\begin{tabular}[t]{l}$\delta_1$\end{tabular}}}}%
    \put(0.03500418,0.23068016){\color[rgb]{0,0,0}\makebox(0,0)[lt]{\lineheight{1.25}\smash{\begin{tabular}[t]{l}$\delta_2$\end{tabular}}}}%
    \put(0.94514239,0.25755097){\color[rgb]{0,0,0}\makebox(0,0)[lt]{\lineheight{1.25}\smash{\begin{tabular}[t]{l}\textcolor{red}{$-4$}\end{tabular}}}}%
    \put(0.34931117,0.27648773){\color[rgb]{0,0,0}\makebox(0,0)[lt]{\lineheight{1.25}\smash{\begin{tabular}[t]{l}\textcolor{red}{$-2$}\end{tabular}}}}%
    \put(0.182768,0.28912946){\color[rgb]{0,0,0}\makebox(0,0)[lt]{\lineheight{1.25}\smash{\begin{tabular}[t]{l}\Large$*$\end{tabular}}}}%
    \put(0.16195112,0.0199079){\color[rgb]{0,0,0}\makebox(0,0)[lt]{\lineheight{1.25}\smash{\begin{tabular}[t]{l}\Large$*$\end{tabular}}}}%
    \put(0.71958934,0.28238081){\color[rgb]{0,0,0}\makebox(0,0)[lt]{\lineheight{1.25}\smash{\begin{tabular}[t]{l}\Large$*$\end{tabular}}}}%
  \end{picture}%
\endgroup%

%% file: figures/surgeryDC2.pdf_tex
\begingroup%
  \makeatletter%
  \providecommand\color[2][]{%
    \errmessage{(Inkscape) Color is used for the text in Inkscape, but the package 'color.sty' is not loaded}%
    \renewcommand\color[2][]{}%
  }%
  \providecommand\transparent[1]{%
    \errmessage{(Inkscape) Transparency is used (non-zero) for the text in Inkscape, but the package 'transparent.sty' is not loaded}%
    \renewcommand\transparent[1]{}%
  }%
  \providecommand\rotatebox[2]{#2}%
  \newcommand*\fsize{\dimexpr\f@size pt\relax}%
  \newcommand*\lineheight[1]{\fontsize{\fsize}{#1\fsize}\selectfont}%
  \ifx\svgwidth\undefined%
    \setlength{\unitlength}{271.69384093bp}%
    \ifx\svgscale\undefined%
      \relax%
    \else%
      \setlength{\unitlength}{\unitlength * \real{\svgscale}}%
    \fi%
  \else%
    \setlength{\unitlength}{\svgwidth}%
  \fi%
  \global\let\svgwidth\undefined%
  \global\let\svgscale\undefined%
  \makeatother%
  \begin{picture}(1,0.40077689)%
    \lineheight{1}%
    \setlength\tabcolsep{0pt}%
    \put(0,0){\includegraphics[width=\unitlength,page=1]{figures/surgeryDC2.pdf}}%
    \put(0.33685203,0.27331312){\color[rgb]{0,0,0}\makebox(0,0)[lt]{\lineheight{1.25}\smash{\begin{tabular}[t]{l}\textcolor{red}{$-2$}\end{tabular}}}}%
    \put(0.01891497,0.03184918){\color[rgb]{0,0,0}\makebox(0,0)[lt]{\lineheight{1.25}\smash{\begin{tabular}[t]{l}$-1$\end{tabular}}}}%
    \put(0.14295456,0.10067948){\color[rgb]{0,0,0}\makebox(0,0)[lt]{\lineheight{1.25}\smash{\begin{tabular}[t]{l}$-1$\end{tabular}}}}%
    \put(0.01854256,0.16848541){\color[rgb]{0,0,0}\makebox(0,0)[lt]{\lineheight{1.25}\smash{\begin{tabular}[t]{l}\Large$*$\end{tabular}}}}%
    \put(0.10242944,0.1676941){\color[rgb]{0,0,0}\makebox(0,0)[lt]{\lineheight{1.25}\smash{\begin{tabular}[t]{l}\Large$*$\end{tabular}}}}%
    \put(0.87421911,0.04698824){\color[rgb]{0,0,0}\makebox(0,0)[lt]{\lineheight{1.25}\smash{\begin{tabular}[t]{l}\textcolor{red}{$-1$}\end{tabular}}}}%
    \put(0.62460224,0.3270357){\color[rgb]{0,0,0}\makebox(0,0)[lt]{\lineheight{1.25}\smash{\begin{tabular}[t]{l}\Large$*$\end{tabular}}}}%
    \put(0.95493709,0.32703563){\color[rgb]{0,0,0}\makebox(0,0)[lt]{\lineheight{1.25}\smash{\begin{tabular}[t]{l}\Large$*$\end{tabular}}}}%
    \put(0.87895468,0.17812495){\color[rgb]{0,0,0}\makebox(0,0)[lt]{\lineheight{1.25}\smash{\begin{tabular}[t]{l}\textcolor{blue}{$+1$}\end{tabular}}}}%
    \put(0.7473402,0.36859926){\color[rgb]{0,0,0}\makebox(0,0)[lt]{\lineheight{1.25}\smash{\begin{tabular}[t]{l}$0$\end{tabular}}}}%
    \put(0.83621447,0.36859926){\color[rgb]{0,0,0}\makebox(0,0)[lt]{\lineheight{1.25}\smash{\begin{tabular}[t]{l}$0$\end{tabular}}}}%
  \end{picture}%
\endgroup%

%% file: figures/graphDC.pdf_tex
\begingroup%
  \makeatletter%
  \providecommand\color[2][]{%
    \errmessage{(Inkscape) Color is used for the text in Inkscape, but the package 'color.sty' is not loaded}%
    \renewcommand\color[2][]{}%
  }%
  \providecommand\transparent[1]{%
    \errmessage{(Inkscape) Transparency is used (non-zero) for the text in Inkscape, but the package 'transparent.sty' is not loaded}%
    \renewcommand\transparent[1]{}%
  }%
  \providecommand\rotatebox[2]{#2}%
  \newcommand*\fsize{\dimexpr\f@size pt\relax}%
  \newcommand*\lineheight[1]{\fontsize{\fsize}{#1\fsize}\selectfont}%
  \ifx\svgwidth\undefined%
    \setlength{\unitlength}{311.13975465bp}%
    \ifx\svgscale\undefined%
      \relax%
    \else%
      \setlength{\unitlength}{\unitlength * \real{\svgscale}}%
    \fi%
  \else%
    \setlength{\unitlength}{\svgwidth}%
  \fi%
  \global\let\svgwidth\undefined%
  \global\let\svgscale\undefined%
  \makeatother%
  \begin{picture}(1,0.44786167)%
    \lineheight{1}%
    \setlength\tabcolsep{0pt}%
    \put(0,0){\includegraphics[width=\unitlength,page=1]{figures/graphDC.pdf}}%
    \put(-0.00262731,0.21045583){\color[rgb]{0,0,0}\makebox(0,0)[lt]{\lineheight{1.25}\smash{\begin{tabular}[t]{l}$-1$\end{tabular}}}}%
    \put(0.10323358,0.23491133){\color[rgb]{0,0,0}\makebox(0,0)[lt]{\lineheight{1.25}\smash{\begin{tabular}[t]{l}$+1$\end{tabular}}}}%
    \put(0.16548258,0.30010574){\color[rgb]{0,0,0}\makebox(0,0)[lt]{\lineheight{1.25}\smash{\begin{tabular}[t]{l}$0$\end{tabular}}}}%
    \put(0.16548447,0.12622402){\color[rgb]{0,0,0}\makebox(0,0)[lt]{\lineheight{1.25}\smash{\begin{tabular}[t]{l}$0$\end{tabular}}}}%
    \put(0.19322601,0.36828322){\color[rgb]{0,0,0}\makebox(0,0)[lt]{\lineheight{1.25}\smash{\begin{tabular}[t]{l}$-3$\end{tabular}}}}%
    \put(0.28500543,0.36577576){\color[rgb]{0,0,0}\makebox(0,0)[lt]{\lineheight{1.25}\smash{\begin{tabular}[t]{l}$-1$\end{tabular}}}}%
    \put(0.39387701,0.41650715){\color[rgb]{0,0,0}\makebox(0,0)[lt]{\lineheight{1.25}\smash{\begin{tabular}[t]{l}$-2$\end{tabular}}}}%
    \put(0.39046367,0.2763724){\color[rgb]{0,0,0}\makebox(0,0)[lt]{\lineheight{1.25}\smash{\begin{tabular}[t]{l}$-3$\end{tabular}}}}%
    \put(0.19322716,0.04840196){\color[rgb]{0,0,0}\makebox(0,0)[lt]{\lineheight{1.25}\smash{\begin{tabular}[t]{l}$-3$\end{tabular}}}}%
    \put(0.28500228,0.09658538){\color[rgb]{0,0,0}\makebox(0,0)[lt]{\lineheight{1.25}\smash{\begin{tabular}[t]{l}$-1$\end{tabular}}}}%
    \put(0.39546107,0.14731691){\color[rgb]{0,0,0}\makebox(0,0)[lt]{\lineheight{1.25}\smash{\begin{tabular}[t]{l}$-3$\end{tabular}}}}%
    \put(0.39204773,0.00718209){\color[rgb]{0,0,0}\makebox(0,0)[lt]{\lineheight{1.25}\smash{\begin{tabular}[t]{l}$-2$\end{tabular}}}}%
    \put(0.68864734,0.2305069){\color[rgb]{0,0,0}\makebox(0,0)[lt]{\lineheight{1.25}\smash{\begin{tabular}[t]{l}$-1$\end{tabular}}}}%
    \put(0.76644485,0.23110639){\color[rgb]{0,0,0}\makebox(0,0)[lt]{\lineheight{1.25}\smash{\begin{tabular}[t]{l}$-4$\end{tabular}}}}%
    \put(0.82475287,0.17724562){\color[rgb]{0,0,0}\makebox(0,0)[lt]{\lineheight{1.25}\smash{\begin{tabular}[t]{l}$-1$\end{tabular}}}}%
    \put(0.9204349,0.30080727){\color[rgb]{0,0,0}\makebox(0,0)[lt]{\lineheight{1.25}\smash{\begin{tabular}[t]{l}$-2$\end{tabular}}}}%
    \put(0.59375375,0.30080358){\color[rgb]{0,0,0}\makebox(0,0)[lt]{\lineheight{1.25}\smash{\begin{tabular}[t]{l}$-2$\end{tabular}}}}%
    \put(0.92019366,0.11148147){\color[rgb]{0,0,0}\makebox(0,0)[lt]{\lineheight{1.25}\smash{\begin{tabular}[t]{l}$-3$\end{tabular}}}}%
    \put(0.58865678,0.11147846){\color[rgb]{0,0,0}\makebox(0,0)[lt]{\lineheight{1.25}\smash{\begin{tabular}[t]{l}$-3$\end{tabular}}}}%
  \end{picture}%
\endgroup%

%% file: figures/graphDC2.pdf_tex
\begingroup%
  \makeatletter%
  \providecommand\color[2][]{%
    \errmessage{(Inkscape) Color is used for the text in Inkscape, but the package 'color.sty' is not loaded}%
    \renewcommand\color[2][]{}%
  }%
  \providecommand\transparent[1]{%
    \errmessage{(Inkscape) Transparency is used (non-zero) for the text in Inkscape, but the package 'transparent.sty' is not loaded}%
    \renewcommand\transparent[1]{}%
  }%
  \providecommand\rotatebox[2]{#2}%
  \newcommand*\fsize{\dimexpr\f@size pt\relax}%
  \newcommand*\lineheight[1]{\fontsize{\fsize}{#1\fsize}\selectfont}%
  \ifx\svgwidth\undefined%
    \setlength{\unitlength}{348.77381092bp}%
    \ifx\svgscale\undefined%
      \relax%
    \else%
      \setlength{\unitlength}{\unitlength * \real{\svgscale}}%
    \fi%
  \else%
    \setlength{\unitlength}{\svgwidth}%
  \fi%
  \global\let\svgwidth\undefined%
  \global\let\svgscale\undefined%
  \makeatother%
  \begin{picture}(1,0.23546845)%
    \lineheight{1}%
    \setlength\tabcolsep{0pt}%
    \put(0,0){\includegraphics[width=\unitlength,page=1]{figures/graphDC2.pdf}}%
    \put(0.07395505,0.14367994){\color[rgb]{0,0,0}\makebox(0,0)[lt]{\lineheight{1.25}\smash{\begin{tabular}[t]{l}$1$\end{tabular}}}}%
    \put(0.28933338,0.20209389){\color[rgb]{0,0,0}\makebox(0,0)[lt]{\lineheight{1.25}\smash{\begin{tabular}[t]{l}$2$\end{tabular}}}}%
    \put(0.00220319,0.20209066){\color[rgb]{0,0,0}\makebox(0,0)[lt]{\lineheight{1.25}\smash{\begin{tabular}[t]{l}$2$\end{tabular}}}}%
    \put(0.28911818,0.03319715){\color[rgb]{0,0,0}\makebox(0,0)[lt]{\lineheight{1.25}\smash{\begin{tabular}[t]{l}$3$\end{tabular}}}}%
    \put(-0.0023438,0.03319446){\color[rgb]{0,0,0}\makebox(0,0)[lt]{\lineheight{1.25}\smash{\begin{tabular}[t]{l}$3$\end{tabular}}}}%
    \put(0.54466659,0.20749728){\color[rgb]{0,0,0}\makebox(0,0)[lt]{\lineheight{1.25}\smash{\begin{tabular}[t]{l}$-2$\end{tabular}}}}%
    \put(0.50754563,0.13542914){\color[rgb]{0,0,0}\makebox(0,0)[lt]{\lineheight{1.25}\smash{\begin{tabular}[t]{l}$-2$\end{tabular}}}}%
    \put(0.54625252,0.07091813){\color[rgb]{0,0,0}\makebox(0,0)[lt]{\lineheight{1.25}\smash{\begin{tabular}[t]{l}$-2$\end{tabular}}}}%
    \put(0.50324487,0.00640711){\color[rgb]{0,0,0}\makebox(0,0)[lt]{\lineheight{1.25}\smash{\begin{tabular}[t]{l}$-2$\end{tabular}}}}%
    \put(0.88601271,0.07091813){\color[rgb]{0,0,0}\makebox(0,0)[lt]{\lineheight{1.25}\smash{\begin{tabular}[t]{l}$-2$\end{tabular}}}}%
    \put(0.92902036,0.00640711){\color[rgb]{0,0,0}\makebox(0,0)[lt]{\lineheight{1.25}\smash{\begin{tabular}[t]{l}$-2$\end{tabular}}}}%
    \put(0.92041883,0.13542914){\color[rgb]{0,0,0}\makebox(0,0)[lt]{\lineheight{1.25}\smash{\begin{tabular}[t]{l}$-2$\end{tabular}}}}%
    \put(0.63556355,0.16269561){\color[rgb]{0,0,0}\makebox(0,0)[lt]{\lineheight{1.25}\smash{\begin{tabular}[t]{l}$-2$\end{tabular}}}}%
    \put(0.71727821,0.16269561){\color[rgb]{0,0,0}\makebox(0,0)[lt]{\lineheight{1.25}\smash{\begin{tabular}[t]{l}$-2$\end{tabular}}}}%
    \put(0.80329346,0.16269561){\color[rgb]{0,0,0}\makebox(0,0)[lt]{\lineheight{1.25}\smash{\begin{tabular}[t]{l}$-2$\end{tabular}}}}%
    \put(0.15136934,0.14367994){\color[rgb]{0,0,0}\makebox(0,0)[lt]{\lineheight{1.25}\smash{\begin{tabular}[t]{l}$4$\end{tabular}}}}%
    \put(0.22448284,0.14367994){\color[rgb]{0,0,0}\makebox(0,0)[lt]{\lineheight{1.25}\smash{\begin{tabular}[t]{l}$1$\end{tabular}}}}%
    \put(0.88872772,0.20749728){\color[rgb]{0,0,0}\makebox(0,0)[lt]{\lineheight{1.25}\smash{\begin{tabular}[t]{l}$-2$\end{tabular}}}}%
  \end{picture}%
\endgroup%

%% file: main.bbl
\begin{thebibliography}{HRRW20}

\bibitem[CDGW]{SnapPy}
Marc Culler, Nathan~M. Dunfield, Matthias Goerner, and Jeffrey~R. Weeks.
\newblock Snap{P}y, a computer program for studying the geometry and topology
  of $3$-manifolds.
\newblock Available at \url{http://snappy.computop.org} (26/01/2024).

\bibitem[DO12]{DO}
Andrew Donald and Brendan Owens.
\newblock Concordance groups of links.
\newblock {\em Algebr. Geom. Topol.}, 12(4):2069--2093, 2012.

\bibitem[Gab86]{G:arborescent}
David Gabai.
\newblock {\em Genera of Arborescent Links: 1986}, volume 339.
\newblock American Mathematical Soc., 1986.

\bibitem[GS99]{Gompf-Stipsicz}
Robert~E. Gompf and Andr\'{a}s~I. Stipsicz.
\newblock {\em {$4$}-manifolds and {K}irby calculus}, volume~20 of {\em
  Graduate Studies in Mathematics}, pages xvi+558.
\newblock American Mathematical Society, Providence, RI, 1999.

\bibitem[Han16]{Hanselman:graph-mfds}
Jonathan Hanselman.
\newblock Bordered {H}eegaard {F}loer homology and graph manifolds.
\newblock {\em Algebr. Geom. Topol.}, 16(6):3103--3166, 2016.

\bibitem[HRRW20]{HRRW}
Jonathan Hanselman, Jacob Rasmussen, Sarah~Dean Rasmussen, and Liam Watson.
\newblock L-spaces, taut foliations, and graph manifolds.
\newblock {\em Compos. Math.}, 156(3):604--612, 2020.

\bibitem[HRW]{HRW}
Jonathan Hanselman, Jacob Rasmussen, and Liam Watson.
\newblock Bordered {F}loer homology for manifolds with torus boundary via
  immersed curves.
\newblock Preprint, arXiv:1604.03466.

\bibitem[HRW22]{HRW-companion}
Jonathan Hanselman, Jacob Rasmussen, and Liam Watson.
\newblock Heegaard {F}loer homology for manifolds with torus boundary:
  properties and examples.
\newblock {\em Proc. Lond. Math. Soc. (3)}, 125(4):879--967, 2022.

\bibitem[HW23a]{HW:cabling}
Jonathan Hanselman and Liam Watson.
\newblock Cabling in terms of immersed curves.
\newblock {\em Geom. Topol.}, 27(3):925--952, 2023.

\bibitem[HW23b]{HW:loop-calculus}
Jonathan Hanselman and Liam Watson.
\newblock A calculus for bordered {F}loer homology, 2023.

\bibitem[LO15]{LiscaOwensQuasi}
Paolo Lisca and Brendan Owens.
\newblock Signatures, {H}eegaard {F}loer correction terms and quasi-alternating
  links.
\newblock {\em Proc. Amer. Math. Soc.}, 143(2):907--914, 2015.

\bibitem[LOT18a]{LOTQGradings}
Robert Lipshitz, Peter Ozsv\'{a}th, and Dylan~P. Thurston.
\newblock Relative {$\Bbb{Q}$}-gradings from bordered {F}loer theory.
\newblock {\em Michigan Math. J.}, 67(4):827--838, 2018.

\bibitem[LOT18b]{LOT}
Robert Lipshitz, Peter~S. Ozsvath, and Dylan~P. Thurston.
\newblock Bordered {H}eegaard {F}loer homology.
\newblock {\em Mem. Amer. Math. Soc.}, 254(1216):viii+279, 2018.

\bibitem[LRS18]{LRS}
Jianfeng Lin, Daniel Ruberman, and Nikolai Saveliev.
\newblock On the {F}r{\o}yshov invariant and monopole {L}efschetz number.
\newblock arXiv:1802.07704, 2018.

\bibitem[MO07]{ManoOwens}
Ciprian Manolescu and Brendan Owens.
\newblock A concordance invariant from the {F}loer homology of double branched
  covers.
\newblock {\em Int. Math. Res. Not. IMRN}, (20):Art. ID rnm077, 21, 2007.

\bibitem[MO08]{ManoOzsvath}
Ciprian Manolescu and Peter Ozsv\'{a}th.
\newblock On the {K}hovanov and knot {F}loer homologies of quasi-alternating
  links.
\newblock In {\em Proceedings of {G}\"{o}kova {G}eometry-{T}opology
  {C}onference 2007}, pages 60--81. G\"{o}kova Geometry/Topology Conference
  (GGT), G\"{o}kova, 2008.

\bibitem[N{\'e}m22]{Nemethi}
Andr{\'a}s N{\'e}methi.
\newblock {\em Normal surface singularities}, volume~74.
\newblock Springer Nature, 2022.

\bibitem[Neu80]{Neumann-invt}
Walter~D. Neumann.
\newblock An invariant of plumbed homology spheres.
\newblock In {\em {Topology {S}ymposium, {S}iegen 1979 ({P}roc. {S}ympos.,
  {U}niv. {S}iegen, {S}iegen, 1979)}}, volume 788 of {\em Lecture Notes in
  Math.}, pages 125--144. Springer, Berlin, 1980.

\bibitem[Neu81]{Neumann}
Walter~D. Neumann.
\newblock A calculus for plumbing applied to the topology of complex surface
  singularities and degenerating complex curves.
\newblock {\em Trans. Amer. Math. Soc.}, 268(2):299--344, 1981.

\bibitem[OS03a]{OSd}
Peter Ozsv\'{a}th and Zolt\'{a}n Szab\'{o}.
\newblock Absolutely graded {F}loer homologies and intersection forms for
  four-manifolds with boundary.
\newblock {\em Adv. Math.}, 173(2):179--261, 2003.

\bibitem[OS03b]{OS:tau}
Peter Ozsv{\'a}th and Zolt{\'a}n Szab{\'o}.
\newblock Knot {F}loer homology and the four-ball genus.
\newblock {\em Geometry \& Topology}, 7(2):615--639, 2003.

\bibitem[Ped09]{Pedersen}
Helge~Moller Pedersen.
\newblock {\em Splice diagrams. {S}ingularity links and universal abelian
  covers}.
\newblock Columbia University, 2009.

\bibitem[Ras10]{Rasmussen:s}
Jacob Rasmussen.
\newblock Khovanov homology and the slice genus.
\newblock {\em Inventiones mathematicae}, 182(2):419--447, 2010.

\bibitem[RR17]{RR}
Jacob Rasmussen and Sarah~Dean Rasmussen.
\newblock Floer simple manifolds and {L}-space intervals.
\newblock {\em Adv. Math.}, 322:738--805, 2017.

\bibitem[Sav00]{Saveliev}
Nikolai Saveliev.
\newblock A surgery formula for the {$\overline\mu$}-invariant.
\newblock {\em Topology Appl.}, 106(1):91--102, 2000.

\bibitem[Sch24]{KnotJob}
Dirk Sch{\"u}tz.
\newblock Knotjob software (grey version).
\newblock Available at
  \url{https://www.maths.dur.ac.uk/users/dirk.schuetz/knotjob.html}, 2024.

\bibitem[Sie80]{Siebenmann}
L.~Siebenmann.
\newblock On vanishing of the {R}ohlin invariant and nonfinitely amphicheiral
  homology {$3$}-spheres.
\newblock In {\em Topology {S}ymposium, {S}iegen 1979 ({P}roc. {S}ympos.,
  {U}niv. {S}iegen, {S}iegen, 1979)}, volume 788 of {\em Lecture Notes in
  Math.}, pages 172--222. Springer, Berlin, 1980.

\bibitem[Ste56]{Stein}
Karl Stein.
\newblock Analytische {Z}erlegungen komplexer {R}{\"a}ume.
\newblock {\em Mathematische Annalen}, 132(1):63--93, 1956.

\bibitem[Sti08]{Stipsicz}
Andr\'{a}s~I. Stipsicz.
\newblock On the {$\overline\mu$}-invariant of rational surface singularities.
\newblock {\em Proc. Amer. Math. Soc.}, 136(11):3815--3823, 2008.

\bibitem[Tur88]{Turaev}
V.~G. Turaev.
\newblock Classification of oriented {M}ontesinos links via spin structures.
\newblock In {\em Topology and geometry---{R}ohlin {S}eminar}, volume 1346 of
  {\em Lecture Notes in Math.}, pages 271--289. Springer, Berlin, 1988.

\end{thebibliography}
